\documentclass{amsart}
\usepackage{amsfonts,amssymb, wasysym}
\usepackage{multicol}

\newtheorem{theorem}{Theorem}
\newtheorem{lemma}{Lemma}
\newtheorem{corollary}{Corollary}

\newcommand{\integers}{{\mathbb Z}}
\newcommand{\realnos}{{\mathbb R}}
\newcommand{\rationalnos}{{\mathbb Q}}

\def\diag{\mathrm{diag}}
\def\span{\mathrm{Span}}
\def\ov{\overline}

\def\Beta{{\rm B}}
\def\Nu{{\rm N}}
\def\Mu{{\rm M}}
\def\Tau{{\rm T}}
\def\Eta{{\rm H}}

\def\Kappa{{\rm K}}

\def\Rho{{\rm P}}

\begin{document}

\title{Fibered Orbifolds and Crystallographic Groups, II }

\author{John G. Ratcliffe and Steven T. Tschantz}

\address{Department of Mathematics, Vanderbilt University, Nashville, TN 37240
\vspace{.1in}}

\email{j.g.ratcliffe@vanderbilt.edu}

\date{}

\begin{abstract}
Let $\Gamma$ be an $n$-dimensional crystallographic group ($n$-space group). 
If $\Gamma$ is a $\integers$-reducible, 
then the flat $n$-orbifold $E^n/\Gamma$ has a nontrivial fibered orbifold structure. 
We prove that this structure can be described by a generalized Calabi construction, 
that is, $E^n/\Gamma$ is represented as the quotient of 
the Cartesian product of two flat orbifolds under the diagonal action of a structure group of isometries. 
We determine the structure group and prove that it is finite if and only if the fibered orbifold structure
has an orthogonally dual fibered orbifold structure. 

A geometric fibration of $E^n/\Gamma$ 
corresponds to a space group extension $1 \to \Nu \to \Gamma \to \Gamma/\Nu \to 1$. 
We give a criterion for the splitting of a  
space group extension in terms of the structure group action 
that is strong enough to detect the splitting 
of all the space group extensions corresponding to the 
standard Seifert fibrations of a compact, connected, flat 3-orbifold. 

If $\Gamma$ is an arbitrary $n$-space group, 
we prove that the group $\mathrm{Isom}(E^n/\Gamma)$ of isometries of $E^n/\Gamma$ is a compact Lie group 
whose component of the identity is a torus of dimension equal to the first Betti 
number of $\Gamma$.  This implies that $\mathrm{Isom}(E^n/\Gamma)$ is finite 
if and only if $\Gamma/[\Gamma, \Gamma]$ is finite. 

We describe how to classify all the geometric fibrations 
of compact, connected, flat $n$-orbifolds, over a 1-orbifold, up to affine equivalence. 
We apply our classification theory to the scientifically important case $n = 3$,  
and classify all the geometric fibrations 
of compact, connected, flat $3$-orbifolds, over a 1-orbifold, up to affine equivalence. 

We end the paper with three applications of the results of our paper. 
The first application is to give nice, explicit, geometric descriptions 
of all the fibered, compact, connected, flat 3-orbifolds. 
The second application is to give an explanation of the enantiomorphic 3-space 
group pairs from the point of view of our classification theory. 
The third application is to apply our splitting criteria to 
the space group extensions corresponding to a geometric fibration 
of a compact, connected, flat 3-orbifold over a 1-orbifold.

\end{abstract}

\maketitle

\section{Introduction} 
An {\it $n$-dimensional crystallographic group} ({\it $n$-space group}) 
is a discrete group $\Gamma$ of isometries of Euclidean $n$-space $E^n$ 
whose orbit space $E^n/\Gamma$ is compact. 
The 3-space groups are the symmetry groups of crystalline structures,  
and so are of fundamental importance in the science of crystallography. 

If an $n$-space group $\Gamma$ is torsion-free, then $E^n/\Gamma$ 
is a compact, connected, flat $n$-manifold, 
and conversely every compact, connected, flat $n$-manifold 
is isometric to $E^n/\Gamma$ for some torsion-free $n$-space group $\Gamma$. 
All torsion-free $n$-space groups, for $n= 2,3,4$ are $\integers$-reducible, 
and so the theory of this paper applies to these groups. 
Compact, connected, flat manifolds occur naturally in the theory of hyperbolic 
manifolds as horispherical cross-sections of cusps of hyperbolic manifolds 
of finite volume. Understanding the geometry of cusps of hyperbolic manifolds 
and orbifolds was our original motivation for studying space groups.  

In our previous paper \cite{R-T}, we proved that the fibered orbifold structures 
on $E^n/\Gamma$ are in one-to-one correspondence 
with the complete normal subgroups of $\Gamma$.  
A normal subgroup $\mathrm{N}$ 
of $\Gamma$ is complete precisely when $\Gamma/{\mathrm N}$ is also a space group. 
As an application, we showed in \cite{R-T} that all the Seifert fibered orbifold structures 
for the reducible 3-space groups described by Conway et al.\ \cite{C-T} have 
an orthogonally dual fibered orbifold structure, 
which we called a co-Seifert fibration.  
The fibers of a Seifert fibration are 1-dimensional, whereas the base of a co-Seifert 
fibration is 1-dimensional. 

In this paper, we prove that a fibered orbifold structure on $E^n/\Gamma$ 
can be described by a generalized Calabi construction, 
that is, $E^n/\Gamma$ is represented as the quotient of 
the Cartesian product of two flat orbifolds under the diagonal action of a structure group of isometries. 
We determine the structure group and prove that it is finite if and only if the fibered orbifold structure
has an orthogonally dual fibered orbifold structure. 
As an application, we show that our generalized Calabi construction simultaneously 
describes, in a geometrically intrinsic way, the dual Seifert and co-Seifert 
fibered orbifold structures for a reducible 3-space group. 
Our generalized Calabi construction fully explains the relationship between 
the dual Seifert and co-Seifert fibered orbifold structures for a reducible 3-space group. 

We also give a geometric condition for the splitting of a
space group extension $1 \to \Nu \to \Gamma \to \Gamma/\Nu$ 
that improves the splitting condition given 
in Theorem 18 of \cite{R-T}.  
Our new splitting condition is strong enough to detect the splitting 
of all the space group extensions corresponding to the 
Seifert fibered orbifold structures for a reducible 3-space group
described by Conway et al.\ \cite{C-T}. 

If $\Gamma$ is an arbitrary $n$-space group, 
we prove that $\mathrm{Isom}(E^n/\Gamma)$ is a compact Lie group 
whose component of the identity is a torus of dimension equal to the first Betti 
number of $\Gamma$.  In particular, $\mathrm{Isom}(E^n/\Gamma)$ is finite 
if and only if $\Gamma/[\Gamma, \Gamma]$ is finite. 

We describe how to classify all the co-Seifert fibrations 
of compact, connected, flat $n$-orbifolds up to affine equivalence. 
We apply our classification theory to the scientifically important case $n = 3$,  
and classify all the co-Seifert fibrations 
of compact, connected, flat $3$-orbifolds up to affine equivalence. 
In the process, we give detailed descriptions of the group of affinities of all 
the compact, connected, flat 2-orbifolds. 
In particular, we show that the group of affinities of a flat pillow or a flat torus 
has a natural free product with amalgamation structure. 

We end the paper with three applications of the results of our paper. 
The first application is to give nice, explicit, geometric descriptions 
of all the fibered, compact, connected, flat 3-orbifolds. 
The second application is to give an explanation of the enantiomorphic 3-space 
group pairs from the point of view of our classification theory. 
The third application is to apply our splitting criteria to 
the space group extensions corresponding to a co-Seifert fibration 
of a compact, connected, flat 3-orbifold. 

\section{Geometric Fibrations of Flat Orbifolds}  

In this section, we review some of the definitions 
and results from our previous paper \cite{R-T}, which will facilitate reading this paper. 
A map $\phi:E^n\to E^n$ is an isometry of $E^n$ 
if and only if there is an $a\in E^n$ and an $A\in {\rm O}(n)$ such that 
$\phi(x) = a + Ax$ for each $x$ in $E^n$. 
We shall write $\phi = a+ A$. 
In particular, every translation $\tau = a + I$ is an isometry of $E^n$. 

Let $\Gamma$ be an $n$-space group. 
Define $\eta:\Gamma \to {\rm O}(n)$ by $\eta(a+A) = A$. 
Then $\eta$ is a homomorphism whose kernel is the group $\mathrm{T}$ 
of translations in $\Gamma$. 
The image of $\eta$ is a finite group $\Pi$ called the 
{\it point group} of $\Gamma$.

Let $\Eta$ be a subgroup of an $n$-space group $\Gamma$. 
Define the {\it span} of $\Eta$ by the formula
$${\rm Span}(\Eta) = {\rm Span}\{a\in E^n:a+I\in \Eta\}.$$
Note that ${\rm Span}(\Eta)$ is a vector subspace $V$ of $E^n$.  
Let $V^\perp$ denote the orthogonal complement of $V$ in $E^n$. 
The following theorem is Theorem 2 in  \cite{R-T}. 

\begin{theorem} 
Let ${\rm N}$ be a normal subgroup of an $n$-space group $\Gamma$, 
and let $V = {\rm Span}(\Nu)$. 
\begin{enumerate}
\item If $b+B\in\Gamma$, then $BV=V$. 
\item If $a+A\in \Nu$,  then $a\in V$ and $ V^\perp\subseteq{\rm Fix}(A)$. 
\item The group $\Nu$ acts effectively on each coset $V+x$ of $V$ in $E^n$ 
as a space group of isometries of $V+x$. 
\end{enumerate}
\end{theorem}

Let $\Gamma$ be an $n$-space group. 
The {\it dimension} of $\Gamma$ is $n$. 
If $\Nu$ is a normal subgroup of $\Gamma$, 
then $\Nu$ is a $m$-space group with $m= \mathrm{dim}(\mathrm{Span}(\Nu))$ 
by Theorem 1(3). 

\vspace{.15in}
\noindent{\bf Definition:}
Let $\Nu$ be a normal subgroup $\Nu$ of an $n$-space group $\Gamma$, and let $V = {\rm Span}(\Nu)$.  Then $\Nu$ is said to be a {\it complete normal subgroup} of $\Gamma$ if 
$$\Nu= \{a+A\in \Gamma: a\in V\ \hbox{and}\ V^\perp\subseteq{\rm Fix}(A)\}.$$
\noindent{\bf Remark 1.}
If $\Nu$ is a normal subgroup of an $n$-space group $\Gamma$, 
then $\Nu$ is contained in a complete normal subgroup $\overline{\Nu}$ 
of $\Gamma$ such that $\Nu$ has finite index in $\overline{\Nu}$, 
moreover $\overline{\Nu}$ is the unique maximal element of the 
commensurability class of normal subgroups of $\Gamma$ containing $\Nu$. 
The group $\overline{\Nu}$ is called the {\it completion} of $\Nu$ in $\Gamma$.  
The following lemma is Lemma 1 in \cite{R-T}. 

\begin{lemma} 
Let $\Nu$ be a complete normal subgroup of an $n$-space group $\Gamma$, 
and let $V={\rm Span}(\Nu)$. 
Then $\Gamma/\Nu$ acts effectively as a space group of isometries of $E^n/V$ 
by the formula
$({\rm N}(b+B))(V+x) = V+ b+Bx.$
\end{lemma}

\noindent{\bf Remark 2.} A normal subgroup $\Nu$ of a space group $\Gamma$ is complete 
precisely when $\Gamma/\Nu$ is a space group by Theorem 5 of \cite{R-T}.

\vspace{.15in}

A {\it flat $n$-orbifold} is a $(E^n,{\rm Isom}(E^n))$-orbifold 
as defined in \S 13.2 of Ratcliffe \cite{R}. 
A connected flat $n$-orbifold has a natural inner metric space structure. 
If $\Gamma$ is a discrete group of isometries of $E^n$, 
then its orbit space $E^n/\Gamma = \{\Gamma x: x\in E^n\}$ 
is a  connected, complete, flat $n$-orbifold, 
and conversely if $M$ is a connected, complete, flat $n$-orbifold, 
then there is a discrete group $\Gamma$ of isometries of $E^n$  
such that $M$ is isometric to $E^n/\Gamma$ by Theorem 13.3.10 of \cite{R}. 

\vspace{.15in}
\noindent{\bf Definition:}
A flat $n$-orbifold $M$ {\it geometrically fibers} over a flat $m$-orbifold $B$, 
with {\it generic fiber} a flat $(n-m)$-orbifold $F$, if there is a surjective map $\eta: M \to B$, 
called the {\it fibration projection},  
such that for each point $y$ of $B$,  
there is an open metric ball $B(y,r)$ of radius $r > 0$ centered at $y$ in $B$
such that $\eta$ is isometrically equivalent on $\eta^{-1}(B(y,r))$
to the natural projection $(F\times B_y)/G_y \to B_y/G_y$, 
where $G_y$ is a finite group acting diagonally on $F\times B_y$, isometrically on $F$,  
and effectively and orthogonally on an open metric ball $B_y$ in $E^m$ of radius $r$. 
This implies that the fiber $\eta^{-1}(y)$ is isometric to $F/G_y$. 
The fiber $\eta^{-1}(y)$ is said to be {\it generic} if $G_y = \{1\}$  
or {\it singular} if $G_y$ is nontrivial. 

\vspace{.15in}
The following theorem is Theorem 4 of \cite{R-T}. 

\begin{theorem}  
Let ${\rm N}$ be a complete normal subgroup of an $n$-space group $\Gamma$, 
and let $V = {\rm Span}({\rm N})$.  
Then the flat orbifold $E^n/\Gamma$ geometrically fibers over the flat orbifold 
$(E^n/V)/(\Gamma/{\rm N})$ with generic fiber the flat orbifold $V/{\rm N}$ and 
fibration projection $\eta_V: E^n/\Gamma \to (E^n/V)/(\Gamma/{\rm N})$ defined by the formula 
$\eta_V(\Gamma x) = (\Gamma/\Nu)(V+x).$
\end{theorem}

A {\it geometrically fibered orbifold structure} on a flat orbifold $M$ 
is the partition of $M$ by the fibers of a geometric fibration projection $\eta$ 
from $M$ to a flat orbifold $B$.   
Let $\Gamma$ be an $n$-space group. 
Given a geometrically fibered orbifold structure on $E^n/\Gamma$,  
there exists, by Theorem 7 of \cite{R-T}, a complete normal subgroup $\mathrm{N}$ of $\Gamma$, 
with $V = \mathrm{Span}(\Nu)$, 
such that the geometrically fibered orbifold structure on $E^n/\Gamma$ is equal to the partition 
of $E^n/\Gamma$ by the fibers of the fibration projection
$$\eta_V: E^n/\Gamma\to (E^n/V)/(\Gamma/\mathrm{N}).$$ 

\section{The Generalized Calabi Construction} 

Let $\Nu$ be a complete normal subgroup of an $n$-space group $\Gamma$, 
let $V = \mathrm{Span}(\Nu)$, and let $V^\perp$ be the orthogonal complement of $V$ in $E^n$.  
Euclidean $n$-space $E^n$ decomposes as the Cartesian product $E^n = V \times V^\perp$. 
Let $b+B\in\Gamma$ and let $x\in E^n$.  Write $b= c+d$ with $c\in V$ and $w\in V^\perp$. 
Write $x = v+w$ with $v\in V$ and $w\in V^\perp$. Then 
$$(b+B)x = b+Bx = c+d + Bv + Bw = (c+Bv) + (d+Bw).$$
Hence the action of $\Gamma$ on $E^n$ corresponds to the diagonal action of $\Gamma$ 
on $V\times V^\perp$ defined by the formula
$$(b+B)(v,w) = (c+Bv,d+Bw).$$
Here $\Gamma$ acts on both $V$ and $V^\perp$ via isometries. 
The kernel of the corresponding homomorphism from $\Gamma$ to $\mathrm{Isom}(V)$ 
is the group
$$\Kappa = \{b+B\in\Gamma: b \in V^\perp\ \hbox{and}\ V \subseteq \mathrm{Fix}(B)\}.$$
We call $\Kappa$ the {\it kernel of the action} of $\Gamma$ on $V$. 
The group $\Kappa$ is a normal subgroup of $\Gamma$. 
The action of $\Gamma$ on $V$ induces an effective action of $\Gamma/\Kappa$ on $V$ 
via isometries.  
Note that $\Nu\cap\Kappa = \{I\}$,  
and each element of $\Nu$ commutes with each element of $\Kappa$.   
Hence $\Nu\Kappa$ is a normal subgroup of $\Gamma$,  
and $\Nu\Kappa$ is the direct product of $\Nu$ and $\Kappa$. 

The action of $\Nu$ on $V^\perp$ is trivial and the action of $\Kappa$ on $V$ is trivial. 
Hence $E^n/\Nu\Kappa$ decomposes as the Cartesian product 
$E^n/\Nu\Kappa = V/\Nu \times V^\perp/\Kappa.$

The action of $\Gamma/\Nu\Kappa$ on $E^n/\Nu\Kappa$ corresponds 
to the diagonal action of $\Gamma/\Nu\Kappa$ on $V/\Nu \times V^\perp/\Kappa$ via isometries 
defined by the formula
$$(\Nu\Kappa(b+B))(\Nu v,\Kappa w) = (\Nu(c+Bv),\Kappa(d+Bw)).$$
Hence we have the following theorem.
\begin{theorem} 
Let $\Nu$ be a complete normal subgroup of an $n$-space group $\Gamma$, 
and let $\Kappa$ be the kernel of the action of $\Gamma$ on $V= \mathrm{Span}(\Nu)$. 
Then the map
$$\chi: E^n/\Gamma \to (V/\Nu\times V^\perp/\Kappa)/(\Gamma/\Nu\Kappa)$$
defined by $\chi(\Gamma x) = (\Gamma/\Nu\Kappa)(\Nu v, \Kappa w)$, 
with  $x = v + w$ and $v\in V$ and $w\in V^\perp$, is an isometry. 
\end{theorem}

We call $\Gamma/\Nu\Kappa$ the {\it structure group} 
of the geometric fibered orbifold structure on $E^n/\Gamma$ 
determined by the complete normal subgroup $\Nu$ of $\Gamma$. 

The group $\Nu$ is the kernel of the action of $\Gamma$ on $V^\perp$, and so we have 
an effective action of $\Gamma/\Nu$ on $V^\perp$. 
The natural projection from $V/\Nu\times V^\perp/\Kappa$ to $V^\perp/\Kappa$ 
induces a continuous surjection
$$\pi^\perp: (V/\Nu\times V^\perp/\Kappa)/(\Gamma/\Nu\Kappa) \to V^\perp/(\Gamma/\Nu).$$
Orthogonal projection from $E^n$ to $V^\perp$ induces an isometry from $E^n/V$ to $V^\perp$ 
which in turn induces an isometry 
$$\psi^\perp: (E^n/V)/(\Gamma/\Nu) \to V^\perp/(\Gamma/\Nu).$$
\begin{theorem} 
The following diagram commutes
\[\begin{array}{ccc}
E^n/\Gamma & {\buildrel \chi\over\longrightarrow} &
(V/\Nu\times V^\perp/\Kappa)/(\Gamma/\Nu\Kappa) \\
\eta_V \downarrow \  & & \downarrow\pi^\perp \\
(E^n/V)/(\Gamma/\Nu) & {\buildrel \psi^\perp\over\longrightarrow}  & V^\perp/(\Gamma/\Nu). 
\end{array}\] 
\end{theorem}
\begin{proof}
Let $x\in E^n$.  Write $x = v+w$ with $v\in V$ and $w\in V^\perp$. 
Then we have 
\begin{eqnarray*}
\pi^\perp(\chi(\Gamma x)) & =  & \pi^\perp((\Gamma/\Nu\Kappa)(\Nu v,\Kappa w)) \\
			      & = & (\Gamma/\Nu) w \\
			      & =  & \psi^\perp((\Gamma/\Nu)(V+x))  = \psi^\perp(\eta_V(\Gamma x)). 
\end{eqnarray*}

\vspace{-.25in}
\end{proof}		

Let ${\rm N}$ be a complete normal subgroup of an $n$-space group $\Gamma$,  
and let $\Kappa$ be the kernel of the action of $\Gamma$ on $V = \mathrm{Span}(\Nu)$. 
If $\mathrm{Span}(\Kappa) = V^\perp$, then $\Kappa$ is a complete normal subgroup 
of $\Gamma$ called the {\it orthogonal dual} of $\Nu$ in $\Gamma$, 
and we write $\Kappa = \Nu^\perp$. 

Suppose $\Kappa = \Nu^\perp$.  
The natural projection from $V/\Nu\times V^\perp/\Kappa$ to $V/\Nu$ 
induces a continuous surjection
$$\pi: (V/\Nu\times V^\perp/\Kappa)/(\Gamma/\Nu\Kappa) \to V/(\Gamma/\Kappa).$$
Orthogonal projection from $E^n$ to $V$ induces an isometry from $E^n/V^\perp$ to $V$ 
which in turn induces an isometry 
$$\psi: (E^n/V^\perp)/(\Gamma/\Kappa) \to V/(\Gamma/\Kappa).$$
The next corollary follows from Theorem 4 by reversing the roles of $\Nu$ and $\Kappa$. 
\begin{corollary} 
The following diagram commutes
\[\begin{array}{ccc}
E^n/\Gamma & {\buildrel \chi\over\longrightarrow} &
(V/\Nu\times V^\perp/\Kappa)/(\Gamma/\Nu\Kappa) \\
\eta_{V ^\perp}\downarrow\ \ \  & & \downarrow\pi \\
(E^n/V^\perp)/(\Gamma/\Kappa) & {\buildrel \psi\over\longrightarrow}  & V/(\Gamma/\Kappa). 
\end{array}\] 
\end{corollary}

Theorem 4 says that the fibration projection $\eta_V: E^n/\Gamma \to (E^n/V)/(\Gamma/\Nu)$  
is equivalent to the projection 
$\pi^\perp: (V/\Nu\times V^\perp/\Kappa)/(\Gamma/\Nu\Kappa) \to V^\perp/(\Gamma/\Nu)$ 
induced by the projection on the second factor of $V/\Nu\times V^\perp/\Kappa$;   
while Corollary 1 says that the orthogonally dual fibration projection 
$\eta_{V^\perp}:E^n/\Gamma \to(E^n/V^\perp)/(\Gamma/\Kappa)$ is equivalent to the projection 
$\pi: (V/\Nu\times V^\perp/\Kappa)/(\Gamma/\Nu\Kappa) \to V/(\Gamma/\Kappa)$ 
induced by the projection on the first factor of $V/\Nu\times V^\perp/\Kappa$. 
Our generalized Calabi construction reveals the intimate relationship between 
a flat orbifold fibration and its orthogonally dual flat orbifold fibration 
as being similar to that of the two sides of the same coin. 

\begin{theorem} 
Let ${\rm N}$ be a complete normal subgroup of an $n$-space group $\Gamma$,  
and let $\Kappa$ be the kernel of the action of $\Gamma$ on $V= \mathrm{Span}(\Nu)$. 
Then $\Nu$ has an orthogonal dual in $\Gamma$ 
 if and only if the structure group $\Gamma/\Nu\Kappa$ is finite. 
\end{theorem}
\begin{proof} 
The group $\Nu\Kappa/\Nu$ is a normal subgroup of the space group $\Gamma/\Nu$. 
Observe that we have
\begin{eqnarray*}
|\Gamma/\Nu\Kappa| < \infty &  \Leftrightarrow & [\Gamma/\Nu:\Nu\Kappa/\Nu] < \infty \\
		&  \Leftrightarrow & \mathrm{dim}(\Nu\Kappa/\Nu) =\mathrm{dim}(\Gamma/\Nu) \\
		&  \Leftrightarrow & \mathrm{dim}(\Kappa) =\mathrm{dim}(V^\perp) \\
		&  \Leftrightarrow & \mathrm{Span}(\Kappa) =V^\perp  \ \Leftrightarrow \ \Kappa = \Nu^\perp. 
\end{eqnarray*} 

\vspace{-.25in} 
\end{proof}

\noindent{\bf Example 1.}  
Let $e_1$ and $e_2$ be the standard basis vectors of $E^2$, 
and let $x \in \realnos$ with $0 \leq x < 1$. 
Let $\Gamma$ be the group generated by $t_1 = e_1+I$ and $t_2 = xe_1+e_2+I$.  
Then $\Gamma$ is a 2-space group, and $E^2/\Gamma$ is a torus. 
Let $\Nu=\langle t_1\rangle$.  
Then $\Nu$ is a complete normal subgroup of $\Gamma$, 
with $V = \mathrm{Span}(\Nu) = \mathrm{Span}\{e_1\}$.  
Let $\Kappa$ be the kernel of the action of $\Gamma$ on $V/\Nu$. 
The structure group $\Gamma/\Nu\Kappa$ is a cyclic group 
generated by $\Nu\Kappa t_2$, which acts on the circle $V/\Nu$, of length one, 
by rotating a distance $x$. 
Let $c, d \in \integers$.  Then $t_1^ct_2^d = (c+dx)e_1+de_2 + I$, 
and $t_1^ct_2^d \in \Kappa$ if and only if $c+dx = 0$.  
Thus if $x$ is irrational, then $\Kappa = \{I\}$, and $\Gamma/\Nu\Kappa$ is infinite.  
If $x = a/b$ with $a, b\in\integers$, $b> 0$, and $a, b$ coprime, 
then $\Kappa$ is generated by $t_1^at_2^{-b}$, and $\Gamma/\Nu\Kappa$ has order $b$. 
Thus $\Nu$ has an orthogonal dual in $\Gamma$ if and only if $x$ is rational.

\begin{theorem} 
Let ${\rm N}$ be a complete normal subgroup of an $n$-space group $\Gamma$,  
and let $\Kappa$ be the kernel of the action of $\Gamma$ on $V= \mathrm{Span}(\Nu)$. 
Then the structure group $\Gamma/\Nu\Kappa$ acts effectively on $V/\Nu$, and $V^\perp/\Kappa$, 
and $V/\Nu \times V^\perp/\Kappa$.  
Moreover $(V/\Nu)/(\Gamma/\Nu\Kappa) = V/(\Gamma/\Kappa)$  
and $(V^\perp/\Kappa)/(\Gamma/\Nu\Kappa) = V^\perp/(\Gamma/\Nu)$. 
The natural projections from $V/\Nu$ to $V/(\Gamma/\Kappa)$, from $V^\perp/\Kappa$ to 
$V^\perp/(\Gamma/\Nu)$, and from $V/\Nu\times V^\perp/\Kappa$ to 
$(V/\Nu\times V^\perp/\Kappa)/(\Gamma/\Nu\Kappa)$ 
are orbifold regular covering projections corresponding, respectively, 
to the short exact sequences 
$$1\to \Nu\Kappa/\Kappa \to \Gamma/\Kappa \to \Gamma/(\Nu\Kappa)\to 1,$$
$$1\to \Nu\Kappa/\Nu \to \Gamma/\Nu \to \Gamma/(\Nu\Kappa)\to 1, $$
$$1\to \Nu\Kappa \to \Gamma \to \Gamma/(\Nu\Kappa)\to 1. $$
Moreover,  $\Nu\Kappa/\Kappa \cong \Nu$ and $\Nu\Kappa/\Nu \cong \Kappa$, 
since $\Nu\cap \Kappa = \{I\}$. 
\end{theorem}
\begin{proof}
Let $b+B$ be an element of $\Gamma$ such that $\Nu\Kappa(b+B)$ acts trivially on $V/\Nu$. 
Write $b = c+d$ with $c\in V$ and $d\in V^\perp$. 
Then for each $v\in V$, we have 
$$\Nu\Kappa(b+B)(\Nu v) = \Nu(c+Bv) = \Nu v.$$
Let $D$ be a fundamental domain for $\Nu$ in $V$, and suppose $v\in D$. 
Then there exists $a+A \in \Nu$ such that $c+Bv = a+ Av$. 
Hence $c+Bv\in (a+A)D$.  As $c+B$ is continuous at $v$, there exists $r > 0$ 
such that if $B(v,r)$ is the open ball with center $v$ and radius $r$ in $E^n$, then 
$$(c+B)(B(v,r)\cap V) \subseteq (a+A)D.$$
Hence $c+Bx = a+Ax$ for all $x\in B(v,r)\cap V$. 
Therefore $c+Bx = a + Ax$ for all $x\in V$. 
Hence $(a+A)^{-1}(c+B)x = x$ for all $x\in V$. 
Now we have 
$$(a+A)^{-1}(c+B) = (-A^{-1}a + A^{-1})(c+B) = -A^{-1}a + A^{-1}c+A^{-1}B.$$
Hence $a = c$ and $V\subseteq \mathrm{Fix}(A^{-1}B)$.  Next, observe that 
\begin{eqnarray*}
(a+A)^{-1}(b+B) & = & (-A^{-1}a + A^{-1})(b+B) \\ 
			& = &  -A^{-1}a + A^{-1}b+A^{-1}B \\
			& = &  -A^{-1}a + A^{-1}c + A^{-1}d+A^{-1}B \ = \ d + A^{-1}B 
\end{eqnarray*}
with $d \in V^\perp$ and $V\subseteq \mathrm{Fix}(A^{-1}B)$. 
Therefore $(a+A)^{-1}(b+B) \in \Kappa$, and so $b+B \in \Nu\Kappa$. 
Hence $\Gamma/\Nu\Kappa$ acts effectively on $V/\Nu$. 

The same argument, with the roles of $V$ and $V^\perp$ reversed, 
shows that $\Gamma/\Nu\Kappa$ acts effectively on $V^\perp/\Kappa$. 
Therefore $\Gamma/\Nu\Kappa$ acts effectively on $V/\Nu\times V^\perp/\Kappa$. 

Observe that 
 $$(V/\Nu)/(\Gamma/\Nu\Kappa) = (V/\Nu\Kappa)/(\Gamma/\Nu\Kappa) = V/\Gamma = 
 V/(\Gamma/\Kappa),$$
 and 
 $$(V^\perp/\Kappa)/(\Gamma/\Nu\Kappa) = (V^\perp/\Nu\Kappa)/(\Gamma/\Nu\Kappa) 
 = V^\perp/\Gamma = V^\perp/(\Gamma/\Nu). $$ 
 
 \vspace{-.2in}
\end{proof}

\begin{corollary} 
Let ${\rm N}$ be a complete normal subgroup of an $n$-space group $\Gamma$,  
and let $\Kappa$ be the kernel of the action of $\Gamma$ on $V= \mathrm{Span}(\Nu)$. 
If the group of isometries of $V/\Nu$ is finite, then $\Nu$ has an orthogonal dual in $\Gamma$. 
\end{corollary}
\begin{proof}
By Theorem 6, the structure group $\Gamma/\Nu\Kappa$ acts effectively on $V/\Nu$. 
If the group of isometries of $V/\Nu$ is finite, then $\Gamma/\Nu\Kappa$ 
is finite, and so $\Nu$ has an orthogonal dual in $\Gamma$ by Theorem 5. 
\end{proof}

\noindent{\bf Example 2.}  
Let $\Nu$ be a complete normal subgroup of an $n$-space group $\Gamma$,  
and let $\Kappa$ be the kernel of the action of $\Gamma$ on $V=\mathrm{Span}(\Nu)$. 
Suppose that $\Nu$ is an infinite dihedral group.  Then $V/\Nu$ is a closed interval.  
The group of isometries of $V/\Nu$ is generated by the reflection of $V/\Nu$ in 
the midpoint of $V/\Nu$, and so has order 2.  
Hence $\Kappa = \Nu^\perp$ by Corollary 2,  
and the structure group has order 1 or 2.  If $\Gamma/\Nu\Kappa$ has order 1, 
then $E^n/\Gamma = V/\Nu\times V^\perp/\Kappa$.  
If $\Gamma/\Nu\Kappa$ has order 2, 
then $E^n/\Gamma$ is isometric to $(V/\Nu\times V^\perp/\Kappa)/(\Gamma/\Nu\Kappa)$,  
which is a twisted $I$-bundle over $V^\perp/(\Gamma/\Nu)$ determined by 
the quotient map $V^\perp/\Kappa \to V^\perp/(\Gamma/\Nu)$;  
moreover $(V/\Nu)/(\Gamma/\Nu\Kappa)$ is a closed interval. 
By Theorem 6, we have that $(V/\Nu)/(\Gamma/\Nu\Kappa)=V/(\Gamma/\Kappa)$. 
Hence $V/(\Gamma/\Kappa)$ is a closed interval. 
The group $\Gamma/\Kappa$ acts as a space group of isometries on $V$ by Lemma 1 of \cite{R-T}. 
Therefore $\Gamma/\Kappa$ is an infinite dihedral group. 
These facts first appeared in Theorems 19 and 20 in \cite{R-T}.

\begin{theorem} 
Let $\Nu$ be an infinite cyclic, complete, normal subgroup of an $n$-space group $\Gamma$, 
and let $\Kappa$ be the kernel of the action of $\Gamma$ on $V = \span(\Nu)$, 
and suppose $\Kappa = \Nu^\perp$. 
Then $V/\Nu$ is a circle,  and 
the structure group $\Gamma/\Nu\Kappa$ is either a finite cyclic group generated 
by a rotation of $V/\Nu$ or a finite dihedral group generated by a rotation and a reflection of $V/\Nu$. 
Moreover $\Gamma/\Nu\Kappa$ is a finite cyclic group generated by a rotation 
if and only if $V/(\Gamma/\Kappa)$ is a circle, and $\Gamma/\Nu\Kappa$ is a finite dihedral group 
generated by a rotation and a reflection if and only if $V/(\Gamma/\Kappa)$ is a closed interval. 
Furthermore $V/(\Gamma/\Kappa)$ is isometric to the base of the geometric 
orbifold fibration of $E^n/\Gamma$ determined by the complete normal subgroup  
$\Kappa$ of $\Gamma$. 
\end{theorem}
\begin{proof}
The structure group $\Gamma/\Nu\Kappa$ is finite by Theorem 5. 
The group $\Gamma/\Nu\Kappa$ acts effectively on the circle $V/\Nu$ by Theorem 6. 
Hence $\Gamma/\Nu\Kappa$ is either a finite cyclic group generated by a rotation of $V/\Nu$ 
or a finite dihedral group generated by a rotation and a reflection of $V/\Nu$. 
Now $(V/\Nu)/(\Gamma/\Nu\Kappa) = V/(\Gamma/\Kappa)$  by Theorem 6.  
Hence $\Gamma/\Nu\Kappa$ is a finite cyclic group generated by a rotation 
if and only if $V/(\Gamma/\Kappa)$ is a circle, and $\Gamma/\Nu\Kappa$ is a finite dihedral group 
generated by a rotation and a reflection if and only if $V/(\Gamma/\Kappa)$ is a closed interval. 
By Corollary 1, the orbifold 
$V/(\Gamma/\Kappa)$ is isometric to the base of the fibration of $E^n/\Gamma$ 
determined by $\Kappa$. 
\end{proof}

\noindent{\bf Example 3.}  
All the Seifert fibrations with generic fiber a circle in Table 1 of \cite{R-T} 
satisfy the hypothesis of Theorem 7. 
The order of the structure group $\Gamma/\Nu\Kappa$ is given in the column headed by ind.
The structure group is cyclic if and only if either its order is $1,2,3$ or its order is at least 4 
and the base of the co-Seifert fibration is a circle. In all other cases, the structure group is 
a dihedral group. 

\vspace{.15in}
\noindent{\bf Example 4.}  
Let $\Gamma$ be an $n$-space group with a first Betti number $\beta_1$ such that 
$0<\beta_1<n$. 
Let $\Nu$ be the subgroup of $\Gamma$ that contains $[\Gamma, \Gamma]$ 
and corresponds to the torsion subgroup of $\Gamma/[\Gamma,\Gamma]$. 
Then $\Nu$ is a complete normal subgroup of $\Gamma$ by Theorem 13 of \cite{R-T}, 
and $\Nu^\perp = Z(\Gamma)$, the center of $\Gamma$, by Theorem 15 of \cite{R-T}. 
Let $V = \mathrm{Span}(\Nu)$.  Then $V^\perp = \mathrm{Span}(Z(\Gamma))$, 
and $V^\perp/Z(\Gamma)$ is a torus. 
By Theorem 3, we have that $E^n/\Gamma$ is isometric to 
$(V/\Nu\times V^\perp/Z(\Gamma)/(\Gamma/\Nu Z(\Gamma))$. 
The group $\Gamma/\Nu$ is a free abelian group of rank $\beta_1$. 
Hence the structure group $\Gamma/\Nu Z(\Gamma)$ is abelian;   
moreover $\Gamma/\Nu Z(\Gamma)$ is finite by Theorem 5.  
This example corresponds to the original Calabi construction 
described in \S 3.6 of \cite{Wolf} when $\Gamma$ is torsion-free. 

\vspace{.15in}
\noindent{\bf Example 5.}  
Let ${\rm N}$ be a complete normal subgroup of an $n$-space group $\Gamma$ such that 
$\Gamma/\Nu$ is torsion-free,  
and let $\Kappa$ be the kernel of the action of $\Gamma$ on $V=\mathrm{Span}(\Nu)$. 
Then $(E^n/V)/(\Gamma/\Nu)$ is a flat manifold and 
$\eta_V: E^n/\Gamma\to (E^n/V)/(\Gamma/\Nu)$ 
is a fiber bundle projection by Theorem 13 of \cite{R-T}. 
Hence $V^\perp/(\Gamma/\Nu)$ is a flat manifold and 
$\pi: (V/\Nu\times V^\perp/\Kappa)/(\Gamma/\Nu\Kappa)\to V^\perp/(\Gamma/\Nu)$ 
is a fiber bundle projection by Theorem 4. 
Moreover $(V/\Nu\times V^\perp/\Kappa)/(\Gamma/\Nu\Kappa)$ is a fiber bundle,  
with structure group $\Gamma/\Nu\Kappa$, in the sense of \S 2 of \cite{S}. 
The associated principle bundle projection, in the sense of \S 8 of \cite{S}, is the quotient map 
$V^\perp/\Kappa \to V^\perp/(\Gamma/\Nu)$, under the action of $\Gamma/\Nu\Kappa$. 
Note that the quotient map $V^\perp/\Kappa \to V^\perp/(\Gamma/\Nu)$ is a regular 
covering projection with fibers equal to the $\Gamma/\Nu\Kappa$ orbits.

\section{Splitting Space Group Extensions} 

Let $\Nu$ be a complete normal subgroup of an $n$-space group $\Gamma$, 
and consider the corresponding space group extension
$$ 1 \to {\rm N}\ {\buildrel i\over \longrightarrow}\ \Gamma\ {\buildrel p\over\longrightarrow} 
\ \Gamma/\Nu\to 1.$$
In this section, we give a geometric condition for the above group extension 
to split ($p$ has a right inverse). Note that $p$ has a right inverse if and only if 
$\Gamma$ has a subgroup $\Sigma$ such that $\Gamma = \Nu\Sigma$ and $\Nu\cap\Sigma = \{I\}$.

\begin{lemma} 
Let $\Nu$ be a complete normal subgroup of an $n$-space group $\Gamma$, 
and let $\Kappa$ be the kernel of the action of $\Gamma$ on $V = {\rm Span}(\Nu)$. 
Let $v_0$ be a point of $V$, and let $\Sigma$ be the stabilizer in $\Gamma$ 
of the set $V^\perp+v_0$. Then 
\begin{enumerate}
\item the group $\Nu\cap\Sigma$ is the stabilizer in $\Nu$ of the point $v_0$ of $V$, 
\item the group $\Kappa$ is a subgroup of $\Sigma$, 
\item the group $\Nu\Sigma/\Nu\Kappa$ is the stabilizer in $\Gamma/\Nu\Kappa$ 
of the point $\Nu v_0$ of $V/\Nu$, 
\item the groups  $\Nu\cap\Sigma$, and $\Sigma/\Kappa$, and  $\Nu\Sigma/\Nu\Kappa$ are finite. 
\end{enumerate}
\end{lemma}
\begin{proof}
Let $b+B\in \Gamma$.  Write $b=c+d$ with $c\in V$ and $d\in V^\perp$. 
Then we have that 
\begin{eqnarray*}
(b+B)(V^\perp +v_0)  =  V^\perp +v_0 & \Leftrightarrow & V^\perp +b+Bv_0 = V^\perp + v_0 \\
& \Leftrightarrow & V^\perp +c+Bv_0  =  V^\perp + v_0 \ \ \Leftrightarrow \ \ c+Bv_0  =  v_0. 
\end{eqnarray*}
Therefore $\Nu\cap\Sigma$ is the stabilizer in $\Nu$ of $v_0$ and $\Kappa \subseteq \Sigma$.  
Now as 
$$(\Nu\Kappa(b+B))\Nu v_0 = \Nu(c+Bv_0), $$
we have that $\Nu\Sigma/\Nu\Kappa$ fixes the point $\Nu v_0$ of $V/\Nu$. 

Suppose the $\Nu\Kappa(b+B)$ fixes the point $\Nu v_0$. 
Then $\Nu(c+Bv_0) = \Nu v_0$. 
Hence there exist $a+A\in\Nu$ such that $(a+A)(c+Bv_0) = v_0$. 
Therefore $(a+A)(b+B) \in \Sigma$, and so $b+B \in \Nu\Sigma$. 
Thus $\Nu\Sigma/\Nu\Kappa$ is the stabilizer in $\Gamma/\Nu\Kappa$ 
of the point $\Nu v_0$ of $V/\Nu$. 

Suppose $b+B\in\Sigma$.  Then $c+Bv_0 = v_0$.  There are only finitely 
many possible $B$, and so there are only finitely many possible $c$. 
Hence $\Nu\Sigma/\Nu\Kappa$ acts effectively on $V/\Nu$ by only finitely many transformations. 
Therefore $\Nu\Sigma/\Nu\Kappa$ is finite. 

The inclusion of $\Nu\cap\Sigma$ into $\Sigma$ induces a monomorphism 
from $\Nu\cap\Sigma$ to $\Sigma/\Kappa$, since $\Nu\cap \Kappa = \{I\}$. 
The inclusion of $\Sigma$ into $\Nu\Sigma$ induces an epimorphism from 
$\Sigma/\Kappa$ onto $\Nu\Sigma/\Nu\Kappa$. Moreover, we have a short exact sequence 
$$1\to \Nu\cap\Sigma \to \Sigma/\Kappa \to \Nu\Sigma/\Nu\Kappa \to 1.$$
The group $\Nu\cap\Sigma$ is the stabilizer in $\Nu$ of $v_0$, and so $\Nu\cap\Sigma$ 
is finite.  Therefore $\Sigma/\Kappa$ is finite. 
\end{proof}

\begin{theorem} 
Let $\Nu$ be a complete normal subgroup of an $n$-space group $\Gamma$, 
and let $\Kappa$ be the kernel of the action of $\Gamma$ on $V = {\rm Span}(\Nu)$. 
Suppose that $\Gamma/\Nu\Kappa$ fixes an ordinary point $\Nu v_0$ of $V/\Nu$. 
Let $\Sigma$ be the stabilizer in $\Gamma$ of the set $V^\perp+v_0$. 
Then $\Gamma = \Nu\Sigma$, and $\Nu\cap\Sigma = \{I\}$, 
and $\mathrm{Span}(\Sigma) = V^\perp$. 

Conversely, if there exists a subgroup $\Rho$ of $\Gamma$ such that 
$\Gamma = \Nu\Rho$, and $\Nu\cap\Rho = \{I\}$, and $\mathrm{Span}(\Rho) = V^\perp$, 
then $\Gamma/\Nu\Kappa$ fixes a point of $V/\Nu$. 
\end{theorem}
\begin{proof}
By Lemma  2(3), we have that $\Gamma = \Nu\Sigma$. 
That $\Nu v_0$ is an ordinary point of $V/\Nu$ means that 
the stabilizer of $v_0$ in $\Nu$ is $\{I\}$. 
Hence $\Nu\cap\Sigma = \{I\}$ by Lemma  2(1). 
By Lemma  2(4), we have that $\Gamma/\Nu\Kappa$ is finite. 
Hence $\Kappa = \Nu^\perp$ by Theorem 5, 
and so $\mathrm{Span}(\Kappa) = V^\perp$. 
By Lemma  2(4), we have that $\Sigma/\Kappa$ is finite. 
Hence $\mathrm{Span}(\Sigma) = \mathrm{Span}(\Kappa) = V^\perp$. 

Conversely, let $\Rho$ be a subgroup of $\Gamma$ such that 
$\Gamma = \Nu\Rho$, and $\Nu\cap\Rho = \{I\}$, and $\mathrm{Span}(\Rho) = V^\perp$. 
By Theorem 5.4.6 of \cite{R}, the group $\Rho$ has a free abelian subgroup $\Eta$ of rank $m$ 
and finite index, there is an $m$-plane $Q$ of $E^n$ such that  $\Eta$ acts effectively 
on $Q$ as a discrete group of translations, and the $m$-plane $Q$ is invariant under $\Rho$. 
Now $P$ is isomorphic to $\Gamma/\Nu$, and so $P$ is space group, 
with $\dim(P) = \dim(\Gamma/\Nu) = \dim(V^\perp)$. 
Moreover $\dim(P) = \dim(\Eta) = m$, and so $\dim(V^\perp)= m$. 
Now $Q = W+v_0$ for some $m$-dimensional vector subspace $W$ of $E^n$.  
Let $a+I \in \Rho$. Then $(a+I)Q = Q$ implies 
that $a\in W$.  As $\mathrm{Span}(\Rho) =V^\perp$, we have that $W = V^\perp$. 
Hence we may assume that $v_0\in V$. Then $\Rho$ is a subgroup of the 
stabilizer $\Sigma$ in $\Gamma$ of the set $V^\perp + v_0$. 
As $\Gamma = \Nu\Rho$, we have that $\Gamma = \Nu\Sigma$, 
and so $\Gamma/\Nu\Kappa$ fixes the point $\Nu v_0$ of $V/\Nu$ by Lemma  2(3).  
\end{proof}

\noindent{\bf Remark 3.}  Theorem 8 generalizes and strengthens Theorem 18 of \cite{R-T}. 

\medskip
\noindent{\bf Example 6.}
We next consider an example that shows that the hypothesis that $\Nu v_0$ is an ordinary point 
cannot be dropped in Theorem 8. 
Let $\Gamma$ be the 3-space group with IT number 113 in Table 1B of \cite{B-Z}. 
Then $\Gamma = \langle t_1,t_2,t_3,\alpha,\beta,\gamma\rangle$ 
where $t_i = e_i+I$ for $i=1,2,3$ are the standard translations, 
and $\alpha = \frac{1}{2}e_1+\frac{1}{2}e_2+A$, 
$\beta=\frac{1}{2}e_1+B$, $\gamma =\frac{1}{2}e_2+C$, and 
$$A = \left(\begin{array}{rrr} -1 & 0 & 0\\ 0 & -1 & 0 \\ 0 & 0 & 1  \end{array}\right),\ \ 
B = \left(\begin{array}{rrr} 0 & 1 & 0  \\ -1 & 0 & 0   \\ 0 & 0 & -1 \end{array}\right), \ \ 
C = \left(\begin{array}{rrr} -1 & 0 & 0  \\ 0 & 1 & 0   \\ 0 & 0 & -1 \end{array}\right).$$
The group $\Nu = \langle t_1,t_2,\alpha,\beta\gamma\rangle$ is a complete normal subgroup 
of $\Gamma$ with $V = {\rm Span}(\Nu) = {\rm Span}\{e_1,e_2\}$. 
The isomorphism type of $\Nu$ is $2\!\ast\!22$ in Conway's notation \cite{Conway} 
or $cmm$ in the IT notation \cite{S}. 
The flat orbifold $V/\Nu$ is a pointed hood.  
A fundamental polygon for the action of $\Nu$ on $V$ is the $45^\circ$ - $45^\circ$ 
right triangle $\triangle$ with vertices $v_1= \frac{1}{4}e_1+\frac{1}{4}e_2$, and 
$v_2 = \frac{5}{4}e_1+\frac{1}{4}e_2$, and $v_3 = \frac{3}{4}e_1+\frac{3}{4}e_2$. 
The short sides $[v_1,v_3]$ and $[v_2,v_3]$ of $\Delta$ are fixed pointwise by the reflections $t_1^{-1}\beta\gamma=BC$ and 
$t_1^2t_2\alpha\beta\gamma = \frac{3}{2}e_1+\frac{3}{2}e_2 + ABC$, respectively. 
The long side $[v_1,v_2]$ of $\triangle$ is flipped about its midpoint 
$v_0 = \frac{3}{4}e_1+\frac{1}{4}e_2$ 
by the halfturn $t_1\alpha$. 

The orbifold $V/\Nu$ has a unique cone point $\Nu v_0$.  
Hence $\Nu v_0$ is fixed by every isometry of $V/\Nu$. 
Therefore the structure group $\Gamma/\Nu\Kappa$ fixes the cone point $\Nu v_0$ of $V/\Nu$. 
However the space group extension $1\to \Nu\to \Gamma\to \Gamma/\Nu\to 1$ 
does not split as explained on p.\ 1649 of \cite{R-T}.
Here $\Kappa = \langle t_3\rangle$ and $\Gamma/\Nu\Kappa$ has order two 
and is generated by $\Nu\Kappa\beta$. 
The element $t_2\beta$ acts on $V$ by rotating $-90^\circ$ around $v_0$. 
Hence $\Nu\Kappa\beta$ acts on $V/\Nu$ as a halfturn around the 
cone point of $V/\Nu$, and so has no other fixed points. 
Note that the cone point of $V/\Nu$ was misidentified on p.\ 1649 of \cite{R-T} 
to be $\Nu v_1$, which is actually one of the two right-angled corner points of $V/\Nu$; 
the other corner point is $\Nu v_3$. 

\vspace{.15in}
\noindent{\bf Definition:} 
Let $\Nu$ be a complete normal subgroup of an $n$-space group $\Gamma$, 
and let $V = \mathrm{Span}(\Nu)$. 
We say that the space group extension
$1 \to \Nu\to\Gamma\to\Gamma/\Nu\to 1$ {\it splits orthogonally}
if $\Gamma$ has a subgroup $\Sigma$ such that 
$\Gamma = \Nu\Sigma$, and $\Nu\cap\Sigma = \{I\}$, and $\mathrm{Span}(\Sigma) = V^\perp$.

\begin{corollary} 
Let $\Nu$ be a torsion-free, complete, normal subgroup of an $n$-space group $\Gamma$, 
and let $\Kappa$ be the kernel of the action of $\Gamma$ on $V = {\rm Span}(\Nu)$. 
Then the space group extension $1\to\Nu\to \Gamma\to \Gamma/\Nu\to 1$ splits orthogonally if and only if the structure group $\Gamma/\Nu\Kappa$ fixes a point of $V/\Nu$. 
\end{corollary}
\begin{proof}
Every point of $V/\Nu$ is an ordinary point, since $\Nu$ is torsion-free,  and so 
the corollary follows from Theorem 8. 
\end{proof}

\noindent{\bf Example 7.} 
Let $e_1$ and $e_2$ be the standard basis vectors of $E^2$, 
and let $\Gamma$ be the group generated by $t_1 = e_1+I$ and $t_2 = e_2+I$.  
Then $\Gamma$ is a 2-space group, and $E^2/\Gamma$ is a torus. 
Let $\Nu$ be a proper, complete, normal subgroup of $\Gamma$. 
Then there exists coprime integers $a$ and $b$ such that $\Nu = \langle t_1^at_2^b\rangle$. 
Let $\Kappa$ be the kernel of the action of $\Gamma$ on $V = \mathrm{Span}(\Nu)$. 
Then $\Kappa = \Nu^\perp = \langle t_1^bt_2^{-a}\rangle$. 
As $a$ and $b$ are coprime, there exists integers $c$ and $d$ such that $ad - bc = 1$. 
From the equation 
$$(b,-a) = (ac + bd)(a,b) -(a^2+b^2)(c,d),$$
we deduce that the structure group $\Gamma/\Nu\Kappa$ is cyclic of order $a^2+b^2$ 
generated by $\Nu\Kappa t_1^ct_2^d$, and  $\Nu\Kappa t_1^ct_2^d$ acts 
on the circle $V/\Nu$ by a $(ac+bd)/(a^2+b^2)$ turn around $V/\Nu$. 
Hence if  $a^2+b^2 > 1$, then $\Gamma/\Nu\Kappa$ does not fix a point of $V/\Nu$. 
Nevertheless, the space group extension $1\to\Nu\to \Gamma\to \Gamma/\Nu\to 1$ splits, 
since $\Gamma = \Nu\langle t_1^ct_2^d\rangle$ and $\Nu\cap\langle t_1^ct_2^d\rangle =\{I\}$. 

\begin{theorem} 
Let $Z(\Gamma)$ be the center of an $n$-space group $\Gamma$, 
and let $\Nu = Z(\Gamma)^\perp$.  
Then the space group extension $1\to Z(\Gamma) \to \Gamma \to \Gamma/Z(\Gamma)\to 1$ 
splits if and only if the structure group $\Gamma/Z(\Gamma)\Nu$ is trivial. 
\end{theorem}
\begin{proof}
Suppose that the space group extension $1\to Z(\Gamma) \to \Gamma \to \Gamma/Z(\Gamma)\to 1$ 
splits.  Then $\Gamma$ has a subgroup $\Sigma$ such that $\Gamma = Z(\Gamma)\Sigma$ 
and $Z(\Gamma)\cap\Sigma = \{I\}$.  Hence $\Sigma$ is a normal subgroup of $\Gamma$ 
and $\Gamma/\Sigma \cong Z(\Gamma)$.  Let $\beta_1$ be the first Betti number of $\Gamma$. 
Then $Z(\Gamma)$ is a free abelian group of rank $\beta_1$ by Theorem 6 of \cite{R-T}. 
Hence $\Gamma/\Sigma$ is a free abelian group of rank $\beta_1$. 
Therefore $\Sigma = \Nu$ by Theorem 15 of \cite{R-T}.  Hence $\Gamma = Z(\Gamma)\Nu$, 
and so the structure group $\Gamma/Z(\Gamma)\Nu$ is trivial. 

Conversely, if $\Gamma/Z(\Gamma)\Nu$ is trivial, then $\Gamma = Z(\Gamma)\Nu$. 
We have that $Z(\Gamma) \cap \Nu = \{I\}$.  
Hence the space group extension $1\to Z(\Gamma) \to \Gamma \to \Gamma/Z(\Gamma)\to 1$ splits. 
\end{proof}

\begin{lemma} 
Let $\Nu$ be a complete normal subgroup of an $n$-space group $\Gamma$, 
and suppose $\Nu$ is a subgroup of the group $\Tau$ of translations of $\Gamma$, 
and $\Sigma$ is a subgroup of $\Gamma$ such that $\Gamma = \Nu\Sigma$ and 
$\Nu\cap\Sigma = \{I\}$. 
Let $V = \mathrm{Span}(\Nu)$ and $W = \mathrm{Span}(\Sigma)$, 
and let $\Pi$ be the point group of $\Gamma$.  Then  $W$ is $\Pi$-invariant, 
$\dim(W) = n - \dim(V)$, and $V \cap W= \{0\}.$
\end{lemma}
\begin{proof}
Let $b+B \in \Gamma$ and $c+I \in \Sigma$. 
As $\Gamma = \Nu\Sigma$, there exists $a+ I \in \Nu$ such that $(a+I)(b+B) \in \Sigma$. 
Hence $(a+b+B)(c+I)(a+b+B)^{-1} = Bc+I$ is in $\Sigma$. 
Therefore $\mathrm{Span}(\Sigma)$ is $\Pi$-invariant. 

As $\Sigma/\Sigma\cap\Tau \cong \Sigma\Tau/\Tau$ is finite, 
$\Sigma\cap\Tau$ has finite index in $\Sigma$. 
Hence 
$$\dim(W)  =  \dim(\Sigma\cap\Tau)  
					      = \dim(\Sigma) 
					      =  \dim(\Gamma/\Nu)  
					      =  \dim(E^n/V) = n-\dim(V).$$
Let $b+I\in\Tau$.  Then there exists $a+I\in \Nu$ and $c+C\in\Sigma$ 
such that $b+ I = (a+I)(c+C)$.  Then $b+I = a+c+C$, and so $C = I$. 
Hence 
$V+W= \mathrm{Span}(T) = E^n.$
Therefore
$$E^n/V = (V+W)/V \cong W/(V\cap W).$$
Now we have
$$\dim(W) = n - \dim(V) = \dim(E^n/V) = \dim(W/V\cap W) = \dim(W) - \dim(V\cap W).$$
Therefore $\dim(V\cap W) = 0$, and so $V\cap W = \{0\}$. 
\end{proof}

\begin{theorem}  
Let $\Nu$ be a 1-dimensional, complete, normal subgroup of a 3-space group $\Gamma$ 
that corresponds to a Seifert, geometrically fibered, orbifold structure on $E^3/\Gamma$ 
described by Conway et al \cite{C-T} or Ratcliffe-Tschantz \cite{R-T}, and let $\Kappa$ 
be the kernel of the action of $\Gamma$ on $V = \mathrm{Span}(\Nu)$. 
Then the following are equivalent: 
\begin{enumerate} 
\item The space group extension $1\to\Nu\to\Gamma\to\Gamma/\Nu\to 1$ splits.
\item The space group extension $1\to\Nu\to\Gamma\to\Gamma/\Nu\to 1$ splits orthogonally.   
\item The structure group $\Gamma/\Nu\Kappa$ fixes a point of $V/\Nu$. 
\item The structure group $\Gamma/\Nu\Kappa$ has order 1 or 2, and if $\Gamma/\Nu\Kappa$ 
has order 2, then either $\Nu$ or $\Gamma/\Kappa$ is an infinite dihedral group. 
\item The structure group $\Gamma/\Nu\Kappa$ has order 1 or 2, and if $\Gamma/\Nu\Kappa$ 
has order 2, then either $V/\Nu$ or $V/(\Gamma/\Kappa)$ is a closed interval.  
\item The index of $\Nu\Kappa$ in $\Gamma$ is 1 or 2, and if $[\Gamma:\Nu\Kappa] =2$, 
then either the generic fiber of the corresponding Seifert fibration or the base 
of the corresponding co-Seifert fibration is a closed interval. 
\end{enumerate}
\end{theorem}
\begin{proof}
Suppose that the space group extension $1\to\Nu\to\Gamma\to\Gamma/\Nu\to 1$ splits. 
If $\Nu$ is an infinite dihedral group, 
then the space group extension splits orthogonally by Theorem 8,  
since the structure group $\Gamma/\Nu\Kappa$ fixes the midpoint of the closed interval $V/\Nu$. 
Moreover $\Gamma/\Nu\Kappa$ has order 1 or 2, 
since the group of isometries of $V/\Nu$ has order 2,  
and $\Gamma/\Nu\Kappa$ acts effectively on $V/\Nu$ by Theorem 6. 
Therefore statements (1)-(6) are equivalent if $\Nu$ is an infinite dihedral group.
Hence, we may assume that $\Nu$ is an infinite cyclic group. 

Now $\Gamma$ has a subgroup $\Sigma$ such that $\Gamma = \Nu\Sigma$ 
and $\Nu\cap\Sigma =\{I\}$.  
Let $W = \mathrm{Span}(\Sigma)$, and let $\Pi$ be the point group of $\Gamma$.  
By Lemma 3, we have that $W$ is $\Pi$-invariant, $\dim(W) = 2$, and $V\cap W=\{0\}$. 

Assume first that $\Gamma$ belongs to either the tetragonal or hexagonal families 
(IT numbers 75-194).  Then $E^3$ has a unique 1-dimensional $\Pi$-invariant vector subspace, 
namely $V$.  Hence $E^3$ has a unique 2-dimensional $\Pi$-invariant vector subspace.  
Therefore $W = V^\perp$, and so the space group extension splits orthogonally. 

Assume next that $\Gamma$ belongs to the orthorhombic family (IT numbers 16-74). 
Then $E^3$ has exactly three $\Pi$-invariant 1-dimensional vector subspaces; 
they are mutually orthogonal, and one of them is $V$. 
Hence $E^3$ has exactly three $\Pi$-invariant 2-dimensional vector subspaces; 
they are mutually orthogonal, two of them contain $V$, and the third is $V^\perp$. 
As $W\cap V = \{0\}$, we must have $W = V^\perp$, 
and so the space group extension splits orthogonally. 

Assume next that $\Gamma$ belongs to the monoclinic family (IT numbers 3-15).  
Then $E^3$ has a primary 1-dimensional $\Pi$-invariant subspace $U$, 
and infinitely many other 1-dimensional $\Pi$-invariant subspaces 
all of which are subspaces of $U^\perp$. 
Hence $E^3$ has a primary 2-dimensional $\Pi$-invariant subspace $U^\perp$ 
and infinitely many other 2-dimensional $\Pi$-invariant subspaces 
all of which contain $U$. 
Thus if $V = U$, then $W = V^\perp$, since $W\cap V = \{0\}$. 
Hence the space group extension splits orthogonally. 
Note that $V = U$ if and only if the base of the corresponding Seifert fibration is either 
a torus, indicated by $\circ$, or a pillow, indicated by 2222. 
Suppose $V \neq U$. 
If  the structure group $\Gamma/\Nu\Kappa$ fixes a point of $V/\Nu$, 
then the space group extension splits orthogonally by Corollary 3. 
If  the structure group $\Gamma/\Nu\Kappa$ does not fix a point of $V/\Nu$, 
we have to prove that the space group extension does not split. 
There are 10 cases, and these have to be considered on a case by case basis. 
The first two cases are considered in Examples 8 and 9. 
The remaining eight cases are handled by similar arguments. 

Finally assume that $\Gamma$ belongs to the triclinic family (IT number 1-2). 
Then the space group extension splits orthogonally by Corollary 3, 
since the structure group $\Gamma/\Nu\Kappa$ fixes a  point of $V/\Nu$. 
Thus (1) and (2) are equivalent. 
The statements (2) and (3) are equivalent by Corollary 3. 

We have that $\Kappa = \Nu^\perp$, and so by Theorem 7, 
the group $\Gamma/\Nu\Kappa$ is either a finite cyclic group generated by a rotation of $V/\Nu$ 
or a finite dihedral group generated by a rotation and a reflection of $V/\Nu$. 
A nontrivial rotation of $V/\Nu$ has no fixed points. 
Hence $\Gamma/\Nu\Kappa$ fixes a point of $V/\Nu$ if and only if either $\Gamma/\Nu\Kappa$ 
has order 1 or $\Gamma/\Nu\Kappa$ has order 2 and $\Gamma/\Nu\Kappa$ acts 
on $V/\Nu$ by a reflection. 
By Theorem 6, we have that $(V/\Nu)/(\Gamma/\Nu\Kappa)=V/(\Gamma/\Kappa)$.
Therefore (3)-(5) are equivalent. 

The generic fiber of the corresponding Seifert fibration is $V/\Nu$ by Theorem 2, 
and the base of the corresponding co-Seifert fibration is isometric to $V/(\Gamma/\Kappa)$ 
by Theorem 7. 
Therefore (5) and (6) are equivalent. 
\end{proof}
\noindent{\bf Remark 4.}  The splitting of the space group extensions in Table 1 of \cite{R-T} 
was determined by brute force computer calculations.  
See Example 8 for an illustration of our methods. 
What is new in Theorem 10 is the equivalence of statements (1) - (3).  
Statement (6) of Theorem 10 is phrased in the terminology of Table 1 of \cite{R-T}.    

\vspace{.15in}
\noindent{\bf Example 8.}
Let $\Gamma$ be the 3-space group with IT number 5 in Table 1B of \cite{B-Z}. 
Then $\Gamma =\langle t_1,t_2,t_3,A\rangle$ where $t_i = e_i+I$ for $i=1,2,3$ are 
the standard translations, and 
$$A = \left(\begin{array}{rrr} 
0 & -1 & 0 \\
-1 & 0 & 0 \\
0 & 0 & -1
\end{array}\right).$$
The group $\Nu = \langle t_1t_2\rangle$ is a complete normal subgroup of $\Gamma$. 
Let $\Kappa$ be the kernel of the action of $\Gamma$ on $V = \span(\Nu) = \span\{e_1+e_2\}$. 
Then $\Kappa =\Nu^\perp= \langle t_1t_2^{-1},t_3\rangle$. 
The structure group $\Gamma/\Nu\Kappa$ is a dihedral group of order 4 
generated by $\Nu\Kappa t_1$ and $\Nu\Kappa A$. 
As 
$t_1 = \frac{1}{2}(e_1+e_2) +\frac{1}{2}(e_1-e_2)+I,$
we have that $\Nu\Kappa t_1$ acts on the circle $V/\Nu$ as a halfturn. 
The element $\Nu\Kappa A$ acts on $V/\Nu$ as a reflection. 

Suppose that the space group extension $1\to\Nu\to\Gamma\to\Gamma/\Nu\to 1$ splits. 
Then there exists a subgroup $\Sigma$ of $\Gamma$ such that $\Gamma = \Nu\Sigma$ 
and $\Nu\cap\Sigma = \{I\}$. 
Hence there exists unique integers $a$ and $b$ such that 
$(t_1t_2)^at_1$ and  $(t_1t_2)^bA$ are in $\Sigma$. 
The mapping $(b+B) \mapsto \Nu(b+B)$ is an isomorphism from $\Sigma$ onto $\Gamma/\Nu$. 
We have that 
$$\Nu A (\Nu t_1)(\Nu A)^{-1} = \Nu At_1A^{-1} = \Nu t_2^{-1} = \Nu (t_1t_2)t_2^{-1} = \Nu t_1.$$
Hence we must have 
$$((t_1t_2)^bA)((t_1t_2)^at_1)((t_1t_2)^bA)^{-1} = (t_1t_2)^at_1,$$
and so we must have 
$$A(t_1^{a+1}t_2^a)A^{-1} = t_1^{a+1}t_2^a,$$
which implies that $-a-1 = a$, and so $a = -1/2$, which is a contradiction. 
Therefore the space group extension does not split. 
This example corresponds to the Seifert fibration $(*_1*_1)$ in Table 1 of \cite{R-T}.

\vspace{.15in}
\noindent{\bf Remark 5.} 
If $\Nu$ is a 1- or 2-dimensional, complete, normal subgroup 
of a 3-space group $\Gamma$, then the space group extension 
$1\to\Nu\to\Gamma\to\Gamma/\Nu\to 1$ splits 
if and only if the space group extension affinely equivalent to it 
in Table 1 of  \cite{R-T} splits, since splitting is an algebraic property. 
Note that splitting orthogonally is a geometric property, and it may 
not be preserved under affine equivalence. 

\begin{lemma} 
Let $\Nu$ be a complete normal subgroup of an $n$-space group $\Gamma$, 
and let $\Kappa$ be the kernel of the action of $\Gamma$ on $V = \span(\Nu)$. 
Suppose that the space group extension $1\to\Nu\to\Gamma\to\Gamma/\Nu\to 1$ splits. 
Let  $b+B$ be an element of $\Gamma$ such that $\Nu(b+B)$ has finite order in $\Gamma/\Nu$. 
Then the element $\Nu\Kappa(b+B)$ of $\Gamma/\Nu\Kappa$ fixes a point of $V/\Nu$. 
\end{lemma}
\begin{proof}
Let $\Sigma$ be a subgroup of $\Gamma$ such that $\Gamma = \Nu\Sigma$ 
and $\Nu\cap\Sigma = \{I\}$. 
Then there exists a unique element $c+C$ of $\Sigma$ such that $\Nu(b+B) = \Nu(c+C)$. 
The element $c+C$ has the same order as $\Nu(b+B)$. 
Hence $c+C$ fixes a point $x$ of $E^n$.  
Write $x = v+w$ with $v\in V$ and $w\in V^\perp$. 
Then $\Nu\Kappa(c+C)$ fixes the point $\Nu v$ of  $V/\Nu$. 
Hence $\Nu\Kappa(b+B)$ fixes the point $\Nu v$ of  $V/\Nu$. 
\end{proof}

\noindent{\bf Example 9.}
Let $\Gamma$ be the 3-space group with IT number 7 in Table 1B of \cite{B-Z}. 
Then $\Gamma = \langle t_1,t_2,t_3,\alpha\rangle$ 
where $t_i = e_i+I$ for $i=1,2,3$ are the standard translations, 
and $\alpha = \frac{1}{2}e_3+\diag(1,-1,1)$.  
The group $\Nu = \langle t_3\rangle$ is a complete normal subgroup of $\Gamma$. 
Let $\Kappa$ be the kernel of the action of $\Gamma$ on $V = \span(\Nu) = \span\{e_3\}$. 
Then $\Kappa = \Nu^\perp = \langle t_1,t_2\rangle$. 
The structure group $\Gamma/\Nu\Kappa$ is generated by $\Nu\Kappa\alpha$, 
and $\Nu\Kappa\alpha$ acts on the circle $V/\Nu$ as a halfturn. 
The element $\Nu\alpha$ has order two in $\Gamma/\Nu$. 
Therefore the space group extension $1\to\Nu\to\Gamma\to\Gamma/\Nu\to 1$ 
does not split by Lemma 4. 
This example corresponds to the Seifert fibration $(*:*:)$ in Table 1 of \cite{R-T}. 

\section{The Isometry Group of a Compact Flat Orbifold} 

Let $\Gamma$ be an $n$-space group, and let $\mathrm{Isom}(E^n/\Gamma)$ be the group of isometries 
of the flat orbifold $E^n/\Gamma$. 
In this section, we give a geometric description of $\mathrm{Isom}(E^n/\Gamma)$ in terms of Lie group theory. 

Let $N_E(\Gamma)$ be the normalizer of $\Gamma$ in $\mathrm{Isom}(E^n)$. 
Define 
$$\Phi: N_E(\Gamma) \to \mathrm{Isom}(E^n/\Gamma)$$
by $\Phi(\eta)= \eta_\star$ where $\eta_\star(\Gamma x) = \Gamma\eta(x)$ for each $x\in E^n$. 

\begin{lemma} 
Let $\Gamma$ be an $n$-space group, and let $N_E(\Gamma)$ be the normalizer of $\Gamma$ in $\mathrm{Isom}(E^n)$. 
Then $\Phi: N_E(\Gamma) \to \mathrm{Isom}(E^n/\Gamma)$ is an epimorphism with kernel $\Gamma$. 
\end{lemma}
\begin{proof}  The map $\Phi$ is an epimorphism by Theorem 13.2.6 of \cite{R}.  
The group $\Gamma$ is obviously in the kernel of $\Phi$.  
Suppose $\eta \in \mathrm{Ker}(\Phi)$. 
Let $D$ be a fundamental domain for $\Gamma$ in $E^n$, and suppose $a\in D$. 
Then there exists $\gamma \in \Gamma$ such that $\eta(a) = \gamma(a)$. 
Hence $\eta(a) \in \gamma(D)$. 
As $\eta$ is continuous, there exists $r> 0$ such that the open ball $B(a,r)$, centered at $a$ of radius $r$, 
is contained in $D$ and $\eta(B(a,r)) \subset \gamma(D)$. 
Then $\eta(x) = \gamma(x)$ for all $x\in B(a,r)$.  Hence $\eta = \gamma$. 
Therefore $\mathrm{Ker}(\Phi) = \Gamma$. 
\end{proof}

The group $\mathrm{Isom}(E^n)$ is a topological group with the subspace topology inherited from 
the space $C(E^n,E^n)$ of continuous self-maps of $E^n$ with the compact-open topology by Theorem 5.2.2 of \cite{R};
moreover $\mathrm{Isom}(E^n)$ is a Lie group by Theorem 5.2.4 of \cite{R}. 

\begin{lemma} 
Let $\Gamma$ be an $n$-space group, and let $N_E(\Gamma)$ be the normalizer of $\Gamma$ in $\mathrm{Isom}(E^n)$. 
Then $N_E(\Gamma)$ is a closed subgroup of $\mathrm{Isom}(E^n)$,  
and therefore $N_E(\Gamma)$ is a Lie subgroup of $\mathrm{Isom}(E^n)$
\end{lemma}
\begin{proof} 
Suppose $\eta_i\to \eta$ in $\mathrm{Isom}(E^n)$ with $\eta_i \in N_E(\Gamma)$ for each $i$. 
Let $\gamma\in \Gamma$.  Then there exists $\gamma_i\in\Gamma$ such that $\eta_i\gamma\eta_i^{-1} = \gamma_i$ for each $i$. 
Now $\eta_i\gamma\eta_i^{-1} \to \eta\gamma\eta^{-1}$.  Hence $\gamma_i \to \eta\gamma\eta^{-1}$. 
The convergent sequence $\{\gamma_i\}$ is eventually constant, 
since $\Gamma$ is a closed discrete subset of $\mathrm{Isom}(E^n)$ by Lemma 3 of \S5.3 of \cite{R}. 
Therefore $\eta\gamma\eta^{-1} \in \Gamma$.  Hence $\eta\Gamma\eta^{-1} \subseteq \Gamma$. 
Likewise $\eta^{-1}\Gamma \eta \subseteq \Gamma$.  Hence $\Gamma \subseteq \eta\Gamma\eta^{-1}$. 
Therefore $\eta \in N_E(\Gamma)$.  Thus $N_E(\Gamma)$ is closed. 
Hence $N_E(\Gamma)$ is a Lie subgroup of $\mathrm{Isom}(E^n)$, 
since every closed subgroup of a Lie group is a Lie subgroup. 
\end{proof}

\begin{lemma} 
Let $\Gamma$ be an $n$-space group, and let $N_E(\Gamma)$ be the normalizer of $\Gamma$ in $\mathrm{Isom}(E^n)$. 
Then the quotient topological group $N_E(\Gamma)/\Gamma$ is compact,  
and the quotient map $\pi: N_E(\Gamma) \to N_E(\Gamma)/\Gamma$ is a covering projection. 
Therefore $N_E(\Gamma)/\Gamma$ is a Lie group with $\pi$ a smooth map. 
\end{lemma}
\begin{proof}
Let $\Tau$ be the group of translations of $\Gamma$, 
and let $L = \{b : b+I \in \Tau\}$. 
Let $a+A \in N_E(\Gamma)$, and let $b+I \in \Tau$. 
Then $(a+A)(b+I)(a+A)^{-1} = Ab+I$.  
Hence $a+A\in N_E(\Tau)$ and $AL = L$. 
As $A$ preserves the length of vectors of $E^n$, 
there are only finitely many possibilities for the image of a basis of the free abelian group $L$. 
Hence the image of the homomorphism from $N_E(\Gamma)$ to $O(n)$, defined by $a+A \mapsto A$, is finite, 
and so the kernel, $N_E(\Gamma)\cap \mathrm{Trans}(E^n)$, has finite index in $N_E(\Gamma)$. 
The group $\mathrm{Trans}(E^n)$, of translations of $E^n$, is a closed subgroup of $\mathrm{Isom}(E^n)$, 
and so $N_E(\Gamma)\cap\mathrm{Trans}(E^n)$ is a closed subgroup of $N_E(\Gamma)$. 

Now $N_E(\Tau)\cap \mathrm{Trans}(E^n) = \mathrm{Trans}(E^n)$. 
Hence $\mathrm{Trans}(E^n)$ is a closed subgroup of $N_E(\Tau)$ of finite index. 
The topological group $\mathrm{Trans}(E^n)/\Tau$ is an $n$-dimensional torus,  
and the quotient map from $\mathrm{Trans}(E^n)$ to $\mathrm{Trans}(E^n)/\Tau$ is a covering projection. 
As $\mathrm{Trans}(E^n)/\Tau$  is a closed connected subgroup of $N_E(\Tau)/\Tau$ of finite index, 
we deduce that $N_E(\Tau)/\Tau$ is compact, and the quotient map from  $N_E(\Tau)$ to $N_E(\Tau)/\Tau$ 
is a covering projection. 
Now $N_E(\Gamma)$ is a closed subgroup of $N_E(\Tau)$. 
Hence $N_E(\Gamma)/\Tau$ is a closed subgroup of $N_E(\Tau)/\Tau$. 
Therefore $N_E(\Gamma)/\Tau$ is compact, and the quotient map from $N_E(\Gamma)$ to $N_E(\Gamma)/\Tau$ 
is a covering projection. 
The group $\Gamma/\Tau$ is a finite normal subgroup of $N_E(\Gamma)/\Tau$ 
that acts freely on $N_E(\Gamma)/\Tau$. 
Hence $N_E(\Gamma)/\Gamma = (N_E(\Gamma)/\Tau)/(\Gamma/\Tau)$ is compact,  
and the quotient map from $N_E(\Gamma)/\Tau$  to $N_E(\Gamma)/\Gamma$ is a covering projection. 
Therefore, the quotient map $\pi: N_E(\Gamma) \to N_E(\Gamma)/\Gamma$ is a covering projection. 
Hence $N_E(\Gamma)/\Gamma$ is a Lie group such that $\pi$ is a smooth map. 
\end{proof}

\begin{lemma} 
Let $\Gamma$ be an $n$-space group, and let $N_E(\Gamma)$ be the normalizer of $\Gamma$ in $\mathrm{Isom}(E^n)$. 
Then $\mathrm{Isom}(E^n/\Gamma)$ is a compact Lie group with $\Phi: N_E(\Gamma) \to \mathrm{Isom}(E^n/\Gamma)$ 
a smooth covering projection. 
\end{lemma}
\begin{proof}
We begin by showing that $\Phi$ is continuous. 
Suppose $\eta_i \to \eta$  in $N_E(\Gamma)$.  Then $\eta_i(x) \to \eta(x)$ for each $x\in E^n$. 
Hence $\Gamma\eta_i(x) \to \Gamma\eta(x)$ in $E^n/\Gamma$ for each $x\in E^n$, 
since the quotient map from $E^n$ to $E^n/\Gamma$ is continuous. 
Hence $\Phi(\eta_i)(\Gamma x) \to \Phi(\eta)(\Gamma x)$ in $E^n/\Gamma$ for each $\Gamma x$ in $E^n/\Gamma$. 
Therefore $\Phi(\eta_i) \to \Phi(\eta)$ in $\mathrm{Isom}(E^n/\Gamma)$ by Theorem 5.2.1 of \cite{R}. 
Hence $\Phi: N_E(\Gamma) \to \mathrm{Isom}(E^n/\Gamma)$ is continuous. 

By Lemma 5, we have that $\Phi: N_E(\Gamma) \to \mathrm{Isom}(E^n/\Gamma)$ is an epimorphism with kernel $\Gamma$.  
Hence $\Phi$ induces a continuous isomorphism 
$\overline{\Phi}: N_E(\Gamma)/\Gamma\to \mathrm{Isom}(E^n/\Gamma)$ such that $\Phi = \overline{\Phi}\pi$. 
By Lemma 7, we have that $N_E(\Gamma)/\Gamma$ is compact. 
Hence $\overline{\Phi}$ is a homeomorphism, since $\mathrm{Isom}(E^n/\Gamma)$ is Hausdorff. 
Therefore $\mathrm{Isom}(E^n/\Gamma)$ is a compact Lie group with $\overline{\Phi}$ an isomorphism of Lie groups. 
By Lemma 7, we have that $\pi: N_E(\Gamma) \to N_E(\Gamma)/\Gamma$ is a smooth covering projection.
Therefore $\Phi$ is a smooth covering projection. 
\end{proof}

\begin{lemma} 
Let $Z(\Gamma)$ be the center of a space group $\Gamma$, and let $\Pi$ be the point group of $\Gamma$
Then $\mathrm{Span}(Z(\Gamma)) =\mathrm{Fix}(\Pi) = \{x\in E^n: Ax = x\ \hbox{for each}\ A \in \Pi\}$. 
\end{lemma}
\begin{proof}
Let $\Tau$ be the group of translations of $\Gamma$.  Then $Z(\Gamma) \subseteq \Tau$ by Theorem 6 of \cite{R-T}. 
Suppose $a+A \in \Gamma$ and $b+I \in \Tau$.  Then $(a+A)(b+I)(a+A)^{-1} = Ab+I$. 
Hence, if $b+I\in Z(\Gamma)$, then $b\in \mathrm{Fix}(A)$.  Therefore $\mathrm{Span}(Z(\Gamma)) \subseteq \mathrm{Fix}(\Pi)$. 

To prove the reverse containment, we take an affine representation of $\Gamma$, 
so that $\Tau$ is generated by $e_1+I, \ldots, e_n+I$ where $e_1,\ldots, e_n$ are the standard basis vectors of $E^n$. 
Then $\Pi$ is a finite subgroup of $\mathrm{GL}(n, \integers)$. 
Suppose $x \in \mathrm{Fix}(\Pi)$.  Then $x$ is a solution of the linear system $\{(A-I)x= 0: A \in \Pi\}$. 
As the entries of $A$ are in $\integers$ for each $A \in \Pi$, 
the row reduced echelon form of the standard matrix of the system $\{(A-I)x= 0: A \in \Pi\}$ has entries in $\rationalnos$. 
Therefore, the solution space $V$ of the linear system $\{(A-I)x= 0: A \in \Pi\}$ 
has a basis  of vectors $v_1,\ldots, v_k$ in $\rationalnos^n$.  By multiplying $v_i$ by the least common multiple 
of the denominators of the entries of $v_i$, we may assume that $v_i \in \integers^n$ for each $i$. 
Then $v_i+I \in \Tau$ for each $i$, and so $v_i+I \in Z(\Gamma)$ for each $i$.  
Therefore $\mathrm{Span}(Z(\Gamma)) =\mathrm{Fix}(\Pi)$. 
\end{proof}

\begin{lemma} 
Let $\Gamma$ be an $n$-space group, let $Z(\Gamma)$ be the center of $\Gamma$, and let $V =\mathrm{Span}(Z(\Gamma))$. 
Let $C_E(\Gamma)$ and $N_E(\Gamma)$ be the centralizer and normalizer of $\Gamma$ in $\mathrm{Isom}(E^n)$, respectively. 
Then $C_E(\Gamma) = \{v+I: v\in V\}$, and $C_E(\Gamma)$ is the connected component of $N_E(\Gamma)$ containing $I$, 
and $C_E(\Gamma)\Gamma/\Gamma$ is the connected component of $N_E(\Gamma)/\Gamma$ containing the identity. 
\end{lemma}
\begin{proof}
Let $\Tau$ be the group of translations in $\Gamma$. 
Let $a+I \in \Tau$ and $b+B\in C_E(\Gamma)$. 
Then $(b+B)(a+I)(b+B)^{-1} = a+I$, implies $Ba = a$ for all $a+I \in \Tau$. 
Hence $B = I$.  
Thus all the elements of $C_E(\Gamma)$ are translations. 

Let $\Pi$ be the point group of $\Gamma$. 
Let $a+A\in \Gamma$ and let $b+I $ be a translation of $E^n$. 
Then $(a+A)(b+I)(a+A)^{-1} = Ab+I$. 
Hence $b+I \in C_E(\Gamma)$ if and only if $b\in \mathrm{Fix}(\Pi)$. 
Therefore $C_E(\Gamma) = \{v+I: v\in V\}$ by Lemma 9. 

Let $a+A\in \Gamma$, and let $b+I \in N_E(\Gamma)$. 
Then $(b+I)(a+A)(-b+I) \in \Gamma$.  Hence $b+a-Ab+A \in \Gamma$, 
and so $(b-Ab+I)(a+A) \in \Gamma$.  Therefore $(I-A)b+I\in \Tau$. 
From the proof of Lemma 7, we have that $N_E(\Gamma)\cap\mathrm{Trans}(E^n)$
is a closed subgroup of $N_E(\Gamma)$ of finite index. 
Hence $N_E(\Gamma)\cap\mathrm{Trans}(E^n)$ is an open subgroup of $N_E(\Gamma)$. 
Since $\Pi$ is finite and Lie groups are locally connected, 
there is an open connected neighborhood $U$ of $I$ in $N_E(\Gamma)$ such that 
if $b+B\in U$, then $B = I$ and $|b|$ is small enough so that if $(I-A)b+I\in \Tau$, then $(I-A)b = 0$ for every $A\in \Pi$. 
If $b+I\in U$, then $b\in \mathrm{Fix}(\Pi)$ . 
Hence $b \in V$ by Lemma 9. 
The set $U$ generates (in the group sense) the connected component $C$ of $N_E(\Gamma)$ containing $I$ 
by Proposition  2.16 of \cite{A}. 
Therefore $C \subseteq \{v+I : v \in V\}$.  Moreover $\{v+I: v\in V\} \subseteq C$, since $\{v+I: v\in V\}$ is connected. 
Therefore $C = \{v+I:v\in V\}$. 
Thus $C = C_E(\Gamma)$. 

Let $\overline{C}$ be the connected component of $N_E(\Gamma)/\Gamma$ containing the identity. 
The quotient map $\pi: N_E(\Gamma) \to N_E(\Gamma)/\Gamma$ is continuous, 
and so $\pi(C)$ is a connected subset of $N_E(\Gamma)/\Gamma$ that contains the identity. 
Hence $\pi(C) \subseteq \overline{C}$. 
The component $\overline{C}$ is path connected, 
since Lie groups are locally path connected. 
Let $\chi \in N_E(\Gamma)/\Gamma$.  Then there is a path $\alpha: [0,1] \to \overline{C}$ 
from the identity of $N_E(\Gamma)/\Gamma$ to $\chi$. 
The quotient map $\pi: N_E(\Gamma) \to N_E(\Gamma)/\Gamma$ is a covering projection by Lemma 7.  
Hence $\alpha$ lifts to a path $\tilde\alpha: [0,1] \to N_E(\Gamma)$ from $I$ to an element $\tilde\chi \in C$. 
Then $\pi(\tilde\chi) = \chi$. Hence $\pi(C)=\overline{C}$.  
Thus $\overline{C} = C\Gamma/\Gamma$. 
\end{proof}

\begin{theorem}  
Let $\Gamma$ be an $n$-space group, with center $Z(\Gamma)$, and let $V =\mathrm{Span}(Z(\Gamma))$. 
Then $\mathrm{Isom}(E^n/\Gamma)$ is a compact Lie group whose connected component of the identity 
is a torus isomorphic to $V/Z(\Gamma)$ via the effective action of $V/Z(\Gamma)$ on $E^n/\Gamma$ 
defined by $(Z(\Gamma)v)(\Gamma x) = \Gamma(v+x)$ for each $v\in V$ and $x\in E^n$. 
\end{theorem}
\begin{proof}
The center $Z(\Gamma)$ is a complete normal subgroup of $\Gamma$ by Theorem 6 of \cite{R-T}. 
The flat orbifold $E^n/\Gamma$ geometrically fibers over the flat orbifold $(E^n/V)/(\Gamma/Z(\Gamma))$ 
with generic fiber $V/Z(\Gamma)$. 
The torus $V/Z(\Gamma)$ acts fiberwise on $E^n/\Gamma$ by the formula $(Z(\Gamma)v)(\Gamma x) = \Gamma(v+x)$. 
The torus $V/Z(\Gamma)$ acts effectively on each generic fiber $(V+x)/Z(\Gamma)$, 
and so the action of $V/Z(\Gamma)$ on $E^n/\Gamma$ is effective. 
The action is continuous, and so the monomorphism $\Psi: V/Z(\Gamma) \to \mathrm{Isom}(E^n/\Gamma)$, 
defined by $\Psi(Z(\Gamma)v)(\Gamma x) = \Gamma(v+x)$, 
maps $V/Z(\Gamma)$ homeomorphically onto a closed connected subgroup of $\mathrm{Isom}(E^n/\Gamma)$. 

By Lemma 5, we have that 
$\Phi: N_E(\Gamma) \to \mathrm{Isom}(E^n/\Gamma)$ is an epimorphism with kernel $\Gamma$. 
By Lemma 6, we have that $N_E(\Gamma)$ is a Lie group. 
By Lemma 8, we have that $\mathrm{Isom}(E^n/\Gamma)$ is a compact Lie group 
with $\Phi:N_E(\Gamma) \to \mathrm{Isom}(E^n/\Gamma)$ a smooth covering projection. 
Hence $\Phi$ induces an isomorphism $\overline{\Phi}: N_E(\Gamma)/\Gamma \to \mathrm{Isom}(E^n/\Gamma)$ of Lie groups 
such that $\Phi = \overline{\Phi}\pi$ where $\pi: N_E(\Gamma) \to N_E(\Gamma)/\Gamma$ is the quotient map. 

Let $C, \overline{C}, K$ be the connected components of $N_E(\Gamma), N_E(\Gamma)/\Gamma, \mathrm{Isom}(E^n/\Gamma)$, 
containing the identity, respectively. 
Then $C = \{v+I: v\in V\}$ and $\overline{C} = \pi(C)$ by Lemma 10, and 
$\Phi(C) = \overline{\Phi}\pi(C) = \overline{\Phi}(\overline{C}) = K.$ 
As $C = \{v+I: v\in V\}$ and $\Phi(v+I)(\Gamma x) = \Gamma(v+x)$ for each $v\in V$ and $x \in E^n$, 
we deduce that $\Psi(V/Z(\Gamma)) = K$. 
Therefore $K$ is a torus isomorphic to $V/Z(\Gamma)$ via the action of $V/Z(\Gamma)$ on $E^n/\Gamma$ 
defined by $(Z(\Gamma)v)(\Gamma x) = \Gamma(v+x)$. 
\end{proof} 

\begin{corollary} 
If $\Gamma$ is an $n$-space group, 
then the rank of the Lie group $\mathrm{Isom}(E^n/\Gamma)$ is the first Betti number of $\Gamma$. 
\end{corollary}
\begin{proof}
The rank of $\mathrm{Isom}(E^n/\Gamma)$ is the dimension of $K$ by Theorem 11. 
Hence the rank of $\mathrm{Isom}(E^n/\Gamma)$ is the rank of $Z(\Gamma)$. 
The rank of $Z(\Gamma)$ is the first Betti number $\beta_1$ of $\Gamma$ by Theorem 6 of \cite{R-T}. 
Hence the rank of $\mathrm{Isom}(E^n/\Gamma)$ is $\beta_1$.
\end{proof}

\begin{corollary} 
If $\Gamma$ is an $n$-space group, then the following are equivalent:
\begin{enumerate}
\item The group $\mathrm{Isom}(E^n/\Gamma)$ is finite. 
\item The center of $\Gamma$ is trivial. 
\item The group $\Gamma/[\Gamma,\Gamma]$ is finite.
\end{enumerate}
\end{corollary}
\begin{proof}
Statements (1) and (2) are equivalent by Theorem 11.  
Statements (2) and (3) are equivalent by Theorem 6 of \cite{R-T}. 
\end{proof}

\noindent{\bf Remark 6.} Let $\Gamma$ be an $n$-space group.  
The group $\Gamma/[\Gamma, \Gamma]$ is the abelianization of $\Gamma$. 
For a list of the abelianizations of all the $n$-space groups for $n=1,2,3$, 
see our paper \cite{R-T-Ab}. 
In particular, the 2-space groups with infinite abelianizations are the space groups 
of the torus, the annulus, the Klein bottle, and the M\"obius band. 
We will describe the groups of isometries of all the compact, connected, flat 2-orbifolds in detail in \S 10. 

\begin{corollary} 
If $\Gamma$ is a $\integers$-irreducible $n$-space group, then $\mathrm{Isom}(E^n/\Gamma)$ is finite.
\end{corollary}
\begin{proof}
A $\integers$-irreducible $n$-space group $\Gamma$ has no proper, complete, normal subgroups by Theorem 11 of \cite{R-T}. 
Hence $Z(\Gamma) =\{I\}$, and so  $\mathrm{Isom}(E^n/\Gamma)$ is finite by Corollary 5. 
\end{proof}

\begin{corollary} 
Let $\Gamma$ be an $n$-space group, and let $N_E(\Gamma)$ be the normalizer of $\Gamma$ in $\mathrm{Isom}(E^n)$. 
Then $N_E(\Gamma)$ is a discrete subgroup of $\mathrm{Isom}(E^n)$ if and only if $Z(\Gamma) = \{I\}$. 
If $Z(\Gamma) =\{I\}$, then $N_E(\Gamma)$ is an $n$-space group with 
$E^n/N_E(\Gamma)$ isometric to $(E^n/\Gamma)/\mathrm{Isom}(E^n/\Gamma)$. 
\end{corollary}
\begin{proof}
We have that $N_E(\Gamma)$ is a Lie group by Lemma 6. Hence $N_E(\Gamma)$ is locally connected. 
Therefore $N_E(\Gamma)$ is discrete if and only if the connected component of $N_E(\Gamma)$ containing $I$ is $\{I\}$. 
The connected component of $N_E(\Gamma)$ containing $I$ is $\{v+I: v \in \mathrm{Span}(Z(\Gamma))\}$ by Lemma 10. 
Hence $N_E(\Gamma)$ is discrete if and only if $Z(\Gamma) = \{I\}$. 

Now assume that $Z(\Gamma) =\{I\}$.  
Then $N_E(\Gamma)$ is a discrete subgroup of $\mathrm{Isom}(E^n)$. 
Therefore $\Gamma$ is a subgroup of $N_E(\Gamma)$ of finite index by Lemma 5. 
As $E^n/N_E(\Gamma)$ is a quotient of $E^n/\Gamma$, we have that $E^n/N_E(\Gamma)$ is compact. 
Hence $E^n/N_E(\Gamma)$ is an $n$-space group. 

The group $\mathrm{Isom}(E^n/\Gamma)$ is finite by Corollary 5. 
By Lemma 5, the epimorphism $\Phi: N_E(\Gamma) \to \mathrm{Isom}(E^n/\Gamma)$ induces an isometry 
from $(E^n/\Gamma)/(N_E(\Gamma)/\Gamma)$ to $(E^n/\Gamma)/\mathrm{Isom}(E^n/\Gamma)$. 
Moreover $(E^n/\Gamma)/(N_E(\Gamma)/\Gamma)$ is isometric to $E^n/N_E(\Gamma)$. 
Hence $E^n/N_E(\Gamma)$ is isometric to $(E^n/\Gamma)/\mathrm{Isom}(E^n/\Gamma)$. 
\end{proof}

Let $\Gamma$ be an $n$-space group.   The group of {\it Euclidean automorphisms} of $\Gamma$, 
denoted by $\mathrm{Aut}_E(\Gamma)$, is the group of all automorphisms of $\Gamma$ 
that are equal to conjugation by an isometry of $E^n$. 
The group $\mathrm{Inn}(\Gamma)$ of inner automorphisms of $\Gamma$ is a normal subgroup of $\mathrm{Aut}_E(\Gamma)$. 
The group of {\it Euclidean outer automorphisms} of $\Gamma$ is the group 
$$\mathrm{Out}_E(\Gamma) = \mathrm{Aut}_E(\Gamma)/\mathrm{Inn}(\Gamma).$$

Let $\zeta$ be an isometry of $E^n/\Gamma$.  Then $\zeta$ lifts to an isometry $\tilde\zeta$ of $E^n$ 
such that $\tilde\zeta\Gamma\tilde\zeta^{-1} = \Gamma$.  
The isometry $\tilde\zeta$ of $E^n$ determines an automorphism $\tilde\zeta_\ast$ of $\Gamma$ 
defined by $\tilde\zeta_\ast(\gamma) = \tilde\zeta\gamma\tilde\zeta^{-1}$. 
If $\tilde\zeta'$ is another lift of $\zeta$, then $\tilde\zeta'=\tilde\zeta\gamma$ for some $\gamma$ in $\Gamma$. 
Hence the element $\tilde\zeta_\ast\mathrm{Inn}(\Gamma)$ of $\mathrm{Out}_E(\Gamma)$ depends only on $\zeta$, 
and we have an epimorphism
$$\Omega: \mathrm{Isom}(E^n/\Gamma) \to \mathrm{Out}_E(\Gamma)$$
defined by $\Omega(\zeta)  = \tilde\zeta_\ast \mathrm{Inn}(\Gamma)$.

\begin{theorem} 
Let $\Gamma$ be an $n$-space group, and let $K$ be the connected component of the identity of 
the compact Lie group $\mathrm{Isom}(E^n/\Gamma)$.  
Then $K$ is the kernel of the epimorphism 
$\Omega: \mathrm{Isom}(E^n/\Gamma) \to \mathrm{Out}_E(\Gamma)$, 
and therefore $\mathrm{Out}_E(\Gamma)$ is a finite group. 
\end{theorem}
\begin{proof}
Define $\Theta: N_E(\Gamma) \to \mathrm{Aut}_E(\Gamma)$
by $\Theta(\eta) = \eta_\ast$ where $\eta_\ast(\gamma) = \eta\gamma\eta^{-1}$ for each $\gamma\in \Gamma$. 
Then $\Theta$ is an epimorphism that maps $\Gamma$ onto $\mathrm{Inn}(\Gamma)$.  
Hence $\Theta$ induces an epimorphism $\overline{\Theta}: N_E(\Gamma)/\Gamma \to \mathrm{Out}_E(\Gamma)$. 
Now the kernel of $\Theta$ is the centralizer $C_E(\Gamma)$ of $\Gamma$ in $\mathrm{Isom}(E^n)$. 
Hence $\mathrm{Ker}(\overline{\Theta})=C_E(\Gamma)\Gamma/\Gamma$. 
By Lemma 10, we have that 
$$C_E(\Gamma) = \{v+I : v \in \mathrm{Span}(Z(\Gamma))\}.$$ 
By Lemma 10, we have that $C_E(\Gamma)$ is the connected component $C$ of $N_E(\Gamma)$ 
containing $I$, and $C\Gamma/\Gamma$ is the connected component of $N_E(\Gamma)/\Gamma$ containing 
the identity.  Moreover, we have an isomorphism $\overline{\Phi}: N_E(\Gamma)/\Gamma \to \mathrm{Isom}(E^n/\Gamma)$ 
of Lie groups and $\Omega = \overline{\Theta}\overline{\Phi}^{-1}$.   Therefore 
$\mathrm{Ker}(\Omega) = \overline{\Phi}(\mathrm{Ker}(\overline{\Theta})) =\overline{\Phi}(C\Gamma/\Gamma) = K.$
\end{proof}

An {\it affinity} $\alpha$ of $E^n$ is a function $\alpha: E^n\to E^n$ 
for which there is an element $a\in E^n$ and a matrix $A \in \mathrm{GL}(n,\realnos)$ such that 
$\alpha(x) = a + Ax$ for all $x$ in $E^n$.  We write simply $\alpha = a + A$. 
The set $\mathrm{Aff}(E^n)$ of all affinities of $E^n$ is a group 
that contains $\mathrm{Isom}(E^n)$ as a subgroup.

Let $\Gamma$ be an $n$-space group.  An {\it affinity} of $E^n/\Gamma$ is a function 
$\beta: E^n/\Gamma \to E^n/\Gamma$ that lifts to an affinity $\tilde \beta$ of $E^n$, 
that is, $\beta(\Gamma x) = \Gamma\tilde\beta(x)$ for all $x$ in $E^n$.  
The set $\mathrm{Aff}(E^n/\Gamma)$ of all affinities of $E^n/\Gamma$ is a group 
that contains $\mathrm{Isom}(E^n/\Gamma)$ as a subgroup. 

Let $N_A(\Gamma)$ be the normalizer of $\Gamma$ in $\mathrm{Aff}(E^n)$. 
Define 
$$\Phi: N_A(\Gamma) \to \mathrm{Aff}(E^n/\Gamma)$$
by $\Phi(\eta)= \eta_\star$ where $\eta_\star(\Gamma x) = \Gamma\eta(x)$ for each $x\in E^n$. 

\begin{lemma} 
Let $\Gamma$ be an $n$-space group, and let $N_A(\Gamma)$ be the normalizer of $\Gamma$ in $\mathrm{Aff}(E^n)$. 
Then $\Phi: N_A(\Gamma) \to \mathrm{Aff}(E^n/\Gamma)$ is an epimorphism with kernel $\Gamma$. 
\end{lemma}
\begin{proof}  Let $\beta \in \mathrm{Aff}(E^n/\Gamma)$.  Then $\beta$ lifts to $\tilde \beta \in \mathrm{Aff}(E^n)$ 
such that $\beta(\Gamma x) = \Gamma\tilde\beta(x)$ for all $x$ in $E^n$.  We have that $\tilde \beta \in N_A(\Gamma)$ 
by the argument in the last paragraph of the proof of Theorem 13.2.6 of \cite{R}. 
Hence $\Phi$ is an epimorphism.  
The proof that $\Gamma = \mathrm{Ker}(\Phi)$ is the same as in the proof of Lemma 5. 
\end{proof}

Define $\Theta: N_A(\Gamma) \to \mathrm{Aut}(\Gamma)$ by $\Theta(\eta) = \eta_\ast$ 
where $\eta_\ast(\gamma) = \eta\gamma\eta^{-1}$ for each $\gamma\in\Gamma$. 
Let $C_A(\Gamma)$ be the centralizer of $\Gamma$ in $\mathrm{Aff}(E^n)$.  
The next lemma is in Gubler \cite{G}. 

\begin{lemma} 
Let $\Gamma$ be an $n$-space group, and let $N_A(\Gamma)$ be the normalizer of $\Gamma$ in $\mathrm{Aff}(E^n)$. 
Then $\Theta: N_A(\Gamma) \to \mathrm{Aut}(\Gamma)$ is an epimorphism with kernel $C_A(\Gamma)$. 
Moreover $C_A(\Gamma) = C_E(\Gamma)$, and $\Theta$ maps $N_E(\Gamma)$ onto $\mathrm{Aut}_E(\Gamma)$. 
\end{lemma}
\begin{proof}
By Bieberbach's Theorem, $\Theta$ is an epimorphism. 
Clearly $\mathrm{Ker}(\Theta) = C_A(\Gamma)$. 
Let $\alpha\in C_A(\Gamma)$.  Then $\alpha = a + A$ for some $a\in E^n$ and $A\in \mathrm{GL}(n,\realnos)$. 
Let $\Tau$ be the group of translations of $\Gamma$, and let $b+I \in \Tau$. 
Then $(a+A)(b+I)(a+A)^{-1} = b+I$ implies that $Ab= b$ for all $b+I \in \Tau$.  Hence $A = I$. 
Therefore $C_A(\Gamma) = C_E(\Gamma)$. 
\end{proof}

Let $\beta$ be an affinity of $E^n/\Gamma$.  Then $\beta$ lifts to an affinity $\tilde\beta$ of $E^n$ 
such that $\tilde\beta\Gamma\tilde\beta^{-1} = \Gamma$.  
The affinity $\tilde\beta$ of $E^n$ determines an automorphism $\tilde\beta_\ast$ of $\Gamma$ 
defined by $\tilde\beta_\ast(\gamma) = \tilde\beta\gamma\tilde\beta^{-1}$. 
If $\tilde\beta'$ is another lift of $\beta$, then $\tilde\beta'=\tilde\beta\gamma$ for some $\gamma$ in $\Gamma$. 
Hence the element $\tilde\beta_\ast\mathrm{Inn}(\Gamma)$ of $\mathrm{Out}(\Gamma)$ depends only on $\beta$, 
and we have an homomorphism
$\Omega: \mathrm{Aff}(E^n/\Gamma) \to \mathrm{Out}(\Gamma)$
defined by $\Omega(\beta)  = \tilde\beta_\ast \mathrm{Inn}(\Gamma)$. 

\begin{theorem} 
Let $\Gamma$ be an $n$-space group, and let $K$ be the connected component of the identity of 
the Lie group $\mathrm{Isom}(E^n/\Gamma)$.  
Then $\Omega: \mathrm{Aff}(E^n/\Gamma) \to \mathrm{Out}(\Gamma)$ is an epimorphism, 
and $\mathrm{Ker}(\Omega) = K$. 
\end{theorem}
\begin{proof}
By Bieberbach's theorem, every automorphism of $\Gamma$ is conjugation 
by an affinity of $E^n$,   
and so $\Omega: \mathrm{Aff}(E^n/\Gamma) \to \mathrm{Out}(\Gamma)$ is an epimorphism. 

Suppose $\beta \in \mathrm{Ker}(\Omega)$. Then $\beta$ lifts to $\tilde\beta  \in \mathrm{Aff}(E^n)$
such that $\tilde \beta_\ast \in \mathrm{Inn}(\Gamma)$.  Hence $\tilde\beta_\ast = \gamma_\ast$ for some $\gamma\in\Gamma$. 
Let $C_A(\Gamma)$ be the centralizer of $\Gamma$ in $\mathrm{Aff}(E^n)$. 
Then $\tilde\beta\gamma^{-1} \in C_A(\Gamma)$. 
Hence $\tilde\beta\in C_A(\Gamma)\Gamma$. 
We have that $C_A(\Gamma) = C_E(\Gamma)$ by Lemma 12, 
and so $\tilde\beta$ is an isometry of $E^n$. 
Hence $\beta$ is an isometry of $E^n/\Gamma$. 
Therefore $\mathrm{Ker}(\Omega) = K$ by Theorem 12. 
\end{proof}

\begin{corollary} 
Let $\Gamma$ be an $n$-space group, and let $K$ be the connected component of the identity of 
the Lie group $\mathrm{Isom}(E^n/\Gamma)$.  Then we have the following commutative diagrams 
of group homomorphisms. 

$$
\begin{array}{ccccccccc}
   &      &   1                  &          &           1           &         &   1                                          &         & \\
   &      & \downarrow &          & \downarrow  &  	    &  \downarrow                         &         & \\
1& \to & Z(\Gamma)  & \to     &  \Gamma       & \to   & \mathrm{Inn}(\Gamma)      &  \to  & 1 \\
   &      & \downarrow &          & \downarrow  &  	    &  \downarrow                         &         & \\
1& \to & C_E(\Gamma)  & \to & N_E(\Gamma) & \to   & \mathrm{Aut}_E(\Gamma) &  \to  & 1 \\
   &      & \downarrow &          & \downarrow  &  	    &  \downarrow                         &         & \\
 1& \to &  K & \to & \mathrm{Isom}(E^n/\Gamma) & \to   & \mathrm{Out}_E(\Gamma) &  \to  & 1 \\
    &      & \downarrow &          & \downarrow  &  	    &  \downarrow                         &         & \\
    &      &   1                  &          &           1           &         &   1                                          &         & \\
    &      &                       &          &                        &         &                                               &         & \\
   &      &   1                  &          &           1           &         &   1                                          &         & \\
   &      & \downarrow &          & \downarrow  &  	    &  \downarrow                         &         & \\
1& \to & Z(\Gamma)  & \to     &  \Gamma       & \to   & \mathrm{Inn}(\Gamma)      &  \to  & 1 \\
   &      & \downarrow &          & \downarrow  &  	    &  \downarrow                         &         & \\
1& \to & C_E(\Gamma)  & \to & N_A(\Gamma) & \to   & \mathrm{Aut}(\Gamma) &  \to  & 1 \\
   &      & \downarrow &          & \downarrow  &  	    &  \downarrow                         &         & \\
 1& \to &  K & \to & \mathrm{Aff}(E^n/\Gamma) & \to   & \mathrm{Out}(\Gamma) &  \to  & 1 \\
    &      & \downarrow &          & \downarrow  &  	    &  \downarrow                         &         & \\
    &      &   1                  &          &           1           &         &   1                                          &         &
\end{array}
$$
Moreover the rows and columns of both diagrams are exact. 
\end{corollary}

Let $\Gamma$ be an $n$-space group with point group $\Pi$. 
The {\it point group} of  $N_E(\Gamma)$ is the group
$$\Pi_E = \{B: b+B \in N_E(\Gamma)\ \hbox{for some}\ b \in E^n\}.$$
The group $\Pi_E$ is finite by the proof of Lemma 7. 
The {\it point group} of  $N_A(\Gamma)$ is the group
$$\Pi_A = \{B: b+B \in N_A(\Gamma)\ \hbox{for some}\ b \in E^n\}.$$
We have that $\Pi_E$ is a subgroup of $\Pi_A$. 
By Lemma 12, we may define
$$\Upsilon: \mathrm{Out}(\Gamma) \to \Pi_A/\Pi$$
by $\Upsilon((b+B)_\ast \mathrm{Inn}(\Gamma)) = \Pi B$ for each $b+B\in N_A(\Gamma)$. 

Define 
$$\mathrm{Out}_E^1(\Gamma) = \{(b+I)_\ast\mathrm{Inn}(\Gamma): b+I \in N_E(\Gamma)\}.$$
Then $\mathrm{Out}_E^1(\Gamma)$ is an abelian subgroup of $\mathrm{Out}_E(\Gamma)$, 
and so $\mathrm{Out}_E^1(\Gamma)$ is a finite abelian group by Theorem 12. 

The {\it translation lattice} of $\Gamma$ is the group $L = \{a\in E^n: a+I \in \Gamma\}$. 
The {\it stabilizer} of $L$ in $\mathrm{GL}(n,\realnos)$ is the group
 $\mathrm{GL}(L) = \{B\in GL(n,\realnos): BL = L\}$.  

\begin{lemma} 
Let $\Gamma$ be an $n$-space group with point group $\Pi$, 
let  $\Pi_E$ be the point group of $N_E(\Gamma)$, and let $\Pi_A$ be the point group of $N_A(\Gamma)$. 
Let $L$ be the translation lattice of $\Gamma$, and let $\mathrm{GL}(L)$ be the stabilizer of $L$ in $\mathrm{GL}(n,\realnos)$. 
Then $\Pi_A$ is a subgroup of the normalizer of $\Pi$ in $\mathrm{GL}(L)$.  
The map $\Upsilon: \mathrm{Out}(\Gamma) \to \Pi_A/\Pi$ is an epimorphism with kernel $\mathrm{Out}_E^1(\Gamma)$, 
and $\Upsilon$ maps $\mathrm{Out}_E(\Gamma)$ onto $\Pi_E/\Pi$.  
The group $\mathrm{Out}(\Gamma)$ is finite if and only if $\Pi_A$ is finite, 
and  $\mathrm{Out}_E(\Gamma) =  \mathrm{Out}(\Gamma)$ if and only if $\Pi_E = \Pi_A$. 
\end{lemma}
\begin{proof} If $b+B\in N_A(\Gamma)$ and $a+A \in \Gamma$, then 
$$(b+B)(a+A)(b+B)^{-1} = Ba + (I -BAB^{-1})b+ BAB^{-1}.$$
Hence $\Pi_A$ is a subgroup of the normalizer of $\Pi$ in $\mathrm{GL}(L)$. 
The map $\Upsilon$ is an epimorphism and $\Upsilon$ maps $\mathrm{Out}_E(\Gamma)$ onto $\Pi_E/\Pi$ by Lemma 12. 
Now $(b+B)_\ast\mathrm{Inn}(\Gamma) \in \mathrm{Ker}(\Upsilon)$ if and only if $B \in \Pi$. 
If $B \in \Pi$, then $a+B \in \Gamma$ for some $a\in E^n$, and  so $(b+B)(a+B)^{-1} = b -a + I$, 
whence $(b+B)_\ast\mathrm{Inn}(\Gamma) = (b-a+I)_\ast\mathrm{Inn}(\Gamma)$. 
Thus $\mathrm{Ker}(\Upsilon) = \mathrm{Out}_E^1(\Gamma)$. 
As the groups $\mathrm{Out}_E^1(\Gamma)$ and $\Pi$ are finite, $\mathrm{Out}(\Gamma)$ is finite 
if and only if $\Pi_A$ is finite. 
As $\Upsilon$ maps $\mathrm{Out}_E(\Gamma)$ onto $\Pi_E/\Pi$, we deduce 
that $\mathrm{Out}_E(\Gamma) =  \mathrm{Out}(\Gamma)$ if and only if $\Pi_E = \Pi_A$. 
\end{proof}

\noindent{\bf Remark 7.} 
If $\Gamma$ is a 1-space group, then  $\Gamma$ is either  an infinite cyclic or infinite dihedral group,  
and so $\mathrm{Out}(\Gamma)$ has order two. 
If $\Gamma$ is a 2-space group, then $\mathrm{Out}(\Gamma)$ is infinite if and only if $E^2/\Gamma$ is either a torus or a pillow by Lemma 13 and Table 5A in \cite{B-Z}.  

\vspace{.15in}

Let $\Gamma_1$ and $\Gamma_2$ be $n$-space groups, 
and let $\phi\in \mathrm{Aff}(E^n)$ such that $\phi\Gamma_1\phi^{-1} = \Gamma_2$. 
Then $\phi$ induces an affinity $\phi_\star: E^n/\Gamma_1\to E^n/\Gamma_2$ 
defined by $\phi_\star(\Gamma_1 x) =\Gamma_2\phi(x)$ for each $x\in E^n$. 
Define $\phi_\sharp:\mathrm{Aff}(E^m/\Gamma_1)\to \mathrm{Aff}(E^n/\Gamma_2)$ 
by $\phi_\sharp(\alpha) = \phi_\star\alpha\phi_\star^{-1}$. 
Then $\phi_\sharp$ is an isomorphism with $(\phi_\sharp)^{-1} = (\phi^{-1})_\sharp$.

Let $\phi_\ast: \Gamma_1 \to \Gamma_2$ be the isomorphism defined by $\phi_\ast(\gamma) = \phi\gamma\phi^{-1}$. 
Define $\phi_\#:\mathrm{Out}(\Gamma_1) \to \mathrm{Out}(\Gamma_2)$ 
by $\phi_\#(\zeta\mathrm{Inn}(\Gamma_1)) = \phi_\ast\zeta\phi_\ast^{-1}\mathrm{Inn}(\Gamma_2)$. 
Then $\phi_\#$ is an isomorphism with $(\phi_\#)^{-1} = (\phi^{-1})_\#$.

\begin{lemma}  
Let $\Gamma_1$ and $\Gamma_2$ be $n$-space groups, 
and let $\phi\in \mathrm{Aff}(E^n)$ such that $\phi\Gamma_1\phi^{-1} = \Gamma_2$. 
Then the following diagram commutes

$$
\begin{array}{ccc}
\mathrm{Aff}(E^n/\Gamma_1)  & {\buildrel \phi_\sharp \over \longrightarrow} & \mathrm{Aff}(E^n/\Gamma_2) \vspace{.05in} \\ 
\Omega\downarrow  &           & \downarrow \Omega \\
\mathrm{Out}(\Gamma_1) &  {\buildrel \phi_\# \over \longrightarrow}  & \mathrm{Out}(\Gamma_2) 
\end{array}
$$
\end{lemma}
\begin{proof}
If $\alpha \in \mathrm{Aff}(E^n/\Gamma)$,  
then we have 
\begin{eqnarray*}
\Omega\phi_\sharp(\alpha) 
& =  &  \Omega(\phi_\star\alpha\phi_\star^{-1})  \\
& =  & (\phi\tilde\alpha\phi^{-1})_\ast\mathrm{Inn}(\Gamma_2) \\
& =  & \phi_\ast\tilde\alpha_\ast\phi_\ast^{-1} \mathrm{Inn}(\Gamma_2) \\
& =  &  \phi_\#(\tilde\alpha_\ast\mathrm{Inn}(\Gamma_1)) \ \ = \ \ \phi_\#(\Omega(\alpha)). 
\end{eqnarray*}

\vspace{-.2in}
\end{proof}

\begin{theorem}  
Let $\Gamma$ be an $n$-space group, and let $G$ be a finite subgroup of $\mathrm{Out}(\Gamma)$. 
Then there exists $C \in \mathrm{GL}(n,\realnos)$ such that $C\Gamma C^{-1}$ is an $n$-space group 
and $C_\#(G)$ is a subgroup of $\mathrm{Out}_E(C\Gamma C^{-1})$.
If $G$ is maximal, then $C_\#(G)= \mathrm{Out}_E(C\Gamma C^{-1})$.
\end{theorem}
\begin{proof}
Let $\Rho$ be the subgroup of $\Pi_A$ such that $\Upsilon(G) = \Rho/\Pi$. 
As $\Rho/\Pi$ and $\Pi$ are finite, $\Rho$ is finite. 
By Lemma 9 of \S 7.5 of \cite{R}, there exists $C\in \mathrm{GL}(n,\realnos)$ 
such that $C\Rho C^{-1} \subset \mathrm{O}(n)$. 
Then $C\Gamma C^{-1}$ is a discrete subgroup of $\mathrm{Isom}(E^n)$, 
and $C_\star: E^n/\Gamma \to E^n/C\Gamma C^{-1}$ is an affine homeomorphism.  
Therefore $E^n/C \Gamma C^{-1}$ is compact.  
Hence $C\Gamma C^{-1}$ is an $n$-space group. 

Let $g\in G$.  Then $g = (b+B)_\ast \mathrm{Inn}(\Gamma)$ for some $B\in \Rho$. 
We have that
\begin{eqnarray*}
C_\#(g)  & =  & C_\ast(b+B)_\ast C_\ast^{-1}\mathrm{Inn}(C\Gamma C^{-1}) \\ 
	       & = &  (C(b+B)C^{-1})_\ast \mathrm{Inn}(C\Gamma C^{-1}) \\ 
	       & = &  (Cb+CBC^{-1})_\ast\mathrm{Inn}(C\Gamma C^{-1}).
\end{eqnarray*}
As $Cb+CBC^{-1} \in \mathrm{Isom}(E^n)$, we have that $C_\#(g) \in \mathrm{Out}_E(C\Gamma C^{-1})$. 
Thus $C_\#(G)$ is a subgroup of $\mathrm{Out}_E(C\Gamma C^{-1})$.

If $G$ is maximal, then $C_\#(G)$ is a maximal finite subgroup of 
$C_\#(\mathrm{Out}(\Gamma)) = \mathrm{Out}(C\Gamma C^{-1})$. 
Hence $C_\#(G)= \mathrm{Out}_E(C\Gamma C^{-1})$, 
since $\mathrm{Out}_E(C\Gamma C^{-1})$ 
is a finite subgroup of $\mathrm{Out}(C\Gamma C^{-1})$. 
\end{proof}
\begin{corollary} 
Let $\Gamma$ be an $n$-space group such that $\mathrm{Out}(\Gamma)$ is finite. 
Then there exists $C \in \mathrm{GL}(n,\realnos)$ such that $C\Gamma C^{-1}$ is an $n$-space group 
and 
$$\mathrm{Out}_E(C\Gamma C^{-1})=\mathrm{Out}(C\Gamma C^{-1}).$$
\end{corollary}

\section{On the affine classification of flat orbifold fibrations}  

In this section, we develop some theory for classifying geometric orbifold fibrations 
of compact, connected, flat $n$-orbifolds up to affine equivalence. 
Geometric flat orbifold fibrations $\eta_i: M_i \to B_i$, for $i = 1,2$, are said to be {\it affinely equivalent} 
if there are affinities $\alpha: M_1 \to M_2$ and $\beta:B_1 \to B_2$ such that $\beta\eta_1 = \eta_2\alpha$. 
We have the following theorem from \cite{R-T}. 

\begin{theorem} 
Let $\Nu_i$ be a complete normal subgroup of an $n$-space group $\Gamma_i$ for $i=1,2$, 
and let $\eta_i:E^n/\Gamma_i \to (E^n/V_i)/(\Gamma_i/\Nu_i)$ 
be the fibration projections determined by $\Nu_i$ for $i=1,2$. 
Then the following are equivalent:
\begin{enumerate}
\item The fibration projections $\eta_1$ and $\eta_2$ are affinely equivalent. 
\item There is an affinity $\phi$ of $E^n$ such that 
$\phi\Gamma_1\phi^{-1} = \Gamma_2$ and $\phi\Nu_1\phi^{-1} = \Nu_2$. 
\item There is an isomorphism $\psi:\Gamma_1\to \Gamma_2$ such that $\psi(\Nu_1) = \Nu_2$.
\end{enumerate}
\end{theorem}

Let $\Nu$ be a complete normal subgroup of an $n$-space group $\Gamma$, 
let $V = \mathrm{Span}(\Nu)$, and let $V^\perp$ be the orthogonal complement of $V$ in $E^n$.  
Euclidean $n$-space $E^n$ decomposes as the Cartesian product 
$E^n = V \times V^\perp$. 
The action of $\Nu$ on $V^\perp$ is trivial. 
Hence $E^n/\Nu$ decomposes as the Cartesian product 
$$E^n/\Nu = V/\Nu \times V^\perp.$$

The action of $\Gamma/\Nu$ on $E^n/\Nu$ corresponds 
to the diagonal action of $\Gamma/\Nu$ on $V/\Nu \times V^\perp$ via isometries 
defined by the formula
$$(\Nu(b+B))(\Nu v,w) = (\Nu(c+Bv),d+Bw)$$
where $b = c + d$ with $c\in V$ and $d\in V^\perp$. 
Hence we have the following theorem.
\begin{theorem} 
Let $\Nu$ be a complete normal subgroup of an $n$-space group $\Gamma$, 
and let $V= \mathrm{Span}(\Nu)$. 
Then the map
$$\chi: E^n/\Gamma \to (V/\Nu\times V^\perp)/(\Gamma/\Nu)$$
defined by $\chi(\Gamma x) = (\Gamma/\Nu)(\Nu v, w)$, 
with  $x = v + w$ and $v\in V$ and $w\in V^\perp$, is an isometry. 
\end{theorem}

The group $\Nu$ is the kernel of the action of $\Gamma$ on $V^\perp$, and so we have 
an effective action of $\Gamma/\Nu$ on $V^\perp$. 
The natural projection from $V/\Nu\times V^\perp$ to $V^\perp$ 
induces a continuous surjection
$$\pi^\perp: (V/\Nu\times V^\perp)/(\Gamma/\Nu) \to V^\perp/(\Gamma/\Nu).$$
Orthogonal projection from $E^n$ to $V^\perp$ induces an isometry from $E^n/V$ to $V^\perp$ 
which in turn induces an isometry 
$$\psi^\perp: (E^n/V)/(\Gamma/\Nu) \to V^\perp/(\Gamma/\Nu).$$
\begin{theorem} 
The following diagram commutes
\[\begin{array}{ccc}
E^n/\Gamma & {\buildrel \chi\over\longrightarrow} &
(V/\Nu\times V^\perp)/(\Gamma/\Nu) \\
\eta_V \downarrow \  & & \downarrow\pi^\perp \\
(E^n/V)/(\Gamma/\Nu) & {\buildrel \psi^\perp\over\longrightarrow}  & V^\perp/(\Gamma/\Nu). 
\end{array}\] 
\end{theorem}
\begin{proof}
Let $x\in E^n$.  Write $x = v+w$ with $v\in V$ and $w\in V^\perp$. 
Then we have 
\begin{eqnarray*}
\pi^\perp(\chi(\Gamma x)) & =  & \pi^\perp((\Gamma/\Nu)(\Nu v, w)) \\
			      & = & (\Gamma/\Nu) w \\
			      & =  & \psi^\perp((\Gamma/\Nu)(V+x))\ \  = \ \ \psi^\perp(\eta_V(\Gamma x)).
\end{eqnarray*}

\vspace{-.2in}
\end{proof}

Let $\Nu$ be a complete normal subgroup of an $n$-space group $\Gamma$, and let $V = \mathrm{Span}(\Nu)$. 
Let $\gamma \in \Gamma$.   Then $\gamma = b+B$ with $b\in E^n$ and $B\in \mathrm{O}(n)$. 
Write $b = \overline b + b'$ with $\overline b \in V$ and $b' \in V^\perp$. 
Let $\overline B$ and $B'$ be the orthogonal transformations of $V$ and $V^\perp$, respectively, 
obtained by restricting $B$. 
Let $\overline \gamma = \overline b + \overline B$ and $\gamma' = b' + B'$. 
Then $\overline \gamma$ and $\gamma'$ are isometries of $V$ and $V^\perp$, respectively. 
If $\gamma_1, \gamma_2 \in \Gamma$, then we have that 
$\overline{\gamma_1\gamma_2} = \overline{\gamma_1}\,\overline{\gamma_2}$ 
and $(\gamma_1\gamma_2)' = \gamma_1'\gamma_2'$. 

Let $\ov \Gamma = \{\ov \gamma: \gamma \in \Gamma\}$. 
Then $\ov \Gamma$ is a subgroup of $\mathrm{Isom}(V)$. 
Example 1 shows that $\ov\Gamma$ is not necessarily discrete. 
The map $\Beta: \Gamma \to \ov \Gamma$ defined by $\Beta(\gamma)= \ov\gamma$ is an epimorphism 
with kernel $\Kappa$ equal to the kernel of the action of $\Gamma$ on $V$. 

Let $\Gamma' = \{\gamma' : \gamma \in \Gamma\}$.  
Then $\Gamma'$ is a subgroup of $\mathrm{Isom}(V^\perp)$. 
The map $\Rho' : \Gamma \to \Gamma'$ 
defined by $\Rho'(\gamma) = \gamma'$ is epimorphism with kernel $\Nu$, 
since $\Nu$ is a complete normal subgroup of $\Gamma$.  
Hence $\Rho'$ induces an isomorphism $\Rho: \Gamma/\Nu \to \Gamma'$ defined by $\Rho(\Nu\gamma) = \gamma'$. 
The group $\Gamma'$ is a space group with $V^\perp/\Gamma' = V^\perp/(\Gamma/\Nu)$. 

Let $\ov \Nu = \{\ov \nu: \nu \in \Nu\}$.  
Then $\ov \Nu$ is a subgroup of $\mathrm{Isom}(V)$. 
The map $\Beta: \Nu \to \ov \Nu$ defined by $\Beta(\nu) = \ov\nu$ is an isomorphism. 
The group $\ov\Nu$ is a space group with $V/\ov{\Nu} = V/\Nu$. 

The diagonal action of $\Gamma$ on $V\times V^\perp$ is given by $\gamma(v,w) = (\overline \gamma v, \gamma' w)$. 
The diagonal action of $\Gamma/\Nu$ on $V/\Nu \times V^\perp$ is given by 
$\Nu\gamma(\Nu v, w) = (\Nu \overline \gamma v, \gamma' w)$. 
The action of $\Gamma/\Nu$ on $V/\Nu$ determines a homomorphism 
$$\Xi : \Gamma/\Nu \to \mathrm{Isom}(V/\Nu)$$
defined by $\Xi(\Nu \gamma) = \overline \gamma_\star$, where $\overline \gamma_\star: V/\Nu \to V/\Nu$ 
is defined by $\overline\gamma_\star(\Nu v) = \Nu \overline\gamma(v)$ for each $v\in V$. 

\begin{theorem} 
Let $\Mu$ be an $m$-space group, let $\Delta$ be an $(n-m)$-space group, 
and let $\Theta:\Delta \to \mathrm{Isom}(E^m/\Mu)$ be a homomorphism. 
Identify $E^n$ with $E^m\times E^{n-m}$, and extend $\Mu$ to a subgroup $\Nu$ of $\mathrm{Isom}(E^n)$
such that the point group of $\Nu$ acts trivially on $(E^m)^\perp= E^{n-m}$. 
Then there exists a unique $n$-space group $\Gamma$ containing $\Nu$ as a complete normal subgroup 
such that $\Gamma' = \Delta$, and if $\Xi:\Gamma/\Nu \to \mathrm{Isom}(E^m/\Nu)$ is 
the homomorphism induced by the action of $\Gamma/\Nu$ on $E^m/\Nu$, then $\Xi = \Theta\Rho$ 
where $\Rho:\Gamma/\Nu \to \Gamma'$ is the isomorphism defined by $\Rho(\Nu\gamma) = \gamma'$ 
for each $\gamma\in\Gamma$. 
\end{theorem}
\begin{proof}  Let $\delta \in \Delta$.  
By Lemma 5, there exists $\tilde\delta \in N_E(\Mu)$ such that $\tilde\delta_\star = \Theta(\delta)$. 
The isometry $\tilde\delta$ is unique up to multiplication by an element of $\Mu$. 
Let $\delta = d + D$, with $d\in E^{n-m}$ and $D\in \mathrm{O}(n-m)$, 
and let $\tilde\delta = \tilde d + \tilde D$, with $\tilde d \in E^m$ and $\tilde D \in \mathrm{O}(m)$.  
Let $\hat d = \tilde d + d$, and let $\hat D = \tilde D \times D$. 
Then $\hat D E^m = E^m$ and $\hat D E^{n-m} = E^{n-m}$. 
Let $\hat \delta = \hat d + \hat D$. 
Then $\hat \delta$ is an isometry of $E^n$ such that $\overline{\hat\delta} = \tilde\delta$ 
and $\big(\hat\delta\big)' = \delta$.  
The isometry $\hat \delta$ is unique up to multiplication by an element of $\Nu$. 
We have that 
$$\hat\delta^{-1} = -\hat D^{-1}\hat d + \hat D^{-1} = -\tilde D^{-1}\tilde d -D^{-1}d + \tilde D^{-1} \times D^{-1},$$
and so $\overline{\hat\delta^{-1} }= \big(\tilde \delta\big)^{-1}$ and  $(\hat\delta^{-1})' = \delta^{-1}$. 
We have that 
$$\ov{\hat\delta\Nu\hat\delta^{-1}} = \ov{\hat\delta}\ov\Nu\ov{\hat\delta^{-1}} = \tilde\delta\Mu\tilde\delta^{-1} = \Mu = \ov\Nu$$
and 
$$(\hat\delta\Nu\hat\delta^{-1})' = \hat\delta'\Nu'(\hat\delta^{-1})' = \delta\{I'\}\delta^{-1} = \{I'\}.$$
Therefore $\hat \delta\Nu\hat\delta^{-1}  = \Nu$. 

Let  $\Gamma$ be the subgroup of $\mathrm{Isom}(E^n)$ 
generated by $\Nu \cup \{\hat \delta : \delta \in \Delta\}$.  
Then $\Gamma$ contains $\Nu$ as a normal subgroup, 
and the point group of $\Gamma$ leaves $E^m$ invariant. 
Suppose $\gamma \in \Gamma$.  Then there exists $\nu \in \Nu$ and $\delta_1, \ldots, \delta_k \in \Delta$ 
and $\epsilon_1,\ldots, \epsilon_k \in \{\pm 1\}$ 
such that $\gamma = \nu \hat{\delta}_1^{\epsilon_1} \cdots \hat{\delta}_k^{\epsilon_k}$. 
Then we have 
$$\gamma' =  \nu' (\hat{\delta}_1^{\epsilon_1})' \cdots (\hat{\delta}_k^{\epsilon_k})' = \delta_1^{\epsilon_1}\cdots \delta_k^{\epsilon_k}.$$
Hence $\Gamma' = \Delta$, and 
we have an epimorphism $\Rho':\Gamma \to \Delta$ defined by $\Rho'(\gamma) = \gamma'$. 
The group $\Nu$ is in the kernel of $\Rho'$, and so $\Rho'$ induces an epimorphism $\Rho: \Gamma/\Nu \to \Delta$ 
defined by $\Rho(\Nu\gamma) = \gamma'$. 
Suppose $\Rho(\Nu\gamma) = I'$.  By Lemma 5, we have that 
\begin{eqnarray*}
 \Rho(\Nu\gamma) = I' \ \ \Rightarrow \ \  \gamma' = I' & \Rightarrow  & \delta_1^{\epsilon_1}\cdots \delta_k^{\epsilon_k} =I' \\
       			     & \Rightarrow & \Theta(\delta_1)^{\epsilon_1} \cdots \Theta(\delta_k)^{\epsilon_k} = \ov I_\star\\
			    & \Rightarrow &(\tilde \delta_1)_\star^{\epsilon_1} \cdots (\tilde \delta_k)_\star^{\epsilon_k} = \ov I_\star\\
		             & \Rightarrow &\Mu(\tilde \delta_1)^{\epsilon_1} \cdots \Mu(\tilde \delta_k)^{\epsilon_k} = \Mu\\
			    & \Rightarrow &\Mu (\tilde \delta_1)^{\epsilon_1} \cdots (\tilde \delta_k)^{\epsilon_k} = \Mu \\
		             & \Rightarrow &\Mu \overline{\hat \delta_1^{\epsilon_1}} \cdots \overline{\hat \delta_k^{\epsilon_k}} = \Mu
			   \ \  \Rightarrow \ \ \Mu \overline\gamma = \Mu  \ \  \Rightarrow \ \ \overline\gamma \in \Mu. 
\end{eqnarray*}
As $\gamma' = I'$ and $\overline \gamma \in \ov\Nu$, we have that $\gamma \in \Nu$. 
Thus $\Rho$ is an isomorphism. 

We next show that $\Gamma$ acts discontinuously on $E^n$. 
Let $C$ be a compact subset of $E^n$. 
Let $K$ and $L$ be the orthogonal projections of $C$ into $E^m$ and $E^{n-m}$, respectively. 
Then $C \subseteq K \times L$.  
As $\Delta$ acts discontinuously on $E^{n-m}$, there exists only finitely many elements $\delta_1, \ldots, \delta_k$ of $\Delta$  
such that $L\cap \delta_iL \neq \emptyset$ for each $i$. 
Let $\gamma_i = \hat \delta_i$ for $i = 1,\ldots, k$. 
The set $K_i = K \cup \overline \gamma_i(K)$ is compact for each $i = 1,\ldots, k$. 
As $\Nu$ acts discontinuously on $E^m$, there is a finite subset $F_i$ of $\Nu$ such that $K_i \cap \nu K_i \neq \emptyset$ 
only if $\nu \in F_i$ for each $i$.  Hence $K \cap \nu\overline\gamma_i(K)\neq\emptyset$ only if $\nu \in F_i$ for each $i$. 
The set $F = F_1\gamma_1\cup \cdots \cup F_k\gamma_k$ is a finite subset of $\Gamma$.  
Suppose $\gamma \in \Gamma$ and  $C\cap \gamma C \neq \emptyset$. 
Then $L \cap \gamma' L \neq \emptyset$, and so $ \gamma' = \delta_i$ for some $i$. 
Hence $\gamma = \nu \gamma_i$ for some $\nu \in \Nu$. 
Now we have that $K\cap \overline \gamma K \neq \emptyset$, and so $K\cap \nu\overline\gamma_iK \neq \emptyset$. 
Hence $\nu \in F_i$.  Therefore $\gamma \in F$.  Thus $\Gamma$ acts discontinuously on $E^n$. 
Therefore $\Gamma$ is a discrete subgroup of $\mathrm{Isom}(E^n)$ by Theorem 5.3.5 of \cite{R}. 

Let $D_\Mu$ and $D_\Delta$ be fundamental domains for $\Mu$ in $E^m$ and $\Delta$ in $E^{n-m}$, respectively. 
Then their topological closures $\overline{D}_\Mu$ and $\overline{D}_\Delta$ are compact sets. 
Let $x \in E^n$.  Write $x = \overline x + x'$ with $\overline x \in E^m$ and $x' \in E^{n-m}$. 
Then there exists $\delta \in \Delta$ such that $\delta x' \in \overline D_\Delta$, 
and there exists $\nu \in \Nu$ such that $\nu \tilde\delta\overline x \in \overline D_\Mu$. 
We have that 
$$\nu\hat \delta x = \overline{ \nu\hat\delta} \overline x +  (\nu \hat\delta)'x' = \nu\tilde \delta \overline x + \delta x' 
\in \overline D_\Mu \times \overline D_\Delta.$$
Hence the quotient map $\pi: E^n \to E^n/\Gamma$ maps the compact set $\overline D_\Mu \times \overline D_\Delta$ 
onto $E^n/\Gamma$.  Therefore $E^n/\Gamma$ is compact.  Thus $\Gamma$ is an $n$-space group. 

Let $\Xi: \Gamma/\Nu \to \mathrm{Isom}(E^n/\Nu)$ be the homomorphism induced by the action of $\Gamma/\Nu$ on $E^m/\Nu$. 
Let $\gamma \in \Gamma$.  Then there exists $\delta \in \Delta$ such that $\Nu\gamma = \Nu\hat \delta$. 
Then $\gamma' = (\hat\delta)' = \delta$, and we have that $\Xi = \Theta\Rho$, since 
$$\Xi(\Nu\gamma) = \Xi(\Nu\hat\delta) = (\overline{\hat\delta})_\star = \tilde \delta_\star = 
\Theta(\delta) = \Theta(\gamma') = \Theta\Rho(\Nu\gamma).$$

Suppose $\gamma$ is an isometry of $E^n$ such that $\gamma\Nu\gamma^{-1} = \Nu$ 
and $\gamma' \in \Delta$ and $\overline\gamma_\star = \Theta(\gamma')$. 
Then $\widehat{\gamma'} \in \Gamma$. 
Now $\overline\gamma_\star = \widetilde{\gamma'}_\star$, 
and so $\overline\gamma = \ov\nu\widetilde{\gamma'}$ for some $\nu\in \Nu$ by Lemma 5. 
Then $\gamma = \nu\widehat{\gamma'}$.  Hence $\gamma\in \Gamma$. 
Thus $\Gamma$ is the unique $n$-space group that contains $\Nu$ as a complete normal subgroup such that $\Gamma' = \Delta$ 
and $\Xi = \Theta\Rho$. 
\end{proof}

Let $\Nu_i$ be a complete normal subgroup of an $n$-space group $\Gamma_i$, with $V_i = \mathrm{Span}(\Nu_i)$, for $i=1,2$. 
Let $\phi$ be an affinity of $E^n$ such that $\phi(\Gamma_1,\Nu_1)\phi^{-1}=(\Gamma_2,\Nu_2)$. 
Write $\phi = c+C$ with $c\in E^n$ and $C\in \mathrm{GL}(n,\realnos)$. 
Let $a+I\in \Nu_1$.  Then $\phi(a+I)\phi^{-1} = Ca+I$.  Hence $CV_1 \subseteq V_2$. 
Let $\overline{C}:V_1 \to V_2$ be the linear transformation obtained by restricting $C$. 
Let $\overline{C'}:V_1^\perp \to V_2$ and $C': V_1^\perp \to V_2^\perp$ be the linear transformations 
obtained by restricting $C$ to $V_1^\perp$ followed by the orthogonal projections to $V_2$ and $V_2^\perp$, respectively. 
Write $c = \overline{c} + c'$ with $\ov{c}\in V_2$ and $c'\in V_2^\perp$. 
Let $\overline{\phi}: V_1 \to V_2$ and $\phi': V_1^\perp \to V_2^\perp$ be the affine transformations 
defined by $\overline{\phi} = \overline{c}+\ov{C}$ and $\phi' = c'+C'$. 

\begin{lemma} 
Let $\Nu_i$ be a complete normal subgroup of an $n$-space group $\Gamma_i$, with $V_i = \mathrm{Span}(\Nu_i)$, for $i=1,2$. 
Let $\phi = c+C$ be an affinity of $E^n$ such that $\phi(\Gamma_1,\Nu_1)\phi^{-1}=(\Gamma_2,\Nu_2)$. 
Then $\ov{C}$ and $C'$ are invertible, with $\ov{C}^{-1} = \ov{C^{-1}}$ and $(C')^{-1} = (C^{-1})'$,  
and $\ov{(C^{-1})'} = -\ov{C}^{-1}\ov{C'}(C')^{-1}$.  
If $b+B\in \Gamma_1$, then $\ov{C'}B' = \ov C \ov B \ov C^{-1} \ov{C'}$.  
Moreover $\ov{C'}V_1^\perp \subseteq \mathrm{Span}(Z(\Nu_2))$. 
\end{lemma}
\begin{proof} We have that 
$$\dim V_1 = \dim \Nu_1 = \dim \Nu_2 = \dim V_2.$$
Let $y\in E^n$ and write $y = \ov{y} + y'$ with $\ov{y}\in V_2$ and $y'\in V_2^\perp$. 
Then 
$$\ov{y} = CC^{-1}(\ov{y}) = \ov{C}\ov{C^{-1}}(\ov{y}),$$
and so $\ov{C}$ is invertible with $\ov{C}^{-1} = \ov{C^{-1}}$. 
We have that 
\begin{eqnarray*}
y' & = & CC^{-1}y' \\
& = &C(\ov{(C^{-1})'}(y') + (C^{-1})'(y')) \\
& = &C\ov{(C^{-1})'}(y')+ C(C^{-1})'(y') \\
& = & \ov{C}\ov{(C^{-1})'}(y')+\ov{C'}(C^{-1})'(y')+ C'(C^{-1})'(y'). 
\end{eqnarray*}
Hence $C'$ is invertible, with $(C')^{-1} = (C^{-1})'$, and $\ov{C}\ov{(C^{-1})'}+\ov{C'}(C^{-1})' =0$. 
Therefore $\ov{(C^{-1})'} = -\ov{C}^{-1}\ov{C'}(C')^{-1}$. 

Let $\gamma = b+ B \in \Gamma_1$.  Then $\phi\gamma\phi^{-1} \in \Gamma_2$.  Let $w \in V_2^\perp$. 
Then we have 
\begin{eqnarray*}
w& = & CBC^{-1}w\\
& = &CB(\ov{(C^{-1})'}(w) + (C^{-1})'(w)) \\
& = &C(\ov B\ov{(C^{-1})'}(w)+ B'(C^{-1})'(w)) \\
& = &C\ov B\ov{(C^{-1})'}(w)+ CB'(C^{-1})'(w) \\
& = & \ov{C}\ov B\ov{(C^{-1})'}(w)+\ov{C'}B'(C^{-1})'(w)+ C'B'(C^{-1})'(w). \\
\end{eqnarray*}
As $\phi\gamma\phi^{-1}(w) \in V_2^\perp$, 
we have that 
$$\ov{C}\ov B\ov{(C^{-1})'}+\ov{C'}B'(C^{-1})'=0.$$
Therefore 
$$\ov{C}\ov B(-\ov{C}^{-1}\ov{C'}(C')^{-1})+\ov{C'}B'(C')^{-1}=0.$$
Hence $\ov{C'}B' = \ov C \ov B\ov{C}^{-1}\ov{C'}$. 

Now suppose $\gamma \in \Nu_1$.  Then $B' = I'$. 
Hence $\ov B\ov{C}^{-1}\ov{C'} = \ov{C}^{-1}\ov{C'}$. 
By Lemma 9, we deduce that $\ov{C}^{-1}\ov{C'}V_1 \subseteq \mathrm{Span}(Z(\Nu_1))$. 
As $\phi Z(\Nu_1)\phi^{-1} = Z(\Nu_2)$, we have  
that $\ov{C'}V_1 \subseteq \mathrm{Span}(Z(\Nu_2))$. 
\end{proof}

\begin{theorem}  
Let $\Nu_i$ be  a complete normal subgroup of an $n$-space group $\Gamma_i$, with $V_i = \mathrm{Span}(\Nu_i)$ for $i=1,2$. 
Let $\Xi_i:\Gamma_i/\Nu_i \to \mathrm{Isom}(V_i/\Nu_i)$ be the homomorphism induced 
by the action of $\Gamma_i/\Nu_i$ on $V_i/\Nu_i$ for $i=1,2$.  
Let $\alpha: V_1 \to V_2$, and $\beta: V_1^\perp \to V_2^\perp$ be affinities such that $\alpha\ov\Nu_1\alpha^{-1} = \ov\Nu_2$ 
and $\beta\Gamma'_1\beta^{-1} = \Gamma_2'$. 
Write $\alpha = \ov c + \ov C$ with $\ov c \in V_2$ and $\ov C: V_1 \to V_2$ a linear isomorphism. 
Let $D: V_1^\perp \to \mathrm{Span}(Z(\Nu_2))$ be a linear transformation 
such that if $b+B\in \Gamma_1$, then $DB' = \ov C\ov B\ov C^{-1} D$. 
Let $\Rho_i:\Gamma_i/\Nu_i\to \Gamma_i'$ be the isomorphism 
defined by $\Rho_i(\Nu_i\gamma) = \gamma'$ for each $i = 1,2$, 
and let $p_1:\Gamma_1/\Nu_1 \to V_1^\perp$ be defined by $p_1(\Nu_1(b+B)) = b'$. 
Let $K_2$ be the connected component of the identity of the Lie group $\mathrm{Isom}(V_2/\Nu_2)$, 
and let $(Dp_1)_\star:\Gamma_1/\Nu_1 \to K_2$ be defined 
by $(Dp_1)_\star(\Nu_1(b+B)) = (Db'+\ov I)_\star$ where $(Db'+ \ov I)_\star$ 
is defined by $(Db'+ \ov I)_\star(\Nu_2 v) = \Nu_2(Db'+v)$ for each $v \in V_2$. 
Then there exists an affinity $\phi = c + C$ of $E^n$ such that $\phi(\Gamma_1,\Nu_1)\phi^{-1} = (\Gamma_2,\Nu_2)$, 
with $\ov\phi = \alpha$, and $\phi' = \beta$, and $\ov{C'} = D$ if and only if 
$$\Xi_2\Rho_2^{-1}\beta_\ast\Rho_1 = (Dp_1)_\star\alpha_\sharp\Xi_1.$$
Moreover $(Dp_1)_\star:\Gamma_1/\Nu_1 \to K_2$ is a crossed homomorphism 
with $\Gamma_1/\Nu_1$ acting on the torus $K_2$ by $(\Nu_1(b+B))(v+\ov I)_\star =  (\ov C\ov B\ov C^{-1}v+\ov I)_\star$ 
for each $v \in  \mathrm{Span}(Z(\Nu_2))$. 
\end{theorem}
\begin{proof}
Suppose that $\phi = c+ C$ is an affinity of $E^n$ such that $\phi(\Gamma_1,\Nu_1)\phi^{-1} = (\Gamma_2,\Nu_2)$, 
with $\ov\phi = \alpha$, and $\phi' = \beta$, and $\ov{C'} = D$. 
Let $\gamma  = b+ B \in \Gamma_1$.  
Then we have 
$$\phi\gamma\phi^{-1} = (c+C)(b+B)(c+C)^{-1} = Cb + (I-CBC^{-1})c+CBC^{-1}.$$
Hence we have 
\begin{eqnarray*}
\ov{\phi\gamma\phi^{-1}} & = & \ov{C}\ov b + \ov{C'}b'+ (\ov I-\ov C\ov B{\ov C}^{-1})\ov c+\ov C\ov B{\ov C}^{-1} \\
& = & ( \ov{C'}b'+\ov I)(\ov c +{\ov C})(\ov b+\ov B)(\ov c + {\ov C})^{-1} \\
& = & (Db'+\ov I)\ov \phi\ov \gamma\ov \phi^{-1} \ \
 = \ \ ( Db'+\ov I)\alpha\ov \gamma\alpha^{-1},  
\end{eqnarray*}
and 
\begin{eqnarray*}
(\phi\gamma\phi^{-1})' & = & C'b'+ (I'-C'B'(C')^{-1})c' + C'B'(C')^{-1} \\
& = & (c'+C')(b'+B')(c'+C')^{-1} \\ 
& = & \phi' \gamma' (\phi')^{-1} \ \ = \ \ \beta \gamma' \beta^{-1}.
\end{eqnarray*}

\begin{eqnarray*} \Xi_2\Rho_2^{-1}\beta_\ast\Rho_1(\Nu_1\gamma) 
& = & \Xi_2\Rho_2^{-1}\beta_\ast(\gamma') \\
& = & \Xi_2\Rho_2^{-1}(\beta\gamma'\beta^{-1}) \\
& = & \Xi_2\Rho_2^{-1}(\phi\gamma\phi^{-1})' \\
& = & \Xi_2(\Nu_2\phi\gamma\phi^{-1})  \\
& =  & (\ov{\phi\gamma\phi^{-1}})_\star \\
& =  & ((Db'+\ov I)\alpha\ov \gamma\alpha^{-1})_\star \\
& =  & (Db'+\ov I)_\star\alpha_\star\ov \gamma_\star \alpha_\star^{-1}  \\
& =  & (Dp_1(\Nu_1(b+B)))_\star \alpha_\sharp(\ov \gamma_\star) \ \,
 =  \ \, (Dp_1(\Nu_1\gamma))_\star \alpha_\sharp(\Xi_1(\Nu_1\gamma)). 
\end{eqnarray*}
Hence we have that $\Xi_2\Rho_2^{-1}\beta_\ast\Rho_1 = (Dp_1)_\star\alpha_\sharp\Xi_1$. 

Conversely, suppose that $\Xi_2\Rho_2^{-1}\beta_\ast\Rho_1 = (Dp_1)_\star\alpha_\sharp\Xi_1$. 
Define an affine transformation $\phi: E^n \to E^n$ by 
$$\phi(x) = \alpha(\overline x) + D(x')+ \beta(x')$$
for each $x\in E^n$, where $x = \ov x + x'$ with $\ov x \in V_1$ and $x'\in V_1^\perp$. 
Then $\phi$ is an affinity of $E^n$, with
$$\phi^{-1}(y) = \alpha^{-1}(\ov y) -\ov C^{-1}D\beta^{-1}(y') + \beta^{-1}(y')$$
for each $y\in E^n$ where $y = \ov y + y'$ with $\ov y \in V_2$  and $y'\in V_2^\perp$. 

Write $\beta' = c' + C'$ with $c' \in V_2^\perp$ and $C': V_1^\perp \to V_2^\perp$ a linear isomorphism. 
Write $\phi = c+ C$ with $c \in E^n$ and $C$ a linear isomorphism of $E^n$. 
Then $c =\ov c + c'$ and $Cx = \ov C\ov x + Dx' + C'x'$ for each $x\in E^n$ with $\ov x \in V_1$ and $x'\in V_1^\perp$. 
We have that $\ov \phi = \alpha$, and $\phi' = \beta$ and $\ov{C'} = D$. 
We also have $\phi^{-1} = -C^{-1}c+C^{-1}$ and 
$$C^{-1}y = \ov C^{-1}\ov y - \ov C^{-1}D(C')^{-1}y' + (C')^{-1}y'$$
for all $y \in E^n$ with $\ov y \in V_2$ and $y' \in V_2^\perp$. 

Let $\gamma \in \Gamma_1$.  Write $\gamma = b+ B$ with $b\in E^n$ and $B\in \mathrm{O}(n)$. 
Then we have 
$$ \phi\gamma\phi^{-1} = (c+C)(b+B)(c+C)^{-1} = Cb+(I-CBC^{-1})c + CBC^{-1}.$$
Let $y \in E^n$.  Write $y = \ov y + y'$ with $\ov y \in V_2$ and $y' \in V_2^\perp$. 
Then we have that 
\begin{eqnarray*} CBC^{-1} 
& = & CB(\ov C^{-1}\ov y - \ov C^{-1}D(C')^{-1}y' + (C')^{-1}y'  \\
& = & C(\ov B\ov C^{-1}\ov y -\ov B\ov C^{-1}D(C')^{-1}y' +B'(C')^{-1}y' \\
& = & \ov C\ov B\ov C^{-1}\ov y -\ov C\ov B\ov C^{-1}D(C')^{-1}y' +DB'(C')^{-1}y' + C'B'(C')^{-1}y' \\
& = & \ov C\ov B\ov C^{-1}\ov y + C'B'(C')^{-1}y'.
\end{eqnarray*}
Hence $CBC^{-1} =  \ov C\ov B\ov C^{-1} \times C'B'(C')^{-1}$ as a linear isomorphism 
of $E^n = V_2\times V_2^\perp$. 

Moreover, we have 
\begin{eqnarray*}
\ov{\phi\gamma\phi^{-1}} & = & \ov{C}\ov b + \ov{C'}b'+ (\ov I-\ov C\ov B{\ov C}^{-1})\ov c+\ov C\ov B{\ov C}^{-1} \\
& = & ( \ov{C'}b'+\ov I)(\ov c +{\ov C})(\ov b+\ov B)(\ov c + {\ov C})^{-1} \\
& = & (Db'+\ov I)\ov \phi\ov \gamma\ov \phi^{-1} \ \
 = \ \ ( Db'+\ov I)\alpha\ov \gamma\alpha^{-1},  
\end{eqnarray*}
and 
\begin{eqnarray*}
(\phi\gamma\phi^{-1})' & = & C'b'+ (I'-C'B'(C')^{-1})c' + C'B'(C')^{-1} \\
& = & (c'+C')(b'+B')(c'+C')^{-1} \\ 
& = & \phi' \gamma' (\phi')^{-1} \ \ = \ \ \beta \gamma' \beta^{-1}.
\end{eqnarray*}

As $\Xi_2\Rho_2^{-1}\beta_\ast\Rho_1 = (Dp_1)_\star\alpha_\sharp\Xi_1$, 
we have that $(\alpha\ov\gamma\alpha^{-1})_\star$ is an isometry of $V_2/\Nu_2$. 
By Lemmas 5 and 11, we have that $\alpha\ov\gamma\alpha^{-1}$ is an isometry of $V_2$. 
Hence $\ov C \ov B \ov C^{-1}$ is an orthogonal transformation of $V_2$. 

As $\beta\Gamma_1'\beta^{-1} = \Gamma_2'$, we have that $C'B'(C')^{-1}$  is an orthogonal transformation of $V_2^\perp$.
Hence 
$CBC^{-1} = \ov C\ov B\ov C^{-1} \times C'B'(C')^{-1}$
is an orthogonal transformation of $E^n = V_2 \times V_2^\perp$.
Therefore $\phi\gamma\phi^{-1}$ is an isometry of $E^n$ for each $\gamma \in \Gamma_1$. 

As $\Gamma_1$ acts discontinuously on $E^n$ and $\phi$ is a homeomorphism of $E^n$, 
we have that $\phi\Gamma_1\phi^{-1}$ acts discontinuously on $E^n$. 
Therefore $\phi\Gamma_1\phi^{-1}$ is a discrete subgroup of $\mathrm{Isom}(E^n)$ by Theorem 5.3.5 of \cite{R}. 
Now $\phi$ induces a homeomorphism 
$\phi_\star: E^n/\Gamma_1 \to E^n/\phi\Gamma_1\phi^{-1}$
defined by $\phi_\star(\Gamma_1 x) = \phi\Gamma_1\phi^{-1}\phi(x)$. 
Hence $E^n/\phi\Gamma_1\phi^{-1}$ is compact. 
Therefore $\phi\Gamma_1\phi^{-1}$ is a $n$-space group. 

Now $\phi_\ast: \Gamma_1 \to \phi\Gamma_1\phi^{-1}$, defined by $\phi_\ast(\gamma) = \phi\gamma\phi^{-1}$,  
is an isomorphism that maps the normal subgroup $\Nu_1$ to the normal subgroup 
$\phi\Nu_1\phi^{-1}$ of $\phi\Gamma_1\phi^{-1}$, 
and $\phi\Gamma_1\phi^{-1}/\phi\Nu_1\phi^{-1}$ is isomorphic to $\Gamma_1/\Nu_1$. 
Hence $\phi\Gamma_1\phi^{-1}/\phi\Nu_1\phi^{-1}$ is a space group. 
Therefore $\phi\Nu_1\phi^{-1}$ is a complete normal subgroup of $\phi\Gamma_1\phi^{-1}$ by Theorem 5 of \cite{R-T}. 

Now suppose $\nu = a+A \in \Nu_1$.  Then $a' = 0$ and $A' = I'$, and so $\nu' = I'$.  
Hence $\ov{\phi\nu\phi^{-1}} = \alpha\ov \nu\alpha^{-1}$ and $(\phi\nu\phi^{-1})' = I'$. 
As $\alpha\ov\Nu_1\alpha^{-1} = \ov\Nu_2$, we have that  $\phi\Nu_1\phi^{-1} = \Nu_2$.  
Moreover, as $\beta\Gamma_1'\beta^{-1}=\Gamma_2'$, we have that $(\phi\Gamma_1\phi^{-1})' = \Gamma_2'$. 

Let $\Xi: \phi\Gamma_1\phi^{-1}/\Nu_2 \to \mathrm{Isom}(V_2/\Nu_2)$ be the homomorphism induced by the action 
of $\phi\Gamma_1\phi^{-1}/\Nu_2$ on $V_2/\Nu_2$.  Let $\gamma =b+B \in \Gamma_1$, 
and let $\Rho:\phi\Gamma_1\phi^{-1}/\Nu_2 \to \Gamma_2'$ 
be the isomorphism defined by $\Rho(\Nu_2\phi\gamma\phi^{-1}) = (\phi\gamma\phi^{-1})'$. 
Then we have that 
\begin{eqnarray*} \Xi\Rho^{-1}(\beta\gamma'\beta^{-1})
& = & \Xi\Rho^{-1}((\phi\gamma\phi^{-1})') \\
& = & \Xi(\Nu_2\phi\gamma\phi^{-1}) \\
& = & (\ov{\phi\gamma\phi^{-1}})_\star \\
& = & ((Db'+\ov{I})\alpha\ov{\gamma}{\alpha}^{-1})_\star \\
& = & (Db'+\ov{I})_\star\alpha_\star\ov{\gamma}_\star{\alpha}^{-1}_\star \\
& = & (Dp_1(\Nu_1(b+B)))_\star\alpha_\sharp(\ov\gamma_\star) \\
& = & (Dp_1(\Nu_1\gamma))_\star\alpha_\sharp(\Xi_1(\Nu_1\gamma)) \\
& = & \Xi_2\Rho_2^{-1}\beta_\ast\Rho_1(\Nu_1\gamma) \ \
= \ \ \Xi_2\Rho_2^{-1}(\beta\gamma'\beta^{-1}). 
\end{eqnarray*} 
Hence we have that $\Xi\Rho^{-1} = \Xi_2\Rho_2^{-1}$.  
Therefore $\phi\Gamma_1\phi^{-1} = \Gamma_2$ by Theorem 18. 
Thus $\phi(\Gamma_1,\Nu_1)\phi^{-1} = (\Gamma_2,\Nu_2)$. 

Let $\gamma = b+B$ and $\gamma_1 = b_1+B_1$ be elements of  $\Gamma_1$.  
Then we have that 
\begin{eqnarray*} (Dp_1)_\star(\Nu_1\gamma\Nu_1\gamma_1) 
& = &(Dp_1)_\star(\Nu_1(b+Bb_1+BB_1))  \\
& = & (D(b+Bb_1)'+\ov I)_\star  \\
& = &(D(b'+B'b_1')+\ov I)_\star \\
& = &(Db'+DB'b_1'+\ov I)_\star\\
& = & (Db'+\ov C \ov B \ov C^{-1}Db_1'+\ov I)_\star \\
& = &(Db'+\ov I)_\star (\ov C \ov B \ov C^{-1}Db_1'+\ov I)_\star \\
& = & (Db'+\ov I)_\star (\Nu_1(b+B))(Db_1'+\ov I)_\star\\
& = & (Dp_1)_\star(\Nu_1\gamma) ((\Nu_1\gamma)(Dp_1)_\star(\Nu_1\gamma_1)). 
\end{eqnarray*} 
Therefore $(Dp_1)_\star : \Gamma_1/\Nu_1 \to K_2$ is a crossed homomorphism. 
\end{proof}


\noindent{\bf Definition:} Let $m$ be a positive integer less than $n$. 
Let $\Mu$ be an $m$-space group, let $\mathrm{Aff}(\Mu) = \mathrm{Aff}(E^m/\Mu)$, and 
let $\Delta$ be an $(n-m)$-space group.  
Define $\mathrm{Hom}_f(\Delta,\mathrm{Aff}(\Mu))$ to be the set of all homomorphism 
from $\Delta$ to $\mathrm{Aff}(\Mu)$ such that the composition 
with $\Omega:\mathrm{Aff}(\Mu)\to \mathrm{Out}(\Mu)$ has finite image. 

\vspace{.15in}
By Lemma 14, the group $\mathrm{Aff}(\Mu)$ acts on the left of $\mathrm{Hom}_f(\Delta,\mathrm{Aff}(\Mu))$ by conjugation, 
that is, if $\alpha\in \mathrm{Aff}(\Mu)$ and $\eta\in \mathrm{Hom}_f(\Delta,\mathrm{Aff}(\Mu))$, 
then $\alpha \eta = \alpha_\ast\eta$ where $\alpha_\ast: \mathrm{Aff}(\Mu) \to \mathrm{Aff}(\Mu)$ is defined 
by $\alpha_\ast(\beta) = \alpha\beta\alpha^{-1}$. 
Let $\mathrm{Aff}(\Mu)\backslash\mathrm{Hom}_f(\Delta,\mathrm{Aff}(\Mu))$ be the set of $\mathrm{Aff}(\Mu)$-orbits. 
The group $\mathrm{Aut}(\Delta)$ acts on the right of $\mathrm{Hom}_f(\Delta,\mathrm{Aff}(\Mu))$ 
by composition of homomorphisms. 
If $\zeta\in \mathrm{Aut}(\Delta)$ and  $\eta\in \mathrm{Hom}_f(\Delta,\mathrm{Aff}(\Mu))$ and $\alpha\in \mathrm{Aff}(\Mu)$, 
then 
$$(\alpha\eta)\zeta = (\alpha_\ast\eta)\zeta = \alpha_\ast(\eta\zeta) = \alpha(\eta\zeta).$$
Hence $\mathrm{Aut}(\Delta)$ acts on the right of 
$\mathrm{Aff}(\Mu)\backslash\mathrm{Hom}_f(\Delta,\mathrm{Aff}(\Mu))$ 
by 
$$(\mathrm{Aff}(\Mu)\eta)\zeta = \mathrm{Aff}(\Mu)(\eta\zeta).$$
Let $\delta, \epsilon \in \Delta$ and $\eta\in \mathrm{Hom}_f(\Delta,\mathrm{Aff}(\Mu))$.  
Then we have that
$$\eta\delta_\ast(\epsilon) = \eta(\delta\epsilon\delta^{-1}) =\eta(\delta)\eta(\epsilon)\eta(\delta)^{-1}= \eta(\delta)_\ast\eta(\epsilon) 
= (\eta(\delta)\eta)(\epsilon). $$
Hence $\eta\delta_\ast = \eta(\delta)\eta$.  Therefore $\mathrm{Inn}(\Delta)$ acts trivially on
 $\mathrm{Aff}(\Mu)\backslash\mathrm{Hom}_f(\Delta,\mathrm{Aff}(\Mu))$. 
Hence $\mathrm{Out}(\Delta)$ acts on the right of 
$\mathrm{Aff}(\Mu)\backslash\mathrm{Hom}_f(\Delta,\mathrm{Aff}(\Mu))$ 
by
$$(\mathrm{Aff}(\Mu)\eta)(\zeta\mathrm{Inn}(\Delta))= \mathrm{Aff}(\Mu)(\eta\zeta).$$

\noindent{\bf Definition:} Define the set  $\mathrm{Aff}(\Delta,\Mu)$ by the formula
$$\mathrm{Aff}(\Delta,\Mu) = (\mathrm{Aff}(\Mu)\backslash\mathrm{Hom}_f(\Delta,\mathrm{Aff}(\Mu)))/\mathrm{Out}(\Delta).$$ 
If $\eta\in \mathrm{Hom}_f(\Delta,\mathrm{Aff}(\Mu))$, let $[\eta] = (\mathrm{Aff}(\Mu)\eta)\mathrm{Out}(\Delta)$ 
be the element of $\mathrm{Aff}(\Delta,\Mu)$ determined by $\eta$. 

\vspace{.15in}
\noindent{\bf Definition:} Define $\mathrm{Iso}(\Delta,\Mu)$ to be the set of isomorphism classes of pairs $(\Gamma, \Nu)$ 
where $\Nu$ is a complete normal subgroup of an $n$-space group $\Gamma$ 
such that $\Nu$ is isomorphic to $\Mu$ and $\Gamma/\Nu$ is isomorphic to $\Delta$. 
We denote the isomorphism class of a pair $(\Gamma,\Nu)$ by $[\Gamma,\Nu]$. 

\vspace{.15in}
Let $\eta \in \mathrm{Hom}_f(\Delta,\mathrm{Aff}(\Mu))$. 
By Theorem 14, there exists $C \in \mathrm{GL}(m,\realnos)$ such that $C\Mu C^{-1}$ is an $m$-space group 
and $C_\#(\Omega\eta(\Delta)) \subseteq \mathrm{Out}_E(C\Mu C^{-1})$. 
By Lemma 14, we have that $\Omega C_\sharp\eta(\Delta) \subseteq \mathrm{Out}_E(C\Mu C^{-1})$. 
By Theorems 12 and 13, we deduce that $C_\sharp\eta(\Delta) \subseteq \mathrm{Isom}(E^m/C\Mu C^{-1})$. 
Extend $C\Mu C^{-1}$ to a subgroup $\Nu$ of $\mathrm{Isom}(E^n)$ such that the point group of $\Nu$ 
acts trivially on $(E^m)^\perp$. 
By Theorem 18, there exists an $n$-space group $\Gamma$ containing $\Nu$ as a complete normal subgroup 
such that $\Gamma' = \Delta$ and if $\Xi: \Gamma/\Nu \to \mathrm{Isom}(E^m/\Nu)$ is the homomorphism 
induced by the action of $\Gamma/\Nu$ on $E^m/\Nu$, then $\Xi = C_\sharp\eta\Rho$ 
where $\Rho: \Gamma/\Nu \to \Gamma'$ is the isomorphism defined by $\Rho(\Nu\gamma) = \gamma'$. 
Define 
$$\psi: \mathrm{Aff}(\Delta,\Mu) \to \mathrm{Iso}(\Delta, \Mu)$$
by $\psi([\eta]) = [\Gamma,\Nu]$.  We next show that $\psi$ is well defined. 

\begin{lemma} 
The function $\psi: \mathrm{Aff}(\Delta,\Mu) \to \mathrm{Iso}(\Delta, \Mu)$ is a well-defined surjection. 
\end{lemma}
\begin{proof}
To see that $\psi$ is well defined, let $\eta \in \mathrm{Hom}_f(\Delta,\mathrm{Aff}(\Mu))$, 
and let $\alpha \in \mathrm{Aff}(\Mu)$ and $\zeta \in \mathrm{Aut}(\Delta)$. 
By Theorem 14, there exists $\hat C \in \mathrm{GL}(m,\realnos)$ such that $\hat C\Mu\hat C^{-1}$ is an $m$-space group 
and 
$$\hat C_\#(\Omega\alpha_\ast\eta\zeta(\Delta))\subseteq \mathrm{Out}_E(\hat C\Mu\hat C^{-1}).$$
Then $\hat C_\sharp\alpha_\ast\eta\zeta(\Delta) \subseteq \mathrm{Isom}(E^m/\hat C\Mu\hat C^{-1})$. 
Extend $\hat C\Mu\hat C^{-1}$ to a subgroup $\hat\Nu$ of $\mathrm{Isom}(E^n)$ such that the point group of $\hat\Nu$ 
acts trivially on $(E^m)^\perp$. 
By Theorem 18, there exists an $n$-space group $\hat\Gamma$ containing $\hat\Nu$ as a complete normal subgroup 
such that $\hat\Gamma' = \Delta$ and if $\hat\Xi: \hat\Gamma/\hat\Nu \to \mathrm{Isom}(E^m/\hat\Nu)$ is the homomorphism 
induced by the action of $\hat\Gamma/\hat\Nu$ on $E^m/\hat\Nu$, then $\hat\Xi = \hat C_\sharp\alpha_\ast\eta\zeta\hat\Rho$ 
where $\hat\Rho:\hat \Gamma/\hat\Nu \to \hat\Gamma'$ is the isomorphism defined by $\hat\Rho(\hat\Nu\gamma) = \gamma'$. 
Lift $\alpha$ to $\tilde\alpha \in N_A(\Mu)$ such that $\tilde\alpha_\star= \alpha$. 
Then $\tilde\alpha_\sharp = \alpha_\ast$ and we have
$$\hat C_\sharp\alpha_\ast\eta\zeta \hat\Rho \hat\Rho^{-1}\zeta^{-1} \Rho = (\hat C\tilde\alpha C^{-1})_\sharp C_\sharp \eta\Rho,$$
and so we have 
$$\hat\Xi\hat\Rho^{-1}\zeta^{-1}\Rho = (\hat C\tilde\alpha C^{-1})_\sharp \Xi.$$
By Bieberbach's theorem, there is a $\tilde\zeta \in N_A(\Delta)$ such that $\tilde\zeta_\ast = \zeta$. 
By Theorem 19, there is $\phi \in \mathrm{Aff}(E^n)$ such that $\phi(\Gamma,\Nu)\phi^{-1} = (\hat\Gamma,\hat\Nu)$ 
with $\ov \phi = \hat C\tilde \alpha C^{-1}$ and $\phi' = \tilde\zeta^{-1}$. 
Thus $\psi:\mathrm{Aff}(\Delta,\Mu) \to \mathrm{Iso}(\Delta, \Mu)$ is well defined. 

To see that $\psi$ is onto, let $[\Gamma_1,\Nu_1] \in \mathrm{Iso}(\Delta,\Mu)$. 
Then we have isomorphisms $\alpha:\ov \Nu_1 \to \Mu$ and $\beta:\Gamma_1/\Nu_1\to\Delta$. 
Let $V =\mathrm{Span}(\Nu_1)$, and let $\Xi_1:\Gamma_1/\Nu_1 \to \mathrm{Isom}(V/\Nu_1)$ 
be the homomorphism induced by the action of $\Gamma_1/\Nu_1$ on $V/\Nu_1$. 
By Bieberbach's theorem, there is an affinity $\tilde\alpha: V \to E^m$ such that $\tilde\alpha\ov\Nu_1\tilde\alpha^{-1} = \Mu$ 
and $\tilde\alpha_\ast = \alpha$. 
Let $\eta = \tilde\alpha_\sharp\Xi_1\beta^{-1}$. 
Then $\eta \in \mathrm{Hom}_f(\Delta,\mathrm{Aff}(\Mu))$ by Theorem 12 and Lemma 14. 
By Theorems 12, 13, and 14, there is a $C\in\mathrm{GL}(m,\realnos)$ such that $C\Mu C^{-1}$ is an $m$-space group and 
$$C_\sharp\eta(\Delta) \subseteq \mathrm{Isom}(E^m/C\Mu C^{-1}).$$
Extend $C\Mu C^{-1}$ to a subgroup $\Nu_2$ of $\mathrm{Isom}(E^n)$ such that the point group of $\Nu_2$ 
acts trivially on $(E^m)^\perp$. 
By Theorem 18, there exists an $n$-space group $\Gamma_2$ containing $\Nu_2$ as a complete normal subgroup 
such that $\Gamma_2' = \Delta$ and if $\Xi_2: \Gamma_2/\Nu_2 \to \mathrm{Isom}(E^m/\Nu_2)$ is the homomorphism 
induced by the action of $\Gamma_2/\Nu_2$ on $E^m/\Nu_2$, then $\Xi_2 = C_\sharp\eta\Rho_2$ 
where $\Rho_2: \Gamma_2/\Nu_2 \to \Gamma_2'$ is the isomorphism defined by $\Rho_2(\Nu_2\gamma) = \gamma'$. 
Then we have
$$\Xi_2\Rho_2^{-1}(\beta\Rho_1^{-1})\Rho_1 = C_\sharp\eta\beta = C_\sharp\tilde\alpha_\sharp\Xi_1 = (C\tilde\alpha)_\sharp \Xi_1.$$
By Theorem 19, there is $\phi \in \mathrm{Aff}(E^n)$ such that $\phi(\Gamma_1,\Nu_1)\phi^{-1} = (\Gamma_2,\Nu_2)$. 
Therefore $\psi([\eta]) = [\Gamma_1,\Nu_1]$.  
Thus $\psi$ is surjective. 
\end{proof}

\begin{theorem}  
Let $m$ be a positive integer less than $n$.  Let $\Mu$ be an $m$-space group 
with trivial center, and let $\Delta$ be an $(n-m)$-space group. 
Then the function $\psi:\mathrm{Aff}(\Delta,\Mu) \to \mathrm{Iso}(\Delta,\Mu)$ is a bijection.
\end{theorem}
\begin{proof}  
To see that $\psi$ is injective, let $\eta_i\in\mathrm{Hom}_f(\Delta,\mathrm{Aff}(\Mu))$ for $i = 1,2$ 
such that $\psi([\eta_1]) = \psi([\eta_2])$. 
By the definition of $\psi$, there exists $C_i\in \mathrm{GL}(m,\realnos)$ such that $C_i\Mu C_i^{-1}$ is an $m$-space group 
and $(C_i)_\sharp\eta(\Delta) \subseteq \mathrm{Isom}(E^m/C_i\Mu C_i^{-1})$ for each $i=1,2$. 
Extend $C_i\Mu C_i^{-1}$ to a subgroup $\Nu_I$ of $\mathrm{Isom}(E^n)$ such that the point group of $\Nu_I$ 
acts trivially on $(E^m)^\perp$. 
By Theorem 18, there exists an $n$-space group $\Gamma_I$ containing $\Nu_i$ as a complete normal subgroup 
such that $\Gamma_i' = \Delta$ and if $\Xi_i: \Gamma_i\Nu_i \to \mathrm{Isom}(E^m/\Nu_i)$ is the homomorphism 
induced by the action of $\Gamma_i/\Nu_i$ on $E^m/\Nu_i$, then $\Xi_i = (C_i)_\sharp\eta_i\Rho_i$ 
where $\Rho_i: \Gamma_i/\Nu_i \to \Gamma_i'$ is the isomorphism defined by $\Rho_i(\Nu_i\gamma) = \gamma'$ for $i = 1,2$. 
Then we have 
$$[\Gamma_1,\Nu_1] = \psi([\eta_1] )= \psi([\eta_2]) = [\Gamma_2,\Nu_2].$$
Hence there exists $\phi\in \mathrm{Aff}(E^n)$ such that $\phi(\Gamma_1,\Nu_1)\phi^{-1} = (\Gamma_2,\Nu_2)$. 
By Lemma 14 and Theorem 19, we have that 
$\Xi_2\Rho_2^{-1}\phi'_\ast\Rho_1 = \ov \phi_\sharp\Xi_1$. 
Hence we have that 
$$(C_2)_\sharp\eta_2\phi'_\ast\Rho_1 = \ov\phi_\sharp(C_1)_\sharp\eta_1\Rho_1.$$
Therefore we have 
$$\eta_2\phi'_\ast = (C_2)_\sharp^{-1}\ov\phi_\sharp(C_1)_\sharp\eta_1 = (C_2^{-1}\ov\phi C_1)_\sharp\eta_1.$$
Hence $[\eta_1] = [\eta_2]$.  Thus $\psi$ is injective. 
By Lemma 16, we have that $\psi$ is surjection.  
Therefore $\psi$ is a bijection. 
\end{proof}


The action of $\Gamma$ on $\Nu$ by conjugation induces a homomorphism 
$$\mathcal{O}: \Gamma/\Nu \to \mathrm{Out}_E(\Nu)$$
defined by $\mathcal{O}(\Nu\gamma) = \gamma_\ast\mathrm{Inn}(\Nu)$ where $\gamma_\ast(\nu) = \gamma\nu\gamma^{-1}$ 
for each $\gamma \in \Gamma$ and $\nu\in\Nu$. 
Let  $\alpha: \Nu_1 \to \Nu_2$ be an isomorphism.  Then $\alpha$ induces an isomorphism
$$\alpha_\#: \mathrm{Out}(\Nu_1) \to \mathrm{Out}(\Nu_2)$$
defined by $\alpha_\#(\zeta\mathrm{Inn}(\Nu_1) )= \alpha\zeta\alpha^{-1}\mathrm{Inn}(\Nu_2)$. 

\begin{lemma}  
Let $\Nu_i$ be a complete normal subgroup of an $n$-space group $\Gamma_i$ for $i = 1,2$.  
Let $\mathcal{O}_i:\Gamma_i/\Nu_i \to \mathrm{Out}_E(\Nu_i)$ be the homomorphism 
induced by the action of $\Gamma_i$ on $\Nu_i$ by conjugation for $i = 1, 2$, and  
let $\alpha:\Nu_1\to \Nu_2$ and $\phi: \Gamma_1\to \Gamma_2$ and $\beta:\Gamma_1/\Nu_1 \to \Gamma_2/\Nu_2$ be isomorphisms  
such that the following diagram commutes
\[\begin{array}{ccccccccc}
1 & \to & \Nu_1 & \rightarrow & \Gamma_1 & \rightarrow & \Gamma_1/\Nu_1 & \to & 1 \\
    &       & \hspace{.12in} \downarrow \, \alpha &    & \hspace{.12in}\downarrow\, \phi  & & \hspace{.12in} \downarrow\, \beta & \\
1 & \to  &  \Nu_2 & \rightarrow  & \Gamma_2 & \rightarrow & \Gamma_2/\Nu_2 & \to & 1,  
\end{array}\] 
where the horizontal maps are inclusions and projections, then $\mathcal{O}_2 = \alpha_\#\mathcal{O}_1\beta^{-1}$. 
\end{lemma}
\begin{proof}  
Let $\gamma\in \Gamma_1$.  Then we have that 
$$\mathcal{O}_2(\Nu_2\phi(\gamma)) = \phi(\gamma)_\ast\mathrm{Inn}(\Nu_2),$$
whereas
\begin{eqnarray*}
\alpha_\#\mathcal{O}_1\beta^{-1}(\Nu_2\phi(\gamma)) & =  & \alpha_\#\mathcal{O}_1(\Nu_1\gamma) \\
                                              & = & \alpha_\#(\gamma_\ast\mathrm{Inn}(\Nu_1)) \ \ = \ \ \alpha\gamma_\ast\alpha^{-1}\mathrm{Inn}(\Nu_2).
\end{eqnarray*}
If $\nu \in \Nu$,  then 
\begin{eqnarray*}
\alpha\gamma_\ast\alpha^{-1}(\nu) & = & \alpha\gamma_\ast\alpha^{-1}(\nu) \\
						       & = & \alpha(\gamma\alpha^{-1}(\nu)\gamma^{-1}) \\
						       & = & \phi(\gamma\phi^{-1}(\nu)\gamma^{-1}) \\
						       & = & \phi(\gamma)\nu\phi(\gamma)^{-1} \ \ = \ \ \phi(\gamma)_\ast(\nu). 
\end{eqnarray*}
Hence $\alpha\gamma_\ast\alpha^{-1} = \phi(\gamma)_\ast$.  
Therefore $\mathcal{O}_2 = \alpha_\#\mathcal{O}_1\beta^{-1}$. 
\end{proof}

\noindent{\bf Definition:} Let $m$ be a positive integer less than $n$. 
Let $\Mu$ be an $m$-space group, and let $\Delta$ be an $(n-m)$-space group. 
Define $\mathrm{Hom}_f(\Delta,\mathrm{Out}(\Mu))$ to be the set of all homomorphisms from $\Delta$ to $\mathrm{Out}(\Mu)$ 
that have finite image. 

\vspace{.15in}
The group $\mathrm{Out}(\Mu)$ acts on the left of $\mathrm{Hom}_f(\Delta,\mathrm{Out}(\Mu))$ by conjugation, 
that is, if $g\in \mathrm{Out}(\Mu)$ and $\eta\in \mathrm{Hom}_f(\Delta,\mathrm{Out}(\Mu))$, 
then $g \eta = g_\ast\eta$ where $g_\ast: \mathrm{Out}(\Mu) \to \mathrm{Out}(\Mu)$ is defined by $g_\ast(h) = ghg^{-1}$. 
Let $\mathrm{Out}(\Mu)\backslash\mathrm{Hom}_f(\Delta,\mathrm{Out}(\Mu))$ be the set of $\mathrm{Out}(\Mu)$-orbits. 
The group $\mathrm{Aut}(\Delta)$ acts on the right of $\mathrm{Hom}_f(\Delta,\mathrm{Out}(\Mu))$ 
by composition of homomorphisms. 
If $\beta\in \mathrm{Aut}(\Delta)$ and  $\eta\in \mathrm{Hom}_f(\Delta,\mathrm{Out}(\Mu))$ and $g\in \mathrm{Out}(\Mu)$, 
then 
$$(g\eta)\beta = (g_\ast\eta)\beta = g_\ast(\eta\beta) = g(\eta\beta).$$
Hence $\mathrm{Aut}(\Delta)$ acts on the right of 
$\mathrm{Out}(\Mu)\backslash\mathrm{Hom}_f(\Delta,\mathrm{Out}(\Mu))$ 
by 
$$(\mathrm{Out}(\Mu)\eta)\beta = \mathrm{Out}(\Mu)(\eta\beta).$$
Let $\delta, \epsilon \in \Delta$ and $\eta\in \mathrm{Hom}_f(\Delta,\mathrm{Out}(\Mu))$.  
Then we have that
$$\eta\delta_\ast(\epsilon) = \eta(\delta\epsilon\delta^{-1}) =\eta(\delta)\eta(\epsilon)\eta(\delta)^{-1}= \eta(\delta)_\ast\eta(\epsilon) 
= (\eta(\delta)\eta)(\epsilon). $$
Hence $\eta\delta_\ast = \eta(\delta)\eta$.  Therefore $\mathrm{Inn}(\Delta)$ acts trivially on
 $\mathrm{Out}(\Mu)\backslash\mathrm{Hom}_f(\Delta,\mathrm{Out}(\Mu))$. 
Hence $\mathrm{Out}(\Delta)$ acts on the right of $\mathrm{Out}(\Mu)\backslash\mathrm{Hom}_f(\Delta,\mathrm{Out}(\Mu))$ 
by
$$(\mathrm{Out}(\Mu)\eta)(\beta\mathrm{Inn}(\Delta)) = \mathrm{Out}(\Mu)(\eta\beta).$$
\noindent{\bf Definition:} Define the set $\mathrm{Out}(\Delta,\Mu)$ by the formula
$$\mathrm{Out}(\Delta,\Mu) = (\mathrm{Out}(\Mu)\backslash\mathrm{Hom}_f(\Delta,\mathrm{Out}(\Mu)))/\mathrm{Out}(\Delta).$$ 
If $\eta\in \mathrm{Hom}_f(\Delta,\mathrm{Out}(\Mu))$, let $[\eta] = (\mathrm{Out}(\Mu)\eta)\mathrm{Out}(\Delta)$ 
be the element of $\mathrm{Out}(\Delta,\Mu)$ determined by $\eta$. 

\vspace{.15in}
Let $(\Gamma,\Nu)$ be a pair such that $[\Gamma,\Nu]\in\mathrm{Iso}(\Delta,\Mu)$. 
Let  $\mathcal{O}: \Gamma/\Nu \to \mathrm{Out}_E(\Nu)$ 
be the homomorphism induced by the action of $\Gamma$ on $\Nu$ by conjugation. 
Let $\alpha: \Nu \to \Mu$ and $\beta:\Delta \to \Gamma/\Nu$ be isomorphisms. 
Then $\alpha_\#\mathcal{O}\beta \in \mathrm{Hom}_f(\Delta,\mathrm{Out}(\Mu))$. 

Let $\alpha':\Nu\to\Mu$ and $\beta':\Delta\to\Gamma/\Nu$ are isomorphisms. 
Observe that 
\begin{eqnarray*}
\alpha'_\#\mathcal{O}\beta' & = & \alpha'_\#\alpha_\#^{-1}\alpha_\#\mathcal{O}\beta\beta^{-1}\beta'  \\
& = & (\alpha'\alpha^{-1})_\#\alpha_\#\mathcal{O}\beta(\beta^{-1}\beta')  \\
& = & (\alpha'\alpha^{-1}\mathrm{Inn}(\Mu))_\ast(\alpha_\#\mathcal{O}\beta(\beta^{-1}\beta'))  \\
& = & (\alpha'\alpha^{-1}\mathrm{Inn}(\Mu))(\alpha_\#\mathcal{O}\beta)(\beta^{-1}\beta').  \\
\end{eqnarray*}
Hence $[\alpha_\#\mathcal{O}\beta]$ in $\mathrm{Out}(\Delta,\Mu)$ does not depend 
on the choice of $\alpha$ and $\beta$,  
and so $(\Gamma,\Nu)$ determines the element $[\alpha_\#\mathcal{O}\beta]$ 
of $\mathrm{Out}(\Delta,\Mu)$ independent of the choice of $\alpha$ and $\beta$. 

Suppose $[\Gamma_i,\Nu_i]\in \mathrm{Iso}(\Delta,\Mu)$ for $i=1,2$, 
and $\phi:(\Gamma_1,\Nu_1) \to (\Gamma_2,\Nu_2)$ is an isomorphism of pairs. 
Let $\alpha:\Nu_1\to \Nu_2$ be the isomorphism obtained by restricting $\phi$, 
and let $\beta: \Gamma_1/\Nu_1\to \Gamma_2/\Nu_2$ be the isomorphism induced by $\phi$. 
Let $\mathcal{O}_i:\Gamma_i/\Nu_i \to \mathrm{Out}_E(\Nu_i)$ be the homomorphism 
induced by the action of $\Gamma_i$ on $\Nu_i$ by conjugation for $i=1,2$. 
Then $\mathcal{O}_2 = \alpha_\#\mathcal{O}_1\beta^{-1}$ by Lemma 17.  
Let $\alpha_1:\Nu_1 \to \Mu$ and $\beta_1:\Delta \to \Gamma_1/\Nu_1$ be isomorphisms. 
Let $\alpha_2 = \alpha_1\alpha^{-1}$ and $\beta_2= \beta\beta_1$. 
Then we have 
$$(\alpha_2)_\#\mathcal{O}_2\beta_2 
 =  (\alpha_1\alpha^{-1})_\#\alpha_\#\mathcal{O}_1\beta^{-1}\beta_2 \\
 = (\alpha_1)_\#\mathcal{O}_1\beta_1. $$
Hence $(\Gamma_1,\Nu_1)$ and $(\Gamma_2,\Nu_2)$ determine the same element of $\mathrm{Out}(\Delta,\Mu)$. 
Therefore there is a function 
$$\omega:\mathrm{Iso}(\Delta,\Mu) \to \mathrm{Out}(\Delta,\Mu)$$
defined by $\omega([\Gamma,\Nu]) = [\alpha_\#\mathcal{O}\beta]$ for any choice of isomorphisms $\alpha:\Nu\to\Mu$ 
and $\beta: \Delta\to \Gamma/\Nu$. 

Let $\eta \in \mathrm{Hom}_f(\Delta, \mathrm{Aff}(\Mu))$.  
Then $\Omega\eta \in \mathrm{Hom}_f(\Delta, \mathrm{Out}(\Mu))$. 
Let $\alpha \in \mathrm{Aff}(\Mu)$ and $\beta\in \mathrm{Aut}(\Delta)$.
Then $\alpha$ lifts to $\tilde\alpha \in N_A(\Mu)$ such that $\tilde\alpha_\star = \alpha$. 
We have that $\tilde\alpha_\sharp = \alpha_\ast$. 
By Lemma 14, we have that 
$$\Omega\alpha_\ast\eta\beta = \Omega\tilde\alpha_\sharp \eta\beta 
= \tilde\alpha_\#(\Omega\eta)\beta = (\tilde\alpha_\ast)_\#(\Omega\eta)\beta.$$ 
Hence we may define a function $\chi:\mathrm{Aff}(\Delta,\Mu) \to \mathrm{Out}(\Delta,\Mu)$ 
by $\chi([\eta]) = [\Omega\eta]$. 
\begin{lemma}  
The function $\chi:\mathrm{Aff}(\Delta,\Mu) \to \mathrm{Out}(\Delta,\Mu)$ is the composition of the function 
$\psi:\mathrm{Aff}(\Delta,\Mu) \to \mathrm{Iso}(\Delta,\Mu)$ followed by 
$\omega:\mathrm{Iso}(\Delta,\Mu) \to \mathrm{Out}(\Delta,\Mu)$.
\end{lemma}
\begin{proof}
Let $\eta \in \mathrm{Hom}_f(\Delta, \mathrm{Aff}(\Mu))$.  
By Theorem 14, there exists $C \in \mathrm{GL}(m,\realnos)$ such that $C\Mu C^{-1}$ is an $m$-space group 
and 
$C_\#(\Omega\eta(\Delta))\subseteq \mathrm{Out}_E(\hat C\Mu C^{-1}).$
Then $C_\sharp\eta(\Delta) \subseteq \mathrm{Isom}(E^m/C\Mu C^{-1})$. 
Extend $C\Mu C^{-1}$ to a subgroup $\Nu$ of $\mathrm{Isom}(E^n)$ such that the point group of $\Nu$ 
acts trivially on $(E^m)^\perp$. 
By Theorem 18, there exists an $n$-space group $\Gamma$ containing $\Nu$ as a complete normal subgroup 
such that $\Gamma' = \Delta$ and if $\Xi: \Gamma/\Nu \to \mathrm{Isom}(E^m/\Nu)$ is the homomorphism 
induced by the action of $\Gamma/\Nu$ on $E^m/\Nu$, then $\Xi = C_\sharp\eta\Rho$ 
where $\Rho:\Gamma/\Nu \to \Gamma'$ is the isomorphism defined by $\Rho(\Nu\gamma) = \gamma'$. 
Then $\psi([\eta]) = [\Gamma,\Nu]$. 

The homomorphism $\mathcal{O}: \Gamma/\Nu \to \mathrm{Out}_E(\Nu)$ induced by the action of $\Gamma$ on $\Nu$ by 
conjugation is given by $\mathcal{O} = \Omega C_\sharp \eta \Rho$.  Now we have that 
$$\omega([\Gamma,\Nu]) = [C_\#^{-1}\mathcal{O}\Rho^{-1}] = [C_\#^{-1}(\Omega C_\sharp \eta\Rho)\Rho^{-1}] 
= [\Omega\eta] = \chi([\eta]).$$
Therefore $\chi = \omega\psi$. 
\end{proof}

\begin{corollary} 
If $\Mu$ has trivial center, then the function $\omega:\mathrm{Iso}(\Delta,\Mu) \to \mathrm{Out}(\Delta,\Mu)$ is a bijection. 
\end{corollary}
\begin{proof}
We have that $\Omega: \mathrm{Aff}(\Mu) \to \mathrm{Out}(\Mu)$ is an isomorphism by Theorems 11, 12, and 13. 
Hence $\chi:\mathrm{Aff}(\Delta,\Mu) \to \mathrm{Out}(\Delta,\Mu)$ is a bijection. 
We have that $\psi:\mathrm{Aff}(\Delta,\Mu) \to \mathrm{Iso}(\Delta,\Mu)$ is a bijection by Theorem 20. 
By Lemma 18, we have that $\chi = \omega\psi$. 
Therefore $\omega$ is a bijection. 
\end{proof}

\begin{theorem}  
If $\Nu$ is a complete normal subgroup of an $n$-space group $\Gamma$ 
such that $\Nu$ has trivial center, 
then the orthogonal dual $\Nu^\perp$ of $\Nu$ exists 
and $\Nu^\perp$  is the centralizer of $\Nu$ in $\Gamma$. 
\end{theorem}
\begin{proof}
The orthogonal dual $\Nu^\perp$ of $\Nu$ exists by Corollaries 2 and 5.
Let $C_\Gamma(\Nu)$ be the centralizer of $\Nu$ in $\Gamma$. 
In general $\Nu^\perp \subseteq C_\Gamma(\Nu)$,  
since $\Nu^\perp$ is the kernel of the action of $\Gamma$ on $V = \mathrm{Span}(\Nu)$. 

Suppose $a+A \in \Nu$ and $b+B \in C_\Gamma(\Nu)$. 
Then we have that 
$$a+A = (b+B)(a+A)(b+B)^{-1} = Ba + (I-BAB^{-1})b + BAB^{-1}.$$
Hence $BAB^{-1} = A$ and $a = Ba + (I-A)b$. 
If $A = I$, then $Ba = a$, and so $V \subseteq \mathrm{Fix}(B)$. 
For arbitrary $a+A \in \Nu$, we have that $a \in V$ by Theorem 1. 
Hence $(I-A)b = 0$. 

Write $b = \ov b + b'$ with $\ov B \in V$ and $b' \in V^\perp$. 
Then $(I-A)\ov b = (I-A)b = 0$, since $V^\perp \subseteq \mathrm{Fix}(A)$ by Theorem 1. 
Hence $\ov b \in \mathrm{Fix}(A)$ for all $A$ in the point group of $\Nu$. 
Therefore $\ov b = 0$ by Lemma 9, since $Z(\Nu) = \{\ov I\}$. 
Hence $b+B$ is in the kernel $\Nu^\perp$ of the action of $\Gamma$ on $V$. 
Therefore $\Nu^\perp = C_\Gamma(\Nu)$. 
\end{proof}

\begin{corollary} 
Let $\Nu_i$ be a complete normal subgroup of an $n$-space group $\Gamma_i$ for $i=1,2$ 
such that the center of $\Nu_1$ is trivial. 
If $(\Gamma_1,\Nu_1)$ is isomorphic to $(\Gamma_2,\Nu_2)$, 
then $(\Gamma_1,\Nu_1^\perp)$ is isomorphic to $(\Gamma_2,\Nu_2^\perp)$.
\end{corollary}

\noindent{\bf Remark 8.}  Corollary 11 is not true if $\Nu_1$ is infinite cyclic and $n = 3$. 
The Seifert fibrations $({\ast}{:}{\times})$ and  $(2{\bar\ast}2{:}2)$, with IT numbers 9 and 15, respectively, 
in Table 1 of \cite{C-T}, with orthogonal dual Klein bottle fibers, 
were replaced by two different affine equivalent Seifert fibrations in Table 1 of \cite{R-T}
with orthogonally dual torus fibers.

 \section{The affine classification of co-Seifert geometric fibrations}  

In this section, we describe the classification of the co-Seifert geometric fibrations 
of compact, connected, flat $n$-orbifolds up to affine equivalence.  
By Theorem 10 of \cite{R-T} this is equivalent to classifying pairs $(\Gamma, \Nu)$, 
consisting of an $n$-space group $\Gamma$ and an $(n-1)$-dimensional, complete, normal subgroup $\Nu$, 
up to  isomorphism.  

Let $\Nu$ be an $(n-1)$-dimensional, complete, normal subgroup of an $n$-space group $\Gamma$. 
Then the quotient group $\Gamma/\Nu$ is either an infinite cyclic group or an infinite dihedral group. 
Assume first that $\Gamma/\Nu$ is an infinite cyclic group.  
Let $\gamma$ be an element of $\Gamma$ such that $\gamma\Nu$ is a generator of $\Gamma/\Nu$. 
Then $\Gamma$ is an HNN extension with base $\Nu$, stable letter $\gamma$, and automorphism $\gamma_\ast$ of $\Nu$ 
defined by $\gamma_\ast(\nu)= \gamma\nu\gamma^{-1}$ for each $\nu$ in $\Nu$.
If $\mu \in \Nu$, then $\gamma\mu\Nu = \gamma\Nu$ and $(\gamma\mu)_\ast = \gamma_\ast\mu_\ast$. 
Hence, the generator $\gamma\Nu$ of $\Gamma/\Nu$ determines a 
unique element $\gamma_\ast\mathrm{Inn}(N)$ of $\mathrm{Out}_E(\Nu)$. 
The other generator $\gamma^{-1}\Nu$ of $\Gamma/\Nu$ 
determines the element $\gamma_\ast^{-1}\mathrm{Inn}(N)$ of $\mathrm{Out}_E(\Nu)$. 
Hence, the pair $(\Gamma, \Nu)$ determines the pair of inverse elements 
$\{\gamma_\ast\mathrm{Inn}(N), \gamma_\ast^{-1}\mathrm{Inn}(N)\}$ of $\mathrm{Out}_E(\Nu)$. 
As usual $(x,y)$ denotes an ordered pair whereas $\{x,y\}$ denotes an unordered pair. 

\begin{lemma} 
Let $\Nu_i$ be a complete normal subgroup of an $n$-space group $\Gamma_i$ 
such that $\Gamma_i/\Nu_i$ is infinite cyclic for $i =1,2$.  
Let $\gamma_i\in \Gamma_i$ be such that $\gamma_i\Nu_i$ generates $\Gamma_i/\Nu_i$ for $i =1,2$. 
Then an isomorphism $\alpha: \Nu_1 \to \Nu_2$ extends to an isomorphism $\phi:\Gamma_1 \to \Gamma_2$ 
if and only if $\alpha(\gamma_1)_\ast\alpha^{-1} \mathrm{Inn}(\Nu_2) = (\gamma_2^{\pm 1})_\ast\mathrm{Inn}(\Nu_2)$. 
\end{lemma}
\begin{proof}
Suppose $\alpha$ extends to an isomorphism $\phi:\Gamma_1 \to \Gamma_2$. 
Then $\phi(\gamma_1)$ is an element of $\Gamma_2$ such that $\phi(\gamma_1)\Nu_2$ generates $\Gamma_2/\Nu_2$. 
Hence $\phi(\gamma_1)\Nu_2 = \gamma^{\pm 1}_2\Nu_2$. 
If $\nu \in \Nu_2$, then 
$$(\alpha(\gamma_1)_\ast\alpha^{-1})(\nu) =\alpha(\gamma_1\alpha^{-1}(\nu)\gamma_1^{-1}) 
= \phi(\gamma_1)\nu\phi(\gamma_1)^{-1}.$$
Hence $\alpha(\gamma_1)_\ast\alpha^{-1} = (\phi(\gamma_1))_\ast$, and  
so $\alpha(\gamma_1)_\ast\alpha^{-1} \mathrm{Inn}(\Nu_2) = (\gamma_2^{\pm 1})_\ast\mathrm{Inn}(\Nu_2)$. 

Conversely, suppose that 
$\alpha(\gamma_1)_\ast\alpha^{-1} \mathrm{Inn}(\Nu_2) = (\gamma_2^{\pm 1})_\ast\mathrm{Inn}(\Nu_2)$. 
By replacing $\gamma_2$ by $\gamma_2^{-1}$, if necessary, we may assume that 
$\alpha(\gamma_1)_\ast\alpha^{-1}\mathrm{Inn}(\Nu_2) = (\gamma_2)_\ast\mathrm{Inn}(\Nu_2)$. 
Then there exists $\mu \in \Nu_2$ such that $\alpha(\gamma_1)_\ast\alpha^{-1}= (\gamma_2)_\ast\mu_\ast$. 
Hence $\alpha(\gamma_1)_\ast\alpha^{-1} = (\gamma_2\mu)_\ast$. 
Define $\phi:\Gamma_1\to \Gamma_2$ by 
$$\phi(\nu\gamma_1^k) = \alpha(\nu)(\gamma_2\mu)^k$$
for each $\nu\in\Nu_1$ and $k\in\integers$. 
If $\nu, \lambda \in \Nu_1$ and $k,\ell \in \integers$, 
then we have 
\begin{eqnarray*}
\phi(\nu\gamma_1^k\lambda\gamma_1^\ell) 
& = & \phi(\nu\gamma_1^k\lambda\gamma_1^{-k}\gamma_1^{k+\ell}) \\
& = & \phi(\nu(\gamma_1)_\ast^k(\lambda)\gamma_1^{k+\ell}) \\
& = & \alpha(\nu(\gamma_1)_\ast^k(\lambda))(\gamma_2\mu)^{k+\ell} \\
& = & \alpha(\nu)\alpha(\gamma_1)_\ast^k\alpha^{-1}\alpha(\lambda)(\gamma_2\mu)^{k+\ell} \\
& = & \alpha(\nu)(\alpha(\gamma_1)_\ast\alpha^{-1})^k\alpha(\lambda)(\gamma_2\mu)^{k+\ell} \\
& = & \alpha(\nu)(\gamma_2\mu)_\ast^k\alpha(\lambda)(\gamma_2\mu)^{k+\ell} \\
& = & \alpha(\nu)(\gamma_2\mu)^k\alpha(\lambda)(\gamma_2\mu)^{-k}(\gamma_2\mu)^{k+\ell} \\
& = & \alpha(\nu)(\gamma_2\mu)^k\alpha(\lambda)(\gamma_2\mu)^{\ell} 
\ \ = \ \ \phi(\nu\gamma_1^k)\phi(\lambda\gamma_1^\ell).
\end{eqnarray*}
Hence $\phi$ is a homomorphism;  
moreover $\phi$ is an isomorphism, since $\phi$ restricts to $\alpha$ on $\Nu_1$ and 
$\phi(\gamma_1) = \gamma_2\mu$. 
\end{proof}

\begin{theorem} 
If $\Delta$ is an infinite cyclic $1$-space group and $\Mu$ is an $(n-1)$-space group, 
then the function $\omega: \mathrm{Iso}(\Delta,\Mu) \to \mathrm{Out}(\Delta,\Mu)$ is a bijection.
\end{theorem}
\begin{proof} Any homomorphism $\eta:\Delta \to \mathrm{Out}(\Mu)$ lifts to a homomorphism 
$\tilde\eta:\Delta\to\mathrm{Aff}(\Mu)$ such that $\Omega\tilde\eta = \eta$, 
since $\Omega: \mathrm{Aff}(\Mu) \to \mathrm{Out}(\Mu)$ is onto by Theorem 13. 
Hence the function $\chi:\mathrm{Aff}(\Delta,\Mu) \to \mathrm{Out}(\Delta,\Mu)$ is surjective. 
Therefore $\omega$ is surjective, since $\chi = \omega\psi$ by Lemma 18. 

Let $\Nu_i$ be a complete normal subgroup of an $n$-space group $\Gamma_i$ 
with $\Gamma_i/\Nu_i$ infinite cyclic for $i = 1,2$ such that 
$\omega([\Gamma_1,\Nu_1]) = \omega([\Gamma_2,\Nu_2])$.  
Let $\alpha_i: \Nu_i \to \Mu$ be an isomorphism for $i=1,2$.  
Let $\gamma_i \in \Gamma_i$ such that $\Nu_i\gamma_i$ generates $\Gamma_i/\Nu_i$ for each $i=1,2$, 
and let $\delta$ be a generator of $\Delta$. 
Define an isomorphism $\beta_i:\Delta\to \Gamma_i/\Nu_i$ by $\beta_i(\delta) = \gamma_i$ for $i =1,2$. 
Let $\mathcal{O}_i: \Gamma_i/\Nu_i \to \mathrm{Out}_E(\Nu_i)$ be the homomorphism induced by the action 
of $\Gamma_i$ on $\Nu_i$ by conjugation for $i=1,2$. 
Then we have 
$$[(\alpha_1)_\#\mathcal{O}_1\beta_1] = [(\alpha_2)_\#\mathcal{O}_2\beta_2].$$
Hence there is an automorphism $\zeta$ of $\Mu$ 
such that 
$$\{\zeta\alpha_1(\gamma_1^{\pm 1})_\ast\alpha_1^{-1}\zeta^{-1}\,\mathrm{Inn}(\Mu)\} = 
\{\alpha_2(\gamma_2^{\pm 1})_\ast\alpha_2^{-1}\mathrm{Inn}(\Mu)\}.$$
After applying the isomorphism $(\alpha_2^{-1})_\#: \mathrm{Out}(\Mu) \to \mathrm{Out}(\Nu_2)$, we have that 
$$\{\alpha_2^{-1}\zeta\alpha_1(\gamma_1^{\pm 1})_\ast\alpha_1^{-1}\zeta^{-1}\alpha_2\mathrm{Inn}(\Nu_2)\} = 
\{(\gamma_2^{\pm 1})_\ast\mathrm{Inn}(\Nu_2)\}.$$
Hence the isomorphism $\alpha_2^{-1}\zeta\alpha_1: \Nu_1 \to \Nu_2$ extends to an isomorphism 
$\phi:(\Gamma_1, \Nu_1) \to (\Gamma_2,\Nu_2)$ by Lemma 19. 
Therefore $[\Gamma_1,\Nu_1]= [\Gamma_2,\Nu_2]$, and so $\omega$ is injective. 
\end{proof}

\begin{lemma} 
Let $\Mu$ be an $(n-1)$-space group, and let $\Delta$ be an infinite cyclic 1-space group 
with generator $\delta$.  
The set $\mathrm{Out}(\Delta,\Mu)$ is in one-to-one correspondence with the set 
of conjugacy classes of pairs of inverse elements of $\mathrm{Out}(\Mu)$ of finite order. 
The element $[\eta]$ of $\mathrm{Out}(\Delta,\Mu)$ corresponds to the conjugacy class of $\{\eta(\delta), \eta(\delta^{-1})\}$. 
\end{lemma}
\begin{proof}
The the set $\mathrm{Hom}_f(\Delta,\mathrm{Out}(\Mu))$ is in one-to-one correspondence 
with the set of elements of $\mathrm{Out}(\Mu)$ of finite order via the mapping $\eta\mapsto \eta(\delta)$, 
since $\Delta$ is the free group generated by $\delta$. 
The infinite cyclic group $\Delta$ has a unique automorphism that maps $\delta$ to $\delta^{-1}$ 
and this automorphism represents the generator of the group $\mathrm{Aut}(\Delta) = \mathrm{Out}(\Delta)$ 
of order 2. 
Therefore the set $\mathrm{Out}(\Delta,\Mu)$ is in one-to-one correspondence with the set 
of conjugacy classes of pairs of inverse elements of $\mathrm{Out}(\Mu)$ of finite order via 
the mapping $[\eta] \mapsto [\{\eta(\delta), \eta(\delta^{-1}\}]$. 
\end{proof}

\begin{theorem} 
Let $\Mu$ be an $(n-1)$-space group, and let $\Delta$ be an infinite cyclic 1-space group.  
The set $\mathrm{Iso}(\Delta,\Mu)$ is in one-to-one correspondence with the 
set of conjugacy classes of pairs of inverse elements of $\mathrm{Out}(\Mu)$ of finite order.  
If $[\Gamma,\Nu] \in \mathrm{Iso}(\Delta,\Mu)$ and 
$\alpha: \Nu \to \Mu$ is an isomorphism and $\gamma$ is an element of $\Gamma$ such that $\Nu\gamma$ generates 
$\Gamma/\Nu$, then $[\Gamma, \Nu]$ corresponds to the conjugacy class 
of the pair of inverse elements $\{\alpha\gamma_\ast^{\pm 1}\alpha^{-1}\mathrm{Inn}(\Mu)\}$ of $\mathrm{Out}(\Mu)$.   
\end{theorem}
\begin{proof}
The function $\omega: \mathrm{Iso}(\Delta,\Mu) \to \mathrm{Out}(\Delta,\Mu)$ is a bijection by Theorem 22. 
Let $[\Gamma,\Nu]\in\mathrm{Iso}(\Delta,\Mu)$, 
and let $\gamma$ be an element of $\Gamma$ such that $\Nu\gamma$ generates $\Gamma/\Nu$. 
Let  $\mathcal{O}: \Gamma/\Nu \to \mathrm{Out}_E(\Nu)$ 
be the homomorphism induced by the action of $\Gamma$ on $\Nu$ by conjugation. 
Let $\alpha: \Nu \to \Mu$ and $\beta:\Delta \to \Gamma/\Nu$ be isomorphisms.  
Then $\omega([\Gamma,\Nu]) = [\alpha_\#\mathcal{O}\beta]$, 
and $ [\alpha_\#\mathcal{O}\beta]$ corresponds to the conjugacy class of the pair of inverse elements 
\begin{eqnarray*}
\{\alpha_\#\mathcal{O}\beta(\beta^{-1}(\Nu\gamma^{\pm 1}))\}  
& = & \{\alpha_\#\mathcal{O}(\Nu\gamma^{\pm 1})\}  \\ 
& = & \{\alpha_\#(\gamma_\ast^{\pm 1}\mathrm{Inn}(\Nu))\}  
\ \ = \ \ \{\alpha\gamma_\ast^{\pm 1}\alpha^{-1}\mathrm{Inn}(\Mu)\}
\end{eqnarray*}
of $\mathrm{Out}(\Mu)$ by Lemma 20. 
\end{proof}

Let $\Delta$ be an infinite dihedral group.  
A set of {\it Coxeter generators} of $\Delta$ is a pair of elements of $\Delta$ of order 2 that generate $\Delta$. 
Any two sets of Coxeter generators of $\Delta$ are conjugate in $\Delta$. 

\begin{lemma} 
Let $\Mu$ be an $(n-1)$-space group, and let $\Delta$ be an infinite dihedral 1-space group 
with Coxeter generators $\delta_1$ and $\delta_2$.  
The set $\mathrm{Out}(\Delta,\Mu)$ is in one-to-one correspondence with the set 
of conjugacy classes of pairs of elements of $\mathrm{Out}(\Mu)$ of order 1 or  2 whose product has finite order. 
The element $[\eta]$ of $\mathrm{Out}(\Delta,\Mu)$ corresponds to the conjugacy class of 
$\{\eta(\delta_1),\eta(\delta_2)\}$. 
\end{lemma}
\begin{proof}
The set $\mathrm{Hom}_f(\Delta,\mathrm{Out}(\Mu))$ is in one-to-one correspondence 
with the set of ordered pairs of elements of $\mathrm{Out}(\Mu)$, of order 1 or 2  whose product has finite order, 
via the mapping $\eta\mapsto (\eta(\delta_1),\eta(\delta_2))$, 
since $\Delta$ is the free product of the cyclic groups of order two generated by $\delta_1$ and $\delta_2$. 
The infinite dihedral group $\Delta$ has a unique automorphism that transposes $\delta_1$ and $\delta_2$,  
and this automorphism represents the generator of the group $\mathrm{Out}(\Delta)$ of order 2. 
Therefore, the set $\mathrm{Out}(\Delta,\Mu)$ is in one-to-one correspondence with the set 
of conjugacy classes of unordered pairs of elements of $\mathrm{Out}(\Mu)$, of order 1 or  2 whose product has 
of finite order, via the mapping $[\eta] \mapsto [\{\eta(\delta_1), \eta(\delta_2\}]$. 
\end{proof}

\begin{theorem} 
Let $\Mu$ be an $(n-1)$-space group with trivial center, 
and let $\Delta$ be an infinite dihedral 1-space group.  
The set $\mathrm{Iso}(\Delta,\Mu)$ is in one-to-one correspondence 
with the set of conjugacy classes of pairs of elements of $\mathrm{Out}(\Mu)$ of order 1 or  2 
whose product has finite order. 
If $[\Gamma,\Nu] \in \mathrm{Iso}(\Delta,\Mu)$ and $\alpha: \Nu \to \Mu$ is an isomorphism,  
and $\gamma_1, \gamma_2$ are elements of $\Gamma$ 
such that $\{\Nu\gamma_1, \Nu\gamma_2\}$ is a set of Coxeter generators of $\Gamma/\Nu$, 
then $[\Gamma, \Nu]$ corresponds to the conjugacy class of the pair of elements 
$\{\alpha(\gamma_1)_\ast\alpha^{-1} \mathrm{Inn}(\Mu), \alpha(\gamma_2)_\ast\alpha^{-1}\mathrm{Inn}(\Mu)\}$ 
of $\mathrm{Out}(\Mu)$.   
\end{theorem}
\begin{proof}
The function $\omega: \mathrm{Iso}(\Delta,\Mu) \to \mathrm{Out}(\Delta,\Mu)$ is a bijection by Corollary 10. 
Let $[\Gamma,\Nu]\in\mathrm{Iso}(\Delta,\Mu)$, 
and let $\gamma_1,\gamma_2$ be elements of $\Gamma$ such that 
$\{\Nu\gamma_1, \Nu\gamma_2\}$ is a set of Coxeter generators of $\Gamma/\Nu$. 
Let  $\mathcal{O}: \Gamma/\Nu \to \mathrm{Out}_E(\Nu)$ 
be the homomorphism induced by the action of $\Gamma$ on $\Nu$ by conjugation. 
Let $\alpha: \Nu \to \Mu$ and $\beta:\Delta \to \Gamma/\Nu$ be isomorphisms.  
Then $\omega([\Gamma,\Nu]) = [\alpha_\#\mathcal{O}\beta]$, 
and $ [\alpha_\#\mathcal{O}\beta]$ corresponds to the conjugacy class of the pair of elements 
\begin{eqnarray*}
\{\alpha_\#\mathcal{O}\beta\beta^{-1}(\Nu\gamma_1), \alpha_\#\mathcal{O}\beta\beta^{-1}(\Nu\gamma_2)\} 
& =  & \{\alpha_\#\mathcal{O}(\Nu\gamma_1), \alpha_\#\mathcal{O}(\Nu\gamma_2)\}  \\ 
& =  & \{\alpha_\#(\gamma_1)_\ast\mathrm{Inn}(\Nu)), \alpha_\#(\gamma_2)_\ast\mathrm{Inn}(\Nu))\}  \\
& =  & \{\alpha(\gamma_1)_\ast\alpha^{-1}\mathrm{Inn}(\Mu), \alpha(\gamma_2)_\ast\alpha^{-1}\mathrm{Inn}(\Mu)\}
\end{eqnarray*}
of $\mathrm{Out}(\Mu)$ by Lemma 21. 
\end{proof}

Let $\Mu$ be an $(n-1)$-space group, and let $\Delta$ be an infinite dihedral 1-space group. 
When $\Mu$ has nontrivial center, the set $\mathrm{Iso}(\Delta,\Mu)$ is best understood by describing the fibers 
of  the surjection $\psi: \mathrm{Aff}(\Delta,\Mu) \to \mathrm{Iso}(\Delta,\Mu)$.

\begin{lemma} 
Let $\Mu$ be an $(n-1)$-space group, and let $\Delta$ be an infinite dihedral 1-space group 
with Coxeter generators $\delta_1$ and $\delta_2$.  
The set $\mathrm{Aff}(\Delta,\Mu)$ is in one-to-one correspondence with the set 
of conjugacy classes of pairs of elements of $\mathrm{Aff}(\Mu)$ of order 1 or  2 whose product has image 
of finite order under the epimorphism $\Omega: \mathrm{Aff}(\Mu) \to \mathrm{Out}(\Mu)$. 
The element $[\eta]$ of $\mathrm{Aff}(\Delta,\Mu)$ corresponds to the conjugacy class of 
$\{\eta(\delta_1),\eta(\delta_2)\}$. 
\end{lemma}
\begin{proof}
Via the mapping $\eta\mapsto (\eta(\delta_1),\eta(\delta_2))$, 
the set $\mathrm{Hom}_f(\Delta,\mathrm{Aff}(\Mu))$ is in one-to-one correspondence 
with the set of ordered pairs of elements of $\mathrm{Aff}(\Mu)$ of order 1 or 2  whose product has image 
of finite order under the epimorphism $\Omega: \mathrm{Aff}(\Mu) \to \mathrm{Out}(\Mu)$, 
since $\Delta$ is the free product of the cyclic groups of order two generated by $\delta_1$ and $\delta_2$.  
The infinite dihedral group $\Delta$ has a unique automorphism that transposes $\delta_1$ and $\delta_2$,  
and this automorphism represents the generator of the group $\mathrm{Out}(\Delta)$ of order 2. 
Therefore, via the mapping $[\eta] \mapsto [\{\eta(\delta), \eta(\delta^{-1}\}]$,  
the set $\mathrm{Aff}(\Delta,\Mu)$ is in one-to-one correspondence with the set 
of conjugacy classes of unordered pairs of elements of $\mathrm{Aff}(\Mu)$ of order 1 or  2 whose product has image 
of finite order under the epimorphism $\Omega: \mathrm{Aff}(\Mu) \to \mathrm{Out}(\Mu)$. 
\end{proof}

In order to simplify matters, we will assume $\Delta$ is the {\it standard} 
infinite dihedral 1-space group with Coxeter generators the reflections $\delta_1=-I'$ and $\delta_2= 1-I'$ of $E^1$.  
We denote the identity maps of $E^n, E^{n-1}, E^1$ by $I, \ov I, I'$, respectively, 
and we will identify $E^1$ with $(E^{n-1})^\perp$ in $E^n$. 
The next theorem gives an algebraic description of 
the fibers of the surjection  $\psi: \mathrm{Aff}(\Delta,\Mu) \to \mathrm{Iso}(\Delta,\Mu)$. 

\begin{theorem}  
Let let $\Mu$ be an $(n-1)$-space group, and let 
$\Delta$ be the infinite dihedral $1$-space group generated by $\delta_1=-I'$ and $\delta_2= 1-I'$.  
Let $\eta_1,\eta_2 \in \mathrm{Hom}_f(\Delta,\mathrm{Aff}(\Mu))$, 
and let $\eta_1(\delta_i) = (a_i+A_i)_\star$, with $a_i+A_i \in N_A(\Mu)$,  for $i = 1,2$. 
Let $E_i$ be the $(-1)$-eigenspace of the restriction of $A_i$ to $\mathrm{Span}(Z(\Mu))$ for $i = 1,2$. 
Then the surjection $\psi: \mathrm{Aff}(\Delta,\Mu) \to \mathrm{Iso}(\Delta,\Mu)$ 
has the property that $\psi([\eta_1]) = \psi([\eta_2])$ if and only if 
there is a vector $v \in E_1\cap E_2$
such that $\{\eta_1(\delta_1), (v+\ov I)_\star \eta_1(\delta_2)\}$ is conjugate to $\{\eta_2(\delta_1),\eta_2(\delta_2)\}$ 
by an element of $\mathrm{Aff}(\Mu)$. 
\end{theorem}
\begin{proof}
By Theorem 14, there exists $C_i\in \mathrm{GL}(n-1,\realnos)$ such that $C_i\Mu C_i^{-1}$ is an $(n-1)$-space group 
and $(C_i)_\sharp\eta_i(\Delta) \subseteq \mathrm{Isom}(E^{n-1}/C_i\Mu C_i^{-1})$ for $i=1,2$. 
Extend  $C_i\Mu C_i^{-1}$ to a subgroup $\Nu_i$ of $\mathrm{Isom}(E^n)$ such that 
the point group of $\Nu_i$ acts trivially on $(E^{n-1})^\perp$ for $i = 1,2$. 
By Theorem 18, there exists an $n$-space group $\Gamma_i$ containing $\Nu_i$ as a complete normal subgroup 
such that $\Gamma_i' = \Delta$ and if $\Xi_i: \Gamma_i/\Nu_i \to \mathrm{Isom}(E^{n-1}/\Nu_i)$ is the homomorphism 
induced by the action of $\Gamma_i/\Nu_i$ on $E^{n-1}/\Nu_i$, then $\Xi_i = (C_i)_\sharp\eta_i\Rho_i$ 
where $\Rho_i:\Gamma_i/\Nu_i \to \Gamma_i'$ is the isomorphism defined by $\Rho_i(\Nu_i\gamma) = \gamma'$. 
Then $\psi([\eta_i]) = [\Gamma_i,\Nu_i]$ for $i=1,2$. 

Suppose $\psi([\eta_1]) = \psi([\eta_2])$.  
Then $[\Gamma_1,\Nu_1] = [\Gamma_2,\Nu_2]$. 
By Bieberbach's theorem, there exists an affinity $\phi = c + C$ of $E^n$ such that 
$\phi(\Gamma_1,\Nu_1)\phi^{-1} = (\Gamma_2,\Nu_2)$. 
By Theorem 19, we have that 
$$\Xi_2\Rho_2^{-1}(\phi')_\ast\Rho_1 = (\ov{C'}p_1)_\star(\ov\phi)_\sharp\Xi_1.$$
Hence we have 
$$(C_2)_\sharp\eta_2(\phi')_\ast= (\ov{C'}p_1)_\star\Rho_1^{-1}(\ov\phi)_\sharp(C_1)_\sharp\eta_1,$$
Now $(\phi')_\star(\{\delta_1,\delta_2\} )$ is a set of Coxeter generators of $\Delta$, 
and hence there exists $\delta \in \Delta$ such that $(\phi')_\star(\{\delta_1,\delta_2\} )= \delta\{\delta_1,\delta_2\}\delta^{-1}$, 
and say that $(\phi')_\star(\delta_1) = \delta\delta_j\delta^{-1}$ and $(\phi')_\star(\delta_2) = \delta\delta_k\delta^{-1}$.  
Upon evaluating at $\delta_1$ and $\delta_2$, we have 
$$(C_2)_\star\eta_2(\delta\delta_j\delta^{-1})(C_2)_\star^{-1} = (\ov\phi C_1)_\star \eta_1(\delta_1)(\ov\phi C_1)_\star^{-1}, $$
and 
$$(C_2)_\star\eta_2(\delta\delta_k\delta^{-1})(C_2)_\star^{-1} = 
(\ov{C'}(e_n) + \ov I)_\star(\ov\phi C_1)_\star \eta_1(\delta_2)(\ov\phi C_1)_\star^{-1}. $$
Therefore 
$$\eta_2(\delta_j) =
\eta_2(\delta)^{-1}(C_2^{-1}\ov\phi C_1)_\star \eta_1(\delta_1)(C_2^{-1}\ov\phi C_1)_\star^{-1}\eta_2(\delta), $$
and 
$$\eta_2(\delta_k) =
\eta_2(\delta)^{-1}(C_2^{-1}\ov\phi C_1)_\star(\ov \phi C_1)_\star^{-1}(\ov{C'}(e_n)+\ov I)_\star(\ov\phi C_1)_\star \eta_1(\delta_2)(C_2^{-1}\ov\phi C_1)_\star^{-1}\eta_2(\delta).$$
Now we have that 
$$(\ov \phi C_1)_\star^{-1}(\ov{C'}(e_n)+\ov I)_\star(\ov\phi C_1)_\star  = (C_1^{-1}\ov C^{-1}\ov{C'}(e_n)+\ov I)_\star.$$
As $\ov{C'}(e_n) \in \mathrm{Span}(Z(\Nu_2))$, we have that $\ov{C}^{-1}\ov{C'}(e_n) \in \mathrm{Span}(Z(\Nu_1))$, 
and so we have that $C_1^{-1}\ov{C}^{-1}\ov{C'}(e_n) \in \mathrm{Span}(Z(\Mu))$.

We have that 
$$(C_1)_\sharp\eta_1(\delta_i) = (C_1a_i + C_1A_iC_1^{-1})_\star$$
for $i =1,2$. 
Let $\hat \delta_i = b_i + B_i \in \Gamma_1$ be formed from $C_1a_i+ C_1A_iC_1^{-1}$ and $\delta_i$, 
as in the proof of Theorem 18, so that $\hat\delta_i' = \delta_i$ for $i =1,2$.  
Then $\ov B_i = C_1A_iC_1^{-1}$ and $B_i' = -I'$ for $i = 1,2$, and $b_1' = 0$ and $b_2' = e_n$. 
As $\ov{C'}(B_i') = \ov C \ov B_i \ov C^{-1} \ov{C'}$, we have that $\ov{C'}(e_n)$ is 
in the $(-1)$-eigenspace of $\ov C \ov B_i \ov C^{-1}$. 
Hence $\ov C^{-1}\ov{C'}(e_n)$ is in the $(-1)$-eigenspace of $\ov B_i$, 
and so $C_1^{-1}\ov C^{-1}\ov{C'}(e_n)$ is in the $(-1)$-eigenspace of $A_i$ for $i = 1,2$.  
Let $v = C_1^{-1}\ov C^{-1}\ov{C'}(e_n)$.  Then $v \in E_1\cap E_2$, and we have that 
$\{\eta_1(\delta_1), (v+\ov I)_\star \eta_1(\delta_2)\}$ 
is conjugate to $\{\eta_2(\delta_1),\eta_2(\delta_2)\}$ by the element $\eta_2(\delta)^{-1}(C_2^{-1}\ov\phi C_1)_\star$ 
of $\mathrm{Aff}(\Mu)$. 

Conversely, suppose there is a vector $v \in E_1\cap E_2$
such that $\{\eta_1(\delta_1), (v+\ov I)_\star \eta_1(\delta_2)\}$ is conjugate to $\{\eta_2(\delta_1),\eta_2(\delta_2)\}$ 
by an element of $\mathrm{Aff}(\Mu)$. 
Let $\xi\in N_A(\Mu)$ such that 
$\eta_2(\delta_j) = \xi_\star \eta_1(\delta_1)\xi_\star^{-1}$ and $\eta_2(\delta_k) =
 \xi_\star(v+\ov I)_\star\eta_1(\delta_2)\xi_\star^{-1}$ with $\{1,2\} = \{j,k\}$. 
Define $\alpha \in \mathrm{Aff}(E^{n-1})$ by $\alpha = C_2\xi C_1^{-1}$.  
Then $\alpha C_1\Mu C_1^{-1}\alpha^{-1} = C_2\Mu C_2^{-1}$, and so $\alpha\ov\Nu_1\alpha^{-1} = \ov \Nu_2$. 
Let $\beta \in\mathrm{Aff}((E^{n-1})^\perp)$ be either the identity map $I'$ if $j = 1$ 
or the reflection $1/2-I'$ if $k = 1$. 
Then $\beta_\ast$ is the automorphism of $\Delta$ that maps $(\delta_1,\delta_2)$ to $(\delta_j, \delta_k)$. 

Write $\xi = a + A$ with $a \in E^{n-1}$ and $A \in \mathrm{GL}(n-1)$. 
Define a linear transformation $D: (E^{n-1})^\perp \to \mathrm{Span}(Z(\Nu_2))$ 
by $D(e_n) = C_2 A v$. 
Let $\phi = c+C \in \mathrm{Aff}(E^n)$ be such that $\ov \phi = \alpha$, and $\phi' = \beta$, and $\ov{C'} = D$. 
Then $\ov C = C_2AC_1^{-1}$.  
For $\hat\delta_i=b_i+B_i$, we have that $\ov B_i = C_1A_iC_1^{-1}$ for $i = 1,2$. 
Now $v$ is in the $(-1)$-eigenspace of $A_i$ for $i= 1,2$, 
and so $Av$ is in the $(-1)$-eigenspace of $AA_iA^{-1}$ for $i= 1,2$. 
Hence $C_2Av$ is in the $(-1)$-eigenspace of $C_2AA_iA^{-1}C_2^{-1}$ for $i= 1,2$.
Therefore $C_2Av$ is in the $(-1)$-eigenspace of $\ov C\ov B_i\ov C^{-1}$ for $i= 1,2$. 
Hence $DB_i' = \ov C\ov B_i\ov C^{-1}D$ for $i = 1,2$. 
As $\Nu_1\hat\delta_1$ and $\Nu_1\hat\delta_2$ generate $\Gamma_1/\Nu_1$, 
we have that if $b+B \in \Gamma_1$, then $DB' = \ov C \ov B\ov C^{-1}D$. 

Observer that 
\begin{eqnarray*}
\Xi_2\Rho_2^{-1}\beta_\ast\Rho_1(\Nu_1\hat\delta_1) 
& = & \Xi_2\Rho_2^{-1}\beta_\ast(\delta_1) \\
& = & (C_2)_\sharp \eta_2(\delta_j) \\ 
& = & (C_2)_\sharp \xi_\ast \eta_1(\delta_1)\xi_\ast^{-1} \\
& = & (C_2)_\sharp\xi_\sharp (C_1)_\sharp^{-1} (C_1)_\sharp\eta_1\Rho_1\Rho_1^{-1}(\delta_1) \\
& = & \alpha_\sharp\Xi_1(\Nu_1\hat\delta_1) \ \
 = \ \  (Dp_1)_\star(\Nu_1\hat\delta_1)\alpha_\sharp\Xi_1(\Nu_1\hat\delta_1),  
\end{eqnarray*}
and
\begin{eqnarray*}
\Xi_2\Rho_2^{-1}\beta_\ast\Rho_1(\Nu_1\hat\delta_2) 
& = & \Xi_2\Rho_2^{-1}\beta_\ast(\delta_2) \\
& = & (C_2)_\sharp \eta_2(\delta_k) \\ 
& = & (C_2)_\sharp \xi_\ast(v+\ov I)_\ast  \eta_1(\delta_2)\xi_\ast^{-1} \\
& = & (C_2)_\ast \xi_\ast(v+\ov I)_\ast \xi_\ast^{-1}(C_2)_\ast^{-1}(C_2)_\ast\xi_\ast \eta_1(\delta_2)\xi_\ast^{-1}(C_2)_\ast^{-1} \\
& = & (C_2 A v+\ov I)_\ast (C_2)_\ast\xi_\ast (C_1)_\ast^{-1} (C_1)_\ast \eta_1(\delta_2)\xi_\ast^{-1}(C_2)_\ast^{-1} \\
& = & (De_n+\ov I)_\ast \alpha_\ast (C_1)_\ast \eta_1(\delta_2)(C_1)^{-1}_\ast \alpha_\ast^{-1} \\
& = &  (Db_2'+\ov I)_\ast \alpha_\sharp (C_1)_\sharp\eta_1\Rho_1\Rho_1^{-1}(\delta_2) \\
& = & (Dp_1)_\star(\Nu_1\hat\delta_2)\alpha_\sharp\Xi_1(\Nu_1\hat\delta_2).   
\end{eqnarray*}

Let $K_2$ be the connected component of the identity of $\mathrm{Isom}(E^{n-1}/\Nu_2)$. 
Then $K_2$ is the kernel of the epimorphism $\Omega: \mathrm{Isom}(E^{n-1}/\Nu_2) \to \mathrm{Out}_E(\Nu_2)$ 
by Theorem 12.  Upon applying $\Omega$, we have that
$$\Omega\Xi_2\Rho_2^{-1}\beta_\ast\Rho_1(\Nu_1\hat\delta_1) = \Omega\alpha_\sharp\Xi_1(\Nu_1\hat\delta_1),$$
and 
$$\Omega\Xi_2\Rho_2^{-1}\beta_\ast\Rho_1(\Nu_1\hat\delta_2) = \Omega\alpha_\sharp\Xi_1(\Nu_1\hat\delta_2).$$
Hence $\Omega\Xi_2\Rho_2^{-1}\beta_\ast\Rho_1 = \Omega\alpha_\sharp\Xi_1$, 
since $\Nu_1\hat\delta_1$ and $\Nu_1\hat\delta_2$ generate $\Gamma_1/\Nu_1$. 
Therefore the ratio of homomorphisms $(\Xi_2\Rho_2^{-1}\beta_\ast\Rho_1)(\alpha_\sharp\Xi_1)^{-1}$ 
maps to the abelian group $K_2$. 
Hence $(\Xi_2\Rho_2^{-1}\beta_\ast\Rho_1)(\alpha_\sharp\Xi_1)^{-1}: \Gamma_1/\Nu_1 \to K_2$ 
is a crossed homomorphism with $\Nu_1(b+B)$ acting on $K_2$ by conjugation by 
\begin{eqnarray*}
\Xi_2\Rho_2^{-1}\beta_\ast\Rho_1(\Nu_1(b+B)) 
& = & \Xi_2\Rho_2^{-1}\beta_\ast(b'+B') \\
& = & \Xi_2\Rho_2^{-1}((c'+C')(b'+B')(c'+C')^{-1})\\
& = & \Xi_2(\Nu_2(c+C)(b+B)(c+C)^{-1}) \\
& = & ((\ov c+\ov C)(\ov b+\ov B)(\ov c+\ov C)^{-1})_\star, 
\end{eqnarray*} 
that is, if $u \in \mathrm{Span}(Z(\Nu_2))$, 
then 
$$(\Nu_2(b+B))(u+I)_\star = (\ov C\ov B\ov C^{-1}u + I)_\star.$$
The mapping  $(Dp_1)_\star: \Gamma_1/\Nu_1 \to K_2$ is a crossed homomorphism 
with respect to the same action of $\Gamma_1/\Nu_1$ on $K_2$ by Theorem 19. 
Hence the crossed homomorphisms $(\Xi_2\Rho_2^{-1}\beta_\ast\Rho_1)(\alpha_\sharp\Xi_1)^{-1}$ 
and $(Dp_1)_\star$ are equal, since they agree on the generators 
$\Nu_1\hat\delta_1$ and $\Nu_1\delta_2$ of $\Gamma_1/\Nu_1$. 
Therefore 
$$\Xi_2\Rho_2^{-1}\beta_\ast\Rho_1= (Dp_1)_\star\alpha_\sharp\Xi_1.$$
Hence $\phi(\Gamma_1,\Nu_1)\phi^{-1} = (\Gamma_2,\Nu_2)$ by Theorem 19. 
Therefore we have that 
$$\psi([\eta_1]) = [\Gamma_1,\Nu_1] = [\Gamma_2,\Nu_2] = \psi([\eta_2]).$$

\vspace{-.2in}
\end{proof}

\begin{theorem} 
Let let $\Mu$ be an $(n-1)$-space group, and let 
$\Delta$ be the infinite dihedral $1$-space group generated by $\delta_1=-I'$ and $\delta_2= 1-I'$.  
Let $[\Gamma,\Nu] \in \mathrm{Iso}(\Delta,\Mu)$,  let $V =  \mathrm{Span}(\Nu)$, 
and let $\alpha: V \to E^{n-1}$ be an affinity such that $\alpha\ov \Nu \alpha^{-1} =  \Mu$. 
Let $\{\Nu\gamma_1,\Nu\gamma_2\}$ be a set of Coxeter generators of $\Gamma/\Nu$, and 
let $\beta:\Delta \to \Gamma/\Nu$ be the isomorphism defined by $\beta(\delta_i) = \Nu\gamma_i$ for $i = 1, 2$. 
Let $\Xi: \Gamma/\Nu \to \mathrm{Isom}(V/\Nu)$ be the homomorphism induced by the action of $\Gamma/\Nu$ 
on $V/\Nu$. 
Write $\gamma_i = b_i + B_i$ with $b_i \in E^n$ and $B_i \in O(n)$ for $i = 1, 2$, and  
let $E_i$ be the $(-1)$-eigenspace of $B_i$ restricted to $\mathrm{Span}(Z(\Nu))$ for $i = 1, 2$. 
Then $[\alpha_\sharp\Xi\beta] \in \psi^{-1}([\Gamma,\Nu])$,   
and $[\alpha_\sharp\Xi\beta]$ corresponds to the conjugacy class of the pair of elements 
$\{(\alpha\ov\gamma_1\alpha^{-1})_\star, (\alpha\ov\gamma_2\alpha^{-1})_\star\}$ of $\mathrm{Aff}(\Mu)$. 
Moreover $\psi^{-1}([\Gamma,\Nu])$ is the set of all the elements  of $\mathrm{Aff}(\Delta,\Mu)$ 
that correspond to the conjugacy class of the pair of elements 
$\{(\alpha\ov\gamma_1\alpha^{-1})_\star, (\alpha(v+\ov I)\ov\gamma_2\alpha^{-1})_\star\}$
of $\mathrm{Aff}(\Mu)$ for any $v \in E_1\cap E_2$. 
In particular, if $E_1\cap E_2 =\{0\}$, then $\psi^{-1}([\Gamma,\Nu]) = \{[\alpha_\sharp\Xi\beta]\}$. 
\end{theorem} 
\begin{proof}
By the proof of Lemma 16, we have that $[\alpha_\sharp\Xi\beta] \in \psi^{-1}([\Gamma,\Nu])$. 
For $i = 1,2$, we have that 
$$\alpha_\sharp \Xi \beta(\delta_i) =\alpha_\sharp\Xi(\Nu\gamma_i) = 
 \alpha_\sharp((\ov \gamma_i)_\star) = \alpha_\star(\ov \gamma_i)_\star\alpha_\star^{-1} 
 = (\alpha\ov\gamma_i\alpha^{-1})_\star.$$
Hence $[\alpha_\sharp\Xi\beta]$ corresponds to the conjugacy class of the pair 
$\{(\alpha\ov\gamma_1\alpha^{-1})_\star, (\alpha\ov\gamma_2\alpha^{-1})_\star\}$ by Lemma 22. 
Write $\alpha = a + A$ with $a \in E^{n-1}$ and $A: V \to E^{n-1}$ a linear isomorphism. 
Then for $i = 1, 2$, we have 
$$\alpha\ov \gamma_i\alpha^{-1} = (a+A)(\ov b_i + \ov B_i)(a+A)^{-1} = A\ov b_i + (I-A\ov B_iA^{-1})a + A\ov B_iA^{-1}.$$
The $(-1)$-eigenspace of $A\ov B_iA^{-1}$ restricted to $\mathrm{Span}(Z(\Mu))$ is $A(E_i)$ for $i = 1,2$. 

Suppose $v \in E_1\cap E_2$. Then  $Av \in A(E_1)\cap A(E_2)$ and 
$$\alpha(v+\ov I)\ov\gamma_2\alpha^{-1} = 
\alpha(v+\ov I)\alpha^{-1}\alpha\ov\gamma_2\alpha^{-1} = (Av + \ov I)\alpha\ov\gamma_2\alpha^{-1}.$$
Now we have that 
\begin{eqnarray*}
((Av + \ov I)\alpha\ov\gamma_2\alpha^{-1})_\star^2 
& =  & ((Av + \ov I)\alpha\ov\gamma_2\alpha^{-1}(Av + \ov I)\alpha\ov\gamma_2\alpha^{-1})_\star \\
& =  & ((Av + \ov I)(-Av+\ov I) \alpha\ov\gamma_2\alpha^{-1} \alpha\ov\gamma_2\alpha^{-1})_\star 
\ \  =\ \  \ov I_\star.
\end{eqnarray*}
Hence the order of  $((Av + \ov I)\alpha\ov\gamma_2\alpha^{-1})_\star$ is at most 2. 
Moreover the element
$$\Omega((\alpha\ov\gamma_1\alpha^{-1})_\star((Av + \ov I)\alpha\ov\gamma_2\alpha^{-1})_\star)$$
of $\mathrm{Out}(\Mu)$ has finite order, since 
$$\Omega((\alpha\ov\gamma_1\alpha^{-1})_\star((Av + \ov I)\alpha\ov\gamma_2\alpha^{-1})_\star) 
= \Omega((\alpha\ov\gamma_1\alpha^{-1})_\star(\alpha\ov\gamma_2\alpha^{-1})_\star).$$
Therefore the conjugacy class of the pair of elements 
$\{(\alpha\ov\gamma_1\alpha^{-1})_\star, (\alpha(v+\ov I)\ov\gamma_2\alpha^{-1})_\star\}$
of $\mathrm{Aff}(\Mu)$ correspond to an element of  $\psi^{-1}([\Gamma,\Nu])$ by Lemma 22 and Theorem 25. 
Thus  $\psi^{-1}([\Gamma,\Nu])$ is the set of all the elements of $\mathrm{Aff}(\Delta,\Mu)$ 
that correspond to the conjugacy class of the pair of elements 
$\{(\alpha\ov\gamma_1\alpha^{-1})_\star, (\alpha(v+\ov I)\ov\gamma_2\alpha^{-1})_\star\}$
of $\mathrm{Aff}(\Mu)$ for any $v \in E_1\cap E_2$ by Lemma 22 and Theorem 25. 
\end{proof}

\noindent{\bf Example 10.}  
Let $e_1$ and $e_2$ be the standard basis vectors of $E^2$. 
Let $\Gamma$ be the group generated by $t_1 = e_1+I$ and $t_2= e_2+I$ and $-I$.  
Then $\Gamma$ is a 2-space group, and $E^2/\Gamma$ is a pillow. 
Let $\Nu=\langle t_1\rangle$.  
Then $\Nu$ is a complete normal subgroup of $\Gamma$, 
with $V = \mathrm{Span}(\Nu) = \mathrm{Span}\{e_1\}$.  
The quotient $\Gamma/\Nu$ is an infinite dihedral group 
generated by $\Nu t_2$ and $\Nu(-I)$. 
Let $\gamma_1 = -I$ and $\gamma_2 = e_2 - I$. 
Then $\Nu\gamma_1$ and $\Nu\gamma_2$ are Coxeter generators of $\Gamma/\Nu$. 

Let $\Delta$ be the standard infinite dihedral group, 
and let $\Mu$ be the standard infinite cyclic 1-space group generated by $\ov t_1 = e_1 + \ov I$. 
By Theorem 26, we have that $\psi^{-1}([\Gamma,\Nu])$ 
consists of all the elements $[\eta] \in \mathrm{Aff}(\Delta,\Mu)$ 
that correspond to the conjugacy class of the pair of elements 
$\{(\ov \gamma_1)_\star, ((v+\ov I)\ov \gamma_2)_\star\}$ of $\mathrm{Isom}(E^1/\Mu)$  for any $v \in E^1$. 
Here $\ov \gamma_1 =\ov \gamma_2 = -\ov I$. 

The reflections $(\ov \gamma_1)_\star$ and $((v+\ov I)\ov \gamma_1)_\star$ 
of the circle $E^1/\Mu$  lie on the same  
connected component $C$ of the Lie group $\mathrm{Isom}(E^1/\Mu)$.  
Define a metric on $C$ so that $\xi: E^1/\Mu \to C$ defined by $\xi(\Mu v) = (v+\ov I)_\star (\ov \gamma_1)_\star$ is an isometry. 
Conjugating by an element of $\mathrm{Isom}(E^1/\Mu)$ is an isometry of $C$ with respect to this metric. 
Hence the distance between $(\ov \gamma_1)_\star$ and $((v+\ov I)\ov \gamma_1)_\star$
is an invariant of the conjugacy class of the pair $\{(\ov \gamma_1)_\star, ((v+\ov I)\ov \gamma_1)_\star\}$. 
Suppose $0\leq v \leq 1/2$. 
Then the distance between $(\ov \gamma_1)_\star$ and $((v+\ov I)\ov \gamma_1)_\star$ is $v$, 
since $((v+\ov I)\ov \gamma_1)_\star = (v+\ov I)_\star (\ov\gamma_1)_\star.$
Hence $\psi^{-1}([\Gamma,\Nu])$ has uncountably many elements. 

Let $v \in \realnos$ with $0 \leq v < 1$,  
and let $\Gamma_v$ be the group generated by $t_1 = e_1+I$ and $t_2 = ve_1+e_2+I$ and $-I$.  
Then $\Gamma_v$ is a 2-space group, and $E^2/\Gamma_v$ is a pillow. 
Let $\Nu=\langle t_1\rangle$.  
Then $\Nu$ is a complete normal subgroup of $\Gamma_v$, 
and $\Gamma_v/\Nu$ is  an infinite dihedral group. 
Let $\gamma_1 = -I$ and $\gamma_2 = ve_1+e_2-I$. 
Then $\Nu\gamma_1$ and $\Nu\gamma_2$ are Coxeter generators of $\Gamma_v/\Nu$. 
Let $\beta_v: \Delta \to \Gamma_v/\Nu$ be the isomorphism defined by $\beta_v(\delta_i) = \gamma_i$ for $i = 1,2$. 
Let $V = \mathrm{Span}(\Nu) = \mathrm{Span}\{e_1\}$, and 
let $\Xi_v: \Gamma_v/\Nu \to \mathrm{Isom}(V/\Nu)$ be the homomorphism induced 
by the action of $\Gamma_v/\Nu$ on $V/\Nu$. 
Then $[\Xi_v\beta_v] \in \psi^{-1}([\Gamma_v,\Nu])$ and $[\Xi_v\beta_v]$ 
corresponds to the pair of elements 
$$\{(\ov\gamma_1)_\star, (\ov\gamma_2)_\star\} = \{(\ov\gamma_1)_\star, ((v+I)\ov\gamma_1)_\star\}$$
of $\mathrm{Isom}(E^1/\Mu)$ by Theorem 26. 
Hence $[\Gamma_v,\Nu] = [\Gamma,\Nu]$ by Theorem 25. 

Let $\Kappa$ be the kernel of the action of $\Gamma_v$ on $V/\Nu$. 
The structure group $\Gamma_v/\Nu\Kappa$ is a dihedral group 
generated by $\Nu\Kappa t_2$ and $\Nu\Kappa (-I)$ ,  
The element $\Nu\Kappa t_2$ acts on the circle $V/\Nu$, of length one, 
by rotating a distance $v$, and  $\Nu\Kappa (-I)$ acts on $V/\Nu$ as a reflection. 
Let $c, d \in \integers$.  Then 
$t_1^ct_2^d = (c+dv)e_1+de_2 + I,$
and $t_1^ct_2^d \in \Kappa$ if and only if $c+dv = 0$.  
Thus if $v$ is irrational, then $\Kappa = \{I\}$, and $\Gamma_v/\Nu\Kappa$ is infinite.  
If $v = a/b$ with $a, b\in\integers$, $b> 0$, and $a, b$ coprime, 
then $\Nu\Kappa t_2$ has order $b$  in $\Gamma_v/\Nu\Kappa$, 
since $\Gamma_v/\Nu\Kappa$ acts effectively on $V/\Nu$ by Theorem 6. 
Thus $\Nu$ has an orthogonal dual in $\Gamma_v$ if and only if $v$ is rational. 
This example shows that  the order of the structure group $\Gamma_v/\Nu\Kappa$ 
is not necessarily an invariant of the affine equivalence class of $(\Gamma_v,\Nu)$. 
See also Example 1. 

\section{Action of the Structure Group} 

Let $\Nu$ be an $(n-1)$-dimensional, complete, normal subgroup of an $n$-space group $\Gamma$, 
and let $\Kappa$ be the kernel of the action of $\Gamma$ on $V = \mathrm{Span}(\Nu)$. 
In order to classify the pairs $(\Gamma,\Nu)$ up to isomorphism,  
we need to know how $\Gamma/\Nu$ acts on $V/\Nu$.  
We next study how to determine the action of $\Gamma/\Nu$ on $V/\Nu$
from the action of the structure group $\Gamma/\Nu\Kappa$ on $V/\Nu \times V^\perp/\Kappa$. 

We have a short exact sequence
$$1 \to \Nu\Kappa/\Nu \to \Gamma/\Nu \to \Gamma/(\Nu\Kappa) \to 1.$$
The group $\Gamma/\Nu$ is either infinite cyclic or infinite dihedral. 
By Lemma 23 below for the dihedral case, 
every normal subgroup of $\Gamma/\Nu$ of infinite index is trivial. 
Hence, if the structure group $\Gamma/(\Nu\Kappa)$ is infinite, 
then $\Kappa \cong \Nu\Kappa/\Nu$ is trivial, 
and so $\Gamma/\Nu$ is the structure group. 

Suppose that the structure group $\Gamma/(\Nu\Kappa)$ is finite. 
Then $\Gamma/(\Nu\Kappa)$ is either finite cyclic or finite dihedral. 
Suppose first that $\Gamma/(\Nu\Kappa)$ is finite cyclic of order $m$. 
Then the number of generators of $\Gamma/(\Nu\Kappa)$ is equal to the Euler phi function of $m$,  
and so $\Gamma/(\Nu\Kappa)$ may have more than two generators, 
and a generator of  $\Gamma/(\Nu\Kappa)$ may not lift to a generator of $\Gamma/\Nu$. 
The next theorem gives a necessary and sufficient condition for a generator of $\Gamma/\Nu\Kappa$ 
to lift to a generator of $\Gamma/\Nu$ with respect to the the quotient map 
from $\Gamma/\Nu$ to $\Gamma/(\Nu\Kappa)$. 

\begin{theorem} 
Let $\Nu$ be a complete, $(n-1)$-dimensional, normal subgroup of an $n$-space group $\Gamma$ 
such that $\Gamma/\Nu$ is infinite cyclic, and let $\Kappa$ be the kernel of the action of $\Gamma$ on $V = \mathrm{Span}(\Nu)$, 
and suppose that the structure group $\Gamma/\Nu\Kappa$ is finite of order $m$. 
Let $\gamma$ be an element of $\Gamma$ such that $\gamma\Nu\Kappa$ generates $\Gamma/\Nu\Kappa$.  
Then there exists an element $\delta$ of $\Gamma$ 
such that $\gamma\Nu\Kappa = \delta\Nu\Kappa$ and $\delta\Nu$ generates $\Gamma/\Nu$ 
if and only if $\gamma\Nu\Kappa$ acts on the circle $V^\perp/\Kappa$ by a rotation of $(360/m)^\circ$. 

\end{theorem}
\begin{proof}
Suppose $\delta\Nu$ is a generator of $\Gamma/\Nu$. 
Since $V^\perp/(\Gamma/\Nu)$ is a circle, $\delta$ acts as a translation $d+I$ on $V^\perp$ with $d\neq 0$. 
As $\Nu\Kappa/\Nu$ is a subgroup of $\Gamma/\Nu$ of index $m$ and $\Kappa \cong \Nu\Kappa/\Nu$ is infinite cyclic,  
the group $\Kappa$ has a generator that acts as a translation $md+I$ on $V^\perp$. 
Therefore $\delta\Nu\Kappa$ acts by a rotation of $(360/m)^\circ$ on the circle $V^\perp/\Kappa$. 

Suppose $\gamma$ is an element of $\Gamma$ such that  
$\gamma\Nu\Kappa$ acts by a rotation of $(360/m)^\circ$ on the circle $V^\perp/\Kappa$. 
Then $\gamma\Nu\Kappa = \delta^{\pm 1}\Nu\Kappa$, since 
the group $\Gamma/\Nu\Kappa$ acts effectively on $V^\perp/\Kappa$. 
\end{proof} 

\begin{lemma}  
Let $\Delta$ be an infinite dihedral group with Coxeter generators $\alpha, \beta$. 
Then the proper normal subgroups of $\Delta$ are the infinite dihedral groups $\langle \alpha, \beta\alpha\beta\rangle$ 
and $\langle \beta,\alpha\beta\alpha\rangle$ of index 2, and the infinite cyclic group $\langle (\alpha\beta)^m\rangle$ 
of index $2m$ for each positive integer $m$. 
\end{lemma}
\begin{proof}
We may assume that $\Delta$ is a discrete group of isometries of $E^1$. 
Then every element of $\Delta$ is either a translation or a reflection. 
Let $\Nu$ be a proper normal subgroup of $\Delta$. 
Suppose $\Nu$ contains a reflection.  Then $\Nu$ contains either $\alpha$ or $\beta$, 
since every reflection in $\Delta$ is conjugate in $\Delta$ to either $\alpha$ or $\beta$. 
Let $\langle\langle \alpha\rangle\rangle$ be the normal closure of $\langle \alpha\rangle$ in $\Delta$. 
Then $\langle\langle \alpha\rangle\rangle$ is the infinite dihedral group $ \langle \alpha, \beta\alpha\beta\rangle$. 
Let  $\langle\langle\beta\rangle\rangle$ be the normal closure of $\langle \beta\rangle$ in $\Delta$. 
Then $\langle\langle\beta\rangle\rangle$ is the infinite dihedral group $ \langle \beta, \alpha\beta\alpha\rangle$. 
Now $\Nu$ contains either $\langle\langle \alpha\rangle\rangle$ or $\langle\langle\beta\rangle\rangle$. 
As both $\langle\langle \alpha\rangle\rangle$ and $\langle\langle\beta\rangle\rangle$
have index 2 in $\Delta$, we have that $\Nu$ is either  $\langle\langle \alpha\rangle\rangle$ or $\langle\langle\beta\rangle\rangle$. 

Now suppose $\Nu$ does not contain a reflection.  
Then $\Nu$ is a subgroup of the group $\langle \alpha\beta\rangle$ of translations of $\Delta$.  
 Hence $\Nu = \langle (\alpha\beta)^m\rangle$ for some positive integer $m$. 
 Moreover each $m$ is possible, since $\langle \alpha\beta\rangle$ is a characteristic subgroup of $\Delta$. 
 \end{proof}
 
 Now assume that the structure group $\Gamma/\Nu\Kappa$ is a finite dihedral group. 
 The next theorem gives a necessary and sufficient condition for a pair of generators of $\Gamma/\Nu\Kappa$ 
to lift to a pair of Coxeter generators of $\Gamma/\Nu$ with respect to the the quotient map 
from $\Gamma/\Nu$ to $\Gamma/(\Nu\Kappa)$.

 \begin{theorem} 
Let $\Nu$ be a complete, $(n-1)$-dimensional, normal subgroup of an $n$-space group $\Gamma$ 
such that $\Gamma/\Nu$ is infinite dihedral, 
and let $\Kappa$ be the kernel of the action of $\Gamma$ on $V = \mathrm{Span}(\Nu)$, 
and suppose that the structure group $\Gamma/\Nu\Kappa$ is finite. 
Let $\gamma_1$ and $\gamma_1$ be elements of $\Gamma$ such that $\{\gamma_1\Nu\Kappa,\gamma_2\Nu\Kappa\}$ 
generates $\Gamma/\Nu\Kappa$. 
Then there exists elements $\delta_1$ and $\delta_2$ of $\Gamma$ 
such that $\{\gamma_1\Nu\Kappa, \gamma_2\Nu\Kappa\} = \{\delta_1\Nu\Kappa, \delta_2\Nu\Kappa\}$,  
and $\{\delta_1\Nu,\delta_2\Nu\}$ is a set of Coxeter generators of $\Gamma/\Nu$  
if and only if either
\begin{enumerate}
\item The order of $\Gamma/\Nu\Kappa$ is 1, or
\item The order of $\Gamma/\Nu\Kappa$ is 2, the group $\Kappa$ is infinite dihedral,  
and one of $\gamma_1\Nu\Kappa$ or $\gamma_2\Nu\Kappa$ is the identity element of $\Gamma/\Nu\Kappa$,  
and the other acts as the reflection of the closed interval $V^\perp/\Kappa$, or 
\item The order of $\Gamma/\Nu\Kappa$ is $2m$ for some positive integer $m$, 
the group $\Kappa$ is infinite cyclic, and both $\gamma_1\Nu\Kappa$ and $\gamma_2\Nu\Kappa$ 
act as reflections of  the circle $V^\perp/\Kappa$, and $\gamma_1\gamma_2\Nu\Kappa$ acts on $V^\perp/\Kappa$ 
as a rotation of $(360/m)^\circ$. 
\end{enumerate}
\end{theorem}
\begin{proof}
Suppose that $\delta_1$ and $\delta_2$ are elements of $\Gamma$ 
such that $\{\delta_1\Nu,\delta_2\Nu\}$ is a set of Coxeter generators of $\Gamma/\Nu$ 
and $\{\gamma_1\Nu\Kappa, \gamma_2\Nu\Kappa\} = \{\delta_1\Nu\Kappa, \delta_2\Nu\Kappa\}$. 
The group $\Nu\Kappa/\Nu$ is a normal subgroup of $\Gamma/\Nu$ of finite index, 
since $(\Gamma/\Nu)/(\Nu\Kappa/\Nu) = \Gamma/\Nu\Kappa$. 
We have that $\Nu\Kappa/\Nu \cong \Kappa$, since $\Nu\cap\Kappa =\{I\}$.   
Suppose that $\Kappa$ is infinite dihedral. 
Then the order of $\Gamma/\Nu\Kappa$ is 1 or 2  by Lemma 23. 
If the order of $\Gamma/\Nu\Kappa$ is 2, 
then $\Nu\Kappa/\Nu$ is either $\langle \delta_1\Nu, \delta_2\delta_1\delta_2\Nu\rangle$ 
or $\langle \delta_2\Nu, \delta_1\delta_2\delta_1\Nu\rangle$ by Lemma 23, and so 
either  $\delta_1\Nu\Kappa = \Nu\Kappa$ or  $\delta_2\Nu\Kappa = \Nu\Kappa$, and 
hence one of  $\gamma_1\Nu\Kappa$  or $\gamma_2\Nu\Kappa$ is  $\Nu\Kappa$ 
and the other acts as the reflection of the closed interval $V^\perp/\Kappa$, 
since $\Gamma/\Nu\Kappa$ acts effectively on $V^\perp/\Kappa$ by Theorem 6.

Suppose that $\Kappa$ is infinite cyclic. 
Then the order of $\Gamma/\Nu\Kappa$ is $2m$ for some positive integer $m$ 
and $\Nu\Kappa/\Nu = \langle (\delta_1\delta_2)^m\Nu\rangle$ by Lemma 23. 
Hence both $\delta_1\Nu\Kappa$ and $\delta_2\Nu\Kappa$ 
act as reflections of  the circle $V^\perp/(\Nu\Kappa/\Nu) = V^\perp/\Kappa$,  
and $\delta_1\delta_2\Nu\Kappa$ acts on $V^\perp/\Kappa$ 
as a rotation of $(360/m)^\circ$. 

Conversely, suppose $\delta_1,\delta_2$ are elements of $\Gamma$ such that 
$\{\delta_1\Nu,\delta_2\Nu\}$ is a set of Coxeter generators of $\Gamma/\Nu$. 
If $\Gamma/\Nu\Kappa$ is trivial, then obviously 
$\{\gamma_1\Nu\Kappa, \gamma_2\Nu\Kappa\} = \{\delta_1\Nu\Kappa, \delta_2\Nu\Kappa\}$. 
Suppose next that statement (2) holds.  
Then $\Nu\Kappa/\Nu$ is an infinite dihedral group, 
and either $\delta_1\Nu\Kappa$ or $\delta_2\Nu\Kappa$ is trivial in 
$(\Gamma/\Nu)/(\Nu\Kappa/\Nu) = \Gamma/\Nu\Kappa$ by Lemma 23. 
Hence $\{\gamma_1\Nu\Kappa,\gamma_2\Nu\Kappa\} = \{\delta_1\Nu\Kappa,\delta_2\Nu\Kappa\}$, 
since $\Gamma/\Nu\Kappa$ has order  2. 

Now suppose statement (3) holds. 
Then $\Nu\Kappa/\Nu = \langle (\delta_1\delta_2)^m\Nu\rangle$ by Lemma 23. 
Hence a generator of the infinite cyclic group $\Kappa$ acts on the line $V^\perp$ 
in the same way that $(\delta_1\delta_2)^m\Nu$ acts on $V^\perp$ as a translation. 
Therefore $\delta_1\delta_2\Nu\Kappa$ acts as a rotation of $(360/m)^\circ$ on the circle $V^\perp/\Kappa$.  
Hence $\gamma_1\gamma_2\Nu\Kappa = (\delta_1\delta_2)^{\pm 1}\Nu\Kappa$, 
since $\Gamma/\Nu\Kappa$ acts effectively on $V^\perp/\Kappa$. 
Therefore there exists $\gamma\in\Gamma$ such that 
$\{\gamma_1\Nu\Kappa,\gamma_2\Nu\Kappa\} = 
\{\gamma\delta_1\gamma^{-1}\Nu\Kappa,\gamma\delta_2\gamma^{-1}\Nu\Kappa\}$.  
Moreover $\{\gamma\delta_1\gamma^{-1}\Nu, \gamma\delta_2\gamma^{-1}\Nu\}$ 
is a set of Coxeter generators of $\Gamma/\Nu$.
\end{proof}


\section{Action of the Structure Group in the 2-Dimensional Case}  

Before we describe the action of the structure group in the 3-dimensional case, 
we will ``warm-up" with a description in the 2-dimensional case. 
We also apply the affine classification of co-Seifert geometric fibrations in \S 7 
to describe the affine classification of all the Seifert geometric fibrations of compact, connected, flat 2-orbifolds.  
We will denote a circle by $\mathrm{O}$ and a closed interval by $\mathrm{I}$. 

\medskip
\begin{table}  
\begin{tabular}{llllllll}
no.  & CN & fibr. & split & dual & split & grp. & structure group action \\
\hline 
1 & $\circ$   & $(\,\cdot\,)$ & Yes & $(\,\cdot\, )$ & Yes & $C_1$ & (idt., idt.)  \\
2 & $2222$ & $(-)$            & Yes & $(-)$          & Yes & $C_2$ & (ref., ref.)\\
3 & $**$       & $(-)$            & Yes & $[\,\cdot\,]$& Yes & $C_1$ & (idt., idt.)\\
4 & $\times\times$ & $(-)$ & No & $(\,\cdot\, )$ & Yes & $C_2$ & (2-rot., ref.)\\
5 & $*\times$  & $(-)$        & No   & $[\,\cdot\,]$ & Yes & $C_2$ & (2-rot, ref.)\\
6 & $*2222$ & $[-]$            & Yes & $[-]$          & Yes & $C_1$ & (idt., idt.)\\
7 & $22*$   & $(-)$             & Yes  & $[-]$         & Yes   &  $C_2$ & (ref., ref.) \\
8 & $22\times$ & $(-)$      & No   & $(-)$         & No    &  $D_2$ &(ref., ref.), (ref.$'$, 2-rot.) \\
9 & $2\hbox{$*$}22$ &  $[-]$  & Yes & $[-]$       & Yes  & $C_2$ & (ref., ref.)
\end{tabular}

\medskip
\caption{The Seifert and dual Seifert fibrations of the 2-space groups.}
\end{table}

Table 1 describes, via the generalized Calabi construction,  
all the Seifert and dual Seifert fibrations of a compact, connected, flat 2-orbifold up to affine equivalence.  
The first column lists the IT number of the corresponding 2-space group $\Gamma$.  
The second column lists the Conway name of the corresponding flat orbifold $E^2/\Gamma$.   
The third column lists the fiber $V/\Nu$ and base $V^\perp/(\Gamma/\Nu)$ 
of the Seifert fibration corresponding to a 1-dimensional, complete, normal subgroup $\Nu$ of $\Gamma$ 
with $V = \mathrm{Span}(\Nu)$. 
Parentheses indicates that the fiber is $\mathrm{O}$, 
and closed brackets indicates that the fiber is $\mathrm{I}$. 
A dot indicates that the base is $\mathrm{O}$ and a dash indicates 
that the base is $\mathrm{I}$. 
The group $\Nu$ is described in \S 10 of \cite{R-T}. 
For the first two rows of Table 1, the group $\Nu$ corresponds to the parameters $(a,b) = (1,0)$ 
in cases (1) and (2) of \S 10 of \cite{R-T}. 
The fourth column indicates whether or not the corresponding 
space group extension $1\to \Nu\to \Gamma\to \Gamma/\Nu\to 1$ splits. 
For a discussion of splitting, see \S 10 of \cite{R-T}. 
The fifth column lists the fiber $V^\perp/\Kappa$ and base $V/(\Gamma/\Kappa)$, 
with $\Kappa = \Nu^\perp$,  of the dual Seifert fibration. 
The sixth column indicates whether or not the corresponding 
space group extension $1\to \Kappa\to \Gamma\to \Gamma/\Kappa\to 1$ splits. 
The seventh column lists the isomorphism type of the structure group $\Gamma/(\Nu\Kappa)$ 
with $C_n$ indicating a cyclic group of order $n$,  
and $D_2$ indicating a dihedral group of order 4. 
The order of the structure groups for the 2-space groups with IT numbers 1 and 2 in Table 1 
were chosen to be as small as possible. 
See Examples 1 and 10 for the full range of structure groups for  2-space groups with IT numbers 1 and 2. 
The isomorphism types of the structure groups for the remaining 2-space groups are unique. 
The last column indicates how the structure group 
$\Gamma/(\Nu\Kappa)$ acts on the Cartesian product of the fibers $V/\Nu \times V^\perp/\Kappa$. 
We denote the identity map by idt., a halfturn by $2$-rot., and a reflection by ref. 
We denote the reflection of $\mathrm{O}$ orthogonal to ref.\ by ref.$'$. 

The actions of the structure group on the fibers of the Seifert fibrations  
are described by Example 2 and Theorem 7, except for the case when the structure group is a dihedral group of order 4. 
The problem with dihedral groups of order 4 is that it is not clear a priori 
which of the three nonidentity elements acts as a halfturn on a circle factor of $V/\Nu \times V^\perp/\Kappa$. 
See Example 12 for the description of the action in the case of Row 8 of Table 1.

\medskip
\noindent{\bf Example 11.} 
Let $\Gamma$ be the group in Table 1 with IT number 4. 
Then $E^2/\Gamma$ is a Klein bottle.  The structure group $G = \Gamma/\Nu\Kappa$ 
has order two. The fibers $V/\Nu$ and $V^\perp/\Kappa$ are both circles. 
The group $G$ acts by a halfturn on $V/\Nu$, since $(V/\Nu)/G =V/(\Gamma/\Kappa)$ is a circle. 
The group $G$ acts by a reflection on $V^\perp/\Kappa$, 
since $(V^\perp/\Kappa)/G=V^\perp/(\Gamma/\Nu)$ is a closed interval.

\medskip
\noindent{\bf Example 12.} 
Let $\Gamma$ be the group with IT number 8 in Table 1A of \cite{B-Z}. 
Then $\Gamma = \langle t_1, t_2, A, \beta\rangle$ 
where $t_i = e_i+I$ for $i=1,2$ are the standard translations, 
and $\beta = \frac{1}{2}e_1+\frac{1}{2}e_2 + B$, and 
$$A = \left(\begin{array}{rr} -1 & 0 \\ 0 & -1 \end{array}\right), \  \hbox{and}\ \ 
B = \left(\begin{array}{rr} 1 & 0 \\ 0 & -1 \end{array}\right).$$
The isomorphism type of $\Gamma$ is $22\times$ in Conway's notation or $pgg$ in IT notation.  
The orbifold $E^2/\Gamma$ is a projective pillow. 
The group $\Nu = \langle t_1\rangle$ is a complete normal subgroup of $\Gamma$, 
with $V= \mathrm{Span}(\Nu) = \mathrm{Span}\{e_1\}$. 
The flat orbifold $V/\Nu$ is a circle. 
Let $\Kappa = \Nu^\perp = \langle t_2\rangle$. 
Then $V^\perp/\Kappa$ is also a circle. 
The structure group $\Gamma/\Nu\Kappa$ is a dihedral group of order 4 
generated by $\Nu\Kappa A$ and $\Nu\Kappa \beta$. 
The element $\Nu\Kappa A$ acts as a reflection on $V/\Nu$ and on $V^\perp/\Kappa$.  
The action of $\Nu\Kappa A$ corresponds to the entry (ref., ref.) in Row 8 of Table 1. 
The element $\Nu\Kappa \beta$ acts on $V/\Nu$ as a halfturn and on $V^\perp/\Kappa$ as a reflection. 
The action of $\Nu\Kappa \beta$ corresponds to the product (2-rot., ref.$'$) of the entries in Row 8 of Table 1.  

\vspace{.05in}
As discussed in \S 10 of \cite{R-T}, the fibration and dual fibration in Rows 1, 2, 6, 8, 9 of Table 1 
are affinely equivalent.  
No other pair of fibrations in Table 1, with the same fibers and the same bases,  
are affinely equivalent, 
since the corresponding 2-space groups are nonisomorphic. 
We next apply the affine classification of co-Seifert geometric fibrations in \S 7 
to prove that every affine equivalence class of a Seifert geometric fibration of a compact, connected, flat 2-orbifold 
is represented by one of the fibrations described in Table 1. 

Suppose $\Delta$ and $\Mu$ are standard 1-space groups. 
We next describe $\mathrm{Iso}(\Delta,\Mu)$. 
Now $E^1/\Mu= \mathrm{O}$ or $\mathrm{I}$. 
The Lie group $\mathrm{Isom}(\mathrm{O})$ is isomorphic to $\mathrm{O}(2)$, 
and $\mathrm{Isom}(\mathrm{I})$ is a group of order 2 generated by the reflection ref.\ of $\mathrm{I}$ 
about its midpoint. 
Moreover $\mathrm{Aff}(\mathrm{O}) = \mathrm{Isom}(\mathrm{O})$ and 
$\mathrm{Aff}(\mathrm{I}) = \mathrm{Isom}(\mathrm{I})$, 
since length preserving affinities of $E^1$ are isometries. 
Hence for both isomorphism types of $\Mu$, the group $\mathrm{Out}(\Mu)$ has order 2 by Theorem 13. 
We represent  $\mathrm{Out}(\Mu)$ by the subgroup \{idt., ref.\} of $\mathrm{Isom}(E^1/\Mu)$ 
that is mapping isomorphically onto $\mathrm{Out}(\Mu)$ by $\Omega: \mathrm{Isom}(E^1/\Mu) \to \mathrm{Out}(\Mu)$. 

Assume first that $\Delta$ is infinite cyclic. 
The set $\mathrm{Iso}(\Delta,\Mu)$ consists of two elements 
corresponding to the pairs of inverse elements $\{$idt., idt.$\}$ and $\{$ref., ref.$\}$ 
of $\mathrm{Isom}(\mathrm{O})$ by Theorem 23. 
Thus there are two affine equivalence classes of fibrations of type $(\,\cdot\,)$ 
corresponding to the case that $\Mu$ is infinite cyclic,  
and two affine equivalence classes of fibrations of type $[\,\cdot\,]$ 
corresponding to the case that $\Mu$ is infinite dihedral. 
The fibration of type $(\,\cdot\,)$,  corresponding to $\{$idt., idt.$\}$, $\{$ref., ref.$\}$, 
is described in Row 1, 4, respectively, of Table 1 by Theorem 27.  
The fibration of type $[\,\cdot\,]$, corresponding to $\{$idt., idt.$\}$, $\{$ref., ref.$\}$, 
is described in Row 3, 5, respectively, of Table 1 by Theorem 27. 

Now assume both $\Delta$ and $\Mu$ are infinite dihedral. 
Then the set $\mathrm{Iso}(\Delta,\Mu)$ consists of three elements 
corresponding to the pairs of elements $\{\mathrm{idt.},\mathrm{idt.}\}$, 
$\{\mathrm{idt.}, \mathrm{ref.}\}$, $\{\mathrm{ref.}, \mathrm{ref.}\}$ of $\mathrm{Isom}(\mathrm{I}$) by Theorem 24. 
Thus there are three affine equivalence classes of fibrations of type $[ - ]$. 
The corresponding fibration of type $[ - ]$ is described in Row 6, 9, 7, respectively, of Table 1 by Theorem 28. 

Finally, assume that $\Delta$ is infinite dihedral and $\Mu$ is infinite cyclic. 
There are two conjugacy classes of isometries of $\mathrm{O}$ of order 2, 
the class of the halfturn 2-rot.\ of $\mathrm{O}$ and the class of a reflection ref.\ of $\mathrm{O}$. 
By Theorem 26, the set $\mathrm{Iso}(\Delta,\Mu)$ consists of six elements 
corresponding to the pairs of elements 
\{idt., idt.\}, \{idt., 2-rot.\}, \{idt., ref.\}, \{2-rot., 2-rot.\}, \{2-rot., ref.\}, \{ref., ref.\} of $\mathrm{Isom}(\mathrm{O})$.  
Only the pair \{ref., ref.\} falls into the case $E_1\cap E_2 \neq \{0\}$ of Theorem 26. 
Thus there are six affine equivalence classes of fibrations of type $(\,-\,)$. 
The corresponding fibration of type $(\,-\,)$ is described in Row 3, 5, 7, 4, 8, 2, respectively, of Table 1 by Theorem 28. 
Thus every affine equivalence class of Seifert geometric fibrations of a compact, connected, flat 2-orbifold 
is represented by one of the fibrations in Table 1.

\section{Action of the Structure Group in the 3-Dimensional Case} 

We are now ready to describe the action of the structure group for the generalized Calabi constructions 
corresponding to the Seifert and dual co-Seifert geometric fibrations of compact, connected, 
flat 3-orbifolds given in Table 1 of \cite{R-T}. 
We did not describe the co-Seifert fibration projections in Table 1 of \cite{R-T} 
because we did not have an efficient way of giving a description. 
Our generalized Calabi construction gives us a simple way to simultaneously describe 
both a Seifert fibration and its dual co-Seifert fibration of a compact, connected, flat 3-orbifold 
via Theorem 4 and Corollary 1.

We will organize our description into 17 tables, 
one for each 2-space group type of the generic fiber of the co-Seifert fibration. 
We will go through the 2-space groups in reverse order because 
the complexity of the description generally decreases as the IT number of the 2-space group increases. 
At the same time, we will apply the affine classification of co-Seifert geometric fibrations in \S 7 
to prove that the computer generated affine classification of the co-Seifert geometric fibrations 
described in Table 1 of \cite{R-T} is correct and complete.  

In each of these 17 tables, the first column lists the IT number of the 3-space group $\Gamma$. 
The second column lists the generic fibers $(V/\Nu, V^\perp/\Kappa$) of the co-Seifert and Seifert fibrations 
corresponding to a 2-dimensional, complete, normal subgroup $\Nu$ of $\Gamma$ with $V = \mathrm{Span}(\Nu)$ 
and $\Kappa = \Nu^\perp$. 
We will denote a circle by $\mathrm{O}$ and a closed interval by $\mathrm{I}$. 
The third column lists the isomorphism type of the structure group $\Gamma/\Nu\Kappa$,  
with $C_n$ indicating a cyclic group of order $n$,  
and $D_n$ indicating a dihedral group of order $2n$. 
The fourth column lists the quotients $(V/(\Gamma/\Kappa),V^\perp/(\Gamma/\Nu))$ 
under the action of the structure group.  
Note that $(V/(\Gamma/\Kappa), V^\perp/(\Gamma/\Nu))$  are the bases of the Seifert and co-Seifert fibrations. 
The fifth column indicates how generators of the 
structure group act on the Cartesian product $V/\Nu \times V^\perp/\Kappa$. 
The structure group action was derived from the standard affine representation 
of $\Gamma$ in Table 1B of \cite{B-Z}. 
We denote a rotation of $\mathrm{O}$ of $(360/n)^\circ$ by $n$-rot., 
and a reflection of $\mathrm{O}$ or $\mathrm{I}$ by ref. 
We denote the identity map by idt. 
The sixth column lists the classifying pairs for the co-Seifert fibrations. 
For $D_n$ actions,  ($\alpha$, ref.), ($\gamma$, $n$-rot.), the classifying pair is $\{\alpha,\beta\}$ 
with $\gamma = \alpha\beta$, except for the cases $n = 3, 6$ at the end of Tables 17 and 18,  
which require an affine deformation.  
Tables 1 -- 18 were first computed by hand, and then they were double checked by a computer calculation.

Let $C_\infty$ be the standard infinite cyclic 1-space group $\langle e_1+I\rangle$, 
and let $D_\infty$ be the standard infinite dihedral 1-space group $\langle e_1+I, -I\rangle$. 
If $\Mu$ is a 2-space group, 
we define the {\it symmetry group} of the flat orbifold $E^2/\Mu$ 
by $\mathrm{Sym}(\Mu) = \mathrm{Isom}(E^2/\Mu)$. 

(17)  The 2-space group $\Mu$ with IT number 17 is $\ast632$ in Conway's notation or $p6m$ in IT notation.  
See space group 4/4/1/1 in Table 1A of \cite{B-Z} for the standard affine representation of $\Mu$. 
The flat orbifold $E^2/\Mu$ is a $30^\circ-60^\circ$ right triangle. 

\begin{lemma} 
If $\Mu$ is the 2-space group $\ast 632$, 
then $\mathrm{Sym}(\Mu) = \mathrm{Aff}(\Mu) =\{\rm idt.\}$, 
and $\Omega:  \mathrm{Aff}(\Mu) \to \mathrm{Out}(\Mu)$ is an isomorphism. 
\end{lemma}
\begin{proof} $\mathrm{Sym}(\Mu) =  \{\rm idt.\}$, 
since a symmetry of $E^2/\Mu$ fixes each corner point. 
We have that $Z(\Mu) = \{I\}$ by Lemma 9. 
Hence $\Omega: \mathrm{Sym}(\Mu) \to \mathrm{Out}_E(\Mu)$ is an isomorphism 
by Theorems 11 and 12. 
We have that $\mathrm{Out}_E(\Mu) = \mathrm{Out}(\Mu)$  by Lemma 13 and Table 5A of \cite{B-Z},  
and $\Omega: \mathrm{Aff}(\Mu) \to \mathrm{Out}(\Mu)$ is an isomorphism by Theorems 11 and 13. 
Therefore $\mathrm{Sym}(\Mu) = \mathrm{Aff}(\Mu)$. 
\end{proof}

\begin{theorem}  
Let $\Mu$ be the 2-space group $\ast 632$. 
Then $\mathrm{Iso}(C_\infty,\Mu)$ has only one element,   
corresponding to the pair of elements {\rm \{\rm idt., idt.\}} of $\mathrm{Sym}(\Mu)$ by Theorem 23, 
and  $\mathrm{Iso}(D_\infty,\Mu)$ has  only one element,   
corresponding to the pair of elements {\rm \{idt., idt.\}} of $\mathrm{Sym}(\Mu)$ by Theorem 24.
The corresponding co-Seifert fibrations are described in Table 2 by Theorems 27 and 28. 
\end{theorem}

\begin{table}  
\begin{tabular}{llllll}
no. & fibers & grp. & quotients &  grp. action & classifying pair \\
\hline 
183  & $(\ast632, \mathrm{O})$ & $C_1$ & $(\ast632, \mathrm{O})$ & (idt., idt.)  & \{idt., idt.\} \\
191  & $(\ast632, \mathrm{I})$   & $C_1$ & $(\ast632, \mathrm{I})$ & (idt., idt.) & \{idt., idt.\} \end{tabular}

\medskip
\caption{The classification of the co-Seifert fibrations of 3-space groups 
whose co-Seifert fiber is of type $\ast 632$ with IT number 17}
\end{table}

(16) The 2-space group $\Mu$ with IT number 16 is $632$ in Conway's notation or $p6$ in IT notation. 
See space group 4/3/1/1 in Table 1A of \cite{B-Z} for the standard affine representation of $\Mu$. 
The flat orbifold $E^2/\Mu$  is a turnover with 3 cone points 
obtained by gluing together two congruent $30^\circ-60^\circ$ right triangles along their boundaries. 
The $632$ turnover is orientable. 
Let c-ref.\ denote the {\it central reflection} between the two triangles. 

\begin{lemma} 
If $\Mu$ is the 2-space group $632$, 
then $\mathrm{Sym}(\Mu) = \mathrm{Aff}(\Mu) =$ \{\rm idt., c-ref.\}, 
and $\Omega:  \mathrm{Aff}(\Mu) \to \mathrm{Out}(\Mu)$ is an isomorphism. 
\end{lemma}
\begin{proof} $\mathrm{Sym}(\Mu)$ = \{\rm idt., c-ref.\}, 
since a symmetry of $E^2/\Mu$ fixes each cone point. 
We have that $\mathrm{Sym}(\Mu) = \mathrm{Aff}(\Mu)$ 
and $\Omega:  \mathrm{Aff}(\Mu) \to \mathrm{Out}(\Mu)$ is an isomorphism as in Lemma 24. 
\end{proof}

\begin{theorem}  
Let $\Mu$ be the 2-space group $632$. 
Then $\mathrm{Iso}(C_\infty,\Mu)$ has two elements,   
corresponding to the pairs of elements {\rm \{idt., idt.\}, \{c-ref., c-ref.\}} of $\mathrm{Sym}(\Mu)$ by Theorem 23, 
and $\mathrm{Iso}(D_\infty,\Mu)$ has three elements,  
corresponding to the pairs of elements {\rm \{idt., idt.\}, \{idt., c-ref.\}, \{c-ref., c-ref.\}} of $\mathrm{Sym}(\Mu)$  by Theorem 24.
The corresponding co-Seifert fibrations are described in Table 3 by Theorems 27 and 28.  
\end{theorem}

\begin{table} 
\begin{tabular}{llllll}
no. & fibers & grp. & quotients &  grp. action & classifying pair \\
\hline 
168 &  $(632, \mathrm{O})$ & $C_1$ & $(632, \mathrm{O})$ & (idt., idt.) & \{idt., idt.\}  \\
175 &  $(632, \mathrm{I})$    & $C_1$ & $(632, \mathrm{I})$ & (idt., idt.) & \{idt., idt.\} \\
177 &  $(632, \mathrm{O})$ & $C_2$ & $(\ast632, \mathrm{I})$ & (c-ref., ref.)  & \{c-ref., c-ref.\} \\
184 &  $(632, \mathrm{O})$    & $C_2$ & $(\ast632, \mathrm{O})$ & (c-ref., 2-rot.) & \{c-ref., c-ref.\} \\
192 &  $(632, \mathrm{I})$    & $C_2$ & $(\ast632, \mathrm{I})$ & (c-ref., ref.) & \{idt., c-ref.\} 
\end{tabular}

\medskip
\caption{The classification of the co-Seifert fibrations of 3-space groups 
whose co-Seifert fiber is of type $632$ with IT number 16}
\end{table}

(15)  The 2-space group $\Mu$ with IT number 15 is $3{\ast}3$ in Conway's notation or $p31m$ in IT notation. 
See space group 4/2/2/1 in Table 1A of \cite{B-Z} for the standard affine representation of $\Mu$. 
The  flat orbifold $E^2/\Mu$ is a cone with one $120^\circ$ cone point and one  $60^\circ$ corner point   
obtained by gluing together two congruent $30^\circ-60^\circ$ right triangles along their sides 
opposite the $30^\circ$ and $90^\circ$ angles. 
Let c-ref.\ denote the {\it central reflection} between the two triangles. 

\begin{lemma} 
If $\Mu$ is the 2-space group $3\ast 3$, 
then $\mathrm{Sym}(\Mu) = \mathrm{Aff}(\Mu) =$ \{\rm idt., c-ref.\}, 
and $\Omega:  \mathrm{Aff}(\Mu) \to \mathrm{Out}(\Mu)$ is an isomorphism. 
\end{lemma}
\begin{proof} $\mathrm{Sym}(\Mu)$ = \{\rm idt., c-ref.\}, 
since a symmetry of $E^2/\Mu$ fixes the cone point and the corner point. 
We have that $\mathrm{Sym}(\Mu) = \mathrm{Aff}(\Mu)$ and $\Omega:  \mathrm{Aff}(\Mu) \to \mathrm{Out}(\Mu)$ is an isomorphism as in Lemma 24. 
\end{proof}

\begin{theorem}  
Let $\Mu$ be the 2-space group \hbox{$3\ast 3$}. 
Then $\mathrm{Iso}(C_\infty,\Mu)$ has two elements,   
corresponding to the pairs of elements {\rm \{idt., idt.\}, \{c-ref., c-ref.\}} of $\mathrm{Sym}(\Mu)$ by Theorem 23,  
and $\mathrm{Iso}(D_\infty,\Mu)$ has three elements,  
corresponding to the pairs of elements {\rm \{idt., idt.\}, \{idt., c-ref.\}, \{c-ref., c-ref.\}} of $\mathrm{Sym}(\Mu)$  by Theorem 24.
The corresponding co-Seifert fibrations  are described in Table 4 by Theorems 27 and 28.  
\end{theorem}

\begin{table} 
\begin{tabular}{llllll}
no. & fibers & grp. & quotients &  grp. action & classifying pair \\
\hline 
157 & $(3\!\ast\! 3, \mathrm{O})$ & $C_1$ & $(3\!\ast\! 3, \mathrm{O})$ & (idt., idt.) & \{idt., idt.\}  \\
162 & $(3\!\ast\! 3, \mathrm{O})$ & $C_2$ & $(\ast 632, \mathrm{I})$ & (c-ref., ref.) & \{c-ref., c-ref.\} \\
185 & $(3\!\ast\! 3, \mathrm{O})$ & $C_2$ & $(\ast 632, \mathrm{O})$ & (c-ref., 2-rot.) & \{c-ref., c-ref.\} \\
189 & $(3\!\ast\! 3, \mathrm{I})$   & $C_1$ & $(3\!\ast\! 3, \mathrm{I})$ & (idt., idt.) & \{idt., idt.\} \\
193 & $(3\!\ast\! 3, \mathrm{I})$   & $C_2$ & $(\ast 632, \mathrm{I})$ & (c-ref., ref.) & \{idt., c-ref.\}
\end{tabular}

\medskip
\caption{The classification of the co-Seifert fibrations of 3-space groups 
whose co-Seifert fiber is of type $3\!\ast\! 3$ with IT number 15}
\end{table}

(14)  The 2-space group $\Mu$ with IT number 14 is $\ast 333$ in Conway's notation or $p3m1$ in IT notation. 
See space group 4/2/1/1 in Table 1A of \cite{B-Z} for the standard affine representation of $\Mu$. 
The flat orbifold $E^2/\Mu$ is an equilateral triangle $\triangle$. 
The symmetry group of $\triangle$ is a dihedral group of order 6. 
There is one conjugacy class of symmetries of order 2 represented by a {\it triangle reflection}  t-ref.\ 
in the perpendicular bisector of a side. 
There is one conjugacy class of symmetries 
of order 3 represented by a rotation 3-rot.\ of $120^\circ$ about the center of the triangle. 
Define the triangle reflection t-ref.$'$ = (t-ref.)(3-rot.).  

\begin{lemma} 
If $\Mu$ is the 2-space group $\ast 333$, 
then $\mathrm{Sym}(\Mu)$ is the dihedral group  $\langle${\rm t-ref., 3-rot.}$\rangle$ of order 6, 
and $\mathrm{Sym}(\Mu) = \mathrm{Aff}(\Mu)$, 
and $\Omega:  \mathrm{Aff}(\Mu) \to \mathrm{Out}(\Mu)$ is an isomorphism. 
\end{lemma}
\begin{proof}  $\mathrm{Sym}(\Mu)$ = $\langle${\rm t-ref., 3-rot.}$\rangle$, 
since a symmetry permutes the corner points. 
We have that $\mathrm{Sym}(\Mu) = \mathrm{Aff}(\Mu)$ 
and $\Omega:  \mathrm{Aff}(\Mu) \to \mathrm{Out}(\Mu)$ is an isomorphism as in Lemma 24. 
\end{proof}

\begin{theorem}  
Let $\Mu$ be the 2-space group $\ast 333$. 
Then $\mathrm{Iso}(C_\infty,\Mu)$ has three elements,   
corresponding to the pairs of elements 
{\rm \{idt., idt.\}, \{t-ref., t-ref.\}, \{3-rot., 3-rot.$^{-1}$\}} of $\mathrm{Sym}(\Mu)$ by Theorem 23,  
and $\mathrm{Iso}(D_\infty,\Mu)$ has four elements,  
corresponding to the pairs of elements {\rm \{idt., idt.\}, \{idt., t-ref.\}, \{t-ref., t-ref.\}, \{t-ref., t-ref.$'$\}} of $\mathrm{Sym} (\Mu)$ 
by Theorem 24.
The corresponding co-Seifert fibrations are described in Table 5 by Theorems 27 and 28.  
\end{theorem}

\begin{table}  
\begin{tabular}{llllll}
no. & fibers & grp. & quotients &  structure group action & classifying pair \\
\hline 
156 & $(\ast 333, \mathrm{O})$ & $C_1$ & $(\ast 333, \mathrm{O})$ & (idt., idt.)  & \{idt., idt.\} \\
160 & $(\ast 333, \mathrm{O})$ & $C_3$ & $(3{\ast}3, \mathrm{O})$ & (3-rot., 3-rot.) & \{3-rot., 3-rot.$^{-1}$\} \\
164 & $(\ast 333, \mathrm{O})$ & $C_2$ & $(\ast 632, \mathrm{I})$ & (t-ref., ref.) & \{t-ref., t-ref.\} \\
166 & $(\ast 333, \mathrm{O})$ & $D_3$ & $(\ast 632, \mathrm{I})$ & (t-ref., ref.), (3-rot., 3-rot.) & \{t-ref., t-ref.$'$\}  \\
186 & $(\ast 333, \mathrm{O})$ & $C_2$ & $(\ast 632, \mathrm{O})$ & (t-ref., 2-rot.) & \{t-ref., t-ref.\} \\
187 & $(\ast 333, \mathrm{I})$   & $C_1$ & $(\ast 333, \mathrm{I})$ & (idt., idt.) & \{idt., idt.\} \\
194 & $(\ast 333, \mathrm{I})$   & $C_2$ & $(\ast 632, \mathrm{I})$ & (t-ref., ref.) & \{idt., t-ref.\} 
\end{tabular}

\medskip
\caption{The classification of the co-Seifert fibrations of 3-space groups 
whose co-Seifert fiber is of type $\ast 333$ with IT number 14}
\end{table}

(13) The 2-space group $\Mu$ with IT number 13 is $333$ in Conway's notation or $p3$ in IT notation. 
See space group 4/1/1/1 in Table 1A of \cite{B-Z} for the standard affine representation of $\Mu$. 
The flat orbifold $E^2/\Mu$ is a turnover with three $120^\circ$ cone points 
obtained by gluing together two congruent equilateral triangles along their boundaries. 
The $333$ turnover is orientable. 
The symmetry group of this orbifold is the direct product of the subgroup of order 2, generated 
by the {\it central reflection} c-ref.\ between the two triangles, and the subgroup of order 6 corresponding 
to the symmetry group of the two triangles. 

There are 3 conjugacy classes of symmetries of order 2,  
the class of the {\it central reflection} c-ref., 
the class of the {\it triangle reflection} t-ref., 
and the class of the halfturn around a cone point 2-rot., defined so that 2-rot.\ = (c-ref.)(t-ref.). 
There is one conjugacy class of symmetries of order 3 represented by a rotation 3-rot.\ 
that cyclically permutes the cone points. 
There is one conjugacy class of dihedral subgroups of order 4, represented by the group 
\{idt., c-ref., t-ref., 2-rot.\}.  
There is one conjugacy class of symmetries of order 6 represented by the group generated by 6-sym.\ = (c-ref.)(3-rot.).  
There are two conjugacy classes of dihedral subgroups of order 6, 
the class of the symmetry group of a triangular side of the turnover generated by t-ref.\ and 3-rot., 
and the class of the orientation preserving subgroup generated by 2-rot.\ and 3-rot. 
Define t-ref.$'$ = (t-ref.)(3-rot.) and 2-rot.$'$ = (c-ref.)(t-ref.$'$). 

\begin{lemma} 
If $\Mu$ is the 2-space group $333$, 
then $\mathrm{Sym}(\Mu)$ is the dihedral group  $\langle${\rm c-ref., t-ref., 3-rot.}$\rangle$ of order 12, 
and $\mathrm{Sym}(\Mu) = \mathrm{Aff}(\Mu)$, 
and $\Omega:  \mathrm{Aff}(\Mu) \to \mathrm{Out}(\Mu)$ is an isomorphism.  
\end{lemma}
\begin{proof} $\mathrm{Sym}(\Mu) = \langle${c-ref., \rm t-ref., 3-rot.}$\rangle$, 
since a symmetry permutes the cone points.  
We have that $\mathrm{Sym}(\Mu) = \mathrm{Aff}(\Mu)$ and $\Omega:  \mathrm{Aff}(\Mu) \to \mathrm{Out}(\Mu)$ is an isomorphism as in Lemma 24. 
\end{proof}

\begin{theorem}  
Let $\Mu$ be the 2-space group $333$. 
Then $\mathrm{Iso}(C_\infty,\Mu)$ has six elements,   
corresponding to the pairs of elements 
{\rm \{idt., idt.\}, \{2-rot., 2-rot.\}, \{c-ref., c-ref.\}, \{t-ref., t-ref.\},  \{3-rot., 3-rot.$^{-1}$\},  \{6-sym., 6-sym.$^{-1}$\}}  
of $\mathrm{Sym}(\Mu)$ by Theorem 23,  $\mathrm{Iso}(D_\infty,\Mu)$ has 13 elements,  
corresponding to the pairs of elements {\rm \{idt., idt.\}, \{idt., 2-rot.\}, \{idt., c-ref.\}, \{idt., t-ref.\}, \{2-rot., 2-rot.\}, 
\{c-ref., c-ref.\}, \{t-ref., t-ref.\}, \{2-rot., c-ref.\}, \{2-rot., t-ref.\}, \{c-ref., t-ref.\}, \{2-rot., 2-rot.$'\}$, 
\{t-ref., t-ref.$'$\}, \{t-ref., 2-rot.$'$\}} of $\mathrm{Sym} (\Mu)$  by Theorem 24.
The corresponding co-Seifert fibrations are described in Table 6 by Theorems 27 and  28.  
\end{theorem}

\begin{table}  
\begin{tabular}{llllll}
no. & fibers & grp. & quotients &  structure group action & classifying pair \\
\hline 
143 & $(333, \mathrm{O})$ & $C_1$ & $(333, \mathrm{O})$ & (idt., idt.)  & \{idt., idt.\} \\
146 & $(333, \mathrm{O})$ & $C_3$ & $(333, \mathrm{O})$ & (3-rot., 3-rot.) &  \{3-rot., 3-rot.$^{-1}$\} \\
147 & $(333, \mathrm{O})$ & $C_2$ & $(632, \mathrm{I})$ & (2-rot., ref.) & \{2-rot., 2-rot.\} \\
148 & $(333, \mathrm{O})$ & $D_3$ & $(632, \mathrm{I})$ & (2-rot., ref.), (3-rot., 3-rot.) &  \{2-rot., 2-rot.$'\}$ \\
149 & $(333, \mathrm{O})$ & $C_2$ & $(\ast 333, \mathrm{I})$ & (c-ref., ref.) & \{c-ref., c-ref.\} \\
150 & $(333, \mathrm{O})$ & $C_2$ & $(3{\ast}3, \mathrm{I})$ & (t-ref., ref.)  & \{t-ref., t-ref.\} \\
155 & $(333, \mathrm{O})$ & $D_3$ & $(\ast 333, \mathrm{I})$ & (t-ref., ref.), (3-rot., 3-rot.) & \{t-ref., t-ref.$'$\} \\
158 & $(333, \mathrm{O})$ & $C_2$ & $(\ast 333, \mathrm{O})$ & (c-ref., 2-rot.) &  \{c-ref., c-ref.\} \\
159 & $(333, \mathrm{O})$ & $C_2$ & $(3{\ast}3, \mathrm{O})$ & (t-ref., 2-rot.) &  \{t-ref., t-ref.\} \\
161 & $(333, \mathrm{O})$ & $C_6$ & $(3{\ast}3, \mathrm{O})$ & (6-sym., 6-rot.) & \{6-sym., 6-sym.$^{-1}$\} \\
163 & $(333, \mathrm{O})$ & $D_2$ & $(\ast 632, \mathrm{I})$ & (c-ref., ref.), (t-ref., 2-rot.) & \{c-ref., 2-rot.\}\\
165 & $(333, \mathrm{O})$ & $D_2$ & $(\ast 632, \mathrm{I})$ & (t-ref., ref.), (c-ref., 2-rot.) & \{t-ref., 2-rot.\} \\
167 & $(333, \mathrm{O})$ & $D_6$ & $(\ast 632, \mathrm{I})$ & (t-ref., ref.), (6-sym., 6-rot.) &  \{t-ref., 2-rot.$'$\} \\
173 & $(333, \mathrm{O})$ & $C_2$ & $(632, \mathrm{O})$ & (2-rot., 2-rot.)  & \{2-rot., 2-rot.\} \\
174 & $(333, \mathrm{I})$  & $C_1$ & $(333, \mathrm{I})$ & (idt., idt.)  & \{idt., idt.\} \\
176 & $(333, \mathrm{I})$  & $C_2$ & $(632, \mathrm{I})$ & (2-rot., ref.) &  \{idt., 2-rot.\} \\
182 & $(333, \mathrm{O})$ & $D_2$ & $(\ast 632, \mathrm{I})$ &  (c-ref., ref.), (2-rot., 2-rot.) & \{c-ref., t-ref.\} \\
188 & $(333, \mathrm{I})$   & $C_2$ & $(\ast 333, \mathrm{I})$ & (c-ref., ref.) &  \{idt., c-ref.\} \\
190 & $(333, \mathrm{I})$   & $C_2$ & $(3{\ast}3, \mathrm{I})$ & (t-ref., ref.)  & \{idt., t-ref.\} 
\end{tabular}

\medskip
\caption{The classification of the co-Seifert fibrations of 3-space groups 
whose co-Seifert fiber is of type $333$ with IT number 13}
\end{table}

\medskip
\noindent{\bf Example 13.}
Let $\Gamma$ be the affine 3-space group with IT number 163 in Table 1B of \cite{B-Z}. 
Then $\Gamma = \langle t_1,t_2,t_3, A,\beta,C\rangle$ 
where $t_i = e_i+I$ for $i=1,2,3$ are the standard translations, 
and $\beta=\frac{1}{2}e_3+B$, and 
$$A = \left(\begin{array}{rrr} 0 & -1 & 0\\ 1 & -1 & 0 \\ 0 & 0 & 1  \end{array}\right),\ \ 
B = \left(\begin{array}{rrr} 0 & -1 & 0  \\ -1 & 0 & 0   \\ 0 & 0 & -1 \end{array}\right), \ \ 
C = \left(\begin{array}{rrr} -1 & 0 & 0  \\ 0 & -1 & 0   \\ 0 & 0 & -1 \end{array}\right).$$
The group $\Nu = \langle t_1,t_2, A\rangle$ is a complete normal subgroup 
of $\Gamma$ with $V = {\rm Span}(\Nu) = {\rm Span}\{e_1, e_2\}$. 
The flat orbifold $V/\Nu$ is a $333$ turnover.  
Let $\Kappa = \Nu^\perp$.  Then $\Kappa = \langle t_3\rangle$. 
The flat orbifold $V^\perp/\Kappa$ is a circle.  
The structure group $\Gamma/\Nu\Kappa$ is a dihedral group of order 4 
generated by $\Nu\Kappa\beta$ and $\Nu\Kappa C$. 
The elements $\Nu\Kappa\beta$ and $\Nu\Kappa C$ act on $V^\perp/\Kappa$ as reflections. 
The elements  $\Nu\Kappa\beta$ and $\Nu\Kappa C$ both fix the cone point 
of $V/\Nu$ represented by $(0,0,0)$.  
The other two cone points of $V/\Nu$ are represented by 
$(2/3,1/3,0)$, which is the fixed point of $t_1A$, and by $(1/3,2/3,0)$,  
which is the fixed point of $t_1t_2A$. 
The element $\Nu\Kappa\beta$ acts as the central reflection of $V/\Nu$, 
since it fixes all three cone points.  
The element $\Nu\Kappa C$ acts as the halfturn around the cone point represented by $(0,0,0)$, 
since $\Nu\Kappa C$ preserves the orientation of $V/\Nu$ 
because $C$ preserves the orientation of $V$.

(12)  The 2-space group $\Mu$ with IT number 12 is $4{\ast}2$ in Conway's notation and $p4g$ in IT notation. 
See space group 3/2/1/2 in Table 1A of \cite{B-Z} for the standard affine representation of $\Mu$. 
The flat orbifold $E^2/\Mu$ is a cone with one $90^\circ$ cone point and one $90^\circ$ corner point 
obtained by glueing together two congruent $45^\circ-45^\circ$ right triangles along two sides opposite 
$45^\circ$ and $90^\circ$ angles.  
Let c-ref.\ denote the  {\it central reflection} between the triangles. 

\begin{lemma} 
If $\Mu$ is the 2-space group $4\ast 2$, 
then $\mathrm{Sym}(\Mu) = \mathrm{Aff}(\Mu) =$ \{\rm idt., c-ref.\}, 
and $\Omega:  \mathrm{Aff}(\Mu) \to \mathrm{Out}(\Mu)$ is an isomorphism. 
\end{lemma}
\begin{proof} $\mathrm{Sym}(\Mu)$ = \{\rm idt., c-ref.\}, 
since a symmetry of $E^2/\Mu$ fixes the cone point and the corner point. 
We have that $\mathrm{Sym}(\Mu) = \mathrm{Aff}(\Mu)$ and $\Omega:  \mathrm{Aff}(\Mu) \to \mathrm{Out}(\Mu)$ is an isomorphism as in Lemma 24. 
\end{proof}

\begin{theorem}  
Let $\Mu$ be the 2-space group \hbox{$4\ast 2$}. 
Then $\mathrm{Iso}(C_\infty,\Mu)$ has two elements,   
corresponding to the pairs of elements {\rm \{idt., idt.\}, \{c-ref., c-ref.\}} of $\mathrm{Sym}(\Mu)$ by Theorem 23, 
and $\mathrm{Iso}(D_\infty,\Mu)$ has three elements,  
corresponding to the pairs of elements {\rm \{idt., idt.\}, \{idt., c-ref.\}, \{c-ref., c-ref.\}} of $\mathrm{Sym}(\Mu)$  by Theorem 24.
The corresponding co-Seifert fibrations are described in Table 7 by Theorems 27 and 28.  
\end{theorem}

\begin{table}  
\begin{tabular}{llllll}
no. & fibers & grp. & quotients &  grp. action & classifying pair \\
\hline 
100 & $(4\!\ast\! 2, \mathrm{O})$   & $C_1$ & $(4{\ast}2, \mathrm{O})$ & (idt., idt.) & \{idt., idt.\} \\
108 & $(4\!\ast\! 2, \mathrm{O})$   & $C_2$ & $(\ast 442, \mathrm{O})$ & (c-ref., 2-rot.) & \{c-ref., c-ref.\} \\
125 & $(4\!\ast\! 2, \mathrm{O})$  & $C_2$ & $(\ast 442, \mathrm{I})$ & (c-ref., ref.) & \{c-ref., c-ref.\}  \\
127 & $(4\!\ast\! 2, \mathrm{I})$    & $C_1$ & $(4{\ast}2, \mathrm{I})$ & (idt., idt.) & \{idt., idt.\} \\
140 & $(4\!\ast\! 2, \mathrm{I})$    & $C_2$ & $(\ast 442, \mathrm{I})$ & (c-ref., ref.) & \{idt., c-ref.\} 
\end{tabular}

\medskip
\caption{The classification of the co-Seifert fibrations of 3-space groups 
whose co-Seifert fiber is of type $4\hbox{$\ast$}2$ with IT number 12}
\end{table}

(11)  The 2-space group $\Mu$ with IT number 11 is $\ast 442$ in Conway's notation and $p4m$ in IT notation. 
See space group 3/2/1/1 in Table 1A of \cite{B-Z} for the standard affine representation of $\Mu$. 
The flat orbifold $E^2/\Mu$ is a $45^\circ - 45^\circ$ right triangle.  
Let t-ref.\ denote the {\it triangle reflection} of $E^2/\Mu$. 

\begin{lemma} 
If $\Mu$ is the 2-space group $\ast 442$, 
then $\mathrm{Sym}(\Mu) = \mathrm{Aff}(\Mu) =$ \{\rm idt., t-ref.\}, 
and $\Omega:  \mathrm{Aff}(\Mu) \to \mathrm{Out}(\Mu)$ is an isomorphism.  
\end{lemma}
\begin{proof} $\mathrm{Sym}(\Mu)$ = \{\rm idt., t-ref.\}, 
since a symmetry of $E^2/\Mu$ fixes the $90^\circ$ corner point 
and permute the other two corner points. 
We have that $\mathrm{Sym}(\Mu) = \mathrm{Aff}(\Mu)$ and $\Omega:  \mathrm{Aff}(\Mu) \to \mathrm{Out}(\Mu)$ is an isomorphism as in Lemma 24. 
\end{proof}

\begin{theorem}  
Let $\Mu$ be the 2-space group $\ast 442$. 
Then $\mathrm{Iso}(C_\infty,\Mu)$ has two elements,   
corresponding to the pairs of elements {\rm \{idt., idt.\}, \{t-ref., t-ref.\}} of $\mathrm{Sym}(\Mu)$ by Theorem 23, 
and $\mathrm{Iso}(D_\infty,\Mu)$ has three elements,  
corresponding to the pairs of elements {\rm \{idt., idt.\}, \{idt., t-ref.\}, \{t-ref., t-ref.\}} of $\mathrm{Sym}(\Mu)$  by Theorem 24.
The corresponding co-Seifert fibrations are described in Table 8 by Theorems 27 and 28.  
\end{theorem}

\begin{table}  
\begin{tabular}{rlllll}
no. & fibers & grp. & quotients &  grp. action & classifying pair \\
\hline 
  99 & $(\ast 442, \mathrm{O})$ & $C_1$ & $(\ast 442, \mathrm{O})$ & (idt., idt.) &  \{idt., idt.\} \\
107 & $(\ast 442, \mathrm{O})$ & $C_2$ & $(\ast 442, \mathrm{O})$ & (t-ref., 2-rot.) & \{t-ref., t-ref.\} \\
123 & $(\ast 442, \mathrm{I})$   & $C_1$ & $(\ast 442, \mathrm{I})$ & (idt., idt.) &  \{idt., idt.\} \\
129 & $(\ast 442, \mathrm{O})$ & $C_2$ & $(\ast 442, \mathrm{I})$ & (t-ref., ref.) & \{t-ref., t-ref.\} \\
139 & $(\ast 442, \mathrm{I})$   & $C_2$ & $(\ast 442, \mathrm{I})$ & (t-ref., ref.) & \{idt., t-ref.\}
\end{tabular}

\medskip
\caption{The classification of the co-Seifert fibrations of 3-space groups 
whose co-Seifert fiber is of type $\ast 442$ with IT number 11}
\end{table}

(10) The 2-space group $\Mu$ with IT number 10 is $442$ in Conway's notation and $p4$ in IT notation. 
See space group 3/1/1/1 in Table 1A of \cite{B-Z} for the standard affine representation of $\Mu$. 
The  flat orbifold $E^2/\Mu$ is a turnover with 3 cone points 
obtained by gluing together two congruent $45^\circ-45^\circ$ right triangles along their boundaries. 
The $442$ turnover is orientable. 
The symmetry group of this orbifold is a dihedral group of order 4 
consisting of the identity symmetry, the {\it halfturn}  2-rot., 
the {\it central reflection} c-ref.\ between the two triangles, and 
the {\it triangle reflection} t-ref. 

\begin{lemma} 
If $\Mu$ is the 2-space group $442$, 
then  $\mathrm{Sym}(\Mu)$ is the dihedral group \{\rm idt., 2-rot, c-ref., t-ref.\}, 
and $\mathrm{Sym}(\Mu) = \mathrm{Aff}(\Mu)$, 
and $\Omega:  \mathrm{Aff}(\Mu) \to \mathrm{Out}(\Mu)$ is an isomorphism.  
\end{lemma}
\begin{proof} $\mathrm{Sym}(\Mu)$ = \{\rm idt., 2-rot, c-ref., t-ref.\}, 
since a symmetry of $E^2/\Mu$ fixes the $180^\circ$ cone point and permutes  
the other two cone points. 
We have that $\mathrm{Sym}(\Mu) = \mathrm{Aff}(\Mu)$ and $\Omega:  \mathrm{Aff}(\Mu) \to \mathrm{Out}(\Mu)$ is an isomorphism as in Lemma 24. 
\end{proof}

\begin{theorem}  
Let $\Mu$ be the 2-space group $442$. 
Then $\mathrm{Iso}(C_\infty,\Mu)$ has four elements,   
corresponding to the pairs of elements {\rm \{idt., idt.\}, \{2-rot., 2-rot.\}, \{c-ref., c-ref.\}, \{t-ref., t-ref.\}} 
of $\mathrm{Sym}(\Mu)$ by Theorem 23, 
and $\mathrm{Iso}(D_\infty,\Mu)$ has 10 elements,  
corresponding to the pairs of elements {\rm \{idt., idt.\},  \{idt., 2-rot.\},  \{idt., c-ref.\}, \{idt., t-ref.\},  \{2-rot., 2-rot.\},  \{c-ref., c-ref.\}, 
\{t-ref., t-ref.\}, \{2-rot., c-ref.\}, \{2-rot., t-ref.\}, \{c-ref., t-ref.\}} of $\mathrm{Sym}(\Mu)$  by Theorem 24.
The corresponding co-Seifert fibrations are described in Table 9 by Theorems 27 and  28.  
\end{theorem}

\begin{table} 
\begin{tabular}{rlllll}
no. & fibers & grp. & quotients &  structure group action & classifying pair \\
\hline 
75 & $(442, \mathrm{O})$ & $C_1$ & $(442, \mathrm{O})$ & (idt., idt.) & \{idt., idt.\} \\
79 & $(442, \mathrm{O})$ & $C_2$ & $(442, \mathrm{O})$ & (2-rot., 2-rot.) & \{2-rot., 2-rot.\} \\
83 & $(442, \mathrm{I})$  & $C_1$ & $(442, \mathrm{I})$ & (idt., idt.) & \{idt., idt.\}  \\
85 & $(442, \mathrm{O})$& $C_2$ & $(442, \mathrm{I})$ & (2-rot., ref.) & \{2-rot., 2-rot.\} \\
87 & $(442, \mathrm{I})$  & $C_2$ & $(442, \mathrm{I})$ & (2-rot., ref.) &  \{idt., 2-rot.\} \\
89 & $(442, \mathrm{O})$ & $C_2$ & $(\ast 442, \mathrm{I})$ & (c-ref., ref.) & \{c-ref., c-ref.\} \\
90 & $(442, \mathrm{O})$ & $C_2$ & $(4{\ast}2, \mathrm{I})$ & (t-ref., ref.) & \{t-ref., t-ref.\} \\
97 & $(442, \mathrm{O})$ & $D_2$ & $(\ast 442, \mathrm{I})$ & (c-ref., ref.), (2-rot., 2-rot.) & \{c-ref., t-ref.\} \\
103 & $(442, \mathrm{O})$& $C_2$ & $(\ast 442, \mathrm{O})$ & (c-ref., 2-rot.) & \{c-ref., c-ref.\} \\
104 & $(442, \mathrm{O})$& $C_2$ & $(4{\ast}2, \mathrm{O})$ & (t-ref., 2-rot.) & \{t-ref., t-ref.\} \\
124 & $(442, \mathrm{I})$  & $C_2$ & $(\ast 442, \mathrm{I})$ & (c-ref., ref.) &  \{idt., c-ref.\} \\
126 & $(442, \mathrm{O})$ & $D_2$ & $(\ast 442, \mathrm{I})$ & (c-ref., ref.), (t-ref., 2-rot.) & \{c-ref., 2-rot.\} \\
128 & $(442, \mathrm{I})$   & $C_2$ & $(4{\ast}2, \mathrm{I})$ & (t-ref., ref.) & \{idt., t-ref.\} \\
130 & $(442, \mathrm{O})$ & $D_2$ & $(\ast 442, \mathrm{I})$ & (t-ref., ref.), (c-ref., 2-rot.) & \{t-ref., 2-rot.\}
\end{tabular}

\medskip
\caption{The classification of the co-Seifert fibrations of 3-space groups 
whose co-Seifert fiber is of type $442$ with IT number 10}
\end{table}

\medskip
\noindent{\bf Example 14.}
Let $\Gamma$ be the 3-space group with IT number 126 in Table 1B of \cite{B-Z}. 
Then $\Gamma = \langle t_1,t_2,t_3, \beta,\gamma, D\rangle$ 
where $t_i = e_i+I$ for $i=1,2,3$ are the standard translations, 
and $\beta=\frac{1}{2}e_1+B$, and $\gamma = \frac{1}{2}e_1+\frac{1}{2}e_3 + C$, and 
$$B = \left(\begin{array}{rrr} 0 & -1 & 0\\ 1 & 0 & 0 \\ 0 & 0 & 1  \end{array}\right),\ \ 
C = \left(\begin{array}{rrr} -1 & 0 & 0  \\ 0 & 1 & 0   \\ 0 & 0 & -1 \end{array}\right), \ \ 
D = \left(\begin{array}{rrr} -1 & 0 & 0  \\ 0 & -1 & 0   \\ 0 & 0 & -1 \end{array}\right).$$
The group $\Nu = \langle t_1,t_2, \beta\rangle$ is a complete normal subgroup 
of $\Gamma$ with $V = {\rm Span}(\Nu) = {\rm Span}\{e_1, e_2\}$. 
The flat orbifold $V/\Nu$ is a $442$ turnover.  
Let $\Kappa = \Nu^\perp$.  Then $\Kappa = \langle t_3\rangle$. 
The flat orbifold $V^\perp/\Kappa$ is a circle.  
The structure group $\Gamma/\Nu\Kappa$ is a dihedral group of order 4 
generated by $\Nu\Kappa\gamma$ and $\Nu\Kappa D$. 
The elements $\Nu\Kappa\gamma$ and $\Nu\Kappa D$ act on $V^\perp/\Kappa$ as reflections. 
The two $90^\circ$ cone points of $V/\Nu$ are represented by 
$(1/4,1/4,0)$, which is the fixed point of $\beta$, and by $(3/4,3/4,0)$,  
which is the fixed point of $t_1\beta$. 
The element $\Nu\Kappa\gamma$ acts as the central reflection of $V/\Nu$, 
since it fixes all three cone points.  
The element $\Nu\Kappa D$ acts as the halfturn of $V/\Nu$, 
since $\Nu\Kappa D$ preserves the orientation of $V/\Nu$ because $D$ preserves the orientation of $V$.

(9) The 2-space group $\Mu$ with IT number 9 is $2{\ast}22$ in Conway's notation and $cmm$ in IT notation. 
See space group 2/2/2/1 in Table 1A of \cite{B-Z} for the standard affine representation of $\Mu$. 
The flat orbifold $E^2/\Mu$ is a {\it pointed hood}. 
The most symmetric pointed hood is the {\it square hood} 
obtained by gluing together two congruent squares along the union of two adjacent sides. 
This orbifold has one $180^\circ$ cone point and two $90^\circ$ corner points. 
The symmetry group of this orbifold is a dihedral group of order 4 
consisting of the identity symmetry, the {\it central reflection} c-ref.\ between the two squares, 
the {\it diagonal reflection} d-ref., and the {\it halfturn} 2-rot.
Both c-ref.\ and  2-rot.\ transpose the two corner points of the square hood, 
whereas  d-ref.\ fixes each of the corner points of the square hood. 

\begin{lemma} 
If $\Mu$ is the 2-space group $2{\ast}22$ and $E^2/\Mu$ is a square hood, 
then  $\mathrm{Sym}(\Mu)$ is the dihedral group \{\rm idt., 2-rot, c-ref., d-ref.\}, 
and $\mathrm{Sym}(\Mu) = \mathrm{Aff}(\Mu)$, 
and $\Omega:  \mathrm{Aff}(\Mu) \to \mathrm{Out}(\Mu)$ is an isomorphism. 
\end{lemma}
\begin{proof} $\mathrm{Sym}(\Mu)$ = \{\rm idt., 2-rot, c-ref., d-ref.\}, 
since a symmetry of $E^2/\Mu$ fixes the cone point 
and permutes the corner points. 
We have that $\mathrm{Sym}(\Mu) = \mathrm{Aff}(\Mu)$ and $\Omega:  \mathrm{Aff}(\Mu) \to \mathrm{Out}(\Mu)$ is an isomorphism as in Lemma 24. 
\end{proof}

\begin{theorem}  
Let $\Mu$ be the 2-space group $2{\ast}22$. 
Then $\mathrm{Iso}(C_\infty,\Mu)$ has four elements,   
corresponding to the pairs of elements {\rm \{idt., idt.\}, \{2-rot., 2-rot.\}, \{c-ref., c-ref.\}, \{d-ref., d-ref.\}} 
of $\mathrm{Sym}(\Mu)$ by Theorem 23, 
and $\mathrm{Iso}(D_\infty,\Mu)$ has 10 elements,  
corresponding to the pairs of elements {\rm \{idt., idt.\},  \{idt., 2-rot.\},  \{idt., c-ref.\}, \{idt., d-ref.\},  \{2-rot., 2-rot.\},  \{c-ref., c-ref.\}, \{d-ref., d-ref.\}, \{2-rot., c-ref.\}, \{2-rot., d-ref.\}, \{c-ref., d-ref.\}} of $\mathrm{Sym}(\Mu)$  by Theorem 24.
The corresponding co-Seifert fibrations are described in Table 10 by Theorems 27 and 28.  
\end{theorem}

\begin{table} 
\begin{tabular}{rlllll}
no. & fibers & grp. & quotients &  structure group action & classifying pair \\
\hline 
35 & $(2{\ast}22, \mathrm{O})$ & $C_1$ & $(2{\ast}22, \mathrm{O})$ & (idt., idt.) & \{idt., idt.\} \\
42 & $(2{\ast}22, \mathrm{O})$ & $C_2$ & $(\ast 2222, \mathrm{O})$ & (c-ref., 2-rot.) & \{c-ref., c-ref.\} \\
65 & $(2{\ast}22, \mathrm{I})$  & $C_1$ & $(2{\ast}22, \mathrm{I})$ & (idt., idt.) & \{idt., idt.\}  \\
67 & $(2{\ast}22, \mathrm{O})$ & $C_2$ & $(\ast 2222, \mathrm{I})$ & (c-ref., ref.) & \{c-ref., c-ref.\} \\
69 & $(2{\ast}22, \mathrm{I})$   & $C_2$ & $(\ast 2222, \mathrm{I})$ &  (c-ref., ref.) & \{idt., c-ref.\} \\
101 & $(2{\ast}22, \mathrm{O})$ & $C_2$ & $(\ast 442, \mathrm{O})$ & (d-ref., 2-rot.) & \{d-ref., d-ref.\} \\
102 & $(2{\ast}22, \mathrm{O})$ & $C_2$ & $(4{\ast}2, \mathrm{O})$ & (2-rot., 2-rot.) & \{2-rot., 2-rot.\} \\
111 & $(2{\ast}22, \mathrm{O})$ & $C_2$ & $(\ast 442, \mathrm{I})$ & (d-ref., ref.) & \{d-ref., d-ref.\} \\
113 & $(2{\ast}22, \mathrm{O})$ & $C_2$ & $(4{\ast}2, \mathrm{I})$ & (2-rot., ref.) & \{2-rot., 2-rot.\} \\
121 & $(2{\ast}22, \mathrm{O})$ & $D_2$ & $(\ast 442, \mathrm{I})$ & (d-ref., ref.), \{c-ref., 2-rot.\} & \{d-ref., 2-rot.\} \\
132 & $(2{\ast}22, \mathrm{I})$  & $C_2$ & $(\ast 442, \mathrm{I})$ & (d-ref., ref.) &  \{idt., d-ref.\} \\
134 & $(2{\ast}22, \mathrm{O})$ & $D_2$ & $(\ast 442, \mathrm{I})$ & (c-ref., ref.), (2-rot., 2-rot.) & \{c-ref., d-ref.\} \\
136 & $(2{\ast}22, \mathrm{I})$  & $C_2$ & $(4{\ast}2, \mathrm{I})$ & (2-rot., ref.) & \{idt., 2-rot.\} \\
138 & $(2{\ast}22, \mathrm{O})$ & $D_2$ & $(\ast 442, \mathrm{I})$ & (c-ref., ref.), (d-ref., 2-rot.) &  \{c-ref., 2-rot.\}
\end{tabular}

\medskip
\caption{The classification of the co-Seifert fibrations of 3-space groups 
whose co-Seifert fiber is of type $2{\ast}22$ with IT number 9}
\end{table}

\medskip
\noindent{\bf Example 15.}
Let $\Gamma$ be the 3-space group with IT number 134 in Table 1B of \cite{B-Z}. 
Then $\Gamma = \langle t_1,t_2,t_3, \beta,\gamma, D\rangle$ 
where $t_i = e_i+I$ for $i=1,2,3$ are the standard translations, 
and $\beta=\frac{1}{2}e_1+\frac{1}{2}e_3+B$, and $\gamma = \frac{1}{2}e_1+\frac{1}{2}e_3 + C$, and 
$$B = \left(\begin{array}{rrr} 0 & -1 & 0\\ 1 & 0 & 0 \\ 0 & 0 & 1  \end{array}\right),\ \ 
C = \left(\begin{array}{rrr} -1 & 0 & 0  \\ 0 & 1 & 0   \\ 0 & 0 & -1 \end{array}\right), \ \ 
D = \left(\begin{array}{rrr} -1 & 0 & 0  \\ 0 & -1 & 0   \\ 0 & 0 & -1 \end{array}\right).$$
The group $\Nu = \langle t_1,t_2, t_3^{-1}\beta^2,t_3^{-1}\beta\gamma D\rangle$ is a complete normal subgroup of $\Gamma$ with $V = {\rm Span}(\Nu) = {\rm Span}\{e_1, e_2\}$. 
The flat orbifold $V/\Nu$ is a a pointed hood.  
Let $\Kappa = \Nu^\perp$.  Then $\Kappa = \langle t_3\rangle$. 
The flat orbifold $V^\perp/\Kappa$ is a circle.  
The structure group $\Gamma/\Nu\Kappa$ is a dihedral group of order 4 
generated by $\Nu\Kappa\gamma$ and $\Nu\Kappa D$. 
The elements $\Nu\Kappa\gamma$ and $\Nu\Kappa D$ act on $V^\perp/\Kappa$ as reflections. 
The product element $\Nu\Kappa\gamma D$ acts as a halfturn on $V^\perp/\Kappa$. 

A fundamental polygon for the action of $\Nu$ on $V$ is the $45^\circ-45^\circ$ 
right triangle $\triangle$ with vertices $v_1=-\frac{1}{4}e_1+\frac{1}{4}e_2$, and $v_2=\frac{3}{4}e_1+\frac{1}{4}e_2$, 
and $v_3 =\frac{1}{4}e_1-\frac{1}{4}e_2$. 
The short sides $[v_1,v_3]$ and $[v_2,v_3]$ of $\triangle$ are fixed pointwise by the reflections 
$\beta^{-1}\gamma D$ and $t_2^{-1}t_3^{-1}\beta\gamma D$, respectively. 
The long side $[v_1,v_2]$ of $\triangle$ is flipped about its midpoint $v_0 = \frac{1}{4}e_1+\frac{1}{4}e_2$ 
by the halfturn $t_3^{-1}\beta^2$. 

The orbifold $V/\Nu$ has a unique cone point $\Nu v_0$ and two right angle corner points 
$\Nu v_1$ and $\Nu v_3$. 
The element $\Nu\Kappa\gamma$ acts as the diagonal reflection of $V/\Nu$, 
since it fixes both corner points.  
The element $\Nu\Kappa D$ acts as the central reflection of $V/\Nu$, 
since $Dv_1 = v_3$ and $D$ fixes the origin which is the midpoint between $v_1$ and $v_3$.

(8) The 2-space group $\Mu$ with IT number 8 is $22\times$ in Conway's notation or $pgg$ in IT notation. 
See space group 2/2/1/3 in Table 1A of \cite{B-Z} for the standard affine representation of $\Mu$. 
The flat orbifold $E^2/\Mu$ is a {\it projective pillow}. 
The most symmetric projective pillow is the {\it square projective pillow} obtained by gluing the opposite sides 
of a square $\Box$ by glide reflections with axes the lines joining the midpoints of opposite sides of $\Box$. 
This orbifold has two $180^\circ$ cone points represented by diagonally opposite vertices of $\Box$, 
and a {\it center point} represented by the center of $\Box$. 
The symmetry group of this orbifold is a dihedral group of order 8, 
represented by the symmetry group of $\Box$, 
 consisting of the identity symmetry, two {\it midline reflections}, m-ref. and m-ref.$'$, 
two {\it diagonal reflections}, d-ref.\ and d-ref.$'$, a {\it halfturn} 2-rot., 
and two order 4 rotations, 4-rot.\ and 4-rot.$^{-1}$, with 4-rot.\ = (d-ref.)(m-ref.).  
There are three conjugacy classes of order 2 symmetries, the classes represented by m-ref., 2-rot., and d-ref. 
There is one conjugacy class of order 4 symmetries. 
There are two conjugacy classes of dihedral subgroups of order 4, 
the group generated by the midline reflections, 
and the group generated by the diagonal reflections. 

\begin{lemma} 
If $\Mu$ is the 2-space group $22\times$ and $E^2/\Mu$ is a square projective pillow, 
then $\mathrm{Sym}(\Mu)$ is the dihedral group $\langle$\rm m-ref., d-ref.$\rangle$ of order 8, 
and $\mathrm{Sym}(\Mu) = \mathrm{Aff}(\Mu)$, 
and $\Omega:  \mathrm{Aff}(\Mu) \to \mathrm{Out}(\Mu)$ is an isomorphism.  
\end{lemma}
\begin{proof} $\mathrm{Sym}(\Mu)$ = $\langle$\rm m-ref., d-ref.$\rangle$, 
since a symmetry of $E^2/\Mu$ fixes the center point and permutes the cone points. 
We have that $\mathrm{Sym}(\Mu) = \mathrm{Aff}(\Mu)$ and $\Omega:  \mathrm{Aff}(\Mu) \to \mathrm{Out}(\Mu)$ is an isomorphism as in Lemma 24. 
\end{proof}

\begin{theorem}  
Let $\Mu$ be the 2-space group $22\times$. 
Then $\mathrm{Iso}(C_\infty,\Mu)$ has five elements,   
corresponding to the pairs of elements {\rm \{idt., idt.\}, \{2-rot., 2-rot.\}, \{m-ref., m-ref.\}, \{d-ref., d-ref.\}, 
\{4-rot., 4-rot.$^{-1}$\}} of $\mathrm{Sym}(\Mu)$ by Theorem 23, 
and $\mathrm{Iso}(D_\infty,\Mu)$ has 12 elements,  
corresponding to the pairs of elements {\rm \{idt., idt.\},  \{idt., 2-rot.\},  \{idt., m-ref.\}, \{idt., d-ref.\},  \{2-rot., 2-rot.\},  \{m-ref., m-ref.\}, 
\{d-ref., d-ref.\}, \{2-rot., m-ref.\}, \{m-ref., m-ref.$'$\}, \{2-rot., d-ref.\}, \{d-ref., d-ref.$'$\}, \{d-ref., m-ref.\}} of $\mathrm{Sym}(\Mu)$  by Theorem 24.
The corresponding co-Seifert fibrations are described in Table 11 by Theorems 27 and 28.  
\end{theorem}

\medskip
\begin{table}  
\begin{tabular}{rlllll}
no. & fibers & grp. & quotients &  structure group action & classifying pair \\
\hline 
  32 & $(22\times, \mathrm{O})$ & $C_1$ & $(22\times, \mathrm{O})$ & (idt., idt.) & \{idt., idt.\} \\
  41 & $(22\times, \mathrm{O})$ & $C_2$ & $(22\ast, \mathrm{O})$ & (m-ref., 2-rot.) &  \{m-ref., m-ref.\} \\
  45 & $(22\times, \mathrm{O})$& $C_2$ & $(2{\ast}22, \mathrm{O})$ & (2-rot., 2-rot.) & \{2-rot., 2-rot.\} \\
  50 & $(22\times, \mathrm{O})$ & $C_2$ & $(2{\ast}22, \mathrm{I})$ & (2-rot., ref.) & \{2-rot., 2-rot.\} \\
  54 & $(22\times, \mathrm{O})$ & $C_2$ & $(22\ast, \mathrm{I})$ & (m-ref., ref.)  &  \{m-ref., m-ref.\} \\
  55 & $(22\times, \mathrm{I})$  & $C_1$ & $(22\times, \mathrm{I})$ & (idt., idt.) & \{idt., idt.\}  \\
  64 & $(22\times, \mathrm{I})$  & $C_2$ & $(22\ast, \mathrm{I})$ & (m-ref., ref.)& \{idt., m-ref.\} \\
  68 & $(22\times, \mathrm{O})$ & $D_2$ & $(\ast 2222, \mathrm{I})$ & (m-ref., ref.), (m-ref.$'$, 2-rot.) & \{m-ref., 2-rot.\} \\
  72 & $(22\times, \mathrm{I})$   & $C_2$ & $(2{\ast}22, \mathrm{I})$ & (2-rot., ref.) &  \{idt., 2-rot.\} \\
  73 & $(22\times, \mathrm{O})$ & $D_2$ & $(\ast 2222, \mathrm{I})$ & (m-ref., ref.), (2-rot., 2-rot.) & \{m-ref., m-ref.$'$\} \\
106 & $(22\times, \mathrm{O})$  & $C_2$ & $(4{\ast}2, \mathrm{O})$ & (d-ref., 2-rot.) & \{d-ref., d-ref.\} \\
110 & $(22\times, \mathrm{O})$  & $C_4$ & $(4{\ast}2, \mathrm{O})$ & (4-rot., 4-rot.) & \{4-rot., 4-rot.$^{-1}$\} \\
117 & $(22\times, \mathrm{O})$  & $C_2$ & $(4{\ast}2, \mathrm{I})$ & (d-ref., ref.) & \{d-ref., d-ref.\} \\
120 & $(22\times, \mathrm{O})$ & $D_2$ & $(\ast 442, \mathrm{I})$ & (d-ref., ref.), (2-rot., 2-rot.) & \{d-ref., d-ref.$'$\} \\
133 & $(22\times, \mathrm{O})$ & $D_2$ & $(\ast 442, \mathrm{I})$ & (d-ref., ref.), (d-ref.$'$, 2-rot.) &  \{d-ref., 2-rot.\} \\
135 & $(22\times, \mathrm{I})$   & $C_2$ & $(4{\ast}2, \mathrm{I})$ & (d-ref., ref.) & \{idt., d-ref.\} \\
142 & $(22\times, \mathrm{O})$ & $D_4$ & $(\ast 442, \mathrm{I})$ & (d-ref., ref.), (4-rot., 4-rot.) & \{d-ref., m-ref.\}
\end{tabular}

\medskip
\caption{The classification of the co-Seifert fibrations of 3-space groups 
whose co-Seifert fiber is of type $22\times$ with IT number 8}
\end{table}

\medskip
\noindent{\bf Example 16.}
Let $\Gamma$ be the 3-space group with IT number 68 in Table 1B of \cite{B-Z}. 
Then $\Gamma = \langle t_1,t_2,t_3, \alpha, \beta, C\rangle$ 
where $t_i = e_i+I$ for $i=1,2,3$ are the standard translations, 
and $\alpha = \frac{1}{2}e_1+\frac{1}{2}e_2+A$,  $\beta=\frac{1}{2}e_3+B$, and 
$$A = \left(\begin{array}{rrr} -1 & 0 & 0\\ 0 & -1 & 0 \\ 0 & 0 & 1  \end{array}\right),\ \ 
B = \left(\begin{array}{rrr} 0 & -1 & 0  \\ -1 & 0 & 0   \\ 0 & 0 & -1 \end{array}\right), \ \ 
C = \left(\begin{array}{rrr} -1 & 0 & 0  \\ 0 & -1 & 0   \\ 0 & 0 & -1 \end{array}\right).$$
The group $\Nu = \langle t_1t_2,t_3, \beta, \alpha C\rangle$ is a complete normal subgroup 
of $\Gamma$ with $V = {\rm Span}(\Nu) = {\rm Span}\{e_1+e_2, e_3\}$. 
The flat orbifold $V/\Nu$ is a projective pillow.  
Let $\Kappa = \Nu^\perp$.  Then $\Kappa = \langle t_1t_2^{-1} \rangle$. 
The flat orbifold $V^\perp/\Kappa$ is a circle.  
The structure group $\Gamma/\Nu\Kappa$ is a dihedral group of order 4 
generated by $\Nu\Kappa t_1$ and $\Nu\Kappa C$. 
The element $\Nu\Kappa t_1$ acts as a halfturn on $V^\perp/\Kappa$, 
since 
$t_1 = \left(\textstyle{\frac{1}{2}e_1+\frac{1}{2}e_2}\right) + \left(\textstyle{\frac{1}{2}e_1-\frac{1}{2}e_2}\right)+I.$
The element $\Nu\Kappa C$ acts as a reflection on $V^\perp/\Kappa$. 

A fundamental polygon for the action of $\Nu$ on $V$ is the rectangle
$\Box$ with vertices 
$v_1=(0,0,1/4)$, $v_2=(1/2,1/2,1/4)$, $v_3 = (0,0,3/4)$, $v_4 = (1/2,1/2,3/4)$.  
The glide reflection $t_3\alpha C$ maps side $[v_1,v_3]$ to side $[v_2,v_4]$, 
and the glide reflection $t_3\alpha C\beta$ maps side $[v_1,v_2]$ to side $[v_3,v_4]$.  
The isometries $\beta$ and $t_3\beta$ act as halfturns on $V$ with fixed points $v_1$ and $v_3$. 
The cone points of $V/\Nu$ are represented by the vertices of $\Box$ with antipodal vertices 
representing the same cone point. 
The element $\Nu\Kappa t_1$ represents a midline reflection of $V/\Nu$, since 
$t_1(\frac{1}{4}e_3) = \frac{1}{2}e_1+\frac{1}{2}e_2+\frac{1}{4}e_3$. 
The element $\Nu\Kappa C$ represents a midline reflection of $V/\Nu$, since 
$t_3C(\frac{1}{4}e_3) = \frac{3}{4}e_3$. 

(7) The 2-space group $\Mu$ with IT number 7 is $22\ast$ in Conway's notation and $pmg$ in IT notation. 
See space group 2/2/1/2 in Table 1A of \cite{B-Z} for the standard affine representation of $\Mu$. 
The flat orbifold $E^2/\Mu$ is a {\it pillowcase}  
obtained by gluing together two congruent rectangles along the union of three of their sides. 
This orbifold has two $180^\circ$ cone points. 
The symmetry group of this orbifold is a dihedral group of order 4 
consisting of the identity symmetry, the {\it central reflection} c-ref.\ between the two rectangles, 
the {\it halfturn} 2-rot.,  and the {\it midline reflection} m-ref.

\begin{lemma} 
If $\Mu$ is the 2-space group $22\ast$, 
then $\mathrm{Sym}(\Mu)$ is the dihedral group \{\rm idt., c-ref., m-ref., 2-rot.\}, 
and $\mathrm{Sym}(\Mu) = \mathrm{Aff}(\Mu)$,  
and $\Omega:  \mathrm{Aff}(\Mu) \to \mathrm{Out}(\Mu)$ is an isomorphism.  
\end{lemma}
\begin{proof} $\mathrm{Sym}(\Mu)$ = \{\rm idt., c-ref., m-ref., 2-rot\}, 
since a symmetry of $E^2/\Mu$ permutes the cone points. 
We have that $\mathrm{Sym}(\Mu) = \mathrm{Aff}(\Mu)$ and $\Omega:  \mathrm{Aff}(\Mu) \to \mathrm{Out}(\Mu)$ is an isomorphism as in Lemma 24. 
\end{proof}

\begin{theorem}  
Let $\Mu$ be the 2-space group $22\ast$. 
Then $\mathrm{Iso}(C_\infty,\Mu)$ has four elements,   
corresponding to the pairs of elements {\rm \{idt., idt.\}, \{2-rot., 2-rot.\},  \{c-ref., c-ref.\}, \{m-ref., m-ref.\}} 
of $\mathrm{Sym}(\Mu)$ by Theorem 23, 
and $\mathrm{Iso}(D_\infty,\Mu)$ has 10 elements,  
corresponding to the pairs of elements {\rm \{idt., idt.\},  \{idt., 2-rot.\},  \{idt., c-ref.\}, \{idt., m-ref.\},  \{2-rot., 2-rot.\},  \{c-ref., c-ref.\}, 
\{m-ref., m-ref.\}, \{2-rot., c-ref.\},  \{2-rot., m-ref.\},  \{c-ref., m-ref.\}} of $\mathrm{Sym}(\Mu)$  by Theorem 24.
The corresponding co-Seifert fibrations are described in Table 12 by Theorems 27 and 28.  
\end{theorem}

\begin{table} 
\begin{tabular}{llllll}
no. & fibers & grp. & quotients &  structure group action & classifying pair \\
\hline 
28 & $(22\ast, \mathrm{O})$ & $C_1$ & $(22\ast, \mathrm{O})$ & (idt., idt.) & \{idt., idt.\} \\
39 & $(22\ast, \mathrm{O})$ & $C_2$ & $(\ast 2222, \mathrm{O})$ & (c-ref., 2-rot.) & \{c-ref., c-ref.\} \\
40 & $(22\ast, \mathrm{O})$ & $C_2$ & $(22\ast, \mathrm{O})$ & (2-rot., 2-rot.) & \{2-rot., 2-rot.\} \\
46 & $(22\ast, \mathrm{O})$ & $C_2$ & $(2{\ast}22, \mathrm{O})$ & (m-ref., 2-rot.) & \{m-ref., m-ref.\} \\
49 & $(22\ast, \mathrm{O})$ & $C_2$ & $(\ast 2222, \mathrm{I})$ & (c-ref., ref.) & \{c-ref., c-ref.\} \\
51 & $(22\ast, \mathrm{I})$   & $C_1$ & $(22\ast, \mathrm{I})$ & (idt., idt.) & \{idt., idt.\}  \\ 
53 & $(22\ast, \mathrm{O})$ & $C_2$ & $(2{\ast}22, \mathrm{I})$ & (m-ref., ref.) & \{m-ref., m-ref.\} \\
57 & $(22\ast, \mathrm{O})$ & $C_2$ & $(22\ast, \mathrm{I})$ & (2-rot., ref.) & \{2-rot., 2-rot.\}  \\
63 & $(22\ast, \mathrm{I})$   & $C_2$ & $(22\ast, \mathrm{I})$ & (2-rot., ref.) & \{idt., 2-rot.\}\\
64 & $(22\ast, \mathrm{O})$ & $D_2$ & $(\ast 2222, \mathrm{I})$ & (m-ref., ref.), (c-ref., 2-rot.) & \{m-ref., 2-rot.\}\\
66 & $(22\ast, \mathrm{O})$ & $D_2$ & $(\ast 2222, \mathrm{I})$ & (c-ref., ref.), (2-rot., 2-rot.) & \{c-ref., m-ref.\} \\
67 & $(22\ast, \mathrm{I})$   & $C_2$ & $(\ast 2222, \mathrm{I})$ & (c-ref., ref.) & \{idt., c-ref.\} \\
72 & $(22\ast, \mathrm{O})$ & $D_2$ & $(\ast 2222, \mathrm{I})$ & (c-ref., ref.), (m-ref., 2-rot.) &  \{c-ref., 2-rot.\} \\ 
74 & $(22\ast, \mathrm{I})$   & $C_2$ & $(2{\ast}22, \mathrm{I})$ & (m-ref., ref.) & \{idt., m-ref.\} 
\end{tabular}

\medskip
\caption{The classification of the co-Seifert fibrations of 3-space groups 
whose co-Seifert fiber is of type $22\ast$ with IT number 7}
\end{table}

\medskip
\noindent{\bf Example 17.}
Let $\Gamma$ be the 3-space group with IT number 64 in Table 1B of \cite{B-Z}. 
Then $\Gamma = \langle t_1,t_2,t_3, \alpha, B, C\rangle$ 
where $t_i = e_i+I$ for $i=1,2,3$ are the standard translations, 
and $\alpha = \frac{1}{2}e_1+\frac{1}{2}e_2+A$,  $\beta=\frac{1}{2}e_3+B$, and 
$$A = \left(\begin{array}{rrr} -1 & 0 & 0\\ 0 & -1 & 0 \\ 0 & 0 & 1  \end{array}\right),\ \ 
B = \left(\begin{array}{rrr} 0 & -1 & 0  \\ -1 & 0 & 0   \\ 0 & 0 & -1 \end{array}\right), \ \ 
C = \left(\begin{array}{rrr} -1 & 0 & 0  \\ 0 & -1 & 0   \\ 0 & 0 & -1 \end{array}\right).$$
The group $\Nu = \langle t_1t_2^{-1},t_3, t_2^{-1}\alpha C, BC\rangle$ is a complete normal subgroup 
of $\Gamma$ with $V = {\rm Span}(\Nu) = {\rm Span}\{e_1-e_2, e_3\}$. 
The flat orbifold $V/\Nu$ is a pillowcase.  
Let $\Kappa = \Nu^\perp$.  Then $\Kappa = \langle t_1t_2\rangle$. 
The flat orbifold $V^\perp/\Kappa$ is a circle.  
The structure group $\Gamma/\Nu\Kappa$ is a dihedral group of order 4 
generated by $\Nu\Kappa t_1$ and $\Nu\Kappa C$. 
The element $\Nu\Kappa t_1$ acts as a halfturn on $V^\perp/\Kappa$, 
since 
$t_1 = \left(\textstyle{\frac{1}{2}e_1+\frac{1}{2}e_2}\right) + \left(\textstyle{\frac{1}{2}e_1-\frac{1}{2}e_2}\right)+I.$
The element $\Nu\Kappa C$ acts as a reflection on $V^\perp/\Kappa$. 

A fundamental polygon for the action of $\Nu$ on $V$ is the rectangle
$\Box$ with vertices 
$v_1=(0,0,1/4)$, $v_2=(1/4,-1/4,1/4)$, $v_3 = (1/4,-1/4,5/4)$, $v_4 = (0,0,5/4)$. 
The reflection $BC$ maps fixes the side $[v_1,v_4]$ pointwise, 
and translation $t_3$ maps side $[v_1,v_2]$ to side $[v_4,v_3]$.  
The halfturn $t_2^{-1}t_3\alpha B$ rotates the side $[v_2,v_3]$ around its midpoint $v_0 = (1/4,-1/4,3/4)$. 
The halfturn $t_2^{-1}\alpha\beta$ fixes the vertex $v_2$. 
The cone points of $V/\Nu$ are represented by $v_0$ and $v_2$. 
Now $BC(t_1^{-1}t_2)t_1$ fixes the point $v_0$. 
Hence the element $\Nu\Kappa t_1$ fixes the cone point $\Nu v_0$ of $V/\Nu$,  
and so $\Nu\Kappa t_1$ acts as the central reflection of $V/\Nu$. 
The point $(1/8,-1/8,1/2)$ represents the center of one side of the pillowcase $V/\Nu$,   
and $BCt_3C$ fixes $(1/8,-1/8,1/2)$. 
Hence the element $\Nu\Kappa C$ represents the midline reflection of $V/\Nu$. 

(6) The 2-space group $\Mu$ with IT number 6 is $\ast 2222$ in Conway's notation or $pmm$ in IT notation. 
See space group 2/2/1/1 in Table 1A of \cite{B-Z} for the standard affine representation of $\Mu$. 
The flat orbifold $E^2/\Mu$ is a rectangle. 
A rectangle has four $90^\circ$ corner points. 
The  most symmetric rectangle is a square $\Box$. 
The symmetry group of $\Box$ is a dihedral group of order 8 
consisting of the identity symmetry, two {\it midline reflections}, m-ref.\ and m-ref.$'$, 
two {\it diagonal reflections},  d-ref.\ and d-ref.$'$, a halfturn 2-rot., 
and two order 4 rotations, 4-rot.\ and 4-rot.$^{-1}$, with 4-rot.\ = (d-ref.)(m-ref.). 
There are three conjugacy classes of order two symmetries, the classes represented by m-ref., 2-rot., and d-ref. 
There is one conjugacy class of order 4 symmetries. 
There are two conjugacy classes of dihedral subgroups of order 4, 
the group generated by the midline reflections, 
and the group generated by the diagonal reflections. 

\begin{lemma} 
If $\Mu$ is the 2-space group $\ast 2222$ and $E^2/\Mu$ is a square, 
then $\mathrm{Sym}(\Mu)$ is the dihedral group $\langle$\rm m-ref., d-ref.$\rangle$ of order 8, 
and $\mathrm{Sym}(\Mu) = \mathrm{Aff}(\Mu)$, 
and $\Omega:  \mathrm{Aff}(\Mu) \to \mathrm{Out}(\Mu)$ is an isomorphism.  
\end{lemma}
\begin{proof} $\mathrm{Sym}(\Mu)$ = $\langle$\rm m-ref., d-ref.$\rangle$, 
since a symmetry of $E^2/\Mu$ permutes the corner points. 
We have that $\mathrm{Sym}(\Mu) = \mathrm{Aff}(\Mu)$ and $\Omega:  \mathrm{Aff}(\Mu) \to \mathrm{Out}(\Mu)$ is an isomorphism as in Lemma 24. 
\end{proof}

\begin{theorem}  
Let $\Mu$ be the 2-space group $\ast 2222$. 
Then $\mathrm{Iso}(C_\infty,\Mu)$ has five elements,   
corresponding to the pairs of elements {\rm \{idt., idt.\}, \{2-rot., 2-rot.\}, \{m-ref., m-ref.\}, \{d-ref., d-ref.\}, 
\{4-rot., 4-rot.$^{-1}$\}} of $\mathrm{Sym}(\Mu)$ by Theorem 23, 
and $\mathrm{Iso}(D_\infty,\Mu)$ has 12 elements,  
corresponding to the pairs of elements {\rm \{idt., idt.\},  \{idt., 2-rot.\},  \{idt., m-ref.\}, \{idt., d-ref.\},  \{2-rot., 2-rot.\},  \{m-ref., m-ref.\}, 
\{d-ref., d-ref.\}, \{2-rot., m-ref.\}, \{m-ref., m-ref.$'$\}, \{2-rot., d-ref.\}, \{d-ref., d-ref.$'$\}, \{d-ref., m-ref.\}} of $\mathrm{Sym}(\Mu)$  by Theorem 24.
The corresponding co-Seifert fibrations are described in Table 13 by Theorems 27 and 28.  
\end{theorem}

\begin{table}  
\begin{tabular}{rlllll}
no. & fibers & grp. & quotients &  structure group action & classifying pair \\
\hline 
  25 & $(\ast 2222, \mathrm{O})$ & $C_1$ & $(\ast 2222, \mathrm{O})$ & (idt., idt.) & \{idt., idt.\} \\
  38 & $(\ast 2222, \mathrm{O})$ & $C_2$ & $(\ast 2222, \mathrm{O})$ & (m-ref., 2-rot.) & \{m-ref., m-ref.\} \\
  44 & $(\ast 2222, \mathrm{O})$ & $C_2$ & $(2{\ast}22, \mathrm{O})$ & (2-rot., 2-rot.) & \{2-rot., 2-rot.\} \\
  47 & $(\ast 2222, \mathrm{I})$   & $C_1$ & $(\ast 2222, \mathrm{I})$ & (idt., idt.) & \{idt., idt.\} \\
  51 & $(\ast 2222, \mathrm{O})$ & $C_2$ & $(\ast 2222, \mathrm{I})$ & (m-ref., ref.) & \{m-ref., m-ref.\} \\
  59 & $(\ast 2222, \mathrm{O})$ & $C_2$ & $(2{\ast}22, \mathrm{I})$ & (2-rot., ref.) & \{2-rot., 2-rot.\}  \\
  63 & $(\ast 2222, \mathrm{O})$ & $D_2$ & $(\ast 2222, \mathrm{I})$ & (m-ref., ref.),  (m-ref.$'$, 2-rot.) &  \{m-ref., 2-rot.\} \\
  65 & $(\ast 2222, \mathrm{I})$   & $C_2$ & $(\ast 2222, \mathrm{I})$ &  (m-ref., ref.) & \{idt., m-ref.\} \\
  71 & $(\ast 2222, \mathrm{I})$   & $C_2$ & $(2{\ast}22, \mathrm{I})$ & (2-rot., ref.) &  \{idt., 2-rot.\} \\
  74 & $(\ast 2222, \mathrm{O})$ & $D_2$ & $(\ast 2222, \mathrm{I})$ & (m-ref., ref.), (2-rot., 2-rot.) & \{m-ref., m-ref.$'$\} \\
105 & $(\ast 2222, \mathrm{O})$ & $C_2$ & $(\ast 442, \mathrm{O})$ & (d-ref., 2-rot.) & \{d-ref., d-ref.\} \\
109 & $(\ast 2222, \mathrm{O})$ & $C_4$ & $(4{\ast}2, \mathrm{O})$ & (4-rot., 4-rot.) & \{4-rot., 4-rot.$^{-1}$\} \\
115 & $(\ast 2222, \mathrm{O})$ & $C_2$ & $(\ast 442, \mathrm{I})$ & (d-ref., ref.) & \{d-ref., d-ref.\} \\
119 & $(\ast 2222, \mathrm{O})$ & $D_2$ & $(\ast 442, \mathrm{I})$ & (d-ref., ref.), (2-rot., 2-rot.) & \{d-ref., d-ref.$'$\} \\
131 & $(\ast 2222, \mathrm{I})$   & $C_2$ & $(\ast 442, \mathrm{I})$ & (d-ref., ref.) & \{idt., d-ref.\} \\
137 & $(\ast 2222, \mathrm{O})$ & $D_2$ & $(\ast 442, \mathrm{I})$ & (d-ref., ref.), (d-ref.$'$, 2-rot.) & \{d-ref., 2-rot.\} \\
141 & $(\ast 2222, \mathrm{O})$ & $D_4$ & $(\ast 442, \mathrm{I})$ & (d-ref., ref.), (4-rot., 4-rot.) &  \{d-ref., m-ref.\}
\end{tabular}

\medskip
\caption{The classification of the co-Seifert fibrations of 3-space groups 
whose co-Seifert fiber is of type $\ast 2222$ with IT number 6}
\end{table}

\medskip
\noindent{\bf Example 18.}
Let $\Gamma$ be the 3-space group with IT number 63 in Table 1B of \cite{B-Z}. 
Then $\Gamma = \langle t_1,t_2,t_3, \alpha, B, C\rangle$ 
where $t_i = e_i+I$ for $i=1,2,3$ are the standard translations, 
and $\alpha = \frac{1}{2}e_3+A$,  and 
$$A = \left(\begin{array}{rrr} -1 & 0 & 0\\ 0 & -1 & 0 \\ 0 & 0 & 1  \end{array}\right),\ \ 
B = \left(\begin{array}{rrr} 0 & -1 & 0  \\ -1 & 0 & 0   \\ 0 & 0 & -1 \end{array}\right), \ \ 
C = \left(\begin{array}{rrr} -1 & 0 & 0  \\ 0 & -1 & 0   \\ 0 & 0 & -1 \end{array}\right).$$
The group $\Nu = \langle t_1t_2^{-1},t_3, \alpha C, BC\rangle$ is a complete normal subgroup 
of $\Gamma$ with $V = {\rm Span}(\Nu) = {\rm Span}\{e_1-e_2, e_3\}$. 
The flat orbifold $V/\Nu$ is a rectangle.  
Let $\Kappa = \Nu^\perp$.  Then $\Kappa = \langle t_1t_2\rangle$.  
The flat orbifold $V^\perp/\Kappa$ is a circle.  
The structure group $\Gamma/\Nu\Kappa$ is a dihedral group of order 4 
generated by $\Nu\Kappa t_1$ and $\Nu\Kappa C$. 
The element $\Nu\Kappa t_1$ acts as a halfturn on $V^\perp/\Kappa$, 
since 
$t_1 = \left(\textstyle{\frac{1}{2}e_1+\frac{1}{2}e_2}\right) + \left(\textstyle{\frac{1}{2}e_1-\frac{1}{2}e_2}\right)+I.$
The element $\Nu\Kappa C$ acts as a reflection on $V^\perp/\Kappa$. 

A fundamental polygon for the action of $\Nu$ on $V$ is the rectangle
$\Box$ with vertices 
$v_1=(0,0,1/4)$, $v_2=(1/2,-1/2,1/4)$, $v_3 = (1/2,-1/2,3/4)$, $v_4 = (0,0,3/4)$. 
The reflection $\alpha C$ fixes the side $[v_1,v_2]$ pointwise, 
the reflection $t_2\alpha C$ fixes the side $[v_3,v_4]$ pointwise,
the reflection $BC$ maps fixes the side $[v_1,v_4]$ pointwise, and
the reflection $t_1t_2^{-1}BC$ fixes the side $[v_2,v_3]$ pointwise. 
The corner points of $V/\Nu$ are represented by the vertices of $\Box$. 
The element $\Nu\Kappa t_1$ acts as a midline reflection of $V/\Nu$, 
since $t_1(v_1) = v_2$ and $t_1(v_4) = v_3$. 
The element $\Nu\Kappa C$ acts as the other midline reflection of $V/\Nu$, since $t_3C(v_1) = v_4$,
and $t_1t_2^{-1}t_3C(v_2) = v_3$. 

(5) Let $\Mu = \langle e_1+I, e_2 + I, A \rangle$ where $e_1, e_2$ are the standard basis vectors of $E^2$ 
and $A$ transposes $e_1$ and $e_2$. 
Then $\Mu$ is a 2-space group with IT number 5.  
The Conway notation for $\Mu$ is $\ast\times$ and the IT notation is $cm$. 
The flat orbifold $E^2/\Mu$ is a M\"obius band. 
We have that $Z(\Mu) = \langle e_1+e_2 +I\rangle$. 

\begin{lemma}  
If $\Mu = \langle e_1+I, e_2 + I, A \rangle$, with $A$ transposing $e_1$ and $e_2$, then 
$$N_A(\Mu) = \{b+B: b_1-b_2 \in \integers \ \, \hbox{and}\ \, B \in \langle -I, A\rangle\}.$$
\end{lemma}
\begin{proof}
Observe that $b+B \in N_A(\Mu)$ if and only if $B\in N_A(\langle e_1+I, e_2+I\rangle)$ 
and $(b+B)A(b+B)^{-1}\in \Mu$.  
Now $B \in N_A(\langle e_1+I, e_2+I\rangle)$ if and only if $B \in \mathrm{GL}(2,\integers)$.
As $(b+B)A(b+B)^{-1} = b-BAB^{-1}b+BAB^{-1}$, we have that $(b+B)A(b+B)^{-1}\in \Mu$ 
if and only if $BAB^{-1} = A$ and $b-Ab \in \integers^2$. 
Hence $b+B \in N_A(\Mu)$ if and only if $B \in \langle -I, A\rangle$ and $b_1-b_2 \in \integers$. 
\end{proof}

A fundamental polygon for the action of $\Mu$ on $E^2$ is the $45^\circ -45^\circ$ right triangle 
with vertices $v_1=(0,0), v_2=(1,0), v_3 = (1,1)$. 
The reflection $A$ fixes the hypotenuse  $[v_1,v_3]$ pointwise.  
The glide reflection $e_1+A$ maps the side $[v_1,v_2]$ to the side $[v_2,v_3]$. 
The {\it boundary} of the M\"obius band $E^2/\Mu$ is represented by the hypotenuse $[v_1,v_3]$, 
and the {\it central circle} of $E^2/\Mu$ is represented by the line segment $[(1/2,0), (1,1/2)]$ 
joining the midpoints of the two short sides of the triangle.  

Restricting a symmetry of the M\"obius band $E^2/\Mu$ to a symmetry of its boundary circle 
gives an isomorphism from the symmetry group of the M\"obius band to the symmetry group of its boundary circle. 
There are two conjugacy classes of elements of order 2, 
the class of the reflection c-ref.\  = $(e_1/2 +e_2/2+I)_\star$ in the central circle of the M\"obius band, 
and the class of a halfturn 2-rot.\ = $(-I)_\star$ about the  point $\Mu(1/2,0)$ on the central circle. 
Note that c-ref.\ restricts to a halfturn of the boundary circle, and 2-rot.\ restricts to a reflection of the boundary circle. 
Define 2-rot.$'$ = $(e_1/2+e_2/2-I)_\star$. 
There is one conjugacy class of dihedral symmetry groups of order 4, represented by the group \{idt., c-ref., 2-rot., 2-rot.$'$\}.

\begin{lemma} 
If $\Mu= \langle e_1+I, e_2 + I, A \rangle$, with $A$ transposing $e_1$ and $e_2$, 
then the Lie group $\mathrm{Sym}(\Mu)$ is isomorphic to $\mathrm{O}(2)$, 
and $\mathrm{Sym}(\Mu) = \mathrm{Aff}(\Mu)$, 
and $\Omega:  \mathrm{Aff}(\Mu) \to \mathrm{Out}(\Mu)$ maps 
the subgroup {\rm \{idt., 2-rot.\}} isomorphically onto $\mathrm{Out}(\Mu)$.  
\end{lemma}
\begin{proof} The Lie group $\mathrm{Sym}(\Mu)$ is isomorphic to $\mathrm{O}(2)$, 
since restricting a symmetry of the M\"obius band $E^2/\Mu$ to a symmetry of its boundary circle $\mathrm{O}$ is 
an isomorphism from $\mathrm{Sym}(\Mu)$ to $\mathrm{Isom}(\mathrm{O})$ by Lemmas 5 and 36. 
We have that $\mathrm{Sym}(\Mu) = \mathrm{Aff}(\Mu)$ by Lemmas 5, 11, and 36. 
The epimorphism $\Omega:  \mathrm{Aff}(\Mu) \to \mathrm{Out}(\Mu)$ maps  
the group \{idt., 2-rot.\} isomorphically onto $\mathrm{Out}(\Mu)$ by Theorem 13, 
since 2-rot.\ acts as a reflection of $\mathrm{O}$. 
\end{proof}

\begin{theorem}  
Let $\Mu= \langle e_1+I, e_2 + I, A \rangle$, with $A$ transposing $e_1$ and $e_2$. 
Then $\mathrm{Iso}(C_\infty,\Mu)$ has two elements,   
corresponding to the pairs of elements {\rm \{idt., idt.\}, \{2-rot., 2-rot.\}} of $\mathrm{Sym}(\Mu)$ by Theorem 23, 
and $\mathrm{Iso}(D_\infty,\Mu)$ has six elements,  
corresponding to the pairs of elements {\rm  \{idt., idt.\},  \{idt., 2-rot.\},  \{idt., c-ref.\}, \{2-rot., 2-rot.\},  \{c-ref., c-ref.\}, \{2-rot., c-ref.\}} of $\mathrm{Sym}(\Mu)$  by Theorem 26.
The corresponding co-Seifert fibrations are described in Table 14 by Theorems 27 and 28.  
Only the pair {\rm \{2-rot., 2-rot.\}} falls into the case $E_1\cap E_2 \neq \{0\}$ of Theorem 26. 
\end{theorem}

\begin{table}  
\begin{tabular}{rlllll}
no. & fibers & grp. & quotients &  structure group action &  classifying pair\\
\hline 
  8 & $(\ast\times, \mathrm{O})$ & $C_1$ & $(\ast\times, \mathrm{O})$ & (idt., idt.) & \{idt., idt.\} \\
12 & $(\ast\times, \mathrm{O})$ & $C_2$ & $(2{\ast}22, \mathrm{I})$ & (2-rot., ref.) &  \{2-rot., 2-rot.\} \\
36 & $(\ast\times, \mathrm{O})$ & $C_2$ & $(2{\ast}22, \mathrm{O})$ & (2-rot., 2-rot.)  & \{2-rot., 2-rot.\} \\
38 & $(\ast\times, \mathrm{I})$   & $C_1$ & $(\ast\times, \mathrm{I})$ & (idt., idt.)  & \{idt., idt.\} \\
39 & $(\ast\times, \mathrm{O})$ & $C_2$ & $(\ast\ast, \mathrm{I})$ & (c-ref., ref.) & \{c-ref., c-ref.\} \\
42 & $(\ast\times, \mathrm{I})$   & $C_2$ & $(\ast\ast, \mathrm{I})$ & (c-ref., ref.) &  \{idt., c-ref.\} \\
63 & $(\ast\times, \mathrm{I})$   & $C_2$ & $(2{\ast}22, \mathrm{I})$ & (2-rot., ref.) &  \{idt., 2-rot.\} \\
64 & $(\ast\times, \mathrm{O})$ & $D_2$ & $(\ast 2222, \mathrm{I})$ & (2-rot., ref.), (2-rot.$'$, 2-rot.) &  \{2-rot., c-ref.\}
\end{tabular}

\medskip
\caption{The classification of the co-Seifert fibrations of 3-space groups 
whose co-Seifert fiber is of type $\ast\times$ with IT number 5}
\end{table}

\medskip
\noindent{\bf Example 19.}
Let $\Gamma$ be the 3-space group with IT number 64 in Table 1B of \cite{B-Z}. 
Then $\Gamma = \langle t_1,t_2,t_3, \alpha, B, C\rangle$ 
where $t_i = e_i+I$ for $i=1,2,3$ are the standard translations, 
and $\alpha = \frac{1}{2}e_1+\frac{1}{2}e_2+\frac{1}{2}e_3+A$,  and 
$$A = \left(\begin{array}{rrr} -1 & 0 & 0\\ 0 & -1 & 0 \\ 0 & 0 & 1  \end{array}\right),\ \ 
B = \left(\begin{array}{rrr} 0 & -1 & 0  \\ -1 & 0 & 0   \\ 0 & 0 & -1 \end{array}\right), \ \ 
C = \left(\begin{array}{rrr} -1 & 0 & 0  \\ 0 & -1 & 0   \\ 0 & 0 & -1 \end{array}\right).$$
The group $\Nu = \langle t_1, t_2, BC\rangle$ is a complete normal subgroup 
of $\Gamma$ with $V = {\rm Span}(\Nu) = {\rm Span}\{e_1, e_2\}$. 
The flat orbifold $V/\Nu$ is a M\"obius band.  
Let $\Kappa = \Nu^\perp$.  Then 
$\Kappa = \langle t_3\rangle$.  
The flat orbifold $V^\perp/\Kappa$ is a circle.  
The structure group $\Gamma/\Nu\Kappa$ is a dihedral group of order 4 
generated by $\Nu\Kappa \alpha$ and $\Nu\Kappa C$. 
The element $\Nu\Kappa C$ acts as 2-rot. on $V/\Nu$ and as a reflection on $V^\perp/\Kappa$. 
The element $\Nu\Kappa \alpha$ acts as 2-rot.$'$ on $V/\Nu$ and as a halfturn on $V^\perp/\Kappa$.

(4) Let $\Mu = \langle e_1+I, e_2+I, e_1/2+A\rangle$ where $A = \mathrm{diag}(1,-1)$. 
Then $\Mu$ is a 2-space group with IT number 4. 
The Conway notation for $\Mu$ is $\times\times$ and the IT notation if $pg$. 
The flat orbifold $E^2/\Mu$ is a Klein bottle. 
We have that $Z(\Mu) = \langle e_1+I \rangle$. 

\begin{lemma}  
If $\Mu = \langle e_1+I, e_2 + I, e_1/2+ A \rangle$, with $A = \mathrm{diag}(1,-1)$, then 
$$N_A(\Mu) = \{b+B: 2b_2 \in \integers \ \, \hbox{and}\ \, B \in \langle -I, A\rangle\}.$$
\end{lemma}
\begin{proof}
Observe that $b+B \in N_A(\Mu)$ if and only if $B\in N_A(\langle e_1+I, e_2+I\rangle)$ 
and $(b+B)(e_1/2+A)(b+B)^{-1}\in \Mu$.  
Now $B \in N_A(\langle e_1+I, e_2+I\rangle)$ if and only if $B \in \mathrm{GL}(2,\integers)$.
As $(b+B)(e_1/2+A)(b+B)^{-1} = b+Be_1/2-BAB^{-1}b+BAB^{-1}$, we have that $(b+B)(e_1/2+A)(b+B)^{-1}\in \Mu$ 
if and only if $BAB^{-1} = A$ and $b-Ab +Be_1/2-e_1/2\in \integers^2$. 
Hence $b+B \in N_A(\Mu)$ if and only if $B \in \langle -I, A\rangle$ and $2b_2 \in \integers$. 
\end{proof}

A fundamental polygon for the action of $\Mu$ on $E^2$ is the rectangle 
with vertices $v_1 = (0,-1/2), v_2 = (1/2,-1/2), v_3 = (1/2,1/2)$ and $v_4= (0,1/2)$. 
The vertical translation $e_2+I$ maps the side $[v_1,v_2]$ to the side $[v_4,v_3]$. 
The horizontal glide reflection $e_1/2 + A$ maps the side $[v_1,v_4]$ to the side $[v_2,v_3]$. 
The Klein bottle $E^2/\Mu$ has two short horizontal geodesics represented by the line segments 
$[v_1, v_2]$ and $[(0,0),(1/2,0)]$. 
The union of these two short geodesics is invariant under any symmetry of the Klein bottle. 
Therefore the horizontal quarterline geodesic, which we call the {\it central circle}, represented by the union 
of the line segments $[(0,-1/4),(1/2,-1/4)]$ and $[(0,1/4),(1/2,1/4)]$ is invariant under any symmetry of the Klein Bottle. 
The symmetry group of the Klein bottle is the direct product of the subgroup of order 2, generated by 
the reflection in the central circle, and the subgroup consisting of the horizontal translations and the reflections 
in a vertical geodesic.   The latter subgroup restricts to the symmetry group of the central circle.  

There are five conjugacy classes of symmetries of order 2, 
the class of the reflection m-ref.\ = $A_\star = (e_1/2+I)_\star$ in the short horizontal geodesics, 
the class of the vertical halfturn 2-sym.\ = $(e_2/2+I)_\star$, 
the class of the reflection c-ref.\ = $(e_2/2 + A)_\star = (e_1/2+e_2/2+I)_\star$ in the central circle,  
the class of the reflection v-ref.\ = $(-A)_\star = (e_1/2-I)_\star$ in the vertical geodesic represented by $[v_1, v_4]$, 
and the class of the halfturn 2-rot.\ = $(e_2/2-I)_\star$ around a pair of antipodal points on the central circle. 
Define v-ref.$' = (-I)_\star$ and 2-rot.$' = (e_1/2 + e_2/2 - I)_\star$.  
There are five conjugacy classes of dihedral symmetry groups of order 4, 
the classes of the groups \{idt., m-ref., 2-sym., c-ref.\},  \{idt.,  m-ref., v-ref., v-ref.$'$\},  \{idt., m-ref., 2-rot., 2-rot.$'$\}, 
 \{idt.,  2-sym., v-ref., 2-rot.$'$\}, and  \{idt.,  c-ref. v-ref., 2-rot.\}. 

\begin{lemma} 
If $\Mu = \langle e_1+I, e_2 + I, e_1/2+ A \rangle$, with $A = \mathrm{diag}(1,-1)$, 
then the Lie group $\mathrm{Sym}(\Mu)$ is isomorphic to $C_2\times\mathrm{O}(2)$, 
and $\mathrm{Sym}(\Mu) = \mathrm{Aff}(\Mu)$, 
and $\Omega:  \mathrm{Aff}(\Mu) \to \mathrm{Out}(\Mu)$ maps 
the subgroup {\rm \{idt., c-ref., v-ref., 2-rot.\}}  isomorphically onto $\mathrm{Out}(\Mu)$.  
Moreover $\Omega${\rm (2-sym)} = $\Omega${\rm (c-ref.)}. 
\end{lemma}
\begin{proof} Let $\mathrm{O}$ be the central circle of the Klein bottle, 
and let $\Lambda$ be the group of all symmetries of the Klein bottle that leave each horizontal geodesic invariant. 
Restriction induces an isomorphism from $\Lambda$ to $\mathrm{Isom}(\mathrm{O})$ by Lemmas 5 and 38. 
We have that  $\mathrm{Sym}(\Mu) = \langle$c-ref.$\rangle \times \Lambda$, 
since every symmetry of the Klein bottle leaves $\mathrm{O}$ invariant,   
and c-ref.\ commutes with every symmetry in $\Lambda$. 
Therefore $\mathrm{Sym}(\Mu)$ is isomorphic to $C_2\times \mathrm{O}(2)$. 
We have that $\mathrm{Sym}(\Mu) = \mathrm{Aff}(\Mu)$ by Lemma 5, 11, and 38. 
The epimorphism $\Omega:  \mathrm{Aff}(\Mu) \to \mathrm{Out}(\Mu)$ maps 
{\rm \{idt., c-ref., v-ref., 2-rot.\}} isomorphically onto $\mathrm{Out}(\Mu)$ by Theorem 13, 
since v-ref.\ acts as a reflection of $\mathrm{O}$. 
Moreover $\Omega$(2-sym) = $\Omega$((c-ref.)(m-ref.)) = $\Omega$(c-ref.). 
\end{proof}

\begin{theorem}  
Let $\Mu = \langle e_1+I, e_2 + I, e_1/2+ A \rangle$ with $A = \mathrm{diag}(1,-1)$. 
Then $\mathrm{Iso}(C_\infty,\Mu)$ has four elements,   
corresponding to the pairs of elements {\rm\{idt., idt.\}, \{2-sym., 2-sym.\}, \{2-rot., 2-rot.\}, \{v-ref., v-ref.\}} of $\mathrm{Sym}(\Mu)$ by Theorem 23, 
and $\mathrm{Iso}(D_\infty,\Mu)$ has 21 elements,  
corresponding to the pairs of elements {\rm  \{idt., idt.\}, \{idt., m-ref.\}, \{idt., 2-sym.\}, \{idt., c-ref.\}, \{idt., v-ref.\}, \{idt., 2-rot.\}, 
 \{m-ref., m-ref.\}, \{2-sym., 2-sym.\},  \{c-ref., c-ref.\},  \{v-ref., v-ref.\}, \{2-rot., 2-rot.\},  
 \{m-ref.,  2-sym.\}, \{2-sym., c-ref.\},  \{m-ref.,  c-ref.\}, \{m-ref.,  v-ref.\}, \{m-ref.,  2-rot.\}, 
\{2-sym., v-ref.\}, \{2-sym., 2-rot.$'$\}, \{v-ref., 2-rot.$'$\}, \{c-ref., v-ref.\}, \{c-ref., 2-rot.\}} of $\mathrm{Sym}(\Mu)$  by Theorem 26.
The corresponding co-Seifert fibrations are described in Table 15 by Theorems 27 and 28.  
Only the three pairs {\rm \{v-ref., v-ref.\}, \{2-rot., 2-rot.\}, \{v-ref., 2-rot.$'$\}} fall into the case $E_1\cap E_2 \neq \{0\}$ of Theorem 26. 
\end{theorem}

\noindent{\bf Remark 9.}  The Seifert fibrations $({\ast}{:}{\times})$ and  $(2{\bar\ast}2{:}2)$, with IT numbers 9 and 15, respectively, 
in Table 1 of \cite{C-T} were replaced by two different affinely equivalent Seifert fibrations in Table 1 of \cite{R-T}. 
The Seifert fibrations $({\ast}{:}{\times})$ and  $(2{\bar\ast}2{:}2)$ have orthogonally dual co-Seifert fibers of type $\times\times$. 
The structure group actions for these fibrations are (c-ref., 2-rot.) and  (v-ref., ref.),  (c-ref., 2-rot.), respectively. 
The action (c-ref., 2-rot.) corresponds to the pair \{c-ref., c-ref.\}, 
and the action (v-ref., ref.),  (c-ref., 2-rot.) corresponds to the pair \{v-ref., 2-rot.\}. 
Hence these co-Seifert fibrations are affinely equivalent to the co-Seifert fibrations described in no.\ 9 and 15 of Table 15, respectively, by Theorems 23 and 26, respectively. 

\medskip

\begin{table}  
\begin{tabular}{rlllll}
no. & fibers & grp. & quotients &  structure group action & classifying pair \\
\hline 
     7& $(\times\times, \mathrm{O})$ & $C_1$ & $(\times\times, \mathrm{O})$ & (idt., idt.) &  \{idt., idt.\}  \\
    9 & $(\times\times, \mathrm{O})$ & $C_2$ & $(\times\times, \mathrm{O})$ & (2-sym., 2-rot.) &  \{2-sym., 2-sym.\} \\
  13 & $(\times\times, \mathrm{O})$& $C_2$ & $(22\ast, \mathrm{I})$ & (v-ref., ref.) &  \{v-ref., v-ref.\} \\
  14 & $(\times\times, \mathrm{O})$ & $C_2$ & $(22\times, \mathrm{I})$ & (2-rot., ref.) & \{2-rot., 2-rot.\} \\
  15 & $(\times\times, \mathrm{O})$ & $D_2$ & $(22\ast, \mathrm{I})$ &  (v-ref., ref.), (2-sym., 2-rot.) & \{v-ref., 2-rot.$'$\} \\
  26 & $(\times\times, \mathrm{I})$  & $C_1$ & $(\times\times, \mathrm{I})$ & (idt., idt.) &  \{idt., idt.\}  \\
  27 & $(\times\times, \mathrm{O})$ & $C_2$ & $(\ast\ast, \mathrm{I})$ & (m-ref., ref.) &  \{m-ref., m-ref.\} \\
  29 & $(\times\times, \mathrm{O})$ & $C_2$ & $(\times\times, \mathrm{I})$ & (2-sym., ref.) &  \{2-sym., 2-sym.\} \\
  29 & $(\times\times, \mathrm{O})$ & $C_2$ & $(22\ast, \mathrm{O})$ & (v-ref., 2-rot.) &  \{v-ref., v-ref.\} \\
  30 & $(\times\times, \mathrm{O})$ & $C_2$ & $(\ast\times, \mathrm{I})$ & (c-ref., ref.) &  \{c-ref., c-ref.\} \\
  33 & $(\times\times, \mathrm{O})$  & $C_2$ & $(22\times, \mathrm{O})$ & (2-rot., 2-rot.) & \{2-rot., 2-rot.\} \\
  36 & $(\times\times, \mathrm{I})$ & $C_2$ & $(\times\times, \mathrm{I})$ & (2-sym., ref.) & \{idt., 2-sym.\} \\
  37 & $(\times\times, \mathrm{O})$   & $D_2$ & $(\ast\ast, \mathrm{I})$ & (m-ref., ref.), (2-sym., 2-rot.) & \{m-ref.,  c-ref.\} \\
  39 & $(\times\times, \mathrm{I})$ & $C_2$ & $(\ast\ast, \mathrm{I})$ & (m-ref., ref.) & \{idt., m-ref.\} \\
  41 & $(\times\times, \mathrm{O})$ & $D_2$ & $(\ast\ast, \mathrm{I})$ & (c-ref., ref.), (m-ref., 2-rot.) & \{c-ref., 2-sym.\} \\
  45 & $(\times\times, \mathrm{O})$   & $D_2$ & $(\ast\ast, \mathrm{I})$ & (m-ref., ref.), (c-ref., 2-rot.) &  \{m-ref.,  2-sym.\} \\
  46 & $(\times\times, \mathrm{I})$ & $C_2$ & $(\ast\times, \mathrm{I})$ & (c-ref., ref.) & \{idt., c-ref.\} \\
  52 & $(\times\times, \mathrm{O})$ & $D_2$ & $(2{\ast}22, \mathrm{I})$ & (c-ref., ref.), (2-rot., 2-rot.) & \{c-ref., v-ref.\} \\
  54 & $(\times\times, \mathrm{O})$  & $D_2$ & $(\ast 2222, \mathrm{I})$ & (m-ref., ref.), (v-ref.$'$, 2-rot.) & \{m-ref.,  v-ref.\} \\
  56 & $(\times\times, \mathrm{O})$ & $D_2$ & $(22\ast, \mathrm{I})$ & (m-ref., ref.),  (2-rot.$'$, 2-rot.) &  \{m-ref., 2-rot.\} \\
  57 & $(\times\times, \mathrm{I})$   & $C_2$ & $(22\ast, \mathrm{I})$ & (v-ref., ref.) & \{idt., v-ref.\} \\
  60 & $(\times\times, \mathrm{O})$ & $D_2$ & $(22\ast, \mathrm{I})$ & (2-sym., ref.), (2-rot.$'$, 2-rot.) & \{2-sym., v-ref.\} \\
  60 & $(\times\times, \mathrm{O})$ & $D_2$ & $(2{\ast}22, \mathrm{I})$ & (c-ref., ref.), (v-ref., 2-rot.) & \{c-ref., 2-rot.\} \\
  61 & $(\times\times, \mathrm{O})$   & $D_2$ & $(22\ast, \mathrm{I})$ & (2-sym., ref.), (v-ref., 2-rot.) & \{2-sym., 2-rot.$'$\} \\
  62 & $(\times\times, \mathrm{I})$ & $C_2$ & $(22\times, \mathrm{I})$ & (2-rot., ref.) & \{idt., 2-rot.\}

\end{tabular}

\medskip
\caption{The classification of the co-Seifert fibrations of 3-space groups 
whose co-Seifert fiber is of type $\times\times$ with IT number 4}
\end{table}

(3) Let $\Mu = \langle e_1+I, e_2+I, A\rangle$ where $A = \mathrm{diag}(1,-1)$. 
Then $\Mu$ is a 2-space group with IT number 3. 
The Conway notation for $\Mu$ is $\ast\ast$ and the IT notation if $pm$. 
The flat orbifold $E^2/\Mu$ is an annulus. 
We have that $Z(\Mu) = \langle e_1+I \rangle$. 

\begin{lemma}  
If $\Mu = \langle e_1+I, e_2 + I, A \rangle$, with $A = \mathrm{diag}(1,-1)$, then 
$$N_A(\Mu) = \{b+B: 2b_2 \in \integers \ \, \hbox{and}\ \, B \in \langle -I, A\rangle\}.$$
\end{lemma}
\begin{proof}
Observe that $b+B \in N_A(\Mu)$ if and only if $B\in N_A(\langle e_1+I, e_2+I\rangle)$ 
and $(b+B)A(b+B)^{-1}\in \Mu$.  
Now $B \in N_A(\langle e_1+I, e_2+I\rangle)$ if and only if $B \in \mathrm{GL}(2,\integers)$.
As $(b+B)A(b+B)^{-1} = b-BAB^{-1}b+BAB^{-1}$, we have that $(b+B)A(b+B)^{-1}\in \Mu$ 
if and only if $BAB^{-1} = A$ and $b-Ab\in \integers^2$. 
Hence $b+B \in N_A(\Mu)$ if and only if $B \in \langle -I, A\rangle$ and $2b_2 \in \integers$. 
\end{proof}

A fundamental polygon for the action of $\Mu$ on $E^2$ is the rectangle with vertices 
$v_1 = (0,0), v_2 = (1,0), v_3 = (1,1/2)$ and $v_4= (0,1/2)$. 
The reflection $A$ fixes the side $[v_1, v_2]$ pointwise, the reflection $e_2+A$ fixes the side $[v_4, v_3]$ pointwise, 
and the horizontal translation $e_1+I$ maps the side $[v_1, v_4]$ to the side $[v_2, v_3]$. 
The central geodesic of the annulus $E^2/\Mu$, which we call the {\it central circle}, represented by the horizontal 
line segment $[(0,1/4),(1,1/4)]$ is invariant under any symmetry of the annulus. 
The symmetry group of $E^2/\Mu$ is the direct product of the subgroup of order 2, generated by 
the reflection in the central circle, and the subgroup consisting of the horizontal rotations and the reflections 
in a pair of antipodal vertical line segments.  The latter subgroup restricts to the symmetry group of the central circle.  

There are five conjugacy classes of symmetries of order 2, 
the class of the reflection c-ref.\  = $(e_2/2+I)_\star$ in the central circle, 
the class of the horizontal halfturn 2-sym.\ = $(e_1/2+I)_\star$, 
the class of the halfturn glide-reflection g-ref.\ = $(e_1/2+e_2/2+I)_\star$ in the central circle, 
the class of a reflection v-ref.\ = $(-I)_\star$ in a pair of antipodal vertical line segments of the annulus, 
and the class of the halfturn 2-rot.\  = $(e_2/2-I)_\star$ around a pair of antipodal points on the central circle. 
Define v-ref.$' = (e_1/2-I)_\star$ and  2-rot.$' = (e_1/2+e_2/2-I)_\star$.  

There are five conjugacy classes of dihedral symmetry groups of order 4, 
the classes of the groups $\{$idt., c-ref., 2-sym., g-ref.$\}$, $\{$idt.,  c-ref., v-ref., 2-rot.$\}$, $\{$idt., 2-sym., v-ref., v-ref.$'$$\}$, $\{$idt.,  2-sym., 2-rot., 2-rot.$'$$\}$, and 
$\{$idt.,  g-ref. v-ref., 2-rot.$'$$\}$. 

\begin{lemma} 
If $\Mu = \langle e_1+I, e_2 + I, A \rangle$, with $A = \mathrm{diag}(1,-1)$, 
then the Lie group $\mathrm{Sym}(\Mu)$ is isomorphic to $C_2\times\mathrm{O}(2)$, 
and $\mathrm{Sym}(\Mu) = \mathrm{Aff}(\Mu)$, 
and $\Omega:  \mathrm{Aff}(\Mu) \to \mathrm{Out}(\Mu)$ maps 
the subgroup {\rm \{idt., c-ref., v-ref., 2-rot.\}}  isomorphically onto $\mathrm{Out}(\Mu)$.  
\end{lemma}
\begin{proof} Let $\mathrm{O}$ be the central circle of the annulus, 
and let $\Lambda$ be the group of all symmetries of the annulus that leave each horizontal geodesic invariant. 
Restriction induces an isomorphism from $\Lambda$ to $\mathrm{Isom}(\mathrm{O})$ by Lemmas 5 and 40. 
We have that  $\mathrm{Sym}(\Mu) = \langle$c-ref.$\rangle \times \Lambda$, 
since every symmetry of the annulus leaves $\mathrm{O}$ invariant,   
and c-ref.\ commutes with every symmetry in $\Lambda$. 
Therefore $\mathrm{Sym}(\Mu)$ is isomorphic to $C_2\times \mathrm{O}(2)$. 
We have that $\mathrm{Sym}(\Mu) = \mathrm{Aff}(\Mu)$ by Lemmas 5, 11,  and 40. 
The epimorphism $\Omega:  \mathrm{Aff}(\Mu) \to \mathrm{Out}(\Mu)$ maps  {\rm \{idt., c-ref., v-ref., 2-rot.\}} isomorphically onto 
$\mathrm{Out}(\Mu)$ by Theorem 13, since v-ref.\ acts as a reflection of $\mathrm{O}$. 
\end{proof}

\begin{theorem}  
Let $\Mu = \langle e_1+I, e_2 + I, A \rangle$ with $A = \mathrm{diag}(1,-1)$. 
Then $\mathrm{Iso}(C_\infty,\Mu)$ has four elements,   
corresponding to the pairs of elements {\rm \{idt., idt.\}, \{c-ref,., c-ref.\}, \{v-ref., v-ref.\}, \{2-rot., 2-rot.\}} of $\mathrm{Sym}(\Mu)$ by Theorem 23, 
and $\mathrm{Iso}(D_\infty,\Mu)$ has 21 elements,  
corresponding to the pairs of elements 
{\rm  \{idt., idt.\}, \{idt., c-ref.\}, \{idt., 2-sym.\}, \{idt., g-ref.\}, \{idt., v-ref.\}, \{idt., 2-rot.\}, 
 \{c-ref., c-ref.\}, \{2-sym., 2-sym.\},  \{g-ref., g-ref.\},  \{v-ref., v-ref.\}, \{2-rot., 2-rot.\},  
 \{c-ref.,  2-sym.\}, \{2-sym., g-ref.\},  \{c-ref.,  g-ref.\}, \{c-ref., v-ref.\}, \{c-ref.,  2-rot.\}, 
\{v-ref.,  2-rot.\}, \{2-sym., v-ref.\}, \{2-sym., 2-rot.\}, \{g-ref., v-ref.\}, \{g-ref., 2-rot.$'$\}} of $\mathrm{Sym}(\Mu)$  by Theorem 26.
The corresponding co-Seifert fibrations are described in Table 15 by Theorems 27 and 28.  
Only the three pairs {\rm \{v-ref., v-ref.\}, \{2-rot., 2-rot.\}, \{v-ref., 2-rot.\}} fall into the case $E_1\cap E_2 \neq \{0\}$ of Theorem 26. 
\end{theorem}

\begin{table}  
\begin{tabular}{rlllll}
no. & fibers & grp. & quotients &  structure group action & classifying pair \\
\hline 
    6 & $(\ast\ast, \mathrm{O})$ & $C_1$ & $(\ast\ast, \mathrm{O})$ & (idt., idt.) &  \{idt., idt.\} \\
    8 & $(\ast\ast, \mathrm{O})$ & $C_2$ & $(\ast\ast, \mathrm{O})$ & (c-ref., 2-rot.) & \{c-ref., c-ref.\}  \\
  10 & $(\ast\ast, \mathrm{O})$ & $C_2$ & $(\ast 2222, \mathrm{I})$ & (v-ref., ref.) &  \{v-ref., v-ref.\} \\
  11 & $(\ast\ast, \mathrm{O})$& $C_2$ & $(22\ast, \mathrm{I})$ & (2-rot., ref.) & \{2-rot., 2-rot.\} \\
  12 & $(\ast\ast, \mathrm{O})$ & $D_2$ & $(\ast 2222, \mathrm{I})$ & (v-ref., ref.), (c-ref., 2-rot.) & \{v-ref.,  2-rot.\} \\
  25 & $(\ast\ast, \mathrm{I})$ & $C_1$ & $(\ast\ast, \mathrm{I})$ & (idt., idt.) &  \{idt., idt.\} \\
  26 & $(\ast\ast, \mathrm{O})$ & $C_2$ & $(\ast\ast, \mathrm{I})$ & (2-sym., ref.) & \{2-sym., 2-sym.\} \\
  26 & $(\ast\ast, \mathrm{O})$  & $C_2$ & $(\ast 2222, \mathrm{O})$ & (v-ref., 2-rot.) &  \{v-ref., v-ref.\} \\
  28 & $(\ast\ast, \mathrm{O})$ & $C_2$ & $(\ast\ast, \mathrm{I})$ & (c-ref., ref.) & \{c-ref., c-ref.\}  \\
  31 & $(\ast\ast, \mathrm{O})$ & $C_2$ & $(\ast\times, \mathrm{I})$ & (g-ref., ref.) & \{g-ref., g-ref.\} \\
  31 & $(\ast\ast, \mathrm{O})$ & $C_2$ & $(22\ast, \mathrm{O})$ & (2-rot., 2-rot.) & \{2-rot., 2-rot.\} \\
  35 & $(\ast\ast, \mathrm{I})$ & $C_2$ & $(\ast\ast, \mathrm{I})$ & (c-ref., ref.) & \{idt., c-ref.\} \\
  36 & $(\ast\ast, \mathrm{O})$  & $D_2$ & $(\ast\ast, \mathrm{I})$ & (2-sym., ref.), (c-ref., 2-rot.) & \{2-sym., g-ref.\} \\
  38 & $(\ast\ast, \mathrm{I})$ & $C_2$ & $(\ast\ast, \mathrm{I})$ & (2-sym., ref.) &  \{idt., 2-sym.\} \\
  40 & $(\ast\ast, \mathrm{O})$   & $D_2$ & $(\ast\ast, \mathrm{I})$ & (c-ref., ref.), (2-sym., 2-rot.) &  \{c-ref.,  g-ref.\} \\
  44 & $(\ast\ast, \mathrm{I})$ & $C_2$ & $(\ast\times, \mathrm{I})$ & (g-ref., ref.) &  \{idt., g-ref.\} \\
  46 & $(\ast\ast, \mathrm{O})$ & $D_2$ & $(\ast\ast, \mathrm{I})$ & (c-ref., ref.), (g-ref., 2-rot.) & \{c-ref., 2-sym.\} \\
  51 & $(\ast\ast, \mathrm{I})$   & $C_2$ & $(\ast 2222, \mathrm{I})$ & (v-ref., ref.) &  \{idt., v-ref.\} \\
  53 & $(\ast\ast, \mathrm{O})$ & $D_2$ & $(\ast 2222, \mathrm{I})$ & (v-ref., ref.), (2-rot., 2-rot.) & \{v-ref., c-ref.\}\\
  55 & $(\ast\ast, \mathrm{O})$ & $D_2$ & $(\ast 2222, \mathrm{I})$ & (2-sym., ref.), (v-ref.$'$, 2-rot.) & \{2-sym., v-ref.\} \\
  57 & $(\ast\ast, \mathrm{O})$  & $D_2$ & $(\ast 2222, \mathrm{I})$ & (c-ref., ref.), (v-ref., 2-rot.) & \{c-ref.,  2-rot.\} \\
  58 & $(\ast\ast, \mathrm{O})$ & $D_2$ & $(2{\ast}22, \mathrm{I})$ & (g-ref., ref.), (2-rot.$'$, 2-rot.) & \{g-ref., v-ref.\} \\
  59 & $(\ast\ast, \mathrm{I})$   & $C_2$ & $(22\ast, \mathrm{I})$ & (2-rot., ref.) &  \{idt., 2-rot.\} \\
  62 & $(\ast\ast, \mathrm{O})$ & $D_2$ & $(22\ast, \mathrm{I})$ & (2-sym., ref.), (2-rot.$'$, 2-rot.) & \{2-sym., 2-rot.\} \\
  62 & $(\ast\ast, \mathrm{O})$ & $D_2$ & $(2{\ast}22, \mathrm{I})$ & (g-ref., ref.), (v-ref., 2-rot.) &  \{g-ref., 2-rot.$'$\}

\end{tabular}

\medskip
\caption{The classification of the co-Seifert fibrations of 3-space groups 
whose co-Seifert fiber is of type $\ast\ast$ with IT number 3}
\end{table}

(2) Let $\Mu = \langle e_1+I, e_2+I, -I\rangle$ where $e_1, e_2$ are the standard basis vectors of $E^2$. 
Then $\Mu$ is a 2-space group with IT number 2. 
The Conway notation for $\Mu$ is $2222$ and the IT notation is $p2$. 
The flat $E^2/\Mu$ is a {\it square pillow} $\Box$. 
A flat 2-orbifold affinely equivalent to $\Box$ is called a {\it pillow}. 
A pillow is orientable and has four $180^\circ$ cone points. 

\begin{lemma}  
If $\Mu = \langle e_1+I, e_2+I, -I\rangle$, then 
$\Omega: \mathrm{Aff}(\Mu) \to \mathrm{Out}(\Mu)$ is an isomorphism. 
\end{lemma}
\begin{proof}
We have that $Z(\Mu) =\{I\}$ by Lemma 9. 
Hence $\Omega: \mathrm{Aff}(\Mu) \to \mathrm{Out}(\Mu)$ is an isomorphism by Theorems 11 and 13. 
\end{proof}

\begin{lemma} 
If $\Mu = \langle e_1+I, e_2+I, -I\rangle$, then 
$$N_A(\Mu) = \left\{\frac{m}{2}e_1+ \frac{n}{2}e_2  + A :  m, n \in \integers\ \hbox{and}\ A \in \mathrm{GL}(2,\integers)\right\}.$$
\end{lemma}
\begin{proof}
Observe that $a+ A \in N_A(\Mu)$ if and only if $A \in N_A(\langle e_1+I, e_2+I\rangle)$ and $(a+I)(-I)(a+I)^{-1} \in \Mu$. 
Now $A \in N_A(\langle e_1+I, e_2+I\rangle)$ if and only if $A \in  \mathrm{GL}(2,\integers)$. 
As $(a+I)(-I)(a+I)^{-1} = 2a - I$, we have that 
$(a+I)(-I)(a+I)^{-1} \in \Mu$ if and only if  $a = \frac{m}{2}e_1+ \frac{n}{2}e_2$ for some $m, n \in \integers$. 
\end{proof}

\begin{lemma}  
Let $\Mu = \langle e_1+I, e_2+I, -I\rangle$. 
The group $\mathrm{Aff}(\Mu)$ has a normal dihedral subgroup $\Kappa$ of order 4 generated by 
$(e_1/2+ I)_\star$ and $(e_2/2 +I)_\star$. 
The map $\eta: \mathrm{Aff}(\Mu) \to \mathrm{PGL}(2,\integers)$, defined by 
$\eta((a+A)_\star) = \pm A$ for each $A \in \mathrm{GL}(2,\integers)$, 
is an epimorphism with kernel $\Kappa$. 
The map $\sigma:  \mathrm{PGL}(2,\integers) \to \mathrm{Aff}(\Mu)$, defined by $\sigma(\pm A) = A_\star$, 
is a monomorphism,  and $\sigma$ is a right inverse of $\eta$. 
\end{lemma}
\begin{proof}
We have that $Z(\Mu)= \{I\}$ by Lemma 9,  and $\Omega: \mathrm{Aff}(\Mu) \to \mathrm{Out}(\Mu)$ is 
an isomorphism by Theorems 11 and 13.  
Hence $\Kappa = \Omega^{-1}(\mathrm{Out}_E^1)$ is a normal subgroup of $\mathrm{Aff}(\Mu)$ by Lemma 13.
Let $\Tau$ be the translation subgroup of $N_A(\Mu)$.  Then $\Kappa = \Phi(\Tau\Mu/\Mu)$ by Lemmas 11 and 12. 
Now $\Tau\Mu/\Mu \cong \Tau/\Tau\cap \Mu$ is a dihedral group of order 4 generated by 
$(e_1/2+ I)\Mu$ and $(e_2/2 +I)\Mu$ by Lemma 43. 
Hence $\Kappa$ is a dihedral group generated by $(e_1/2+ I)_\star$ and $(e_2/2 +I)_\star$. 

The point group $\Pi_A$ of $N_A(\Mu)$ is $\mathrm{GL}(2,\integers)$ by Lemma 43. 
Hence the map $\eta: \mathrm{Aff}(\Mu) \to \mathrm{PGL}(2,\integers)$, defined by 
$\eta((a+A)_\star) = \pm A$ for each $A \in \mathrm{GL}(2,\integers)\}$, 
is an epimorphism with kernel $\Kappa$ by Lemma 13. 
The map $\sigma:  \mathrm{PGL}(2,\integers) \to \mathrm{Aff}(\Mu)$ 
is a well defined homomorphism by Lemma 11. 
The map $\sigma$ is a monomorphism, since $\sigma$ is a right inverse of $\eta$.  
\end{proof}

\begin{lemma} 
Let 
$$A = \left(\begin{array}{rr} 0 & -1 \\ 1 & 0 \end{array}\right),\quad  B = \left(\begin{array}{rr} 0 & -1 \\ 1 & 1 \end{array}\right), 
\quad C = \left(\begin{array}{rr} 0 & 1 \\ 1 & 0 \end{array}\right). $$
The group $\mathrm{PGL}(2,\integers)$ is the free product of the dihedral subgroup $\langle \pm A,  \pm C\rangle$ of order 4 
and the dihedral subgroup $\langle\pm B,  \pm C \rangle$ of order 6 amalgamated 
along the subgroup $\langle \pm C\rangle$ of order 2. 
Every finite subgroup of  $\mathrm{PGL}(2,\integers)$ is conjugate to a subgroup 
of either $\langle \pm A,  \pm C\rangle$ or $\langle\pm B,  \pm C \rangle$. 
\end{lemma}
\begin{proof}
We have that $A^2 = -I, C^2 = I$, and $CAC^{-1} = -A$. 
Hence $\langle \pm A, \pm C\rangle$ is a dihedral group of order 4. 
We have that $B^3 = -I$ and $C B C^{-1} = B^{-1}.$
Hence $\langle \pm B, \pm C\rangle$ is a dihedral group of order 6. 
It follows from the group presentation for $\mathrm{PGL}(2,\integers)$ 
given in \S 7.2 of \cite{C-M} that $\mathrm{PGL}(2,\integers)$ is the free 
product of the subgroups $\langle \pm A, \pm C\rangle$ and $\langle \pm B, \pm C\rangle$ 
amalgamated along the subgroup $\langle \pm C \rangle$. 
By Theorem 8 of \S 4.3 of \cite{Serre}, every finite subgroup of $\mathrm{PGL}(2,\integers)$ 
is conjugate to a subgroup of either $\langle \pm A, \pm C\rangle$ or $\langle \pm B, \pm C\rangle$. 
\end{proof}

\begin{lemma} 
Let $\Mu = \langle e_1+I, e_2+I, -I\rangle$, let $\Kappa = \langle (e_1/2+ I)_\star, (e_2/2 +I)_\star \rangle$, and  
let $A, B, C$ be defined as in the previous lemma. 
The group $\mathrm{Aff}(\Mu)$  is the free product of the subgroup $\langle \Kappa, A_\star, C_\star\rangle$ of order 16
and the subgroup $\langle \Kappa, B_\star, C_\star\rangle$ of order 24 amalgamated 
along the subgroup $\langle \Kappa, C_\star\rangle$ of order 8. 
Every finite subgroup of $\mathrm{Aff}(\Mu)$ is conjugate to a subgroup of either
$\langle \Kappa, A_\star, C_\star\rangle$ or $\langle \Kappa, B_\star, C_\star\rangle$. 
\end{lemma}
\begin{proof} This follows easily from Lemmas 44 and 45. 
\end{proof}

Let $\Delta = \langle e_1+I, e_1/2 + \sqrt{3}e_2/2+I, -I\rangle$. 
Then $\Delta$ is a 2-space group, and the flat orbifold $E^2/\Delta$ is a {\it tetrahedral pillow} $\triangle$. 

\begin{lemma} 
Let $\Mu = \langle e_1+I, e_2+I, -I\rangle$, and let $\Delta = \langle e_1+I, e_1/2 + \sqrt{3}e_2/2+I, -I\rangle$. 
Let $\Kappa = \langle (e_1/2+ I)_\star, (e_2/2 +I)_\star \rangle$, and  
let 
$$A = \left(\begin{array}{rr} 0 & -1 \\ 1 & 0 \end{array}\right),\ \  B = \left(\begin{array}{rr} 0 & -1 \\ 1 & 1 \end{array}\right), 
\ \ C = \left(\begin{array}{rr} 0 & 1 \\ 1 & 0 \end{array}\right), \ \ 
D  = \left(\begin{array}{rr} 1 & 1/2 \\ 0 & \sqrt{3}/2 \end{array}\right). $$
Then $\mathrm{Sym}(\Mu) = \langle \Kappa, A_\star, C_\star\rangle$,  and $D\Mu D^{-1}=\Delta$, 
and $D_\sharp: \mathrm{Aff}(\Mu) \to \mathrm{Aff}(\Delta)$ is an isomorphism, 
and $D_\sharp^{-1}(\mathrm{Sym}(\Delta)) = \langle \Kappa, B_\star, C_\star\rangle$. 
\end{lemma}
\begin{proof}
We have that $Z(\Mu) = \{I\}$ by Lemma 9, and so $\mathrm{Sym}(\Mu)$ is finite by Corollary 5. 
Now $\mathrm{Sym}(\Mu)$ is a finite subgroup of $\mathrm{Aff}(\Mu)$ that contains $\langle \Kappa, A_\star, C_\star\rangle$. 
Hence $\mathrm{Sym}(\Mu)= \langle \Kappa, A_\star, C_\star\rangle$, since $\langle \Kappa, A_\star, C_\star\rangle$ 
is a maximal finite subgroup of $\mathrm{Aff}(\Mu)$ by Lemma 46. 

Clearly $D\Mu D^{-1}=\Delta$, and so $D_\sharp: \mathrm{Aff}(\Mu) \to \mathrm{Aff}(\Delta)$ is an isomorphism. 
Now $D_\sharp(\langle \Kappa, B_\star, C_\star\rangle)$ is a subgroup of $\mathrm{Sym}(\Delta)$, 
since $DBD^{-1}, DCD^{-1} \in \mathrm{O}(2)$.  
Hence $D_\sharp^{-1}(\mathrm{Sym}(\Delta))$ is a finite subgroup of $\mathrm{Aff}(\Mu)$ 
that contains $\langle \Kappa, B_\star, C_\star\rangle$. 
Therefore  $D_\sharp^{-1}(\mathrm{Sym}(\Delta)) = \langle \Kappa, B_\star, C_\star\rangle$, 
since $\langle \Kappa, B_\star, C_\star\rangle$ is a maximal finite subgroup of $\mathrm{Aff}(\Mu)$ by Lemma 46. 
\end{proof}

A square pillow is formed by identifying the boundaries of two congruent squares.  
The symmetry group of a square pillow is the direct product of the subgroup of order 2 generated 
by the central reflection between the two squares, 
and the subgroup of order 8 corresponding to the symmetry group of the two squares. 
A fundamental domain for the square pillow $\Box = E^2/\Mu$ is the rectangle with vertices $(0,0), (1/2,0), (1/2,1), (0, 1)$. 
This rectangle is subdivided into two congruent squares that correspond to the two sides of the pillow $\Box$. 

There are seven conjugacy classes of symmetries of order 2 of $\Box$, 
the class of the {\it central halfturn} c-rot.\ = $(e_1/2+e_2/2+I)_\star$ about the centers of the squares, 
the class of the {\it midline halfturn} m-rot.\ = $(e_1/2+ I)_\star$ about the midpoints of opposite sides of the squares,  
the class of the {\it central reflection} c-ref.\ = $(AC)_\star$ between the two squares, 
the class of the {\it midline reflection} m-ref.\ = $(e_1/2+AC)_\star$,  
the class of the {\it antipodal map} 2-sym.\ = $(e_1/2+e_2/2 + AC)_\star$, 
the class of the {\it diagonal reflection} d-ref.\ = $C_\star$, 
and the class of the {\it diagonal halfturn} d-rot.\ = $A_\star$. 

There are two conjugacy classes of symmetries of order 4 of $\Box$, 
the class of the $90^\circ$ rotation 4-rot.\ $= (e_1/2+ A)_\star$ about the centers of the squares, 
and the class of 4-sym.\ = $(e_1/2+ C)_\star$. 

There are nine conjugacy classes of dihedral symmetry groups of order 4 of $\Box$, 
the classes of the groups $\Kappa = \{$idt., c-rot., m-rot., m-rot.$'$$\}$,  
$\{$idt.,  c-rot., d-ref., d-ref.$'$$\}$, 
$\{$idt., c-ref., m-ref., m-rot.$\}$, 
$\{$idt.,  2-sym., m-ref., m-rot.$'$$\}$, 
$\{$idt.,  c-rot. m-ref., m-ref.$'$$\}$, 
$\{$idt., c-ref., c-rot., 2-sym.$\}$, 
$\{$idt.,  c-rot., d-rot., d-rot.$'$$\}$, 
$\{$idt., c-ref., d-ref., d-rot.$\}$, and 
$\{$idt.,  2-sym., d-ref., d-rot.$'$$\}$. 

There are four conjugacy classes of dihedral symmetry groups of order 8 of $\Box$, 
the classes of the groups 
$\langle$m-rot., d-rot.$\rangle$, 
$\langle$m-ref., d-ref.$\rangle$, 
$\langle$m-rot., d-ref.$\rangle$, and 
$\langle$m-ref., d-rot.$\rangle$.  
Moreover 4-rot.\ = (m-rot.)(d-rot.) = (m-ref.)(d-ref.) and 4-sym.\ = (m-rot.)(d-ref.) = (m-ref.)(d-rot.). 

A tetrahedral pillow is realized by the boundary of a regular tetrahedron. 
The symmetry group of a tetrahedral pillow corresponds to the symmetric group on its four cone points (vertices). 
A fundamental domain for the tetrahedral pillow $\triangle = E^2/\Delta$ is 
the equilateral triangle with vertices $(0,0), (1,0), (1/2, \sqrt{3}/2)$. 
This triangle is subdivided into four congruent equilateral triangles that correspond 
to the four faces of a regular tetrahedron. 

There are two conjugacy classes of symmetries of order 2 of $\triangle$, 
the class of the {\it halfturn} 2-rot.\ = $(e_1/2+I)_\star$, 
with axis joining the midpoints of a pair of opposite edges of the tetrahedron, 
corresponding to the product of two disjoint transpositions of vertices, and 
the class of the {\it reflection} ref.\ = $(DCD^{-1})_\star$ corresponding to a transposition of vertices.

There is one conjugacy class of symmetries of order 3 of $\triangle$, 
the class of the $120^\circ$ rotation 3-rot.\ = $(DBD^{-1})_\star$ corresponding to a 3-cycle.   
There is one conjugacy class of elements of order 4 of $\triangle$, 
the class of 4-cyc.\ = $(e_1/2+ DCD^{-1})_\star$, corresponding to a 4-cycle.

There are two conjugacy classes of dihedral symmetry groups of order 4 of $\triangle$, 
the class of the group of halfturns, 
and the class of the group generated by two perpendicular reflections.

There is one conjugacy class of dihedral symmetry groups of order 6 of $\triangle$, 
the class of the stabilizer of a face of $\triangle$. 
There is one conjugacy class of dihedral symmetry groups of order 8 of $\triangle$, 
since all Sylow 2-subgroups of the symmetry group of $\triangle$ are conjugate.  

The square pillow $\Box$ is affinely equivalent to the tetrahedral pillow $\triangle$ by Lemma 47.  
The symmetries m-rot., d-ref.\ and 4-sym.\ of $\Box$ are conjugate by $D_\star$ to the symmetries 
2-rot., ref.\ and 4-cyc.\ of $\triangle$, respectively. 
The affinity 3-aff.\ $ = B_\star$ of $\Box$ is conjugate by $D_\star$ to the symmetry 3-rot.\ of $\triangle$. 
The group $\{$idt., c-rot, m-rot., m-rot.$'$$\}$ of symmetries of $\Box$ is conjugate by $D_\star$ to the group of halfturns 
of $\triangle$. 
The group $\{$idt., c-rot, d-ref., d-ref.$'$$\}$ of symmetries of $\Box$ is conjugate by $D_\star$ to a group of symmetries 
of $\triangle$ generated by two perpendicular reflections. 

Define the affinity 2-aff.\  of $\Box$ by 2-aff.\ = (d-ref.)(3-aff.). 
Define the reflection ref.$'$ of $\triangle$ by ref.$'$ = (ref.)(3-rot.). 
The group $\langle$d-ref., 2-aff.$\rangle$ of affinities of $\Box$ is conjugate by $D_\star$ 
to the group $\langle$ref., ref.$'\rangle$ of symmetries of $\triangle$ of order 6. 
The group $\langle$m-rot., d-ref.$\rangle$ of symmetries of $\Box$ is conjugate by $D_\star$ 
to the dihedral group of symmetries $\langle$2-rot., ref.$\rangle$ of $\triangle$ of order 8.

Every affinity of $\Box$ of finite order has order 1, 2, 3, or 4 by Lemma 46. 
There are six conjugacy classes of affinities of $\Box$ of order 2 represented 
by m-rot., c-ref., m-ref., 2-sym., d-ref., and d-rot. 
There is one conjugacy classes of affinities of $\Box$ of order 3 represented by 3-aff.  
There are two conjugacy classes of affinities of $\Box$ of order 4 represented by 4-rot.\ and 4-sym. 

\begin{theorem} 
Let $\Mu = \langle e_1+I, e_2+I, -I\rangle$. 
Then $\mathrm{Iso}(C_\infty, \Mu)$ has 10 elements corresponding to the pairs of elements 
{\rm \{idt., idt.\}, \{m-rot., m-rot.\}, \{c-ref., c-ref.\}, \{m-ref., m-ref.\}, \{2-sym., 2-sym.\}, \{d-ref., def.\}, \{d-rot., d-rot.\}, 
\{3-aff., 3-aff.$^{-1}$\}, \{4-rot., 4-rot.$^{-1}$\}, \{4-sym., 4-sym.$^{-1}$\}} of $\mathrm{Aff}(\Mu)$ by Theorem 23,  
and $\mathrm{Iso}(D_\infty, \Mu)$ has 40 elements corresponding to the pairs of elements 
{\rm \{idt., idt.\}, \{idt., m-rot\}, \{idt., c-ref.\}, \{idt., m-ref.\}, \{idt., 2-sym.\}, \{idt., d-ref.\}, \{idt., d-rot.\}, 
\{m-rot., m-rot\}, \{c-ref., c-ref.\}, \{m-ref., m-ref.\}, \{2-sym., 2-sym.\}, \{d-ref., d-ref.\}, \{d-rot., d-rot.\}
\{c-rot, m-rot.\}, \{c-rot., c-ref.\}, \{c-rot., m-ref.\}, \{c-rot., 2-sym.\}, \{c-rot., d-ref.\}, \{c-rot., d-rot.\}, 
\{m-rot., c-ref.\}, \{m-rot., m-ref.\},  \{m-rot.$'$, m-ref.\}, \{m-rot.$'$, 2-sym.\}, \{m-rot., d-ref.\}, \{m-rot., d-rot.\}, 
\{c-ref., m-ref.\}, \{c-ref., 2-sym.\}, \{c-ref., d-ref.\}, \{c-ref., d-rot.\}, 
\{m-ref., m-ref.$'$\}, \{m-ref, 2-sym.\}, \{m-ref, d-ref.\}, \{m-ref., d-rot.\}, 
\{2-sym., d-ref.\}, \{2-sym, d-rot.$'$\}, \{d-ref., d-ref.$'$\}, \{d-ref., d-rot.\}, \{d-ref., d-rot.$'$\}, \{d-rot., d-rot.$'$\}, \{d-ref., 2-aff.\}} 
of $\mathrm{Aff}(\Mu)$ by Theorem 26. 
The corresponding co-Seifert fibrations are described in Table 17 by Theorems 27 and 28.  
\end{theorem}

\begin{table}  
\begin{tabular}{rlllll}
no. & fibers & grp. & quotients & structure group action & classifying pair \\
\hline 
    3 & $(2222, \mathrm{O})$ & $C_1$ & $(2222, \mathrm{O})$ & (idt., idt.) &  \{idt., idt.\} \\
    5 & $(2222, \mathrm{O})$ & $C_2$ & $(2222, \mathrm{O})$ & (m-rot., 2-rot.) & \{m-rot., m-rot\} \\
  10 & $(2222, \mathrm{I})$ & $C_1$ & $(2222, \mathrm{I})$ & (idt., idt.)  &  \{idt., idt.\} \\
  12 & $(2222, \mathrm{I})$ & $C_2$ & $(2222, \mathrm{I})$ & (m-rot., ref.) & \{idt., m-rot\} \\
  13 & $(2222, \mathrm{O})$ & $ C_2$ & $(2222, \mathrm{I})$ & (m-rot., ref.) & \{m-rot., m-rot\} \\
  15 & $(2222, \mathrm{O})$ & $D_2$ & $(2222, \mathrm{I})$ & (c-rot., ref.), (m-rot.$'$, 2-rot.) & \{c-rot, m-rot.\} \\
  16 & $(2222, \mathrm{O})$ & $C_2$ & $(\ast 2222, \mathrm{I})$ & (c-ref., ref.) &  \{c-ref., c-ref.\} \\
  17 & $(2222, \mathrm{O})$  & $C_2$ & $(22\ast, \mathrm{I})$ & (m-ref., ref.) & \{m-ref., m-ref.\} \\
  18 & $(2222, \mathrm{O})$ & $C_2$ & $(22\times, \mathrm{I})$ & (2-sym., ref.) & \{2-sym., 2-sym.\} \\
  20 & $(2222, \mathrm{O})$ & $D_2$ & $(22\ast, \mathrm{I})$ & (2-sym., ref.), (m-rot.$'$, 2-rot.)  & \{2-sym., m-ref.\} \\
  21 & $(2222, \mathrm{O})$ & $C_2$ & $(2{\ast}22, \mathrm{I})$ & (d-ref., ref.) &  \{d-ref., d-ref.\} \\
  21 & $(2222, \mathrm{O})$ & $D_2$ & $(\ast 2222, \mathrm{I})$ & (c-ref., ref.), (m-rot., 2-rot.) & \{c-ref., m-ref.\} \\
  22 & $(2222, \mathrm{O})$ & $D_2$ & $(\ast 2222, \mathrm{I})$ & (d-ref., ref.), (c-rot., 2-rot.) &  \{d-ref., d-ref.$'$\}  \\
  23 & $(2222, \mathrm{O})$ & $D_2$ & $(2{\ast}22, \mathrm{I})$ & (c-ref., ref.), (c-rot., 2-rot.) & \{c-ref., 2-sym.\} \\
  24 & $(2222, \mathrm{O})$ & $D_2$ & $(2{\ast}22, \mathrm{I})$ & (m-ref., ref.), (c-rot., 2-rot.) & \{m-ref., m-ref.$'$\} \\
  27 & $(2222, \mathrm{O})$ & $C_2$ & $(\ast 2222, \mathrm{O})$ & (c-ref., 2-rot.) &  \{c-ref., c-ref.\} \\
  30 & $(2222, \mathrm{O})$ & $C_2$ & $(22\ast, \mathrm{O})$ & (m-ref., 2-rot.) & \{m-ref., m-ref.\} \\
  34 & $(2222, \mathrm{O})$ & $C_2$ & $(22\times, \mathrm{O})$ & (2-sym., 2-rot.) & \{2-sym., 2-sym.\} \\
  37 & $(2222, \mathrm{O})$ & $C_2$ & $(2{\ast}22, \mathrm{O})$ & (d-ref., 2-rot.) &  \{d-ref., d-ref.\} \\
  43 & $(2222, \mathrm{O})$ & $C_4$ & $(22\times, \mathrm{O})$ & (4-sym., 4-rot.) & \{4-sym., 4-sym.$^{-1}$\} \\
  48 & $(2222, \mathrm{O})$ & $D_2$ & $(2{\ast}22, \mathrm{I})$ & (c-rot., ref.), (2-sym., 2-rot.) &  \{c-rot., c-ref.\} \\
  49 & $(2222, \mathrm{I})$ & $C_2$ & $(\ast 2222, \mathrm{I})$ & (c-ref., ref.) & \{idt., c-ref.\} \\
  50 & $(2222, \mathrm{O})$ & $D_2$ & $(\ast 2222, \mathrm{I})$ & (c-ref., ref.), (m-ref., 2-rot.) & \{c-ref., m-rot.\} \\ 
  52 & $(2222, \mathrm{O})$ & $D_2$ & $(22\ast, \mathrm{I})$ & (m-ref., ref.), (2-sym., 2-rot.) &  \{m-ref., m-rot.$'$\} \\
  52 & $(2222, \mathrm{O})$ & $D_2$ & $(2{\ast}22, \mathrm{I})$ & (c-rot., ref.), (m-ref.$'$, 2-rot.) & \{c-rot., m-ref.\} \\
  53 & $(2222, \mathrm{I})$  & $C_2$ & $(22\ast, \mathrm{I})$ & (m-ref., ref.) &  \{idt., m-ref.\} \\
  54 & $(2222, \mathrm{O})$ & $D_2$ & $(\ast 2222, \mathrm{I})$ & (m-ref., ref.), (c-ref., 2-rot.) &  \{m-ref., m-rot.\} \\
  56 & $(2222, \mathrm{O})$ & $D_2$ & $(2{\ast}22, \mathrm{I})$ & (c-rot., ref.), (c-ref., 2-rot.) & \{c-rot., 2-sym.\} \\
  58 & $(2222, \mathrm{I})$ & $C_2$ & $(22\times, \mathrm{I})$ & (2-sym., ref.) & \{idt., 2-sym.\} \\
  60 & $(2222, \mathrm{O})$ & $D_2$ & $(22\ast, \mathrm{I})$ & (2-sym., ref.), (m-ref., 2-rot.) &  \{2-sym., m-rot.$'$\} \\
  66 & $(2222, \mathrm{I})$ & $C_2$ & $(2{\ast}22, \mathrm{I})$ & (d-ref., ref.) &  \{idt., d-ref.\} \\
  68 & $(2222, \mathrm{O})$ & $D_2$ & $(\ast 2222, \mathrm{I})$ & (c-rot., ref.), (d-ref.$'$, 2-rot.) & \{c-rot., d-ref.\} \\
  70 & $(2222, \mathrm{O})$ & $D_4$ & $(2{\ast}22, \mathrm{I})$ & (m-rot., ref.), (4-sym., 4-rot.) & \{m-rot., d-ref.\} \\
  77 & $(2222, \mathrm{O})$ & $C_2$ & $(442, \mathrm{O})$ & (d-rot., 2-rot.) & \{d-rot., d-rot.\} \\
  80 & $(2222, \mathrm{O})$ & $C_4$ & $(442, \mathrm{O})$ & (4-rot., 4-rot.) & \{4-rot., 4-rot.$^{-1}$\} \\
  81 & $(2222, \mathrm{O})$ & $C_2$ & $(442, \mathrm{I})$ & (d-rot., ref.) &  \{d-rot., d-rot.\} \\
  82 & $(2222, \mathrm{O})$ & $D_2$ & $(442, \mathrm{I})$ & (d-rot., ref.), (c-rot., 2-rot.) & \{d-rot., d-rot.$'$\} \\
  84 & $(2222, \mathrm{I})$ & $C_2$ & $(442, \mathrm{I})$ & (d-rot., ref.) & \{idt., d-rot.\} \\
  86 & $(2222, \mathrm{O})$ & $D_2$ & $(442, \mathrm{I})$ & (c-rot., ref.), (d-rot.$'$, 2-rot.) & \{c-rot., d-rot.\} \\
  88 & $(2222, \mathrm{O})$ & $D_4$ & $(442, \mathrm{I})$ & (m-rot., ref.), (4-rot., 4-rot.) & \{m-rot., d-rot.\} \\
  93 & $(2222, \mathrm{O})$ & $D_2$ & $(\ast 442, \mathrm{I})$ & (c-ref., ref.), (d-rot., 2-rot.) & \{c-ref., d-ref.\} \\
  94 & $(2222, \mathrm{O})$ & $D_2$ & $(4{\ast}2, \mathrm{I})$ & (2-sym., ref.), (d-rot.$'$, 2-rot.) & \{2-sym., d-ref.\} \\
  98 & $(2222, \mathrm{O})$ & $D_4$ & $(\ast 442, \mathrm{I})$ & (m-ref., ref.), (4-rot., 4-rot.) &  \{m-ref, d-ref.\} \\
112 & $(2222, \mathrm{O})$ & $D_2$ & $(\ast 442, \mathrm{I})$ & (c-ref., ref.), (d-ref., 2-rot.) & \{c-ref., d-rot.\} \\
114 & $(2222, \mathrm{O})$ & $D_2$ & $(4{\ast}2, \mathrm{I})$ & (2-sym., ref.), (d-ref., 2-rot.) & \{2-sym, d-rot.$'$\} \\
116 & $(2222, \mathrm{O})$ & $D_2$ & $(\ast 442, \mathrm{I})$ & (d-ref., ref.), (c-ref., 2-rot.) & \{d-ref., d-rot.\} \\
118 & $(2222, \mathrm{O})$ & $D_2$ & $(4{\ast}2, \mathrm{I})$ & (d-ref., ref.), (2-sym., 2-rot.) & \{d-ref., d-rot.$'$\} \\
122 & $(2222, \mathrm{O})$ & $D_4$ & $(4{\ast}2, \mathrm{I})$ & (m-ref., ref.), (4-sym., 4-rot.) &  \{m-ref., d-rot.\} \\
171 & $(2222, \mathrm{O})$ & $C_3$ & $(632, \mathrm{O})$ & (3-rot., 3-rot.) & \{3-aff., 3-aff.$^{-1}$\} \\
180 & $(2222, \mathrm{O})$ & $D_3$ & $(\ast632, \mathrm{I})$ & (ref., ref.), (3-rot., 3-rot.) & \{d-ref., 2-aff.\}
\end{tabular}

\medskip
\caption{The classification of the co-Seifert fibrations of 3-space groups 
whose co-Seifert fiber is of type $2222$ with IT number 2}
\end{table}

(1) Let $\Mu = \langle e_1+I, e_2+I\rangle$ where $e_1, e_2$ are the standard basis vectors of $E^2$. 
Then $\Mu$ is a 2-space group with IT number 1. The Conway notation for $\Mu$ is $\circ$ and the IT notation is $p1$. 
The flat orbifold $E^2/\Mu$ is a {\it square torus} $\Box$.  
A flat 2-orbifold affinely equivalent to $\Box$ is called a {\it torus}. 
A torus is orientable. 

\begin{lemma}  
If $\Mu = \langle e_1+I, e_2+I\rangle$, then 
$\mathrm{Out}(\Mu)= \mathrm{Aut}(\Mu) = \mathrm{GL}(2,\integers)$, and  
$$N_A(\Mu) = \left\{a  + A :  a \in E^2\ \hbox{and}\ A \in \mathrm{GL}(2,\integers)\right\}.$$
\end{lemma}
\begin{proof}
We have that $\mathrm{Out}(\Mu)= \mathrm{Aut}(\Mu)$, since $\Mu$ is abelian, 
and $\mathrm{Aut}(\Mu) = \mathrm{GL}(2,\integers)$, since every automorphism of $\Mu$ 
extends to a unique linear automorphism of $E^2$ corresponding to an element of $\mathrm{GL}(2,\integers)$. 

Observe that $a+ A \in N_A(\Mu)$ if and only if $A \in N_A(\langle e_1+I, e_2+I\rangle)$. 
Hence $a+ A \in N_A(\Mu)$ if and only if  $A \in \mathrm{GL}(2, \integers)$. 
\end{proof}

\begin{lemma}  
Let $\Mu = \langle e_1+I, e_2+I, -I\rangle$. 
The group $\mathrm{Aff}(\Mu)$ has a normal, infinite, abelian subgroup $\Kappa = \{(a+I)_\star: a \in E^2\}$. 
The map $\eta: \mathrm{Aff}(\Mu) \to \mathrm{GL}(2,\integers)$, defined by 
$\eta((a+A)_\star) = A$ for each $A \in \mathrm{GL}(2,\integers)$, 
is an epimorphism with kernel $\Kappa$. 
The map $\sigma:  \mathrm{GL}(2,\integers) \to \mathrm{Aff}(\Mu)$, defined by $\sigma(A) = A_\star$, 
is a monomorphism,  and $\sigma$ is a right inverse of $\eta$. 
\end{lemma}
\begin{proof}
The map $\Phi: N_A(\Mu) \to \mathrm{Aff}(\Mu)$, defined by $\Phi(a+A) = (a+A)_\star$,  
is an epimorphism with kernel $\Mu$ by Lemma 11. 
Let $\Tau$ be the translation subgroup of $N_A(\Mu)$. 
Then $\Phi$ maps the normal subgroup $\Tau/\Mu$ of $N_A(\Mu)/\Mu$ 
onto the normal subgroup $\Kappa$ of $\mathrm{Aff}(\Mu)$. 
We have that 
$$(N_A(\Mu)/\Mu)/(\Tau/\Mu)= N_A(\Mu)/\Tau = \mathrm{GL}(2,\integers).$$
Hence $\eta: \mathrm{Aff}(\Mu) \to \mathrm{GL}(2,\integers)$ is an epimorphism with kernel $\Kappa$. 
The map $\sigma:  \mathrm{GL}(2,\integers) \to \mathrm{Aff}(\Mu)$ 
is a well defined homomorphism by Lemma 11. 
The map $\sigma$ is a monomorphism, since $\sigma$ is a right inverse of $\eta$.  
\end{proof}

\begin{lemma} 
Let 
$$A = \left(\begin{array}{rr} 0 & -1 \\ 1 & 0 \end{array}\right),\quad  B = \left(\begin{array}{rr} 0 & -1 \\ 1 & 1 \end{array}\right), 
\quad C = \left(\begin{array}{rr} 0 & 1 \\ 1 & 0 \end{array}\right). $$
The group $\mathrm{GL}(2,\integers)$ is the free product of the dihedral subgroup $\langle A,  C\rangle$ of order 8 
and the dihedral subgroup $\langle B,  C \rangle$ of order 12 amalgamated 
along the dihedral subgroup $\langle -I, C \rangle$ of order 4. 
Every finite subgroup of  $\mathrm{GL}(2,\integers)$ is conjugate to a subgroup 
of either $\langle A,  C\rangle$ or $\langle B,  C \rangle$. 
\end{lemma}
\begin{proof}
We have that $A^2 = -I, C^2 = I$, and $CAC^{-1} = A^{-1}$. 
Hence $\langle A, C\rangle$ is a dihedral group of order 8. 
We have that $B^3 = -I$ and $C B C^{-1} = B^{-1}$. 
Hence $\langle B, C\rangle$ is a dihedral group of order 12. 
It follows from the group presentation for $\mathrm{GL}(2,\integers)$ 
given in \S 7.2 of \cite{C-M} that $\mathrm{GL}(2,\integers)$ is the free 
product of the subgroups $\langle A, C\rangle$ and $\langle B, C\rangle$ 
amalgamated along the subgroup $\langle -I, C \rangle$. 
By Theorem 8 of \S 4.3 of \cite{Serre}, every finite subgroup of $\mathrm{GL}(2,\integers)$ 
is conjugate to a subgroup of either $\langle A, C\rangle$ or $\langle B, C\rangle$. 
\end{proof}

\begin{lemma} 
The group $\mathrm{GL}(2,\integers)$ has seven conjugacy classes of elements of finite order,  
the class of $I$, three conjugacy classes of elements of order 2, 
represented by $-I, CA, C$, one conjugacy class of elements of order 3 represented by $B^2$, 
one conjugacy class of elements of order 4 represented by $A$, 
and one conjugacy class of element of order 6 represented by $B$. 
\end{lemma}
\begin{proof} By Lemma 50 an element of $\mathrm{GL}(2,\integers)$ of finite order 
is conjugate to either $I, -I, CA, C, CB, B^2, A,$ or $B$.  
As $-I$ commutes with every element of $\mathrm{GL}(2,\integers)$, 
we have that $-I$ is conjugate to only itself. 
The element $CA$ is not conjugate to $C$ by a simple proof by contradiction. 
The element $C$ is conjugate to $CB$ in $\mathrm{GL}(2,\integers)$. 
\end{proof}

\begin{lemma} 
Let $\Mu = \langle e_1+I, e_2+I\rangle$, let $\Kappa = \{(a+I)_\star: a \in E^2\}$, and  
let $A, B, C$ be defined as in Lemma 50. 
The group $\mathrm{Aff}(\Mu)$  is the free product of the subgroup $\langle \Kappa, A_\star, C_\star\rangle$ 
and the subgroup $\langle \Kappa, B_\star, C_\star\rangle$ amalgamated 
along the subgroup $\langle \Kappa,  (-I)_\star, C_\star\rangle$.  
Every finite subgroup of $\mathrm{Aff}(\Mu)$ is conjugate to a subgroup of either
$\langle \Kappa, A_\star, C_\star\rangle$ or $\langle \Kappa, B_\star, C_\star\rangle$. 
\end{lemma}
\begin{proof} This follows easily from Lemmas 49 and 50. 
\end{proof}

Let $\Lambda = \langle e_1+I, e_1/2 + \sqrt{3}e_2/2+I\rangle$. 
Then $\Lambda$ is a 2-space group, and the flat orbifold $E^2/\Lambda$ is a {\it hexagonal torus} $\varhexagon$. 

\begin{lemma} 
Let $\Mu = \langle e_1+I, e_2+I\rangle$, and let $\Lambda = \langle e_1+I, e_1/2 + \sqrt{3}e_2/2+I\rangle$. 
Let $\Kappa = \{(a+I)_\star: a \in E^2\}$, and  
let 
$$A = \left(\begin{array}{rr} 0 & -1 \\ 1 & 0 \end{array}\right),\ \  B = \left(\begin{array}{rr} 0 & -1 \\ 1 & 1 \end{array}\right), 
\ \ C = \left(\begin{array}{rr} 0 & 1 \\ 1 & 0 \end{array}\right), \ \ 
D  = \left(\begin{array}{rr} 1 & 1/2 \\ 0 & \sqrt{3}/2 \end{array}\right). $$
Then $\mathrm{Sym}(\Mu) = \langle \Kappa, A_\star, C_\star\rangle$,  and $D\Mu D^{-1}=\Lambda$, 
and $D_\sharp: \mathrm{Aff}(\Mu) \to \mathrm{Aff}(\Lambda)$ is an isomorphism, 
and $D_\sharp^{-1}(\mathrm{Sym}(\Lambda)) = \langle \Kappa, B_\star, C_\star\rangle$. 
\end{lemma}
\begin{proof}
Now $\mathrm{Out}_E(\Mu) = \mathrm{Aut}_E(\Mu)$ is a finite subgroup of 
$\mathrm{Aut}(\Mu) = \mathrm{GL}(2,\integers)$ that contains $\langle A, C\rangle$. 
Hence $\mathrm{Aut}_E(\Mu)= \langle A, C\rangle$, since $\langle A, C\rangle$ 
is a maximal finite subgroup of $\mathrm{GL}(2,\integers)$ by Lemma 50. 
The group $\Kappa$ is the connected component of $\mathrm{Sym}(\Mu)$ that contains $I_\star$. 
Hence $\mathrm{Sym}(\Mu) = \langle \Kappa, A_\star, C_\star\rangle$ by Theorem 12. 

Clearly $D\Mu D^{-1}=\Lambda$, and so 
$D_\#: \mathrm{Aut}(\Mu) \to \mathrm{Aut}(\Lambda)$ and 
$D_\sharp: \mathrm{Aff}(\Mu) \to \mathrm{Aff}(\Lambda)$ are isomorphisms. 
Now $D_\#(\langle B, C\rangle)$ is a subgroup of $\mathrm{Aut}_E(\Lambda)$,  
since $DBD^{-1}$, $DCD^{-1} \in \mathrm{O}(2)$.  
Hence $D_\#^{-1}(\mathrm{Aut}_E(\Lambda))$ is a finite subgroup of $\mathrm{Aff}(\Mu)$ 
that contains $\langle B, C\rangle$. 
Therefore $D_\#^{-1}(\mathrm{Aut}_E(\Lambda)) = \langle B, C\rangle$,
since $\langle B, C\rangle$ is a maximal finite subgroup of $\mathrm{GL}(2,\integers)$ by Lemma 50. 
Hence $D_\sharp^{-1}(\mathrm{Sym}(\Lambda)) = \langle \Kappa, B_\star, C_\star\rangle$ 
by Theorem 12 and Lemma 14. 
\end{proof}

The square torus $\Box = E^2/\Mu$ is formed by identifying the opposite sides 
of the square fundamental domain for $\Mu$, with vertices $(\pm 1/2, \pm 1/2)$, by translations. 
By Lemma 53, the group $\mathrm{Isom}(\Box) = \mathrm{Sym}(\Mu)$ is the semidirect product of the translation subgroup 
$\Kappa = \{(a+I)_\star: a \in E^2\}$ 
and the dihedral subgroup $\langle A_\star, C_\star\rangle$ of order 8 
induced by the symmetry group $\langle A, C\rangle$ of the square fundamental domain of $\Mu$. 
The elements of $\Kappa$ of order at most 2 form the dihedral group of order 4, 
$$\Kappa_2 = \{I_\star, (e_1/2+I)_\star, (e_2/2+I)_\star, (e_1/2+e_2/2+I)_\star\}. $$

By Lemma 54 below, 
the group $\mathrm{Isom}(\Box)$ has six conjugacy classes of elements of order 2, 
the class of the {\it horizontal halfturn}  h-rot.\ = $(e_1/2+I)_\star$
(or {\it vertical halfturn} v-rot. = $(e_2/2+I)_\star$), 
the class of the {\it antipodal map} 2-sym. = $(e_1/2+e_2/2+I)_\star$, 
the class of the {\it  halfturn} 2-rot.\ = $(-I)_\star$, 
the class of the {\it  horizontal reflection} h-ref.\ = $(AC)_\star$
(or  {\it vertical reflection} v-ref.\ = $(CA)_\star$), 
the class of the {\it horizontal glide-reflection}  h-grf.\ = $(e_1/2+CA)_\star$, 
(or  {\it vertical glide-reflection} v-grf.\  = $(e_2/2+AC)_\star$), and 
the class of the {\it diagonal reflection} d-ref.\ = $C_\star$
(or perpendicular diagonal reflection e-ref.\ = $(-C)_\star$). 

In the following discussion and Table 18,  
an apostrophe ``$\,\hbox{'}\,$" on a symmetry means that the symmetry is multiplied on the left by h-rot., 
a prime symbol ``$\,'\,$" on a symmetry means that the symmetry is multiplied on the left by v-rot., and  
a double prime symbol ``$\,''\,"$ means that the symmetry is multiplied on the left by 2-sym. 
These alterations do not change the conjugacy class of the symmetry. 

By Lemma 55 below, 
the group $\mathrm{Isom}(\Box)$ has 12 conjugacy classes of dihedral subgroups of order 4, 
the classes of the groups $\Kappa_2 = \{$idt., h-rot., v-rot., 2-sym.$\}$, 
$\{$idt, h-rot., 2-rot., 2-rot.'$\}$, 
$\{$idt, 2-sym., 2-rot., 2-rot.$''$$\}$, 
$\{$idt., h-rot., v-ref., h-grf.$\}$, 
$\{$idt, v-rot, v-ref, v-ref.$'$$\}$, 
$\{$idt, 2-sym., d-ref, d-ref.$''$$\}$, 
$\{$idt, v-rot, h-grf., h-grf.$'$$\}$, 
$\{$idt, 2-sym, v-ref, h-grf.$'$$\}$, 
$\{$idt., 2-rot. h-ref., v-ref.$\}$, 
$\{$idt., 2-rot., h-ref.', h-grf.$\}$, 
$\{$idt., 2-rot., h-grf.$'$, v-grf.'$\}$,  and 
$\{$idt., 2-rot., d-ref., e-ref.$\}$. 

The isometries 4-rot.\ = $A_\star$, and 4-sym.\ = $(e_2/2 + C)_\star$ 
are symmetries of $\Box$ of order 4, 
and the groups $\langle$h-ref., d-ref.$\rangle$, $\langle$d-ref., h-grf.$\rangle$, 
and $\langle$v-rot., d-ref.$\rangle$ are dihedral subgroups of $\mathrm{Isom}(\Box)$ of order 8. 
Moreover 4-rot.\ = (h-ref.)(d-ref.) and 4-rot.$'$ = (d-ref.)(h-grf.) and 4-sym.\ = (v-rot.)(d-ref.). 


The hexagonal torus $\varhexagon = E^2/\Lambda$ is formed by identifying the opposite sides of the regular hexagon 
fundamental domain, with vertices $(\pm 1/2,\pm \sqrt{3}/6), (0, \pm \sqrt{3}/3)$, by translations. 
By Lemma 53, the group $\mathrm{Isom}(\varhexagon)$ is the semidirect product 
of the subgroup $\Kappa = \{(a+I)_\star: a \in E^2\}$ 
and the dihedral subgroup $\langle (DBD^{-1})_\star, (DCD^{-1})_\star\rangle$ of order 12 
induced by the symmetry group $\langle DBD^{-1}, DCD^{-1}\rangle$ of the regular hexagon fundamental domain of $\Lambda$. 

We denote the symmetry of $\varhexagon$ represented by the $180^\circ, 120^\circ, 60^\circ$ rotation 
about the center of the regular hexagon fundamental domain of $\Lambda$ 
by 2-rot.\ = $(-I)_\star$, 3-rot.\ = $(DB^2D^{-1})_\star$, 6-rot.\ = $(DBD^{-1})_\star$, respectively. 
Define affinities of $\Box$  
by 3-aff.\ $= (B^2)_\star$ and 6-aff.\ $= B_\star$. 
The affinities 3-aff.\ and 6-aff.\ of $\Box$ are conjugate by $D_\star$ 
to the symmetries 3-rot.\ and 6-rot.\ of $\varhexagon$, respectively.

Define the reflection l-ref.\ of $\varhexagon$ by  l-ref.\ = $(DCD^{-1})_\star$. 
Then l.ref.\ is represented by the reflection in the line segment joining 
the two opposite vertices $\pm (1/2,\sqrt{3}/6)$ of the regular hexagon.  
Define the reflection m-ref.\ of $\varhexagon$ by m-ref.\ = $(DCBD^{-1})_\star$. 
Now $DCBD^{-1} = CA$, and so m-ref.\ is represented by the reflection in the line segment joining 
the midpoints $(\pm 1/2, 0)$ of the two vertical sides of the regular hexagon. 
Define the reflections n-ref.\  and o-ref.\ of $\varhexagon$ by n-ref.\ = $(DCB^2D^{-1})_\star$ and 
o-ref.\ = $(DCB^3D^{-1})_\star = (D(-C)D^{-1})_\star$. 
The groups $\langle$l-ref., n-ref.$\rangle$ and $\langle$m-ref., o-ref.$\rangle$ are dihedral groups of order 6   
with 3-rot.\ = (l-ref.)(n-ref.) = (m-ref.)(o-ref.). 
The group $\langle$l-ref., m-ref.$\rangle$ is a dihedral group of order 12 with 6-rot.\ = (l-ref.)(m-ref.). 
The symmetries 2-rot., d-ref., e-ref.\ of $\Box$ are conjugate by $D_\star$ to the symmetries 2-rot., l-ref., o-ref.\  of $\varhexagon$, respectively. 

Define affinities of $\Box$ 
by m-aff.\ = $(CB)_\star$ and n-aff.\ = $(CB^2)_\star$. 
The affinities m-aff.\ and n-aff.\ of $\Box$ are conjugate by $D_\star$ 
to the symmetries m-ref.\ and n-ref.\ of $\varhexagon$, respectively. 
The groups $\langle$d-ref., n-aff.$\rangle$, $\langle$m-aff., e-ref.$\rangle$, $\langle$d-ref., m-aff.$\rangle$ 
are conjugate by $D_\star$ to the groups $\langle$l-ref., n-ref.$\rangle$, $\langle$m-ref., o-ref.$\rangle$, 
$\langle$l-ref., m-ref.$\rangle$, respectively. 

\begin{lemma} 
Let $\Mu = \langle e_1+I, e_2+I\rangle$, and let $\kappa = (k+K)_\star$ be an element of $\mathrm{Aff}(\Mu)$ 
such that $K$ has finite order.  Then $K$ is conjugate to exactly one of  $I, -I, CA, C, B^2, A$ or $B$. 
\begin{enumerate}
\item If $K = I$ and $\kappa$ has order 2, then $\kappa \in \Kappa_2$ and $\kappa$ is conjugate to {\rm h-rot.\ =} $(e_1/2+I)_\star$. 
\item If $K =  -I$, then $\kappa$ is conjugate to {\rm 2-rot.\ =} $(-I)_\star$. 
\item If $K$ is conjugate to $CA$ and $\kappa$ has order 2, 
then $\kappa$ is conjugate to exactly one of {\rm  v-ref.\ =} $(CA)_\star$ or 
{\rm  h-grf.\ =} $(e_1/2 +CA)_\star$.
\item If $K$ is conjugate to $C$ and $\kappa$ has order 2 , then $\kappa$ is conjugate to {\rm  d-ref.\ =} $C_\star$.
\item If $K$ is conjugate to $B^2$, then $\kappa$ is conjugate to {\rm  3-aff.\ =} $(B^2)_\star$. 
\item If $K$ is conjugate to $A$, then $\kappa$ is conjugate to {\rm  4-rot.\ =} $A_\star$. 
\item If $K$ is conjugate to $B$, then $\kappa$ is conjugate to {\rm  6-aff.\ =} $B_\star$. 
\end{enumerate}
\end{lemma}
\begin{proof} The matrix $K$ is conjugate to exactly one of  $I, -I, CA, C, B^2, A$ or $B$ by Lemma 51. 

(1)  As $\kappa^2 = 2k+I$, we have that $2k \in \integers^2$.  Hence $\kappa \in \Kappa_2$. 
We have that $C(e_1/2+I)C^{-1} = e_2/2+I$ and $B(e_2/2+I)B^{-1} = -e_1/2+e_2/2$. 
Hence $\kappa$ is conjugate to $(e_1/2+I)_\star$. 

(2) We have that $(k/2+I)(-I)(-k/2+I) = k-I$. 

(3) We have that $(k+CA)^2 = 2k_1e_1+I$, and so $2k_1\in \integers$.  
Hence either $\kappa = (k_2e_2+CA)_\star$ or $\kappa = (e_1/2+k_2e_2+CA)_\star$. 
We have that $(k_2e_2/2+I)CA(-k_2e_2/2+I) = k_2e_2+CA$ and 
$(k_2e_2/2+I)(e_1/2+CA)(-k_2e_2/2+I) = e_1/2+k_2e_2+CA$. 
The isometries v-ref.\ and h-grf.\ are not conjugate in $\mathrm{Aff}(\Mu)$, 
since v-ref.\ fixes points of $\Box$ whereas h-grf.\ does not. 

(4) We have that $(k + C)^2 = (k_1+k_2)e_1+(k_1+k_2)e_2+I$. 
Hence $k_1+k_2 \in \integers$, and so $\kappa = (k_1e_1-k_1e_2 +C)_\star$. 
We have that $(k_1e_1/2 -k_1e_2/2+I)C(-k_1e_1/2+k_1e_2/2 +I) = k_1e_1-k_1e_2 +C$. 

(5)  We have that $(a+I)B^2(-a+I) = a -B^2a +B^2$.  The equation $a- B^2a = k$ has 
a solution, since $\mathrm{ker}(I-B^2) = \mathrm{Fix}(B^2) = \{0\}$. 

(6) The proofs of (6) and (7) are similar to the proof of (5). 
\end{proof}

\begin{lemma}  
Let $\Mu = \langle e_1+I,e_2+I\rangle$. 
Let $\kappa = (k+K)_\star$ and $\lambda = (\ell +L)_\star$ be distinct elements of order 2 of $\mathrm{Aff}(\Mu)$ 
such that $\Omega(\kappa\lambda)$ has finite order in $\mathrm{Out}(\Mu)$. 
Let $\Eta = \langle K, L \rangle$.  
Then $K$ and $L$ have order 1 or 2 and either $\Eta$ is a cyclic group of order 1 or 2 or $\Eta$ is a dihedral group 
of order 4, 6, 8 or 12. 
\begin{enumerate}
\item If $\Eta = \{I\}$, then $\langle \kappa, \lambda\rangle = \Kappa_2$ 
and $\{\kappa, \lambda\}$ is conjugate to $\{${\rm h-rot.,  v-rot.}$\}$.
\item If $\Eta$ has order 2, then $\Eta$ is conjugate to exactly one of $\langle -I\rangle$, $\langle CA\rangle$
 or $\langle C\rangle$. 
\begin{enumerate}
\item If $K = I$ and $L = -I$, then $\{\kappa, \lambda\}$ is conjugate to $\{${\rm h-rot.,  2-rot.}$\}$. 
\item If $K = L =-I$, then $\{\kappa, \lambda\}$ is conjugate to 
$\{${\rm 2-rot.}, $(v+I)_\star${\rm 2-rot.}$\}$ with $Kv = -v$.  
\item If $K = I$ and $L$ is conjugate to $CA$, then $\{\kappa, \lambda\}$ is conjugate to exactly one of 
$\{${\rm h-rot.,  v-ref.}$\}$, $\{${\rm v-rot.,  v-ref.}$\}$, $\{${\rm 2-sym.,  v-ref.}$\}$, $\{${\rm h-rot.,  h-grf.}$\}$, 
$\{${\rm v-rot.,  h-grf.}$\}$ or $\{${\rm 2-sym.,  h-grf.}$\}$. 
\item If $K = L$ and $L$ is conjugate to $CA$, then $\{\kappa, \lambda\}$ is conjugate to 
exactly one of $\{${\rm v-ref.}, $(v+I)_\star${\rm v-ref.}$\}$, $\{${\rm v-ref.}, $(v+I)_\star${\rm h-grf.}$\}$, 
or $\{${\rm h-grf.}, $(v+I)_\star${\rm h-grf.}$\}$ with $Kv = -v$.  
\item If $K = I$ and $L$ is conjugate to $C$, then $\{\kappa, \lambda\}$ is conjugate to exactly one of 
$\{${\rm 2-sym.,  d-ref.}$\}$ or $\{${\rm v-rot.,  d-ref.}$\}$. 
\item If $K = L$ and $L$ is conjugate to $C$, then $\{\kappa, \lambda\}$ is conjugate to 
$\{${\rm d-ref.}, $(v+I)_\star${\rm d-ref.}$\}$ with $Kv = -v$. 
\end{enumerate}
\item If $\Eta$ has order 4,  then  $\Eta$ is conjugate to exactly one of  $\langle -I, CA \rangle$ or $\langle -I, C\rangle$. 
\begin{enumerate}
\item If $K = -I$ and $L$ is conjugate to $CA$, then $\{\kappa, \lambda\}$ is conjugate to exactly one of  
$\{${\rm 2-rot.}, $(v+I)_\star${\rm v-ref.}$\}$ or $\{${\rm 2-rot.}, $(v+I)_\star${\rm h-grf.}$\}$ with $L(v) = -v$. 
\item If $K = -L$ and $K$ is conjugate to $CA$, then $\{\kappa, \lambda\}$ is conjugate to exactly one of 
$\{${\rm v-ref., h-ref.}$\}$,  $\{${\rm h-ref.'}, {\rm h-grf.}$\}$ or $\{${\rm h-grf.}$'$, {\rm v-grf.'}$\}$. 
\item If $K = -I$ and $L$ is conjugate to $C$, then $\{\kappa, \lambda\}$ is conjugate to 
$\{${\rm 2-rot.}, $(v+I)_\star${\rm d-ref.}$\}$ with $L(v) = -v$. 
\item If $K = -L$ and $K$ is conjugate to $C$, then $\{\kappa, \lambda\}$ is conjugate to $\{${\rm d-ref., e-ref.}$\}$. 
\end{enumerate}
\item If $\Eta$ has order 8,  then $\Eta$ is conjugate to $\langle A, C\rangle$ and 
$\{\kappa, \lambda\}$ is conjugate to exactly one of $\{${\rm d-ref.,  h-ref.}$\}$ or $\{${\rm d-ref.,  h-grf.}$\}$. 
\item If $\Eta$ has order 6,  then $\Eta$ is conjugate to exactly one of $\langle B^2, C\rangle$ or $\langle B^2, CB\rangle$. 
\begin{enumerate}
\item If $\Eta$ is conjugate to $\langle B^2, C\rangle$, then $\{\kappa, \lambda\}$ is conjugate to $\{${\rm d-ref.,  n-aff.}$\}$.
\item If $\Eta$ is conjugate to $\langle B^2, CB\rangle$, then $\{\kappa, \lambda\}$ is conjugate to $\{${\rm m-aff., e-ref.}$\}$. 
\end{enumerate}
\item If $\Eta$ has order 12, then $\Eta$ is conjugate to $\langle B, C\rangle$ and 
$\{\kappa, \lambda\}$ is conjugate to $\{${\rm d-ref., m-aff.}$\}$.
\end{enumerate}
\end{lemma}
\begin{proof}
As $\kappa^2 = (k +Kk + K^2)_\star$, we have that $k + Kk \in \integers^2$ and $K^2 = I$.  
Hence $K$ has order 1 or 2. Likewise $L$ has order 1 or 2. 
Now $KL$ has finite order by Theorem 13 and Lemma 49, 
moreover $KL$ has order 1, 2, 3, 4 or 6 by Lemma 51. 
Hence either $\Eta$ is cyclic of order 1 or 2 or $\Eta$ is dihedral of order 4, 6, 8 or 12.

(1) We have that $\langle \kappa,\lambda\rangle = \Kappa_2$ by Lemma 54(1),  
moreover $B_\star\{$h-rot., v-rot.$\}B^{-1}_\star = \{$v-rot., 2-sym.$\}$ and 
$B_\star\{$2-sym., h-rot.$\}B_\star^{-1} = \{$h-rot., v-rot.$\}$. 

(2)  If $\Eta$ has order 2,  then $\Eta$ is conjugate to exactly one of $\langle -I\rangle$, $\langle CA\rangle$
 or $\langle C\rangle$ by Lemma 51; moreover if $\Eta$ is conjugate to $\langle -I\rangle$, then $\Eta = \langle -I\rangle$. 
 
(a) By conjugating $\{\kappa, \lambda\}$, we may assume that $\lambda$ = 2-rot.\ by Lemma 54(2). 
Then $\{\kappa, \lambda\}$ is conjugate to $\{${\rm h-rot.,  2-rot.}$\}$ as in the proof of Lemma 54(1). 

(b) By Lemma 54(2), we may assume that $\kappa = K_\star = $ 2-rot. 
Then $L = (\ell + K)_\star = (\ell+I)_\star$2-rot.\ with $K\ell = -\ell$. 

(c) By conjugating $\{\kappa, \lambda\}$, we may assume that $\lambda$ = v-ref.\ or h-grf.\ by Lemma 54(3). 
Then $\kappa \in \Kappa_2$ by Lemma 54(1), and so there are six possibilities.  
These six pairs are nonconjugate since v-ref.\ is not conjugate to h-grf., the centralizer of $CA$ in $\mathrm{GL}(2,\integers)$ 
is $\langle -I, CA\rangle$, and $\langle -I, CA\rangle$ centralizes $\Kappa_2$. 

(d) By Lemma 54(3),  we may assume that $K = CA$ and either 
$\kappa = K_\star =$ v-ref.\ or $\kappa = (e_1/2+K)_\star = $ h-grf. 
Now either $\lambda = (\ell_2e_2 +K)_\star$ or $\lambda =  (e_1/2 + \ell_2e_2 + K)_\star$. 
Hence either $\lambda = (\ell_2e_2+I)_\star$v-ref. or $\lambda = (\ell_2e_2+I)_\star$h-grf.\ 
with $K(\ell_2e_2) = - \ell_2e_2$. 
Moreover $(-v/2+I)_\star\{$v-ref., $(v+I)_\star$h-grf.$\}(v/2+I)_\star = \{(-v+I)_\star$v-ref., h-grf.$\}$. 

(e) By conjugating $\{\kappa, \lambda\}$, we may assume that $\lambda$ = d-ref.\ by Lemma 54(4). 
Then $\kappa \in \Kappa_2$ by Lemma 54(1), and so there are three possibilities.  
We have that $C(e_1/2+C)C^{-1} = e_2/2 + C$.  
The pair $\{${\rm v-rot.,  d-ref.}$\}$  is not conjugate to $\{${\rm 2-sym,  d-ref.}$\}$, 
since (v-rot.)(d-ref.) has order 4 while (2-sym.)(d-ref.) has order 2. 

(f) By Lemma 54(4),  we may assume that $K = C$ and $\kappa = K_\star =$ d-ref. 
Then $\lambda = (\ell + K)_\star = (\ell + I)_\star$d-ref.\ with $K\ell = -\ell$. 

(3)  If $\Eta$ has order 4, then $\Eta$ is conjugate to exactly one of $\langle -I, CA\rangle$ or $\langle -I, C\rangle$ 
by Lemmas 50 and 51.

(a) By conjugating $\{\kappa, \lambda\}$, we may assume that $L = CA$ and $\kappa$ = 2-rot. 
By the proof of Lemma 54(3), either $\lambda = (\ell_2e_2+CA)_\star$ or $(e_1/2+\ell_2e_2+CA)_\star$. 
If $\lambda = (\ell_2e_2+CA)_\star$, then $\lambda = (\ell_2e_2+I)_\star$v-ref.\ with $CA(\ell_2e_2) = -\ell_2e_2$. 
If $\lambda = (e_1/2+\ell_2e_2+CA)_\star$, then $\lambda = (\ell_2e_2+I)_\star$h-grf.\ with $CA(\ell_2e_2) = -\ell_2e_2$.

(b) By conjugating $\{\kappa, \lambda\}$, we may assume that $K = CA$ and $\kappa\lambda$ = 2-rot. 
By the proof of Lemma 54(3), we have that $\kappa = (k_2e_2 +CA)_\star$ or $(e_1/2+k_2e_2 +CA)_\star$ 
and $\lambda = (\ell_1e_1-CA)_\star$ or $(\ell_1e_2 +e_2/2-CA)_\star$.  As $\kappa\lambda$ = 2-rot., 
we have that $\{\kappa, \lambda\}$ is either $\{$v-ref., h-ref.$\}$,  $\{$v-ref.$'$, v-grf.$\}$, 
$\{$h-grf., h-ref.'$\}$, or $\{$h-grf.$'$, v-grf.'$\}$.  Moreover $C_\star\{$v-ref.$'$, v-grf.$\}C_\star^{-1} = \{$h-ref.', h-grf.$\}$. 

(c) By conjugating $\{\kappa, \lambda\}$, we may assume that $L = C$ and $\kappa$ = 2-rot. 
By the proof of Lemma 54(4), we have that $\lambda = (\ell_1e_1-\ell_1e_2+C)_\star = (\ell_1e_1-\ell_1e_2+I)_\star$d-ref.\ 
with $C(\ell_1e_1-\ell_1e_2) = -\ell_1e_1+\ell_1e_2$. 

(d) By conjugating $\{\kappa, \lambda\}$, we may assume that $K = C$ and $\kappa\lambda$ = 2-rot. 
By the proof of Lemma 54(4), we have that $\kappa = (k_1e_1 - k_1e_2 +C)_\star$.  
Now $(\ell - C)^2 = (\ell_1-\ell_2)e_1 + (\ell_2-\ell_1)e_2 + I$.  
Hence $\ell_1-\ell_2 \in \integers$. Hence $\lambda = (\ell_1e_1 + \ell_1e_2-C)_\star$. 
Now $(k_1e_1 - k_1e_2 +C)(\ell_1e_1 + \ell_1e_2-C) = (k_1+\ell_1)e_1 + (\ell_1-k_1)e_1 - I$. 
Hence $k_1+\ell_1, \ell_1-k_1 \in \integers$.  Thus we may take $\ell_1 = k_1$ and $k_1 = 0$ or $1/2$. 
Then $\{\kappa, \lambda\}$ is either $\{$d-ref., e-ref.$\}$ or $\{$d-ref.$''$, e-ref.$''\}$.  
Moreover $(e_1/2+I)C(-e_1/2+I) = e_1/2-e_2/2+C$ and $(e_1/2+I)(-C)(-e_1/2+I) = e_1/2+e_2/2-C$.  

(4)  If $\Eta$ has order 8,  then $\Eta$ is conjugate to $\langle A, C\rangle$ by Lemma 50. 
By conjugating $\{\kappa, \lambda\}$, we may assume that $K = C$ and $L = CA$ and $\kappa\lambda$ = 4-rot. 
By Lemma 54(4), we have that $\kappa = (k_1e_1-k_1e_2+C)_\star$, and by Lemma 54(3), 
we have that $\lambda = (\ell_2e_2+CA)_\star$ or $(e_1/2+\ell_2e_2+CA)_\star$. 
Now $(k_1e_1 - k_1e_2 +C)(\ell_2e_2+CA) = (k_1+\ell_2)e_1 - k_1e_2 + A$, 
and we may take $k_1 = \ell_2 = 0$, 
and $(k_1e_1 - k_1e_2 +C)(e_1/2+\ell_2e_2+CA) = (k_1+\ell_2)e_1 +(1/2- k_1)e_2 + A$, 
and we may take $k_1 = \ell_2 = 1/2$. 
Then  $\{\kappa, \lambda\}$ is either $\{$d-ref., h-ref.$\}$ or $\{$d-ref.$''$, h-grf.$'\}$. 
Moreover $(e_1/4-e_2/4+I)C(-e_1/4+e_2/4+I) = e_1/2-e_2/2+C$ and 
$(e_1/4-e_2/4+I)(e_1/2+CA)(-e_1/4+e_2/4+I) = e_1/2 -e_2/2+CA$. 

(5) If $\Eta$ has order 6,  then $\Eta$ is conjugate to exactly one of 
$\langle B^2, C\rangle$ or $\langle B^2, CB\rangle$ by Lemma 50. 

(a) By conjugating $\{\kappa, \lambda\}$, we may assume that $K = C$ and $L = CB^2$ and $\kappa\lambda$ = 3-aff. 
By Lemma 54(4), we have that $\kappa = (k_1e_1-k_1e_2+C)_\star$. 
Now $(\ell +CB^2)^2 = 2\ell_1e_1-\ell_1e_2+I$.  
Hence $\lambda = (\ell_2e_2+CB^2)_\star$. 
Now $(k_1e_1-k_1e_2 +C)(\ell_2e_2 +CB^2) = (k_1+\ell_2)e_1-k_1e_2 +B^2$. 
Hence $\kappa = C_\star$ = d-ref.\ and $\lambda = (CB^2)_\star$ = n-aff.

(b) By conjugating $\{\kappa, \lambda\}$, we may assume that $K = CB$ and $L = CB^3 = -C$ and $\kappa\lambda$ = 3-aff. 
Now $(k+CB)^2 = (2k_1 + k_2)e_1+I$. 
Hence $\kappa = (k_1e_1-2k_1e_2 + CB)_\star$. 
Now $\lambda = (\ell_1e_1+\ell_1e_2-C)_\star$ and 
$(k_1e_1-2k_1e_2 + CB)(\ell_1e_1+\ell_1e_2-C) = (k_1+2\ell_1)e_1 +(-2k_1-\ell_1)e_2+B^2$. 
Hence we may take $\ell_1 = -2k_1$.  Then $3k_1 \in \integers$, 
and we may take $(k_1,\ell_1) = (0,0), (1/3,1/3)$ or $(2/3,2/3)$. 
Now $(-e_1/3+2e_2/3+I)CB(e_1/3-2e_2/3+I) = -2e_1/3+4e_2/3+CB$ and 
$(-e_1/3+2e_2/3+I)(-C)(e_1/3-2e_2/3+I) = e_1/3+e_2/3-C$. 
Moreover  $(e_1/3+e_2/3+I)CB(-e_1/3-e_2/3+I) = -e_1/3+2e_2/3+CB$ and 
$(e_1/3+e_2/3+I)(-C)(-e_1/3-e_2/3+I) = 2e_1/3+2e_2/3-C$. 

(6) If $\Eta$ has order 12,  then $\Eta$ is conjugate to $\langle B, C\rangle$ by Lemma 50. 
By conjugating $\{\kappa, \lambda\}$, we may assume that $K = C$ and $L = CB$ and $\kappa\lambda$ = 6-aff. 
By Lemma 54(4), we have that $\kappa = (k_1e_1-k_1e_2+C)_\star$. 
By the proof of 5(b), we have that $\lambda = (\ell_1e_2-2\ell_1e_2 + CB)_\star$. 
Now $(k_1e_1-k_1e_2+C) (\ell_1e_1-2\ell_1e_2 + CB) = (k_1-2\ell_1)e_1+(\ell_1-k_1)e_2+B$. 
Hence $\kappa = C_\star$ = d-ref.\ and $\lambda = (CB)_\star$ = m-aff. 
\end{proof}

\begin{theorem} 
Let $\Mu = \langle e_1+I, e_2+I\rangle$. 
Then $\mathrm{Iso}(C_\infty, \Mu)$ has seven elements corresponding to the pairs of elements 
{\rm \{idt., idt.\}, \{2-rot., 2-rot.\}, \{h-ref., h-ref.\}, \{d-ref., d-ref.\}, \{3-aff., 3-aff.$^{-1}$\}, 
\{4-rot., 4-rot.$^{-1}$\}, \{6-aff., 6-aff.$^{-1}$\}} of $\mathrm{Aff}(\Mu)$ by Theorem 23 and Lemmas 48 and 51,   
and $\mathrm{Iso}(D_\infty, \Mu)$ has 34 elements corresponding to the pairs of elements 
{\rm \{idt., idt.\}, \{idt., h-rot\}, \{idt., 2-rot.\}, \{idt., v-ref.\}, \{idt., h-grf.\}, \{idt., d-ref.\}, 
\{h-rot., h-rot\}, \{2-rot., 2-rot.\}, \{v-ref., v-ref.\}, \{h-grf., h-grf.\}, \{d-ref., d-ref.\}, 
\{h-rot, v-rot.\}, \{h-rot., 2-rot.\}, \{h-rot., v-ref.\}, \{v-rot., v-ref.\}, \{2-sym., v-ref.\}, \{h-rot., h-grf.\}, \{v-rot., h-grf.\}, 
\{2-sym., h-grf.$'$\}, \{v-rot., d-ref.\}, \{2-sym., d-ref.\},\{2-rot., v-ref.\}, \{2-rot., h-grf.\}, \{2-rot., d-ref.\},  
 \{v-ref., h-ref.\}, \{v-ref., h-grf.\}, \{h-ref.', h-grf.\}, \{h-ref., d-ref.\}, \{h-grf.$'$, v-grf.'\}, \{h-grf., d-ref.\}, 
\{d-ref., e-ref.\}, \{d-ref., n-aff.\}, \{m-aff., e-ref.\}, \{d-ref., m-aff.\}}  
of $\mathrm{Aff}(\Mu)$ by Theorem 26 and Lemmas 54 and  55. 
The corresponding co-Seifert fibrations are described in Table 18 by Theorems 27 and 28.  
Only the eight pairs {\rm \{2-rot., 2-rot.\}, \{v-ref., v-ref.\}, \{h-grf., h-grf.\},  \{v-ref., h-grf.\},\{d-ref., d-ref.\}, 
\{2-rot., v-ref.\}, \{2-rot., h-grf.\}, \{2-rot., d-ref.\}} fall into the case $E_1\cap E_2 \neq \{0\}$ of Theorem 26. 
\end{theorem}

\begin{table}  
\begin{tabular}{rlllll}
no. & fibers & grp. & quotients &  structure group action & classifying pair \\
\hline 
    1 & $(\circ, \mathrm{O})$ & $C_1$ & $(\circ, \mathrm{O})$ & (idt., idt.)  & \{idt., idt.\}  \\
    2 & $(\circ, \mathrm{O})$ & $C_2$ & $(2222, \mathrm{I})$ & (2-rot., ref.) & \{2-rot., 2-rot.\} \\
    3 & $(\circ, \mathrm{O})$ & $C_2$ & $(\ast\ast, \mathrm{I})$ & (v-ref., ref.) & \{v-ref., v-ref.\} \\
    4 & $(\circ, \mathrm{O})$& $C_2$ & $(\times\times, \mathrm{I})$ & (h-grf., ref.) & \{h-grf., h-grf.\} \\
    4 & $(\circ, \mathrm{O})$ & $ C_2$ & $(2222, \mathrm{O})$ & (2-rot., 2-rot.) & \{2-rot., 2-rot.\} \\
    5 & $(\circ, \mathrm{O})$ & $C_2$ & $(\ast\times, \mathrm{I})$ & (e-ref., ref.) & \{e-ref., e-ref.\} \\
    5 & $(\circ, \mathrm{O})$ & $D_2$ & $(\ast\ast, \mathrm{I})$ & (v-ref., ref.), (h-rot., 2-rot.) & \{v-ref., h-grf.\} \\
    6 & $(\circ, \mathrm{I})$  & $C_1$ & $(\circ, \mathrm{I})$ & (idt., idt.)  & \{idt., idt.\}  \\
    7 & $(\circ, \mathrm{O})$ & $C_2$ & $(\circ, \mathrm{I})$ & (v-rot., ref.) & \{v-rot., v-rot\} \\
    7 & $(\circ, \mathrm{O})$ & $C_2$ & $(\ast\ast, \mathrm{O})$ & (v-ref., 2-rot.) & \{v-ref., v-ref.\}\\
    8 & $(\circ, \mathrm{I})$ & $C_2$ & $(\circ, \mathrm{I})$ & (h-rot., ref.) & \{idt., h-rot\} \\
    9 & $(\circ, \mathrm{O})$ & $D_2$ & $(\circ, \mathrm{I})$ & (v-rot., ref.), (h-rot., 2-rot.) &  \{v-rot, 2-sym.\} \\
    9 & $(\circ, \mathrm{O})$ & $C_2$ & $(\ast\times, \mathrm{O})$ & (d-ref., 2-rot.)  & \{d-ref., d-ref.\} \\
  11 & $(\circ, \mathrm{I})$ & $C_2$ & $(2222, \mathrm{I})$ &(2-rot., ref.)  & \{idt., 2-rot.\} \\
  13 & $(\circ, \mathrm{O})$ & $D_2$ & $(\ast 2222, \mathrm{I})$ & (h-ref., ref.),  (v-ref., 2-rot.) & \{h-ref., 2-rot.\} \\
  14 & $(\circ, \mathrm{O})$ & $D_2$ & $(2222, \mathrm{I})$ & (v-rot., ref.),  (2-rot.$'$, 2-rot.) & \{v-rot., 2-rot.\} \\
  14 & $(\circ, \mathrm{O})$ & $D_2$ & $(22\ast, \mathrm{I})$ & (v-grf., ref.), (v-ref.$'$, 2-rot.) & \{v-grf., 2-rot.\} \\
  15 & $(\circ, \mathrm{O})$ & $D_2$ & $(2{\ast}22, \mathrm{I})$ & (e-ref., ref.), (d-ref., 2-rot.) & \{e-ref., 2-rot.\} \\
  17 & $(\circ, \mathrm{O})$ & $D_2$ & $(\ast 2222, \mathrm{I})$ & (v-ref., ref.), (2-rot., 2-rot.) &  \{v-ref., h-ref.\} \\
  18 & $(\circ, \mathrm{O})$ & $D_2$ & $(22\ast, \mathrm{I})$ & (h-grf., ref.), (2-rot.', 2-rot.)  & \{h-grf., h-ref.\} \\
  19 & $(\circ, \mathrm{O})$ & $D_2$ & $(22\times, \mathrm{I})$ & (v-grf., ref.), (2-rot.$''$, 2-rot.) &  \{v-grf., h-grf.\} \\
  20 & $(\circ, \mathrm{O})$ & $D_2$ & $(2{\ast}22, \mathrm{I})$ & (d-ref., ref.), (2-rot., 2-rot.)  & \{d-ref., e-ref.\} \\
  28 & $(\circ, \mathrm{I})$ & $C_2$ & $({\ast}{\ast}, \mathrm{I})$ & (h-ref., ref.) &  \{idt., h-ref.\} \\
  29 & $(\circ, \mathrm{O})$ & $D_2$ & $({\ast}{\ast}, \mathrm{I})$ & (v-rot., ref.), (h-ref., 2-rot.) & \{v-rot., v-grf.\} \\
  30 & $(\circ, \mathrm{O})$ & $D_2$ & $({\ast}{\ast}, \mathrm{I})$ & (v-rot., ref.), (v-grf., 2-rot.) & \{v-rot., h-ref.\} \\
  31 & $(\circ, \mathrm{I})$ & $C_2$ & $({\times}{\times}, \mathrm{I})$ & (v-grf., ref.)  & \{idt., v-grf.\} \\
  32 & $(\circ, \mathrm{O})$ & $D_2$ & $({\ast}{\ast}, \mathrm{I})$ & (h-rot., ref.), (h-ref.', 2-rot.) & \{h-rot., h-ref.\} \\
  33 & $(\circ, \mathrm{O})$ & $D_2$ & $({\times}{\times}, \mathrm{I})$ & (h-rot., ref.), (v-grf.', 2-rot.) &  \{h-rot., v-grf.\} \\
  33 & $(\circ, \mathrm{O})$ & $D_2$ & $({\ast}{\times}, \mathrm{I})$ & (2-sym., ref.), (h-ref., 2-rot.) & \{2-sym., v-grf.'\} \\
  34 & $(\circ, \mathrm{O})$ & $D_2$ & $({\ast}{\times}, \mathrm{I})$ & (2-sym., ref.), (v-grf., 2-rot.) & \{2-sym., h-ref.'\} \\
  40 & $(\circ, \mathrm{I})$ & $C_2$ & $({\ast}{\times}, \mathrm{I})$ & (e-ref., ref.)  &  \{idt., e-ref.\} \\
  41 & $(\circ, \mathrm{O})$ & $D_2$ & $({\ast}{\ast}, \mathrm{I})$ & (e-ref.$''$, ref.), (e-ref., 2-rot.)  & \{e-ref.$''$, 2-sym.\} \\
  43 & $(\circ, \mathrm{O})$ & $D_4$ & $({\ast}{\times}, \mathrm{I})$ & (v-rot., ref.), (4-sym., 4-rot.) &  \{v-rot., d-ref.\} \\
  76 & $(\circ, \mathrm{O})$ & $C_4$ & $(442, \mathrm{O})$ & (4-rot., 4-rot.)  & \{4-rot., 4-rot.$^{-1}$\} \\
  91 & $(\circ, \mathrm{O})$ & $D_4$ & $({\ast}442, \mathrm{I})$ & (h-ref., ref.), (4-rot., 4-rot.) & \{h-ref., d-ref.\} \\
  92 & $(\circ, \mathrm{O})$ & $D_4$ & $(4{\ast}2, \mathrm{I})$ & (v-grf., ref.), (4-rot.$'$, 4-rot.) & \{v-grf., d-ref.\}  \\
144 & $(\circ, \mathrm{O})$ & $C_3$ & $(333, \mathrm{O})$ & (3-rot., 3-rot.)  &  \{3-aff., 3-aff.$^{-1}$\} \\
151 & $(\circ, \mathrm{O})$ & $D_3$ & $({\ast}333, \mathrm{I})$ & (l-ref., ref.), (3-rot., 3-rot.)  & \{d-ref., n-aff.\} \\
152 & $(\circ, \mathrm{O})$ & $D_3$ & $(3{\ast}3, \mathrm{I})$ & (m-ref., ref.), (3-rot., 3-rot.) & \{m-aff., e-ref.\} \\
169 & $(\circ, \mathrm{O})$ & $C_6$ & $(632, \mathrm{O})$ & (6-rot., 6-rot.)  & \{6-aff., 6-aff.$^{-1}$\} \\
178 & $(\circ, \mathrm{O})$ & $D_6$ & $({\ast}632, \mathrm{I})$ & (l-ref., ref.), (6-rot., 6-rot.) & \{d-ref., m-aff.\} \\
\end{tabular}

\medskip
\caption{The classification of the co-Seifert fibrations of 3-space groups 
whose co-Seifert fiber is of type $\circ$ with IT number 1}
\end{table}

\section{Applications}

As our first application, we will discuss how our tables 
give nice, explicit, geometric descriptions of all the fibered, compact, connected, flat  3-orbifolds. 
For a preliminary discussion of our description, see \S 11 of \cite{R-T}. 

Let $\Nu$ be a 2-dimensional, complete, normal subgroup of a 3-space group $\Gamma$,  
and let $\Kappa$ be the kernel of the action of $\Gamma$ on $V = \mathrm{Span}(\Nu)$, 
and suppose $\Kappa$ is the orthogonal dual of $\Nu$.  
First assume that the base $V^\perp/(\Gamma/\Nu)$ of the co-Seifert fibration of $E^3/\Gamma$ determined by $\Nu$ 
is a circle, that is, the second quotient in Column 4 of Tables 2-18 is $\mathrm{O}$. 
Then $E^3/\Gamma$ is the quotient space of $V/\Nu\times V^\perp/\Kappa$ under the diagonal action 
of a cyclic structure group $\Gamma/\Nu\Kappa$ of order $n$ given in Column 5 of Tables 2-18.  
The structure group acts on the circle $V^\perp/\Kappa$ by a rotation of $(360/n)^\circ$. 
This implies that the orbifold $E^3/\Gamma$ is the mapping torus of $V/\Nu$ under the action 
of the generator of the structure group,   
in other words,  if the action is given by the pair ($\alpha$, $n$-rot.) in Column 5 of Tables 2-18,  
then $E^3/\Gamma$ is equivalent to the quotient 
$$(V/\Nu\times \mathrm{I})/\{(x,0) \sim (\alpha(x),1)\}.$$ 

Now assume that the base $V^\perp/(\Gamma/\Nu)$ of the co-Seifert fibration of $E^3/\Gamma$ determined by $\Nu$ 
is a closed interval, that is, the second quotient in Column 4 of Tables 2-18 is $\mathrm{I} = [0,1]$. 
If the structure group $\Gamma/\Nu\Kappa$ is trivial, then $E^3/\Gamma$ is the Cartesian product $V/\Nu\times V^\perp/\Kappa$ 
with $V^\perp/\Kappa$ a closed interval. 

If the structure group has order 2 and $V^\perp/\Kappa$ is a closed interval, that is, the second fiber in Column 2 
of Tables 2-18 is $\mathrm{I}$, 
then $E^3/\Gamma$ is the twisted I-bundle obtained from $V/\Nu\times V^\perp/\Kappa$ under the action of the structure group,  
that is,  if the action is given by the pair ($\alpha$, ref.) in Column 5 of Tables 2-18,  
then $E^3/\Gamma$ is equivalent to the quotient 
$$(V/\Nu\times \mathrm{I})/\{(x,t) \sim (\alpha(x),1-t)\}.$$

If the structure group is dihedral of order $2n$, and $V^\perp/\Kappa$ is a circle,  
and the action is given by ($\alpha$, ref.),  ($\gamma$, $n$-rot.), with $\gamma = \alpha\beta$, 
then $E^3/\Gamma$ is equivalent to the quotient 
$$(V/\Nu\times \mathrm{I})/\{(x,0) \sim (\beta(x),0)\ \&\ (x,1) \sim (\alpha(x), 1)\}.$$ 

Our second application is to give an explanation of the enantiomorphic 3-space group pairs 
from the point of view of our paper.   
An enantiomorphic 3-space group pair consists of a pair $(\Gamma_1, \Gamma_2)$ of isomorphic 
3-space groups all of whose elements are orientation preserving such that there is no orientation preserving 
affinity of $E^3$ that conjugates one to the other. 
There are 11 enantiomorphic 3-space group pairs up to isomorphism, 
and 10 of them have 2-dimensional, complete, normal subgroups, 
moreover these  2-dimensional, complete, normal subgroups are unique.  

The first enantiomorphic pair $(\Gamma_1, \Gamma_2)$ has IT numbers 76 and 78. 
The geometric co-Seifert fibration of $E^3/\Gamma_1$ is described in Table 18. 
The action of the cyclic structure group of order 4 on $\circ \times \mathrm{O}$ is given by $($4-rot., 4-rot.$)$. 
The action of the structure group for $E^3/\Gamma_2$ 
is given by $($4-rot.$^{-1}$, 4-rot.), 
since the classifying pair is $\{$4-rot., 4-rot.$^{-1}\}$.  
The isometry 4-rot.\ of the torus $\circ$ is defined by 4-rot.\ = $A_\star$. 
By Theorems 19 and 22 and Lemma 48, the pair $(\Gamma_1, \Gamma_2)$ is enantiomorphic 
because there is no element of $\mathrm{SL}(2,\integers)$ that conjugates  $A$ to $A^{-1}$ and 
there is no element of  $\mathrm{GL}(2,\integers)$ of determinant $-1$ that conjugates $A$ to $A$. 
The enantiomorphic IT number pairs (144, 145), (169, 170), (171, 172) have a similar description. 

The second enantiomorphic pair $(\Gamma_1, \Gamma_2)$ has IT numbers 91 and 95. 
The co-Seifert fibration of $E^3/\Gamma_1$ is described in Table 18. 
The action of the dihedral structure group of order 8 on $\circ \times \mathrm{O}$ is given by 
$($h-ref., ref.$)$, $($4-rot., 4-rot.$)$. 
The action of the structure group for $E^3/\Gamma_2$ is given by 
$($d-ref., ref.$)$, $($4-rot.$^{-1}$, 4-rot.$)$, 
since the classifying pair is  $\{$h-ref., d-ref.$\}$.
By Theorems 19 and 26 and Lemma 55(4), 
the pair $(\Gamma_1, \Gamma_2)$ is enantiomorphic because there is no orientation preserving affinity of 
the torus $\circ$ that conjugates (h-ref., d-ref.) to (d-ref., h-ref.) and there is no orientation reversing affinity of $\circ$ 
that conjugates (h-ref., d-ref.) to (h-ref., d-ref.). 
The enantiomorphic IT number pairs (92, 96), (151, 153), (152, 154), (178, 179), (180, 181)  have a similar description.

The third application is to splitting space group extensions.  
Let $\Nu$ be a 2-dimensional, complete, normal subgroup of a 3-space group $\Gamma$,  
and let $\Kappa$ be the kernel of the action of $\Gamma$ on $V = \mathrm{Span}(\Nu)$, 
and suppose $\Kappa$ is the orthogonal dual of $\Nu$.  
The Tables 2-18 give an explicit description of the action of the structure group $\Gamma/\Nu\Kappa$
on the generic fiber $V/\Nu$ of  the geometric co-Seifert fibration of the flat 3-orbifold $E^3/\Gamma$ determined by $\Nu$. 
This explicit description allows one to deduce that the space group extension 
$1\to \Nu \to \Gamma \to \Gamma/\Nu \to 1$ splits orthogonally when $\Gamma/\Nu\Kappa$ fixes 
an ordinary point of $V/\Nu$ by Theorem 8. 
Table 1 of \cite{R-T} lists whether or not the extension $1\to \Nu \to \Gamma \to \Gamma/\Nu \to 1$ splits. 
If $V/\Nu$ is a torus, then the extension splits orthogonally whenever it splits by  the first part of Theorem 8. 
If $V/\Nu$ is a pillow, then for IT number 15 in Table 17, 
the extension splits but not orthogonally by the second part of Theorem 8.


\begin{thebibliography}{99}

\bibitem{A} J. F. Adams, {\it Lectures on Lie Groups}, W. A. Benjamin, New York, 1969. 


\bibitem{B-Z} H. Brown, R. B\"ulow, J. Neub\"user, H. Wondratschek, H. Zassenhaus, 
{\it Crystallographic groups of four-dimensional space}, 
John Wiley \& Sons, New York, 1978. 


\bibitem{Conway} J. H. Conway, The orbifold notation for surface groups, 
In: {\it Groups, Combinatorics and Geometry}, London Math. Soc. Lec. Notes Ser. {\bf 165}, 
Cambridge Univ. Press (1992), 438-447. 

\bibitem{C-T} J. H. Conway, O. D. Friedrichs, D. H. Huson, W. P. Thurston, 
On three-dimensional space groups, 
{\it Beitr\"age Algebra Geom.} {\bf 42} (2001), 475-507. 

\bibitem{C-M} H. S. M. Coxeter and W. O. J. Moser, {\it Generators and Relations for Discrete Groups} 
Fourth Ed., Springer-Verlag, Berlin, 1980. 

\bibitem{G} M. Gubler,  Normalizer groups and automorphism groups of symmetry groups, 
Z. Krist. 158 (1982), 1-26.


\bibitem{IT} T. Hahn, {\it International tables for crystallography, Vol. A, 
Space-group symmetry, Second Edition}, Ed. by T. Hahn.,    
D. Reidel Publishing Co., Dordrecht, 1987. 


\bibitem{R} J.G. Ratcliffe, {\it Foundations of Hyperbolic Manifolds, Second Edition}, 
Graduate Texts in Math., vol. {\bf 149}, Springer-Verlag, Berlin, Heidelberg, and
New York, 2006. 

\bibitem{R-T-Ab} J. G. Ratcliffe and S. T. Tschantz, 
Abelianization of space groups, 
{\it Acta Cryst. } {\bf A65} (2009), 18-27.


\bibitem{R-T} J. G. Ratcliffe and S. T. Tschantz, 
Fibered orbifolds and crystallographic groups, 
{\it Algebr. Geom. Topol. } {\bf 10} (2010), 1627-1664. 

\bibitem{Serre} J-P. Serre, {\it Trees}, Springer-Verlag, Berlin, 1980. 

\bibitem{S} N. Steenrod, {\it The Topology of Fibre Bundles}, 
Princeton Univ. Press, Princeton, New Jersey, 1951. 

\bibitem{Wolf} J. A. Wolf, {\it Spaces of Constant Curvature, Fifth Edition}, 
Publish or Perish, Inc., Wilmington, Delaware, 1984. 


\end{thebibliography}
\end{document}